\documentclass[12pt]{amsart}
\usepackage[margin=1.25in]{geometry}
\usepackage{amsmath,amsthm,amscd,amsfonts,wasysym,amssymb,epic,eepic,bbm,tikz-cd,mathtools,array}
\usepackage{dynkin-diagrams}
\usepackage[pagebackref,colorlinks=true,linkcolor=blue,citecolor=blue]{hyperref}
\usepackage[noabbrev]{cleveref}
\usepackage{MnSymbol}
\usepackage{marginnote}
\usepackage{dynkin-diagrams}
\usepackage{amssymb}
\usepackage{float}
\usepackage{amscd}
\usepackage{amsfonts}
\usepackage[utf8]{inputenc}
 \usepackage{longtable}
 \usepackage{stmaryrd}

\usepackage{mathtools}

\usepackage[all,arc]{xy}
\usepackage{enumerate}
\usepackage{mathrsfs}
\usepackage{color}   
\usepackage{hyperref}
\hypersetup{
    colorlinks=true, 
    linktoc=all,     
    linkcolor=blue,  
}
\usepackage{amsthm}
\usepackage{amsmath}
\usepackage{graphicx}
\usepackage{wasysym}
\usepackage{pgf,tikz}
\usetikzlibrary{positioning}
\usepackage{ marvosym }
\usepackage{tikz-cd}
\usepackage{ textcomp }
\usepackage{bbm}
\usepackage{ stmaryrd }

\newcommand{\cc}{\mathbb C}
\newcommand{\ff}{\mathbb F}

\newcommand{\zz}{\mathbb Z}

\newcommand{\qq}{\mathbb Q}
\newcommand{\rr}{\mathbb R}
\newcommand{\A}{\mathbb A}

\newcommand{\la}{\langle}
\newcommand{\ra}{\rangle}
\newcommand{\lra}{\longrightarrow}
\newcommand{\hra}{\hookrightarrow}

\newcommand{\al}{\alpha}
\newcommand{\be}{\beta}
\newcommand{\ga}{\gamma}
\newcommand{\de}{\delta}
\newcommand{\De}{\Delta}
\newcommand{\Ga}{\Gamma}
\newcommand{\ep}{\epsilon}
\newcommand{\lam}{\lambda}
\newcommand{\Lam}{\Lambda}

\newcommand{\sig}{\sigma}
\newcommand{\ka}{\kappa}

\DeclareMathOperator{\Gal}{Gal}
\DeclareMathOperator{\Aut}{Aut}

\DeclareMathOperator{\ad}{ad}

\DeclareMathOperator{\out}{out}

\DeclareMathOperator{\ind}{ind}
\DeclareMathOperator{\GL}{GL}

\DeclareMathOperator{\Hom}{Hom}
\DeclareMathOperator{\Out}{Out}
\DeclareMathOperator{\Ind}{Ind}
\DeclareMathOperator{\im}{im}
\DeclareMathOperator{\PGL}{PGL}
\DeclareMathOperator{\SU}{SU}
\DeclareMathOperator{\rS}{S}
\DeclareMathOperator{\U}{U}
\DeclareMathOperator{\B}{B}
\DeclareMathOperator{\M}{M}
\DeclareMathOperator{\T}{T}
\DeclareMathOperator{\rL}{L}
\DeclareMathOperator{\K}{K}

\DeclareMathOperator{\SO}{SO}

\DeclareMathOperator{\SL}{SL}
\DeclareMathOperator{\Sp}{Sp}
\DeclareMathOperator{\Spin}{Spin}

\DeclareMathOperator{\Lie}{Lie}
\DeclareMathOperator{\sspan}{span}
\DeclareMathOperator{\diag}{diag}

\DeclareMathOperator{\Spec}{Spec}
\DeclareMathOperator{\Stab}{Stab}

\DeclareMathOperator{\coker}{coker}
\DeclareMathOperator{\rk}{rk}
\DeclareMathOperator{\Ad}{Ad}
\DeclareMathOperator{\I}{I}
\DeclareMathOperator{\Res}{Res}
\DeclareMathOperator{\Orb}{O}

\DeclareMathOperator{\inv}{inv}

\DeclareMathOperator{\TF}{TF}

\newcommand{\fu}{\mathfrak u}
\newcommand{\fg}{\mathfrak g}

\newcommand{\fc}{\mathfrak c}
\newcommand{\fe}{\mathfrak e}
\newcommand{\fa}{\mathfrak a}

\newcommand{\fl}{\mathfrak l}

\newcommand{\fgl}{\mathfrak{gl}}
\newcommand{\fX}{\mathfrak X}
\newcommand{\fN}{\mathfrak N}
\newcommand{\fC}{\mathfrak C}
\newcommand{\fsl}{\mathfrak{sl}}

\newcommand{\modmod}{\backslash\backslash}

\newcommand{\calh}{\mathcal{H}}

\newcommand{\calg}{\mathcal{G}}
\newcommand{\calo}{\mathcal{O}}
\newcommand{\calx}{\mathcal{X}}
\newcommand{{\rP}}{\mathrm{P}}
\newcommand{{\rM}}{\mathrm{M}}
\newcommand{{\X}}{\mathrm{X}}

\newcommand{{\cale}}{\mathcal{E}}
\newcommand{\cala}{\mathcal{A}}

\newcommand{\cald}{\mathcal{D}}

\newcommand{\caln}{\mathcal N}
\newcommand{\calm}{\mathcal M}

\newcommand{\heart}{\heartsuit}
\newcommand{\Nm}{\mathrm{Nm}}

\newcommand{\Gm}{\mathbb{G}_m}

\newcommand{\iso}{\overset{\sim}{\longrightarrow}}
\newcommand{\car}{\textbf{car}}
\newcommand{\kbar}{\overline{k}}

\newcommand{\rH}{\mathrm{H}}

\newcommand{\Ax}{\mathrm{A}_{\mathrm{X}}}
\newcommand{\rA}{\mathrm{A}}

\newcommand{\D}{{\mathcal{D}}}
\newcommand{\G}{{\mathrm{G}}}
\newcommand{\Y}{{\mathrm{Y}}}

\makeatletter
\def\Ddots{\mathinner{\mkern1mu\raise\p@
\vbox{\kern7\p@\hbox{.}}\mkern2mu
\raise4\p@\hbox{.}\mkern2mu\raise7\p@\hbox{.}\mkern1mu}}
\makeatother

\newenvironment{psmatrix}
  {\left(\begin{smallmatrix}}
  {\end{smallmatrix}\right)}

\newtheorem{Thm}{Theorem}[section]
\newtheorem{Prop}[Thm]{Proposition}
\newtheorem{LemDef}[Thm]{Lemma/Definition}
\newtheorem{Lem}[Thm]{Lemma}
\newtheorem{Cor}[Thm]{Corollary}

\newtheorem{Conj}[Thm]{Conjecture}
\newtheorem{Quest}[Thm]{Question}
\newtheorem{Assumption}[Thm]{Assumption}
\newtheorem{Axiom}[Thm]{Axiom}

\theoremstyle{definition}
\newtheorem{Def}[Thm]{Definition}

\theoremstyle{remark}
\newtheorem{Rem}[Thm]{Remark}
\newtheorem{Ex}[Thm]{Example}

\theoremstyle{definition}

\newcommand{\quash}[1]{}
\numberwithin{equation}{section}

\title[Endoscopy for symmetric varieties]{Symmetric varieties for endoscopic groups}
\author{Spencer Leslie}

\date\today

\address{Department of Mathematics, Boston College, Chestnut Hill, MA, USA}

\email{spencer.leslie@bc.edu}


\begin{document}

\begin{abstract}We introduce a notion of endoscopic varieties for a symmetric variety $\X$, and establish the foundational properties of these varieties such as matching of stable semi-simple orbits. To do this, we introduce certain automorphism groups of homogeneous spherical varieties, which encode the fine rational structure needed to work over non-algebraically closed fields. In particular, we establish the existence and uniqueness of the corresponding symmetric varieties under a mild restriction of the characteristic of the field of definition. We conjecture that this construction plays a role analogous to endoscopic groups in the context of the relative trace formula. As evidence, we show how our construction gives a pre-stabilization of regular elliptic terms of relative trace formulae for many symmetric varieties. 

When the cotangent bundle of the symmetric variety is hyperspherical, we relate our theory to the Hamiltonian variety of the Langlands dual group introduced by Ben-Zvi, Sakellaridis, and Venkatesh, proving some structural conjectures for this variety in the symmetric setting. 
\end{abstract}

\maketitle
\setcounter{tocdepth}{1}
\tableofcontents

\section{Introduction}

The primary goal of this paper is to develop a theory of endoscopy in the relative Langlands program of Jacquet and Sakellaridis--Venkatesh. That is, given a connected reductive group $\G$ over a field $k$ and a smooth symmetric variety $\X$, we develop a method of associating symmetric varieties $\X_\fe$ for endoscopic groups $\G_\fe$ in terms of the dual group $\check{\G}_\X$ which is suitable for stabilizing the relative trace formula associated to many pairs $(\G,\X)$. As discussed below, this requires proving the existence and uniqueness of the appropriate symmetric varieties of endoscopic groups, for which we introduce a new invariant (the geometric cocyle) to homogeneous spherical $k$-varieties to differentiate between possible $k$-forms of the variety $\X$ with different arithmetic properties. We then  establish the foundational properties required of a theory of endoscopy, including the matching of geometric semi-simple orbits necessary for comparison of relative trace formulae. In sufficiently nice settings, we place these results in the context of the relative duality conjectures of Ben-Zvi, Sakellaridis, and Venkatesh \cite{BZSV}.

In Section \ref{Section: stabilize}, we show how our notion of endoscopic datum of the symmetric variety gives a pre-stabilization of the regular elliptic part of the relative trace formula for symmetric varieties. This relies on several cohomological preliminaries established in the companion paper \cite{LesliestabFJ}, where we also apply these results to formulate the local harmonic-analytic conjectures (smooth transfer, fundamental lemma, etc.) generalizing notions introduced in \cite{Leslieendoscopy,LeslieUFJFL,Lesliedescent} and stabilize the elliptic part of a large class of relative trace formulae subject to these conjectures. Nevertheless, the main step in the prestabilization is demonstrated in Proposition \ref{Prop: important}, in which all of our existence and uniqueness results become necessary. In a sense, the notion of endoscopic datum of a symmetric variety $\X=\rH\backslash \G$ is formulated to prove Proposition \ref{Prop: important}. 

\subsection{Motivation for a relative theory of endoscopy} Suppose that $\G$ is a connected reductive group over a field $k$, which we assume for the moment is global. In this setting, one is interested in the study of period integrals of automorphic forms and their relations to functoriality and arithmetic quantities such as $L$-functions.  While certainly not the most general setting, this typically involves a closed subgroup $\rH\subset \G$ and one seeks to analyze the period integral
\[
\varphi\in L^2(Z_\rH(\G)(\A)\G(k)\backslash\G(\A))\longmapsto \displaystyle \int_{Z_\rH(\G)(\A)\rH(k)\backslash\rH(\A)}^\ast \varphi(h)dh,
\]
where $Z_\rH(\G)=\rH\cap Z(\G)$ and the asterisk indicates the common need to regularize the integral.

In this context, one of the central tools is the relative trace formula (RTF). Introduced by Jacquet in the 1980's \cite{JacquetWalds}, this gives a distribution designed to relate period questions on different groups via comparisons: by matching up the geometric terms (orbital integrals) of two RTFs associated to $(\G_1,\rH_1)$ and $(\G_2,\rH_2)$, we can extract relations between the respective spectral terms ($L$-values and period integrals). A famous recent example of this method is the proof of the Gan--Gross--Prasad and Ichino--Ikeda conjectures for unitary groups, now a theorem by the combination of many works \cite{YunJR,ZhangRankin,beuzart2019isolation,beuzart2020global}. 

In general, many of the good properties present in the GGP setting (trivial generic stabilizers on the geometric side, local multiplicity one on the spectral side) fail for more general RTFs, and several new problems arise. As noted by Getz and Wambach in their twisted base change formalism \cite{GetzWambach}, many RTFs require a form of \emph{endoscopy and stabilization} analogous to the Arthur--Selberg trace formula in order to be effectively compared. Spectrally, this is connected to questions of whether periods of automorphic forms (or more properly, the squares of their norms) are Eulerian or are sums of Euler products. See \cite[Section 14]{BZSV} for a collection of conjectures of when such formulae are expected to hold; we expect the theory of endoscopy developed here will play a role in establishing cases of these conjectures, and discuss a specific example in Section \ref{Section: Example}.

When $k$ is a local field, one is interested in understanding the space of $\rH(k)$-invariant functionals $\Hom_{\rH(k)}(\pi,\cc)$ for an irreducible admissible representation $\pi$ of $\G(k)$. This is most naturally studied on the level of packets, as illustrated by the local Gan--Gross--Prasad conjectures and the conjectures of Sakellaridis and Venkatesh \cite{SakVenk}. Indeed, questions about the distribution of $\rH$-invariant functionals within packets, as well as general dimension formulae such as \cite{Wandimension,prasad2015relative} often have interpretations in terms of endoscopic transfer. To give an example, the multiplicity formulae for unitary Shalika models in \cite{WanBPfuture} (see also \cite{chen2021unitary}) are explicitly computed via a local trace formula technique combined with the endoscopic properties of $L$-packets for unitary groups.

\subsubsection{Prior work}
As a first case, we developed a putative theory of endoscopy for the case of unitary Friedberg--Jacquet (FJ) periods where $\G=\U_{2n}$ is a unitary group of even rank and $\rH=\U_n\times \U_n$ \cite{Leslieendoscopy, LeslieUFJFL,Lesliedescent}, establishing many of the necessary harmonic-analytic results such as the fundamental lemma and results toward smooth transfer. The stabilization of the elliptic part of this trace formula will appear in \cite{LesliestabFJ}. This case has important applications with arithmetic: in ongoing joint work with Jingwei Xiao and Wei Zhang, we relate these {periods to central values of base change $L$-functions} on associated general linear groups by comparing a relative trace formula to the \emph{stabilized relative trace formula for unitary FJ periods}. Unfortunately, the notion of endoscopic symmetric variety used in \cite{Leslieendoscopy} relies on certain degenerations to relate relative orbital integrals to more classical objects in the setting of endoscopy for Lie algebras. As a result, it does not clearly extend to more general period questions. 

The purpose of this paper is to develop a general framework of endoscopy for symmetric varieties which recovers the theory of \cite{Leslieendoscopy, LeslieUFJFL,Lesliedescent}, allows for the (pre-)stabilization of relative trace formulas for more general symmetric varieties, and is compatible with the conjectures of Sakellaridis--Venkatesh and the more recent relative duality conjectures of \cite{BZSV} (see Section \ref{Section: BZSV intro} below).

 \subsection{The dual group and semi-simple descent}\label{Sec: endoscopy intro} Suppose now that $\G$ is a connected reductive group over a field $k$ of either characteristic zero or large enough positive characteristic in a sense made precise later. Suppose also that $\X$ is a homogeneous spherical $\G$-variety with $x_0\in \X(k)\neq \emptyset$. Setting $\rH:=\mathrm{Stab}_{\G}(x_0)$, we obtain $\X\simeq\rH\backslash\G$. Assume also that $\G$ is quasi-split, and let $\rA\subset \B$ denote a $k$-rational maximal torus and a Borel subgroup. Sphericity of $\X$ is the statement that $\B$ acts on $\X$ with a Zariski-open orbit. As first observed by Brion \cite{BrionVers}, there exists a based root datum $\Psi_\X$ associated to $\X$ which controls the large scale geometry of $\X$.

Motivated by Langlands duality, one may associate a complex reductive group $\check{\G}_\X$ by passing to the dual root datum. In \cite{SakVenk}, Sakellaridis and Venkatesh refined the dual group to a hypothetical homomorphism
\begin{equation}\label{eqn: intro dist morphism}
    \varphi_{\X}: \check{\G}_\X\times \SL_2(\cc)\lra \check{\G}
\end{equation}
to formulate a Plancherel theorem for spherical varieties over p-adic fields. The existence and uniqueness (up to conjugacy) of this morphism was established by Knop and Schalke \cite{KnopSchalke}, who refer to such maps as distinguished  morphisms. As clarified in \cite{KnopFunctorial}, there is an essentially unique way to induce an action of the Galois group $\Ga$ of $k$ on $\check{\G}_\X$ so that $\varphi_\X$ intertwines with the $L$-action on $\check{\G}$.

Now assume that $\X$ is symmetric, in the sense that there is a $k$-rational involutive automorphism $\theta$ of $\G$ such that 
\[
(\G^\theta)^\circ\subset \rH\subset N_{\G}(\G^\theta);
\]
we always assume that $\rH$ is smooth (automatic if $\rH$ connected or if $\mathrm{char}(k)$ is zero). We now assume that {$\mathrm{char}(k)$ is good for $\G$} (cf. Section \ref{section: prelim groups}).

The notion of endoscopic group is closely tied to semi-simple descent on $\check{\G}$. Our starting point for the relative setting is the construction of $\varphi_\X$ by Knop--Schalke (which is compatible with ideas of Nadler \cite{NadlerReal}): the image of $\check{\G}^\ast_\X:=\varphi_{\X}(\check{\G}_\X)$ is (the connected component of the identity of) a symmetric subgroup of a Levi subgroup $\hat{\G}_\X\subset\check{\G}$ (Proposition \ref{Prop: dual involution}). We let
 \[
 \hat{{\X}}:=\check{\G}_{{\X}}^\ast\backslash\hat{\G}_\X.
 \]
 This is a complex symmetric variety equipped with a distinguished point $\hat{x}_0.$ The equivariance of $\varphi_{{\X}}$ implies that $\Ga$ acts on $\hat{{\X}}$. An important property of symmetric varieties is that they are \emph{quasi-Hamiltonian}: they come with a group-valued moment map 
 \[
 \hat{s}:\hat{{\X}}\lra \hat{\G}_\X\subset \check{\G},
 \]
which is $\Ga$-equivariant. We refer to it as the symmetrization map for $\hat{\X}$.
 
In Section \ref{Section: endoscopy defs}, we isolate a closed subvariety of semi-simple points $x\in \hat{\X}^{\heart}\subset\hat{{\X}}^{ss}$ and let $(\check{\G}_x,\hat{\G}_{\X,x},\check{\G}_{{\X},x})$ be the \textbf{descendant at $x;$} that is, $\check{\G}_x$ (resp., $\hat{\G}_{\X,x}$) is the centralizer of $\check{s}(x)$ in $\check{\G}$ (resp., $\hat{\G}_\X$) and $\check{\G}_{{\X},x}$ is the stabilizer of $x\in \hat{\X}$ in $\check{\G}_\X$. The constraint on $x\in \hat{\X}^{ss}$ is to ensure that the morphism from $\SL_2(\cc)\lra \check{\G}$ induced from \eqref{eqn: intro dist morphism} factors through $\check{\G}_x$. This restriction is motivated by harmonic-analytic considerations, but plays a role in our arguments (cf. Proposition \ref{Prop: endoscopic roots}).

Set $\hat{{\X}}_{x}:=\varphi_\X(\check{\G}_{{\X},x})\backslash\hat{\G}_{\X,x}$. The symmetrization map induces a closed immersion
 \[
{\eta}_x:\hat{{\X}}_{x}\lra \hat{{\X}},
 \]
 and a commutative diagram
 \begin{equation}\label{eqn: intro diagram}
 \begin{tikzcd}
\check{\G}_{{\X},x}\ar[d]\ar[r]&\hat{\G}_{x}\ar[d]\ar[r]&\hat{{\X}}_{x}\ar[d,"\eta_x"]\\
 \check{\G}_{{\X}}\ar[r,"\varphi_{{\X}}"]&\hat{\G}\ar[r]&\hat{{\X}}
 \end{tikzcd}
 \end{equation}
 Note that if $\eta:\check{\G}_x^\circ\subset \check{\G}$ is the inclusion, the triple $\fe:=(\check{\G}_x^\circ,\hat{s}(x),\eta)$ is  an endoscopic datum for $\G$. Let $\G_\fe$ denote the corresponding quasi-split reductive group over $k$ satisfying that $\check{\G}_{{\fe}}=\check{\G}_x^\circ$. 

 It is natural to ask if the top row of \eqref{eqn: intro diagram} corresponds to a $k$-rational spherical $\G_\fe$-variety $\X_\fe$ in the sense that $\check{\G}_{\X_\fe} = \check{\G}_{{\X},x}$ and the induced morphism
 \[
  \varphi_{\X}|_{\check{\G}_{{\X},x}}: \check{\G}_{\X,x}\times \SL_2(\cc)\lra \check{\G}_x
 \]
 is the analogue of \eqref{eqn: intro dist morphism} for $(\G_\fe,\X_\fe)$. This is essentially the content of Theorem \ref{Thm: exists} (which we state as Theorem \ref{Thm: exists intro} below), though more work is necessary to state this result. More precisely, we establish the existence and uniqueness (under mild assumptions) of the pair $(\G_\fe,\X_\fe)$ suitable for relative functoriality and for the stabilization of the RTF associated to $(\G,\X)$.

\begin{Rem}[Why symmetric varieties?]
    A natural question is why restrict to symmetric varieties among all spherical varieties. There are several components to this.
    \begin{enumerate}
        \item Semi-simple descent at \emph{all} semi-simple points is central to the theory of elliptic endoscopy, both classically and in the relative setting. For symmetric varieties, there is a well-understood theory for descent via the symmetrization map, so that the existence of the dual involution in Proposition \ref{Prop: dual involution} seems to play a crucial role. It is not clear how to extend this to the general spherical setting, though certainly the existence of a group-valued moment map is a reasonable desideratum. Certain \emph{ad hoc} variations of these techniques can be used to include affine spherical varieties such as $\G_2\backslash\Spin_7$ or $\SL_3\backslash\G_2$, but it is not clear to us if this is the correct approach in these settings.
      \item We make concrete use of the notion of $(\Ga,\theta)$-indices due to \cite{Helminckrational} as an enhancement of the spherical datum of a symmetric variety over a non-algebraically closed field in proving the existence of endoscopic symmetric varieties in Sections \ref{Section: endo invo} and \ref{Section: endo varieties}. It is possible that a similar argument may be handled purely in terms of homogeneous spherical data.
        \item Finally, there are issues of characteristic. While Section \ref{Appendix: spherical roots} discusses general spherical homogeneous spaces, it is essentially restricted to characteristic zero due to the reliance on \cite{Losev, BorovoiGagliardi}. It is likely that this restriction may be weakened to include fields of sufficiently large positive characteristic, but this is not clear. In the symmetric setting, the work of Helminck \cite{Helmincktwo, Helminckrational} and Levy \cite{Levy} allow us to handle positive characteristic settings with minor restrictions.
    \end{enumerate}
\end{Rem}
\begin{Rem}[What about hyperspherical varieties?]
     Generalizing in a different direction, one may ask about the setting of \emph{hyperspherical varieties} as studied in \cite{BZSV} as a good context for ``relative Langlands duality.'' As defined in Section 3.5 of \emph{loc. cit.}, a hyperspherical $\G$-variety $\M$ is a Hamiltonian $\G$-variety with a moment map $\mu:M\to \fg^\ast$ satisfying several assumptions including that $M$ is smooth, affine, and that the stabilizer in $\G$ of a generic point of $M$ is connected. When $M=T^\ast\X$ with $\X$ spherical, then Proposition 3.7.4 of \cite{BZSV} implies that if $M$ is hyperspherical, $\X$ is affine and ``has no spherical roots of type $N$;'' see Section \ref{Section: spherical datum}. More generally, $M$ is determined by the data
     \begin{itemize}
         \item an injective morphism $\rH\times \SL_2\to \G$ with $\rH$ reductive,
         \item a symplectic $\rH$-representation $S$.
     \end{itemize}
     
  Our results and methods readily extend to spherical varieties obtained by parabolic induction from a symmetric variety on a Levi subgroup as well as {Whittaker-type inductions from such varieties}. These Whittaker inductions include several examples of hyperspherical varieties which are not cotangent bundles of symmetric varieties, such as generalized Shalika models and their unitary variants. For the sake of space, we omit a detailed account of these cases, as the core existence and uniqueness results ultimately reduce to the symmetric case. We plan to return to this more general setting in future work.

  On the other hand, this article \emph{does not} impose the restriction that generic stabilizers on $T^\ast\X$ are connected: we allow $\X$ to be a general (smooth) symmetric variety for the majority of the paper, and only impose a restriction on roots of type $N$ when connecting to conjectures of \cite{BZSV}. Indeed, many very natural symmetric varieties have type $N$ roots, such as $\mathrm{O}_n\backslash\GL_n$ or $A_1\cdot A_5\backslash E_6$. While these cases are ruled out of \cite{SakVenk,BZSV}, there is good reason to expect that these theories extend in some fashion to include such case (see \cite{chen2023slices} for a geometric example  and \cite{PollackWanZydor} for an arithmetic one). We hope that working in this generality will fit naturally into future extensions of the framework of \cite{BZSV}.
\end{Rem}

 \subsection{Rationality questions} We first address the question of \emph{uniqueness} of the variety $\X_\fe$ on $\G_\fe$. The problem here is that the data of the dual group $\check{\G}_\X$ and the morphism $\varphi_\X$ in \eqref{eqn: intro dist morphism} -- even with appropriate Galois actions incorporated -- do not determine a unique spherical $k$-variety $\X$, even up to inner twist.

 \begin{Ex}
     An important motivating example is the unitary FJ variety $$\X=\U_n\times \U_n\backslash\U_{2n},$$ where we take all the unitary groups to be quasi-split over $k$. In this case one calculates that $\check{\G}_\X = \Sp_{2n}(\cc)$ and the morphism $\varphi_\X$ is trivial on the $\SL_2(\cc)$ factor and may be arranged to be the natural inclusion
     \[
     \varphi_\X:\Sp_{2n}(\cc)\lra \GL_{2n}(\cc);
     \]
     we may arrange so that the $\Ga$-action on $\GL_{2n}(\cc)$ preserves the standard pinning and fixes $\Sp_{2n}(\cc)$ point-wise.

     On the other hand, we may consider the symmetric $\U_{2n}$-variety $$\X'=\Res_{E/k}(\GL_n)\backslash\U_{2n},$$ where $E/k$ is the quadratic extension of $k$ over which $\G$ splits. One may then show that $\X_{E}\simeq\X_{E}'$, where the subscript denotes the base change to $\Spec(E)$, and that $\X'$ induces the same distinguished morphism  \eqref{eqn: intro dist morphism}. As explored in \cite{chen2021unitary}, these two varieties give rise to deeply related but distinct arithmetic problems in the local and global relative Langlands program.\qed
 \end{Ex}

\subsubsection{Outer forms of $\X$} 
 
Suppose now that $\G$ is quasi-split over $k$ and $\overline{\X}$ is a spherical variety of $\G_{\kbar}$. We will say that $\overline{\X}$ \emph{descends to $k$} if there exists a $\G$-variety $\X$ such that $\X_{\kbar}\simeq \overline{\X}$. When $\overline{\X}$ descends to $k$ the isomorphism classes of $\G$-forms of $\overline{\X}$ are a torsor for $H^1(k,\cala_\X)$, where 
\[
\cala_\X\simeq \rH\backslash N_{\G}(\rH)
\]
is the group of $\G$-equivariant automorphisms of $\X=\rH\backslash\G$. 

Recall that the Langlands $L$-group ${}^L\G$ only determines the inner class of $\G$. Thus, our goal is to enhance \eqref{eqn: intro dist morphism} so as to determine $\X$ up to ``$\G$-inner twist,'' in a sense to be clarified presently. For any such $\G$-variety $\X=\rH\backslash\G$ there is a canonical morphism 
\[
\out_{\rH}^\G:\cala_\X \lra \Out(\rH);
\]
we set $\Out_{\X}(\rH)$ to be the image of this map\footnote{More canonically, we can work with the limit of such groups over $x\in \X(k)$, which is denoted $\Out(\X)$}. We say that two $\G$-varieties $\X$ and $\X'$ are $\G$-inner twists if there is a $\G_{\kbar}$-equivariant isomorphism $\psi:\X_{\kbar}\iso \X'_{\kbar}$ such that the $1$-cocycle $\sig\in \Ga\mapsto \psi^{-1}{}^\sig\psi$ takes values in $$\cala_\X^\flat:=\ker[\out_{\rH}^\G].$$ It is straightforward to see that if $\X'=\rH'\backslash\G$ is another $k$-form of $\X$ and we consider the corresponding class $[X']\in H^1(k,\cala_\X)$, then $\rH'$ is an inner form of $\rH$ if and only if the image of $[\X']$ in $H^1(k,\Out_\X(\rH))$ is trivial. However, this result depends on the choice of $\X$ as a $k$-forms of $\X_{\kbar}$.

To remedy this, we introduce a quotient $\Out_{\X}(\rH)\to \Aut_d(\X)$ called the group of ``doubling automorphisms'' of $\X$. When $\rH$ is connected, this group is implicit in the arguments of \cite{Losev} and essentially encodes the actions of $\cala_\X$ on $\B$-orbits of $\X$; see Remark \ref{Rem: geometric} and Example \ref{Ex: Symplectic example}. We develop a different definition when $\rH$ is disconnected that enjoys better functorial properties in Section \ref{Section: doubling aut gen}. In Section \ref{Section: geometric cocycle}, we associate to any $k$-rational spherical $\G$-variety a $1$-cocycle $\mu_\X:\Ga\lra \Aut_d(\X)(\kbar)$, referred to as the \emph{geometric cocycle} for $\X$. Moreover, we show in Lemma \ref{Lem: surjective on dist cohom} and Corollary \ref{Lem: examples of conj} that $H^1(k,\cala_\X)\to H^1(k,\Aut_d(\X))$ is surjective most cases of interest (we conjecture this always holds in Conjecture \ref{Conj: cohom surj}).

We say that $\X$ is \emph{well adapted} if the map $\Out_{\X}(\rH)\to \Aut_d(\X)$ is an isomorphism. In this setting, two $k$-forms $\X$ and $\X'$ of $\overline{\X}$ are $\G$-inner twists of each other if and only if their geometric cocycles are cohomologous. 

It turns out that not all spherical varieties are well adapted. Nevertheless, we completely determine which symmetric varieties are well adapted in Theorem \ref{Thm: Outer in terms of dist}. A nice consequence is that if $\X$ has no roots of type $N$, then it is well-adapted. Along the way, we also show that when a symmetric variety $\X=\rH\backslash\G$ is well adapted, there exists a $\G$-inner twist $\X'$ possessing a point $x\in \X'(k)$ such that $\mathrm{Stab}_{\G}(x)$ is the \emph{quasi-split} inner form of $\rH$; see Theorem \ref{Thm: quasi-split G-inner form}. Proving these two statements involves reducing to the case where $\G$ is $k$-simple and simply connected and appealing to the classification of symmetric varieties over $\kbar$.  This argument depends on an extension of the notion of a $z$-extension to the case of symmetric varieties in Proposition \ref{Prop: good cover}; in fact, this construction  works more generally for any finite order automorphism of $\G$ and may be of independent interest.

\subsection{Endoscopic symmetric varieties} We may now state our main existence theorem. We refer the reader to Section \ref{Section: endo varieties} for any undefined terms. Recall that to each spherical $\G$-variety $\X$, there is a well-defined notion of ``canonical torus'' $\rA_\X$, which encodes the possible Borel semiinvariant rational functions on $\X$; see Section \ref{Section: spherical datum}.

\begin{Thm}\label{Thm: exists intro}
Assume that $\X=\rH\backslash\G$ is a symmetric space associated to a $k$-rational involution $\theta$ of $\G$. Assume also that $\X$ is well-adapted. Let\footnote{See Section \ref{Section: endo varieties} for the definition of $\hat{\X}^\heart$. This condition essentially requires that the $\SL_2(\cc)$-factor of \eqref{eqn: intro dist morphism} factors through $\G_x$. In particular, endoscopic varieties are as ``{non-tempered}'' as $\X$.} $x\in \X^{\heart,\Ga}$. Let $\G_\fe$ be the (quasi-split) endoscopic group associated to $\fe_x:=(\hat{\G}_x^\circ,x,\eta)$. There exists an $k$-rational involution $\theta_\fe:\G_\fe\lra \G_\fe$ and a $k$-rational subgroup $(\G_\fe^{\theta_\fe})^\circ\subset \rH_\fe\subset N_{\G_\fe}(\G^{\theta_\fe}_\fe)$ such that if $\X_\fe:=\rH_\fe\backslash\G_\fe$  is the associated symmetric variety, 
    \begin{enumerate}
    \item\label{item lattice i} $\Ax =\rA_{\X_\fe}$ (up to a Galois $1$-cocycle valued in the little Weyl group $W_{\hat{\X}}$);
        \item\label{item diagram i} the dual sequence
        \[
\check{\G}_{\X_\fe}\lra\hat{\G}_{\X_\fe}\lra\check{\G}_\fe
        \]of the $\G_\fe$-variety $\X_\fe$ is conjugate by an element of $\check{\G}_\fe$ to the top row of \eqref{diag: dual at x};
        \item there is a canonical morphism $\Aut_d(\X)\to \Aut_d(\X_\fe)$.
\end{enumerate}
If Conjecture \ref{Conj: cohom surj} holds for $\X_{\fe}$, then $\X_\fe$ may be chosen so that its geometric $1$-cocycle $\mu_{\X_\fe}:\Ga\lra \Aut_d(\X_\fe)$ is cohomologous to the cocycle obtained by
 \begin{equation}\label{item rep i}
\begin{tikzcd}
    \Ga\ar[dr,"\mu_{\X,\fe}"]\ar[d,"\mu_\X"]&\\
    \Aut_d(\X)\ar[r]& \Aut_d(\X_\fe).
\end{tikzcd}
 \end{equation}
 When $\X_\fe$ is well-adapted, this determines it up to $\G_\fe$-inner form.
\end{Thm}
\begin{Rem}
    We verify that Conjecture \ref{Conj: cohom surj} holds if $\rH_\fe$ is connected (Lemma \ref{Lem: surjective on dist cohom}) and in many natural settings when $\rH_\fe$ is not connected (cf. Corollary \ref{Lem: examples of conj}). 
\end{Rem}

Our proof of this relies heavily on our restriction to symmetric varieties, where we may consider first the existence of the involution over the algebraic closure and then consider questions pertaining to the fine structure of $\X_\fe$ over $k$ as in the theorem. For this, we show in Section \ref{Section: endo invo} that there is a good notion of endoscopic root systems with involution; this essentially shows the existence of $\theta_\fe$ (equivalently, the symmetric variety $\overline{\X}_\fe$) over $\kbar$. In Proposition \ref{Prop: endoscopic roots}, we utilize the more subtle notion of a $(\Ga,\theta)$-index from \cite{Helminckrational} to descend this involution to $k$ (see Remark \ref{Rem: what about Helminck}). This corresponds to the statement that the symmetric variety $\overline{\X}_\fe$ descends to $\X_{\fe}$ over $k$. This is not sufficient for the purposes of endoscopy, and the additional constraint on the geometric cocycle isolates the appropriate inner class of $k$-forms. Finally, Theorem \ref{Thm: quasi-split G-inner form} shows that one may typically isolate a unique form in a given $\G_\fe$-inner class possessing points with quasi-split stabilizers. We remark that all of these conditions are satisfied if $\X$ has no spherical roots of type $N$.

With this theorem, we formulate the notion of an endoscopic datum for $(\G,\X)$ as a quintuple
\[
{\fe}=(\G_\fe,\theta_\fe,\X_\fe,\ka,\eta_\fe),
\]
where $(\G_\fe,\ka,,\eta_\fe)$ is a pure endoscopic triple for $\G$ requiring the semi-simple element $\ka\in \check{\G}^{ss}$ to satisfy several constraints, $\theta_\fe$ is the $k$-rational involution on $\G_\fe$ as in the theorem, and $\X_\fe$ is the corresponding symmetric variety. The data of $\theta_\fe$ is technically not unique up to isomorphism due to the issue of pure inner forms (see Definition \ref{Def: endo iso}), but we include it to incorporate the existence and uniqueness of the quasi-split $\G_\fe$-inner form (Theorem \ref{Thm: quasi-split G-inner form}).
\begin{Rem}
    We do not assume that $\X$ not possess type $N$ roots, but only that it be well adapted. As seen in Section \ref{Section: color autos}, these two notions are not unrelated, but the latter assumption is much weaker. 
    
    The presence of such roots can cause the distinguished morphism $\varphi_\X$ to fail to be injective, and even if it is injective, the construction of $\check{\G}_\X$ requires a certain renormalization which loses information. What saves us is our ability to keep track of the involution $\theta$ induced on a maximal torus $\rA\subset \G$ which transfers to $\G_\fe$. 
\end{Rem}
\begin{Ex}
    We show how to recover the notion of endoscopy in \cite{Leslieendoscopy}. Let $\G=\U_{2n}$ be a quasi-split unitary group associated to a quadratic extension $E/k$ of fields, and let $\X=\U_{n}\times \U_{n}\backslash\U_{2n}$ as before.  As noted,  $\check{\G}_{\X}=\Sp_{2n}$ embeds into $\check{\G}=\hat{\G}_{\X}$ in the obvious way. 
 This embedding identifies the image with the $\Ga$-fixed points of $\check{\G}$ so that ${}^L\X=\Sp_{2n}(\cc)\times \Gal(E/k)$.  The dual symmetric variety is thus the variety of non-degenerate alternating forms on $\cc^{2n}$ 
\begin{align*}
\hat{\X}= \Sp_{2n}(\cc)\backslash\GL_{2n}(\cc)&\cong \mathrm{Alt}(\cc^{2n})\\
                    [g]&\longmapsto \check{\theta}(g)^{-1}g ,
\end{align*}
with distinguished base point $x_0:=J\in \hat{\X}^{\Gal(E/k)}$ corresponding to the symplectic form cutting out $\check{\G}_\X$. In this setting $\Aut_d(\X)\cong \mu_2$ and the geometric cocycle $\mu_\X:\Gal(E/k)\to \mu_2(\kbar)$ is the unique non-trivial character $\omega_{E/k}$.

 
 For each  pair $(a,b)\in \zz^{2}_{\geq0}$ such that $a+b=n$, consider the element $x_{a,b}:=s_{a,b} J$ with
 \[
s_{a,b} = \left(\begin{array}{ccc}
     I_a&&  \\
     & -I_{2b}&\\
     &&I_a
\end{array}\right).
\]
Passing to the descendant of $\hat{\X}$ at this point, we obtain $\check{\G}_{\X,a,b}\lra \check{\G}_{a,b}$, where 
\[
\text{$\check{\G}_{\X,a,b}=\Sp_{2a}(\cc)\times \Sp_{2b}(\cc)$ and $\check{\G}_{a,b}=\GL_{2a}(\cc)\times \GL_{2b}(\cc)$}.
\]
By Theorem \ref{Thm: exists intro} we obtain the endoscopic datum
\[
{\fe}_{a,b}=(\G_{a,b},\theta_{a,b},\X_{a,b},x_{a,b},\eta_{a,b}).
\]
 where $\X_{a,b}=\rH_{a,b}\backslash\G_{a,b}$ such that $\rH_{a,b}= \rH_a\times \rH_b$, $\G_{a,b}= \G_a\times \G_b$ with
\[
\rH_a= \U_a\times \U_a\text{ and }\G_a=\U_{2a}
\]
and similarly for $\rH_b$ and $\G_b$. Noting that $\Aut_{d}(\X_{a,b})\simeq\mu_2\times \mu_2$ with $\Aut_{d}(\X_n)\to \Aut_d(\X_{a,b})$ being the diagonal map, this is the unique outer class of $k$-forms satisfying \eqref{item rep i}.
 Finally, we note that everything above is compatible with the Galois actions, so that we obtain the following Cartesian diagram of $L$-groups
\begin{equation}\label{diag endo intro}
    \begin{tikzcd}
{}^L\X_{a,b}\ar[d]\ar[r]&{}^L\G_{a,b}\ar[d]\\
{}^L\X\ar[r]&{}^L\G.
\end{tikzcd}
\end{equation}
It is not hard to show that all elliptic endoscopic data of $(\G,\X)$ are isomorphic to a form of this type. This is compatible with our goal of stabilizing the relative trace formula along transfer from these varieties as discussed in \cite{Leslieendoscopy, LeslieUFJFL,Lesliedescent}.
\end{Ex}

\subsection{Point comparison and prestabilization} An important property of endoscopic groups $\G_\fe$ of $\G$ is the point matching of stable semi-simple conjugacy classes of $\G_\fe$ with those of $\G$. This is foundational to the notion of smooth transfer of orbital integrals. As evidence for our theory, we establish the existence of such a matching of stable semi-simple orbits in Theorem \ref{Thm: point comparison}.
\begin{Thm}\label{Thm: point match intro}
Suppose that $\G$ is a reductive group over $k$. Let $\theta$ be a $k$-rational involution and let  $\X=\rH\backslash\G$ be an associated symmetric variety. Let $\fe = (\G_\fe,\theta_\fe,\X_\fe,\ka,\eta)$ be an endoscopic datum for $(\G,\X)$. There exists a canonical map of affine $k$-schemes
\[
\fa_\fe:\X_\fe\sslash \rH_\fe\lra\X\sslash \rH
\]
between the categorical quotients.    
\end{Thm}
 After a simple argument extending the Chevalley-style isomorphism for $\G^\theta\backslash\G$ of \cite{Richardson} to all $(\G^\theta)^\circ\subset \rH\subset N_\G(\G^\theta)$ (Lemma \ref{Lem: chevalley}), the proof relies on this isomorphism as well as important structural properties of the dual symmetric variety $\hat{\X}$.

In particular, one may now formulate the notion of matching stable orbits necessary to formulate the local transfer and fundamental lemma conjectures. This requires enhancing the discussion to the stack-theoretic level via the morphism $\X/\rH\to \X\sslash \rH$, since it is not true in general that we may explicitly relate $k$-points of $\X\sslash \rH$ to stable semi-simple orbits of $\X(k)$ alone. We discuss these notions in more detail in the companion paper \cite{LesliestabFJ}.

 We expect the preceding theorem will play a role in establishing of analogues of Ng\^{o}'s work on geometric stabilization in this context, with a view toward establishing fundamental lemmas such as the one established in \cite{LeslieUFJFL}. Building on \cite{LeslieSpringer,Garcia-Prada}, forthcoming work of T. Hameister and B. Morrissey \cite{HameisterMorrissey} completes the structural description of the Hitchin fibration for symmetric pairs and proves analogues of Ng\^{o}'s results on the gerbe structure of the Hitchin fibration with respect to a certain stabilizer group scheme when $\X$ has abelian regular stabilizers, and one may ask whether there exists in that setting an interpretation of our notion of endoscopic data analogous to Ng\^{o}'s re-interpretation \cite{Ngo06}. We also mention in this direction work in the context of Hitchin fibrations for symmetric pairs and more general Vinberg gradings of \cite{garcia2023vinberg,gallego2023multiplicative}. 

In Section \ref{Section: stabilize}, we show how Theorem \ref{Thm: point match intro} allows for the pre-stabilization of the regular elliptic part of the relative trace formula for $(\G,\X)$, subject to simplifying assumptions. The more general statement and proof of the cohomological preliminaries is handled in \cite{LesliestabFJ}, where we show that the (regular elliptic part of the) geometric side of the RTF may be written as
\[
\sum_{a\in(\X^{re}\sslash\rH)(k)}\sum_{\ka\in H^1_{ab}(k,{\rH}_a)^D}\Orb_a^\ka({f}),
\] where $\X^{re}$ denotes the regular elliptic locus and ${\rH}_a$ denotes the stabilizer of a representative in $\X(k)$ of the class $a$ (well-defined up to inner form, so that the abelianized cohomology group is well-defined), and $\Orb_a^\ka({f})$ is an adelic $\ka$-orbital integral; see Section \ref{Section: stabilize}. To  illustrate the main application of Theorems \ref{Thm: exists intro} and \ref{Thm: point match intro}, we establish Proposition \ref{Prop: important} which switches the order of summation to give
\[
\sum_{\fe/\sim}\left(\sum_{b\in [\X_\fe^{re}\sslash\rH_\fe]^{\X}(k)}\Orb_{\fa_\fe(b)}^\ka({f})\right)
\]summing over isomorphism classes of elliptic endoscopic data for $(\G,\X)$ and the $\X$-regular locus of $[\X_\fe^{re}\sslash\rH_\fe](k)$. One may now formulate the relevant smooth transfer and fundamental lemma conjectures toward the stabilization of this formula; to reasons of space, we defer this to \cite{LesliestabFJ}.

\subsection{Relative Langlands duality and endoscopy}\label{Section: BZSV intro} Let us now assume that $\X=\rH\backslash\G$ satisfies the constraint that generic (semi-simple) stabilizers are connected; this is related to the notion of ``simply connected'' variety studied in \cite{Lesliedescent}. In this context, the cotangent bundle $\mathcal{M}=T^\ast\X$ is a hyperspherical $\G$-variety\footnote{For the purposes of this article, we ignore the more technical restrictions and grading discussed in \cite{BZSV}.} in the sense of \cite{BZSV}. Ben-Zvi, Sakellaridis, and Venkatesh have conjectured the existence of a \emph{dual Hamiltonian $\check{\G}$-variety} $\check{\calm}$, the symplectic geometry of which encodes the arithmetic properties of the various distinction questions associated to the variety $\X$. Section 4 of \emph{ibid.} constructs the putative variety $\check{\calm}$ for hyperspherical $T^\ast\X$, and conjectures its various properties. 

A consequence of Theorem \ref{Thm: Outer in terms of dist} is that $\X$ is well adapted whenever $T^\ast\X$ is hyperspherical. We show in Section \ref{Section: calculate cocycle} and Section \ref{Sec: symplectic rep} that the geometric cocyle $\mu_\X$ encodes the data of a symplectic representation $S_\X$ of $\check{\G}_\X$ \emph{together with a natural extension} of this representation to ${}^L\X=\check{\G}_\X\rtimes \Ga$; see Theorem \ref{Thm: unique}. In fact, this holds for any well adapted symmetric variety. 

Happily, our construction of the representation $S_\X$ is identical to the conjectured formula of \cite[Section 4.3]{BZSV}! Their formula was derived from analysis of the $L$-factors associated to spherical varieties and agrees precisely with the representation arising from our analysis of rational forms of $\X$. As a consequence, we establish some of the conjectures of these authors when $\X$ is symmetric and $T^\ast\X$ is hyperspherical (Section \ref{Section: BZSV conj}). Moreover, we obtain a natural Galois action on $\check{\calm}$ via the action on $S_\X$; this appears vital for applications beyond the context of split groups (cf. \cite{GetzWambach}).  We expect this connection (and the notion of well-adapted variety) will play a role in understanding their construction in greater generality.

One of the points of relative Langlands duality is to give a robust setting in which to study ``relative functoriality''; that is, the relationship between distinction questions of various groups through Langlands functoriality. Our construction gives rise to such a setting: given a pair $(\G,\X)$ of quasi-split reductive group $\G$ and symmetric variety, we construct a family of pairs $(\G_\fe,\X_\fe)$ where $\G_{\fe}$ is an endoscopic group of $\G$. A basic assertion is that this construction is compatible with functoriality. For example, when $k$ is local we expect that an $\X_{\fe}$-distinguished $L$-packet $\Pi$ of $\G_\fe(k)$ lifts under endoscopic transfer to a $\X$-distinguished packet on $\G(k)$ via functoriality along diagrams generalizing \eqref{diag endo intro}. The natural expectation is that if $\check{\calm}$ is dual to $T^\ast\X$ and $\check{\calm}_\fe$ is dual to $T^\ast\X_{\fe}$, then there ought to be a (derived) Lagrangian correspondence 
\[
\check{\calm}_\fe\longleftarrow \mathcal{L}_\fe\lra \check{\calm}
\]
compatible with the moment maps and the restriction morphism $\check{\fg}^\ast\to \check{\fg}^\ast_\fe$
Under the assumption that the generic stabilizer of $\X$ is abelian and $\G_{der}$ is simple, we verify that this is the case in Proposition \ref{Prop: langrangan}. The restriction on the generic stabilizers is purely for convenience; we will return to the geometric and symplectic relationship between $\check{\calm}$ and $\check{\calm}_\fe$ in future work.

 \begin{Ex}
    Continuing with the example $\G=\U_{2n}$ and $\X= \U_n\times \U_n\backslash\U_{2n}$, we find that $S_\X =T^\ast(\cc^{2n})$, where $\cc^{2n}$ denotes the standard representation of $\check{\G}_\X=\Sp_{2n}(\cc)$. Thus,
    \[
    \check{\calm}_n=T^\ast(\Sp_{2n}(\cc)\backslash\GL_{2n}v\times \cc^{2n}).
    \]
    For $\X$ as above, the $\Gal(E/k)$ action is non-trivial on $S_\X$, acting by swapping the two factors of $T^\ast(\cc^{2n})$. On the other hand, this action is trivial for $\X' = \Res_{E/k}(\GL_n)\backslash\U_{2n}$. In either case, there is a natural closed immersion
    \[
    \left[\Sp_{2a}(\cc)\backslash\GL_{2a}(\cc)\times \cc^{2a}\right]\times \left[\Sp_{2b}(\cc)\backslash\GL_{2b}(\cc)\times \cc^{2b}\right]\subset \Sp_{2n}(\cc)\backslash\GL_{2n}(\cc)\times \cc^{2n},
    \]
    giving rise to the Lagrangian correspondence
    \[
   \check{\calm}_a\times \check{\calm}_b\longleftarrow \mathcal{L}_{a,b}\lra \check{\calm}_n,
    \]
    with $\mathcal{L}_{a,b}$ the fiber product over this closed embedding.
\end{Ex}

In Section 14 of \cite{BZSV}, the authors formulate local and global conjectures relating $\rH$-invariant forms (or $\rH$-period integrals) to fixed points of Langlands parameters on $\check{\calm}$. We hope the results of the present work will aid in establishing these conjectures, at least for varieties related to symmetric varieties. In the number field setting, this can only be a guiding principle at the moment. Nevertheless, Section \ref{Section: Example unitary} discusses the relation between the endoscopy for the Galois symmetric variety $\U_n\backslash\Res_{E/k}(\GL_n)$, unitary periods of Eisenstein series, and the fundamental lemma proved in \cite{LeslieUFJFL}. 

\subsection{Outline} We end with a brief summary of the papers contents. We begin with some notation and conventions in Section \ref{Section: Prelim}.  In Section \ref{Appendix: spherical roots}, we recall the necessary generalities of spherical varieties as well as the rationality results of Borovoi and Gagliardi. In Section \ref{Section: doubling aut}, we introduce the group of doubling automorphisms, the geometric cocycle,  and the notion of well adaptedness. In Section \ref{Section: involutions}, we recall certain generalities on involutions on root systems, including the notions of $\theta$-index and $(\Ga,\theta)$-index due to \cite{Helminckrational}. The new material is our introduction of endoscopic root systems with involution in Section \ref{Section: endo invo}, which is the combinatorial core of the existence result in Theorem \ref{Thm: exists intro}. Section \ref{Section: symmetric data} elucidates the spherical data of symmetric varieties at length. Particular attention is paid to the explicit calculation of the various root systems associated to symmetric varieties. We also show that the notion of $z$-extension generalizes quite naturally to the setting of symmetric varieties in Proposition \ref{Prop: good cover}. Section \ref{Section: indices} includes the necessary results regarding $(\Ga,\theta)$-indices.

Section \ref{Section: color autos} proves Theorems \ref{Thm: Outer in terms of dist} and \ref{Thm: quasi-split G-inner form}, calculating the group $\Out_{\X}(\rH)$ for all symmetric varieties and showing that in all well-adapted cases the quasi-split $\G$-inner form exists over $k$. The argument proceeds via a reduction to the absolutely simple case, where we ultimately compute the various groups case-by-case.

In Section \ref{Section: dual groups}, we recall the dual group of a spherical variety, its construction and properties. We also refine the functoriality of the dual group in Appendix \ref{Appendix: derived subgroup}. We then introduce the dual symmetric variety $\hat{\X}$ in Section \ref{Sec: dual symm space}, and prove its basic properties in Section \ref{Sec dual basic}. Section \ref{Section: endoscopy defs} defines and proves the existence of endoscopic symmetric varieties. Section \ref{Section:orbit match} then establishes the foundational result on point matching for these varieties.

In Section \ref{Sec: symplectic rep}, we show that under the additional assumption that $\rH$ is geometrically connected the geometric cocycle constructed in Section \ref{Section: spherical data} is completely encoded by a symplectic representation of ${}^L\X$, with the cocycle determining the Galois action. This is Theorem \ref{Thm: unique}. Restricting further to the case of so-called excellent symmetric varieties, we discuss the compatibility of our construction with the dual Hamiltonian variety of \cite{BZSV}. Finally, Section \ref{Section: Example} discusses examples of unitary periods and relations to endoscopy, establishes the pre-stabilization of the regular elliptic part of the RTF in Section \ref{Section: stabilize}, and ends with examples of endoscopic varieties when $\G_{der}$ is $k$-simple in Table \ref{tab:Examples}.


\subsection{Acknowledgements} 
 I want to thank Jayce Getz for his patience with my many questions and for his generosity of ideas. Essentially all of this is an attempt to work out some of the ideas first discussed in \cite{GetzWambach} which were explained to me by him. I also want to thank Yiannis Sakellaridis for several helpful conversations and for encouragement regarding this work, including how to relate unitary periods and the arguments of \cite{LeslieUFJFL} in the context of \cite{BZSV}, and Sol Friedberg for several helpful remarks. I also thank Rapha\"{e}l Beuzart-Plessis, Sanath Devalapurkar, Ben Howard, Keerthi Madapusi, Siddharth Mahendraker, Aaron Pollack, and Wei Zhang for helpful conversations about several aspects of this work. I especially thank Friedrich Knop for answering questions about spherical varieties.

This work was partially supported by an AMS-Simons Travel Award and by NSF grants DMS-1902865 and DMS-2200852.

\part{$\G$-outer forms and symmetric varieties}

\section{Preliminaries}\label{Section: Prelim}

\subsection{Groups and varieties}\label{section: prelim groups}
We use $k$ for a field, which we always assume does not have characteristic $2$.  We generally assume that $\G$ is a connected reductive group over a field $k$. We will either assume that $k$ has characteristic zero or good characteristic for $\G$. We let $\kbar$ denote a fixed algebraic closure of $k$ and $k^{sep}\subset \kbar$ the separable closure. Throughout much of the paper, $\Ga=\Gal(k^{sep}/k)$. 

We let $\G_{der}\subset \G$ denote the derive subgroup of $\G$, $\G_{ad}$ its adjoint quotient, and $\G_{sc}$ the simply connected cover of $\G_{der}$. We write $Z(\G)$ for the center, and $Z_{\G}(\rH)$ for the centralizer of a subgroup $\rH\subset \G$.

For a spherical $\G$-variety $\X$, we set $\Aut^{\G}(\X)$ to the group $k$-scheme of $\G$-equivariant automorphisms (cf. \cite{Losev,KnopAutomorphisms}). In particular, if $\X=\rH\backslash\G$ is homogeneous, then 
\[
\Aut^\G(\X) =\rH\backslash N_{\G}(\rH),
\]
where $N_{\G}(\rH)$ denotes the normalizer of $\rH$ in $\G$.

If $\rH$ is acted on by a group of automorphisms $\Ga_{\rH}\subset \Ga$, we denote by $\Ind_{\Ga/\Ga_{\rH}}(\rH)$ the induced group \cite[Section 4]{BorelAutomorphic}.

Now assume that $\X=\rH\backslash\G$ is symmetric, in the sense that there is a $k$-rational involutive automorphism $\theta$ of $\G$ such that 
\[
(\G^\theta)^\circ\subset \rH\subset N_{\G}(\G^\theta);
\]
we always assume that $\rH$ is smooth (automatic if $\rH$ connected or if $\mathrm{char}(k)$ is zero). When $\mathrm{char}(k)\neq 2$ is positive, we always assume that it is good for $\G$. When $\G_{sc}$ is absolutely simple, this is equivalent to 
\begin{itemize}
    \item $\mathrm{char}(k)\neq2$ if $\G_{\kbar}$ is of Cartan types $A$, $B$, $C$, or $D$,
    \item $\mathrm{char}(k)\neq2,3$ if $\G_{\kbar}$ is of Cartan types $G_2$, $F_4$, $E_6$, $E_7$, 
        \item $\mathrm{char}(k)\neq2,3,5$ if $\G_{\kbar}$ is of Cartan type $E_8$,
\end{itemize} In general, $p$ is good for $\G$ if it is good for each irreducible factor of $\G_{sc}$ This will be our standard assumption throughout, though it is not optimal for much of the paper (for which it suffices to assume $\mathrm{char}(k)\neq 2$). It is made to ensure that our use of certain calculations regarding the nilpotent cone for symmetric varieties in positive characteristic (see \cite{Levy}) behave similarly to the characteristic zero setting.


If $\rA\subset \G$ is a torus that is $\theta$-stable, we abuse notation and use $\theta$ to denote the induced involution on the character lattice $\theta:X^\ast(\rA)\to X^\ast(\rA)$ and $\check{\theta}$ the involution on $X_\ast(\rA)$.
\subsection{Invariant theory}\label{Section: Prelim inv}
For any field $k$ and any non-singular affine algebraic variety $\X$ over $k$ with $\mathrm{G}$ a connected reductive algebraic group over $k$ acting algebraically on $\X$, we set $\X^{rss}$ to be the invariant-theoretic regular semi-simple locus. That is, $x\in \X^{rss}(\kbar)$ if and only if its $\mathrm{G}_{\kbar}$-orbit is of maximal possible dimension and Zariski-closed. We also recall the semi-simple locus $\X^{ss}$ of points with Zariski-closed orbits. 

\begin{Def}
    Let an algebraic group $\G$ act on an algebraic variety $\X$. A pair consisting of an algebraic variety $\Y$ and a $\G$-invariant morphism $\pi:\X\to \Y$ is called the \emph{categorical quotient} of $\X$ by the action of $\G$ if for any other such pair $(\pi',\Y')$, there exists a unique morphism $\varphi:\Y\to \Y'$ such that $\pi'=\varphi\circ\pi$. Clearly, if a categorical quotient exists, it is unique up to a unique isomorphism with respect to the obvious notion of isomorphism between pairs. We will denote it by $(\pi_\X,\X\sslash \G)$; by a common abuse of terminology, we often drop reference to the morphism $\pi_\X$ and call $\X\sslash\G$ the categorical quotient.
\end{Def}
\begin{Prop}\cite[Theorem 2.2.2]{AizGourdescent}\label{prop: quotients exist}
   Suppose $\G$ is (not-necessarily connected) reductive and $\X$ affine. Then the categorical quotient $\X\sslash \G$ always exists and is given by $\Spec(k[\X]^{\G})$. Moreover, every fiber of $\pi_\X$ contains a unique Zariski-closed orbit. 
\end{Prop}

Suppose now that $\G$ is reductive over a field $k$ and $\X$ is an affine $\G$-variety with categorical quotient $\X\sslash \G$. Suppose that $\cala$ is another (possibly disconnected, but smooth) reductive group with an algebraic action on $\X$ commuting with $\G$. Then one obtains an algebraic action\quash{The map
\[
\X\overset{\phi}{\lra}\X\overset{\pi_{\X}}{\lra} \X\sslash \G
\]
is $\G$-equivariant, so that by the definition, there exists $\phi':\X\sslash \G\to \X\sslash \G$ satisfying
$\pi_\X\circ \phi=\phi'\circ \pi_\X$. Applying the same argument to $\phi^{-1}$, we see that $\phi'\in \Aut(\X\sslash \G)$, and it is easy to see that this induces an injective homomorphism }
\begin{equation}\label{eqn: descend autos}
  m'\sslash\G:  \cala\times \X\sslash \G\lra \X\sslash\G,
\end{equation}
by noting that $\cala\times \X\sslash\G$ is the categorical quotient of $\cala\times \X$ by $\G$, where $\G$ acts trivially on $\cala$. With this, one applies the universal property to the action map $\cala\times \X\to \X$ composed with $\X\to \X\sslash\G$ to obtain the morphism \eqref{eqn: descend autos}, and verifies that the appropriate diagrams commute. This is straightforward.
\begin{Lem}\label{Lem: mod out by autos}
 Setting $\Y=\X\sslash \G$, the categorical quotients $\X\sslash\cala$ and $\Y\sslash\cala$ exist and sit in a commutative diagram of $k$-schemes
    \[
    \begin{tikzcd}
        \X\ar[r]\ar[d]& \X\sslash\cala\ar[d]\\
        \Y\ar[r]&\Y\sslash\cala,
    \end{tikzcd}
    \]
    where each arrow is the corresponding quotient map.
\end{Lem}
\begin{proof}
    If the two claimed quotients exist, the diagram exists by properties of the categorical quotient and repeated applications of the existence of the embedding \eqref{eqn: descend autos}. The quotients exist by Proposition \ref{prop: quotients exist} since by assumption $\cala$ is reductive. 
\end{proof}

\subsection{The $\ast$-action and the $L$-group}\label{Section: pre star}
Assume now that $\G$ is a reductive group over $k$  and let $\rA$ denote a maximal torus. If $\mathrm{SAut}(\G)$ denotes the set of semi-linear automorphisms of $\G(\kbar)$ (cf. \cite{BorovoiModels}), the $k$-rational structure of $\G$ induces a morphism
 \[
 \mu: \Ga\lra \mathrm{SAut}(\G),
 \]
which splits the sequence
\[
1\lra \Aut(\G)(\kbar)\lra \mathrm{SAut}({\G})\lra \Ga.
\]
 As explained in \cite{BorovoiModels}, this induces a $\Ga$-action on $\G(\kbar)$, which we refer to as $\mu_\sig$ for $\sig\in \Ga$. 

Over the algebraic closure $\kbar$, there exists a Borel subgroup $\B\supset \rA$, unique up to $N_\G(\rA)$-conjugacy. Fixing one gives rise to a based root datum $\Psi=(X^\ast(\rA), \De,X_\ast(\rA),\check{\De})$. 

For any $\sig\in \Ga$, there exists $g_\sig\in \G(k^{sep})$ such that $\sig(\B,\rA) = \Ad(g_\sig)(\B,\rA)$. This induces a $\Ga$-action on $\Psi$. 
 Fixing a pinning $\{x_\al\}_{\al\in \De},$ where $x_\al\in U_\al(k^{sep})$ is a root vector for each simple root realizes
$
 \Aut(\Psi(\G))\simeq \Aut(\G,\B,\rA,\{x_\al\}_\al),
$
 as a subgroup of $\Aut(\G)$ is a manner that splits the short exact sequence
 \[
1\lra \mathrm{Inn}(\G)\lra \Aut(\G)\overset{\psi}{\lra} \Aut(\Psi(\G))\lra 1.
\]  Augmenting $\sig\mapsto \Ad(g_\sig)$ by a cocycle valued in $\rA$ if necessary to ensure that
 \[
 \mu_\sig(\B,\rA,\{x_\al\}_\al) = \Ad(g_\sig)(\B, \rA,\{x_\al\}_\al),
 \]
 the assignment $\mu_{\G}(\sig) = \Ad(g_\sig)^{-1}\circ \mu_\sig$ thus gives a homomorphism $\mu_{\G}^\ast:\Ga\lra \mathrm{SAut}({\G})$ such that the corresponding $k$-form $\G^\ast$ is quasi-split. The induced $\Ga$-action on $\Psi$ is known as the \textbf{$\ast$-action}.
 
Passing to the dual based root datum $\check{\Psi}:=(X_\ast(\rA),\check{\De},X^\ast(\rA),\De)$, let ${\check{\G}}$ be the Langlands dual complex reductive group associated to this datum. Since there is a canonical isomorphism
 \[
 \Aut(\Psi(\G))\simeq \Aut(\check{\Psi}({\check{\G}})),
 \]
 we obtain a homomorphism $\mu^\vee_{\G}: \Ga\lra \Aut(\check{\Psi}({\check{\G}}))$, and upon fixing a pinning $\{x_{\check{\al}}\}$ for $(\check{B},\check{T})$, we obtain an action of $\Ga$ on $(\check{\G},\check{B},\check{T},\{x_{\check{\al}}\})$. The (Galois form of the) $L$-group 
 \[
 {}^L\G:=\check{\G}\rtimes \Ga
 \]
 is the corresponding semi-direct product. 
 

 
\quash{
 This is compatible with the notion of Tits $\Ga$-indices. In that theory, a reductive $k$-group gives an index $(X^\ast(T), \De, \De_0, \sig_\ast)$, where $\De$ is a $\Ga$-basis (arising from a minimal parabolic $k$-group containing a maximally $k$-split maximal torus $T\subset P\subset \G$), $X_0 = \ker\Nm$, and $\De_0 = \De\cap X_0$. 
 The connection with the preceding remark is that $g_\sig\mapsto w_\sig$ under the natural morphism 
 \[
 \begin{tikzcd}
 N_{\G}(T)\ar[r]& W(G,T)\\
 N_{\G}(T)\cap Z_{\G}(A)\ar[r]\ar[u]&W_0=W(Z_{\G}(A),T)\ar[u].
 \end{tikzcd}
 \]
 In particular, the index knows the quasi-split inner form of $\G$. It also knows the Dynkin diagram of the $k$-anisotropic kernel by remembering $\De_0$, 
}

 \subsection{Galois actions on groups and varieties}\label{Sec: forms and action}
With $\G$ as before, assume $\X$ is a $\G$-variety over $k$. These $k$-rational structures induce morphisms
 \[
 \mu: \Ga\lra \mathrm{SAut}(\G),\qquad \mu_\X:\Ga\lra \mathrm{SAut}^\G(\X)
 \]
with the first morphism as above and the second morphism satisfies $\mu$-equivariance and $\mathrm{SAut}^{\G}(\X)$ denotes the group of $\mu$-equivariant semi-linear automorphisms of $\X$ as in \cite[Section 2.10]{BorovoiGagliardi}. That is, for each $\sig\in \Ga$,
\begin{equation}\label{eqn: ga equivariant}
\mu_{\X,\sig}(g\cdot x) = \mu_\sig(g)\cdot \mu_{\X,\sig}(x).
\end{equation}
 The $k$-rational structure of $\G$ induces an exact sequence
\[
1\lra \Aut^{\G}(\X)(\kbar)\lra \mathrm{SAut}^{\G}(\X)\lra \Ga,
\]
which $\mu_\X$ splits.  These morphisms induce $\Ga$-actions on $\G(\kbar)$ and $\X(\kbar)$, such that for each $\sig\in \Ga$, \eqref{eqn: ga equivariant} holds. 
  \quash{
Assume now that $\X$ is $\G$-homogeneous.  Fix a base point $x_0\in \X(k)$ and let $\Stab_{\G}(x_0)=\rH$. This induces an isomorphism $\X\simeq \rH\backslash\G$. By \cite[Lemma 4.1]{BorovoiModels}, for each $\ga\in \Ga$ there exists $g_\ga\in N_\G(\rH)(\kbar)$ such that
 \[
 \mu_{\ga}(g\rH(\kbar)) = \nu_\ga(g)g_\ga^{-1} \rH(\kbar).
 \]
 As simple calculation shows that $g_\sig\sig(g_\tau)\rH(\kbar)=g_{\sig\tau}\rH(\kbar)$. Thus, we obtain a $1$-cocycle
 \begin{equation}\label{eqn: cocycle from Galois}
      \mu_{\X}:\Ga\overset{\mu}{\lra}N_\G(\rH)(\kbar)/\rH(\kbar)\simeq \Aut^{\G}(\X)(\kbar).
 \end{equation}
 A different choice of base point $x_0=1\rH(\kbar)\in \X(\kbar)$ changes $\mu_{\X}$ by a coboundary. In particular, a $k$-model $\X$ induces a canonical element $c_\X\in H^1(k, \Aut^\G(\X))$. 
 
 Since $\X$ is quasi-projective, a straightforward extension of Theorem 9.17 of \cite{BorovoiModels} shows that this gives a canonical bijection between the set of \emph{all} $\G$-equivariant $k$-models of $\X_{\kbar}$ and $H^1(k,\Aut^{\G}(\X))$. On the other hand, if we fix the base point $x_0\in \X(\kbar)$, the $1$-cocycle is canonical.
}
\subsubsection{$\G$-inner vs. $\G$-outer forms}\label{Section: outer forms}
Assume now that $\X$ is $\G$-homogeneous with $\X(k)\neq \emptyset$. Fix a base point $x_0\in \X(k)$, and let $\Stab_{\G}(x_0)=\rH$. This induces an isomorphism $\X\simeq \rH\backslash\G$ identifying $\rH\cdot 1 = x_0$. A standard calculation shows that
\[
\X(k)=\bigsqcup_{s\in \ker^1(\rH,\G;k)}x_s\cdot \G(k),
\]
where $\ker^1(\rH,\G;k) = \ker[H^1(k,\rH)\to H^1(k,\G)]$ and $x_s\in \X(k)$ satisfying that $\rH_s=\mathrm{Stab}_{\G}(x_s)$ is obtaining by twisting $\rH$ by a $1$-cocycle in the class $s$. In particular, all stabilizers are pure inner twists of $\rH$.

Having fixed $x_0$, the normalizer $N_{\G}(\rH)$ of $\rH$ has a natural map to $\Aut(\rH)$ via the conjugation action, which we may compose with the canonical quotient to obtain a $k$-morphism
\[
\mathrm{out}_{\rH}^{\G}:\Aut^{\G}(\X)\lra \Out(\rH);
\]
here we rely on the $k$-isomorphism $\Aut^{\G}(\X)\simeq \rH\backslash N_{\G}(\rH)$. We claim that this morphism does not depend on the choice of $x_0\in \X(k)$. Indeed, for any other $x\in \X(k)$ with $\rH'=\Stab_{\G}(x)$, there exists $g\in \G(\kbar)$ satisfying $x_0=x\cdot g^{-1}$ so that $\rH' = g\rH g^{-1}$. Then $\Ad(g)$ induces in a isomorphism $\Aut(\rH)\iso \Aut(\rH')$ such that changing $g\mapsto hg$ composes $\Ad(g)$ with an inner automorphism of $\rH'$, so that the induced isomorphism $\Out(\rH)\to \Out(\rH')$ is independent of the choice of $g$. Combining with $N_{\G}(\rH') = gN_{G}(\rH)g^{-1}$, we find that the diagram
\[
\begin{tikzcd}
    \rH\backslash N_{\G}(\rH)\ar[r,"\Ad(g)"]\ar[d]& \rH'\backslash N_{\G}(\rH')\ar[d]\\
    \Out(\rH)\ar[r,"\Ad(g)"]&\Out(\rH')
\end{tikzcd}
\]
commutes. Let $\Out_{\X}(\rH)$ denote the image. One may thus define $$\Out(\X):=\lim_{x} \Out_\X(\rH_x)$$ and $\out:\Aut^\G(\X)\to \Out(\X)$ is independent of the choice of base point. Denote by $\mathcal{A}_\X^\flat\subset \Aut^{\G}(\X)$ the kernel of this map. 

We will often fix $x_0\in \X(k)$ and work with $\Out_\X(\rH)$ below. Then both $\mathcal{A}^\flat_\X$ and $\Out_{\X}(\rH)$ are defined over $k$ with the natural inclusion $\Out_\X(\rH)\subset \mathrm{Out}(\rH)$ being defined over $k$. The following lemma is straightforward.
\begin{Lem}
    Suppose that $\G$ and $\X=\rH\backslash\G$ as above.  The map $\mathrm{out}:=\mathrm{out}_{\rH}^{\G}$ is $\Ga$-equivariant for the natural action on $\Out(\rH)$. Finally, $\mathcal{A}^\flat_\X(\kbar)$ is normal inside $\mathrm{SAut}^\G(\X)$.
\end{Lem} 

In particular, if we define $\mathrm{S}\Out_\X(\rH)\subset \mathrm{SOut}(\rH)$ as $\mathrm{S}\Out_\X(\rH) = \mathrm{SAut}(\X)/\mathcal{A}_\X^\flat(\kbar)$, there is a morphism 
$
\tilde\mu_\X:\Ga\lra \mathrm{SOut}_\X(\rH)
$
fitting into a diagram
\[
\begin{tikzcd}
   \Aut^{\G}(\X)(\kbar)\ar[r]\ar[d]&\mathrm{SAut}^{\G}(\X)\ar[r]\ar[d]&\Ga\ar[d,"="]\ar[dl,"\tilde\mu_\X"]\\
   \Out_\X(\rH)(\kbar)\ar[r]&\mathrm{S}\Out_\X(\rH)\ar[r]&\Ga.
\end{tikzcd}
\]

Suppose now that $\X'$ is another $\G$-variety such that there exists a $\G_{\kbar}$-equivariant isomorphism 
\[
\psi:\X_{\kbar}\iso \X'_{\kbar}.
\] 
satisfying that for each $\sig\in \Ga$
\[
\psi(\mu_\sig(g\cdot x)) = \nu_\sig (g)\cdot\psi(\mu_\sig (x)).
\]
We say that $\X'$ is a \textbf{$\G$-form} of $\X$. If $\X'(k)\neq\emptyset$, then $\X'\simeq\rH'\backslash\G$ for some subgroup $\rH'\subset \G$.
\begin{Rem}\label{Rem: G-form is restrictive}
    Note that $\X$ and $\X'$ being $\G$-forms is more restrictive than $\X_{\kbar}\simeq \X'_{\kbar}$. They are required to be isomorphic in a $\G_{\kbar}$ equivariant way. Indeed, for $\rH_1\backslash\G$ and $\rH_2\backslash\G$ to be $\G$-forms, it is necessary that $\rH_1=g\rH_2g^{-1}$ for some $g\in \G(\kbar)$ \cite[Lemma 4.1]{BorovoiModels}.
\end{Rem}

Fix a $\G$-form $\X'$ and isomorphism $\psi$ as above. For any $\sig\in \Ga$, the automorphism $\psi^{-1}\circ {}^\sig\psi:\X_{\kbar}\lra \X_{\kbar}$ is $\G_{\kbar}$-equivariant. Here ${}^\sig\psi(x):=\tilde{\mu}'_{\sig}\circ\psi(\mu_{\sig^{-1}}(x))$, where $\tilde{\mu}'$ denotes the corresponding semi-linear automorphism of $\X'$. We thus obtain a cocycle $$c_{\X,\X'}:=[\sig \lra \psi^{-1}\circ {}^\sig\psi]\in Z^1(k,\Aut^{\G}(\X)).$$
Augmenting the choice of isomorphism $\psi$ via automorphisms corresponds to twisting the cocycle by a coboundary, so we obtain a cohomology class $[c_{\X,\X'}] \in H^1(k,\Aut^{\G}(\X))$ which depends only on $\X'$.
\begin{Def}
    Suppose that $\X=\rH\backslash\G$ and let $\X'=\rH'\backslash\G$ be a $\G$-form of $\X$.  We say $\X'$ is a $\G$-\textbf{inner form} if there exists a choice of $\psi$ such that the cocycle $c_{\X,\X'}\in Z^1(k,\Aut^{\G}(\X))$ takes values in $\cala_\X^\flat(\kbar).$ 
      A $\G$-form of $\X$ which is not a $\G$-inner form is called a \textbf{$\G$-outer form} of $\X=\rH\backslash\G$.
\end{Def}
\begin{Lem}\label{Lem: outer forms classify}
    Suppose that $\G$ is a  reductive $k$-group and suppose that $\X=\rH\backslash\G$ is a homogeneous $k$-variety and let $\X'$ be a $\G$-form of $\X$. The cocycle $c_{\X,\X'}$ induces the trivial cohomology class in $H^1(k,\Out_\X(\rH))$ if and only if $\X$ and $\X'$ are $
    \G$-inner forms. 
\end{Lem}
\begin{proof} This is immediate by considering the long exact sequence in Galois cohomology associated to the commutative diagram with exact rows
\[
\begin{tikzcd}
   1\ar[r] &\mathcal{A}_{\X}^\flat\ar[r]\ar[d]&\Aut^{\G}(\X)\ar[r]\ar[d] &\Out_{\X}(\rH)\ar[d]\ar[r]&1\\
    1\ar[r] &\rH_{ad}\ar[r]&\Aut(\rH)\ar[r] &\Out(\rH)\ar[r]&1.  \qedhere
\end{tikzcd}
\]
\end{proof}

For any $k$-subgroup $\rH\subset \G$, recall the notation
\[
\ker^1(\rH,\G;k):=\ker[H^1(k,\rH)\to H^1(k,\G)].
\]
The following is an easy exercise.

\begin{Lem}\label{Lem: G inner forms}
     Let $\G$ be a reductive group over $k$, $\rH\subset \G$ as $k$-rational reductive subgroup, and $\X=\rH\backslash\G$. Set $\cala_\X = \rH\backslash N_{\G}(\rH)$. \begin{enumerate}
         \item   There is a natural exact sequence
     \[
    \ker^1(\rH,\G;k)\to\ker^1(N_\G(\rH),\G;k)\to  H^1(k,\cala_\X).
     \]
     \item  If $\cala_\X^\flat = \ker[\cala_\X\to \Out(\rH)]$, and $\rH\subset\rH^\flat\subset N_\G(\rH)$ is the subgroup corresponding to $\cala^\flat_\X$, then there is a natural exact sequence
     sequence
     \[
    \ker^1(\rH,\G;k)\to\ker^1(\rH^\flat,\G;k)\to  H^1(k,\cala^\flat_\X).
     \]
     \end{enumerate}
     In particular, two $\G$-forms $\X'=\rH'\backslash\G$ and $\X=\rH\backslash\G$ are $\G$-equivariantly isomorphic if and only if  $\rH$ to $\rH'$ are \emph{pure inner forms} with $\rH'$ corresponding to a class in $\ker^1(\rH,\G;k)$.
\end{Lem}

We remark that the preceding claim relies on the notion of $\G$-equivariant isomorphism $\psi: \X\to \X'$, which requires the diagram
\[
\begin{tikzcd}
    \G\times \X\ar[r,"m"]\ar[d,"id\times \psi"]&\X\ar[d,"\psi"]\\
    G\times \X'\ar[r,"m'"]&\X'
    \end{tikzcd}
\]
to commute. If we allow more general automorphisms of $\G$ in the preceding diagram, certain $\G$-inner forms may be shown to be isomorphic to $\X$.

\section{Spherical data and automorphisms}\label{Appendix: spherical roots}

In this section, we recall generalities of spherical varieties, their combinatorial data, and rationality properties. We then analyze certain quotients of the automorphism group of a homogeneous spherical variety in Section \ref{Section: doubling aut}. The main tool developed is the notion of geometric cocycle in Definition \ref{Def: geometric cocycle}.
\subsection{Generalities}\label{Section: spherical datum}
Let $\G$ be a connected reductive group over a field $k$, which for this section we assume has characteristic zero. This is essentially for convenience as our main references \cite{Losev, BorovoiGagliardi} work in this setting. However, most of these results hold with mild assumptions on the residue characteristic. In any case, we discuss the positive characteristic setting for symmetric varieties in Section \ref{Section: symmetric data}.

Fix a Borel subgroup $\B$ with $\rA=\B/[\B,\B]$ the canonical Cartan of $\G$. The canonicity of the quotient implies that $\rA$ inherits a unique $k$-structure, and we use $\rA$ to denote the resulting $k$-torus. In particular, for any maximally $k$-split maximal torus $T\subset \G$ there is a $k$-isomorphism $T\iso \rA$. Let $X^\ast(\rA)=X^\ast(\B)$ denote the character lattice, with $\De\subset \Phi^+\subset \Phi$ be the associated simple and positive roots. Then we have the based root datum $\Psi:=(X^\ast(\rA),\De,X_\ast(\rA),\check{\De})$.

Suppose that $\X=\rH\backslash\G$ is a homogeneous spherical variety over $k$. Let $\mathring{\X}$ denote the open $B$-orbit. Then $\rA$ acts on $ \mathring{\X}\sslash[B,B]$ through a quotient $\rA\lra \Ax$; we call the torus $\Ax$ the (canonical) torus of $\X$.

We let $\fX:=X^\ast(\Ax)$ denote the weight lattice of $\Ax$. It corresponds to the lattice of $B$-semiinvariants 
\[
0\lra \kbar^\times \lra \kbar(\X)^{(B)}\lra \fX\lra 0.
\]
We note that since we are not assuming that $k$ is algebraically closed, this should be viewed as a sequence of $\Ga$-modules. We also have the \textbf{spherical weights} of $\X$
\[
\fX^+:=\{\lam\in X^\ast(\rA)^{+}: V(\lam)_{\kbar}^{\rH_{\kbar}}\neq 0\},
\]
which satisfies 
\[
\kbar[\X]=\bigoplus_{\lam\in \fX^+}V(\lam).
\]

One considers the set $\mathcal{V}:=\mathcal{Z}(\X)$ of \emph{central discrete valuations} on $\kbar(\X)$; that is, valuations $\nu:\kbar(\X)\lra \zz\cup \{\infty\}$ that are $\G$-invariant and trivial on $\B$-invariants $\kbar(\X)^{\B}=\kbar$. On any $\B$-semiinvariant, such a valuation depends only on the corresponding character, so that there is a map 
\[
\mathcal{V}\lra \fa_{\X,\qq}:=\Hom(\fX,\qq).
\]
It is known that this map is injective and that the image realizes $\mathcal{V}$ as a  finitely-generated convex cone which is not contained in any hyperplane. Indeed, it gives a chamber for the so-called little Weyl group $W_\X$, which acts on $\fX$ and $\fa_{\X,\qq}$. 

Let $\De_\X:=\{\ga_1,\ldots,\ga_s\}\subset \fX_\qq$ be a minimal set of ``outward normal vectors'' such that
\[
\mathcal{V}:= \{a\in \fa_{\X,\qq}: a(\ga_i)\leq0,\:\:i\in \{1,\ldots,s\}\};
\] these are the \emph{spherical roots} of $\X$. At the moment, these roots are only defined up to a rational scalar.

\subsubsection{Normalizations of spherical roots}\label{App: normalizations}
For our purposes, there are $3$ natural root systems associated to $\X$, each of which plays a distinct role in the structure theory. Let $W_{\X}$ denote the little Weyl group of $\X$; all root systems have this as their Weyl group, and differ only in replacing simple roots with scalar multiples.
\begin{enumerate}
    \item \textbf{The $n$-spherical roots:}  Suppose that $\lam,\mu,\nu\in \fX^+$ such that $V(\nu)\subset V(\lam)\cdot V(\mu)$ in $\kbar[\X]$. Denote
\[
\Lam_\X^+:= \zz_{>0}\{\lam+\mu-\nu : V(\nu)\subset V(\lam)\cdot V(\mu)\}.
\]
A theorem of Knop states that this is a free monoid; let $\De^n_{\X}$ denote its (unique) set of free generators. This is the set of $n$-spherical roots of $\X$, and we let $\Phi^n_{\X}=W_{\X}\cdot\De^n_{\X}$. 
\begin{Rem}
    Knop shows in \cite{KnopAutomorphisms} that the \emph{root lattice} $\Lam_\X:=\zz\Lam_\X^+$  of $\X$ is precisely the kernel of 
    \[
    X^\ast(\Ax)\lra X^\ast(\Aut^{\G}(\X)),
    \]
    at least when $\X$ is quasi-affine.
\end{Rem}
\item \textbf{The minimal root system:} By construction, we have inclusion
\[
\Lam_\X\subset \fX.
\]
However, $\De^n_{\X}$ need not be primitive in $\fX$. Knop defines $\De^{min}_{\X}$ as the set of primitive elements in $\fX$ associated to the extremal rays of the cone generated by $\De^n_{\X}$ in $\fX\otimes_{\zz}\qq$. Set $\Phi^{min}_{\X}=W_{\X}\cdot\De^{min}_{\X}$.
\begin{Rem}
    We take this to be the definition of \emph{spherical roots} of $\X$, and set $\De_\X:=\De_{\X}^{min}$.
\end{Rem}
\item\textbf{The normalized root system:}  Motivated by integrality constraints imposed by Langlands functoriality (see Section \ref{Section: dual groups}), Sakellaridis--Venkatesh introduce an additional normalization in \cite{SakVenk}. 
    The \textbf{normalized spherical roots} of $\X$, denoted as $\De_{\X}^{sv}$, are obtained by replacing $\al\in\De_{\X}$ with the corresponding primitive elements in the root lattice of $\G$, $\zz\Phi\subset X^\ast(\rA)$.
\end{enumerate}

With this definition, it may happen that $\De_{\X}^{sv}\not\subset\fX$, so one must also re-normalize 
\[
{\fX}^{sv}:=\fX+\zz\De_{\X}^{sv}.
\]
Relying on work of Brion, there is a natural trichotomy of spherical roots.
\begin{Def}
    Let $\ga\in \De_{\X}^{min}$ be an unnormalized spherical root. Then $\ga$ is either
    \begin{enumerate}
        \item \label{type T} proportional to a root $\al\in \Phi$ of $\G$ such that $\al\in \fX$. We say $\ga$ is a root of type $T$,
             \item \label{type N} proportional to a root $\al\in \Phi$ of $\G$ that does not lie in $\fX$. We say $\ga$ is a root of type $N$, or
                  \item \label{type G} proportional to a sum $\al+\be$ where $\al,\be\in \Phi$ are strongly orthogonal roots of $\G$. We say $\ga$ is a root of type $G$.
    \end{enumerate}
\end{Def}
A standard assumption in the harmonic analysis of spherical varieties is to assume that $\X$ has no spherical roots of type $N$. This implies that $\fX^{sv} = \fX$ \cite[Proposition 3.1.6]{SakVenk}. This is a very restrictive assumption in the context of symmetric varieties, so we will not impose this constraint until Section \ref{Section: hamiltonian endoscopy} when we connect with the ideas of \cite{BZSV}.

Let $\De_\X:=\{\ga_1,\ldots,\ga_s\}\subset \fX$ denote the spherical roots. 
\begin{Def} For each $i$, $ \ga^{sv}_i = \sum_{\al\in \De} n_\al \al$ with $n_\al\in \zz$ and $\gcd(n_\al) =1$. For a given spherical root $\ga$, we define its \emph{support} $|\ga|$ to be $\{\al\in \De: n_\al>0\}$. More generally, we set $|\Sigma_0|=\bigcup_{\ga\in \Sigma_0}|\ga|$ for any subset $\Sigma_0\subset \De_\X$.
\end{Def}

\subsubsection{Associated coroots}
There is an obvious bijection $\De_\X\iso \De_{\X}^{sv}$, sending a spherical root $\ga$ to its normalization $\ga^{sv}$. For any spherical root $\ga\in\De_\X$, either $\ga^{sv}\in \Phi^+$ (types $T$ and $N$ roots) or $\ga^{sv}=\ga_1+\ga_2$ is a sum of two strongly orthogonal roots $\ga_1,\ga_2\in \Phi^+$ (type $G$). For any root $\ga$ of type $\G$, there is a unique pair of such roots $\{\ga_1,\ga_2\}$ such that
\[
\check{\ga}_1-\check{\ga}_2 = \check{\de}_1-\check{\de}_2, \text{ for appropriate simple roots}\de_1,\de_2\in \De.
\]
This constraint implies that the restrictions of $\check{\ga}_i$ to ${\fX}^{sv}$ coincide and define a unique element $\check{\ga}\in \check{\fX}^{sv} =\Hom({\fX}^{sv},\zz).$
For spherical roots $\ga\in \Phi^+$ of type $T$ or $N$, we set $\check{\ga}$ for the corresponding coroot in the usual sense. Let $\check{\De}_{\X}:=\{\check{\ga}:\ga\in \De^{sv}_\X\}$ and
\begin{align}
\hat{\De}_\X=\{\check{\ga}:\ga\text{ of type $T$ or $N$}\}\sqcup\{\check{\ga}_1,\check{\ga}_2:\ga\text{ of type $G$}\}.
\end{align}
This latter set is called the associated coroots of $\X$; the corresponding roots $\hat{\Sigma}_\X$ is the set of associated roots. The table below gives the possible set of associated roots for spherical roots of type $\G$. The first column refers to the underlying root system of the rank $1$ boundary degeneration of $\X$ associated to the spherical root $\ga$; see \cite{KnopSchalke} for details.

\begin{table}[htp]
  \caption{Associated roots}
    \label{tab:associated roots}
    \centering
    \begin{tabular}{|c|c|c|c|}
	\hline 
	$|\ga|$ & $\ga_1,\:\ga_2$ & $\ga_1^\vee,\ga_2^\vee$ &$\de_1^\vee,\:\de_2^\vee$\\[.25 cm] \hline 
	$\mathbf{D_2}$ & $\al_1,\:\al_2$  & $\al_1^\vee,\al_2^\vee$ &   $\al_1^\vee,\al_2^\vee$ \\[.25 cm] \hline
	$\mathbf{D_l \: (l\geq 3)}$ & \shortstack{$(\al_1+\cdots+\al_{l-2})+\al_{l-1}$\\$(\al_1+\cdots+\al_{l-2})+\al_{l}$} &\shortstack{$(\al^\vee_1+\cdots+\al^\vee_{l-2})+\al^\vee_{l-1}$\\$(\al^\vee_1+\cdots+\al^\vee_{l-2})+\al^\vee_{l}$} & $\al_{l-1}^\vee,\al_l^\vee$  \\[.25 cm] \hline
	$\mathbf{B_3}$ & $\al_1+\al_2+2\al_3,\: \al_2+\al_3$ & $\al^\vee_1+\al^\vee_2+\al^\vee_3,\: 2\al^\vee_2+\al^\vee_3$&$\al_1^\vee,\al_2^\vee$ \\[.25 cm] \hline
    \end{tabular}
  
\end{table}

\subsection{Combinatorial homogeneous spherical data}\label{Section: spherical data}

We say a simple root $\al\in \De$ is \emph{parabolic} for $\X$ if $x\cdot P_\al =  x\cdot \B$ for any $x\in \X$ in the (unique) open $\B$-orbit $\mathring{\X}\subset \X$; here, $P_\al$ is the minimal non-solvable parabolic subgroup associated to $\al$. We set $\De^p_\X\subset \De$ to be the set of parabolic simple roots of $\X$.


Let $\mathcal{D}(\X) = \D^B(\X)$ denote the set of $\B$-stable prime
divisors of $\X$, called the \emph{colors of $\X$}. Each color $D\in \D(\X)$ defines a $\B$-invariant valuation on $\kbar(\X)^{(\B)}$. We thus obtain a map
\[
\rho:\D(\X)\lra \fa_{\X,\qq}
\]
sending $D$ to the associated valuation $\nu_D$. For $D\in \D(\X)$, let $P_D$ denote the stabilizer of $D$ in $\G$. Clearly, $P_D\supset B$ is a parabolic subgroup. For $\al\in\De$, let $P_\al$ denote the corresponding minimal parabolic
subgroup of $\G$ containing $\B$. Let $\De(D)$ denote the set of $\al\in \De$ such that $P_\al\not\subset P_D$. We obtain a map
\[
\varsigma: \D(\X)\lra \mathcal{P}(\De),
\]
where $\mathcal{P}(\De)$ denotes the power set of $\De$. 

Consider the map 
\begin{equation}\label{eqn: color function}
    \rho\times \varsigma:\D(\X)\lra \fa_{\X,\qq}\times \mathcal{P}(\De),
\end{equation}
and let $\Omega$ denote its image. The fibers of this map over $\Omega$ contain either one or $2$ colors \cite[Appendix B]{BorovoiModels}. Let $\Omega^{(1)},\:\Omega^{(2)}$ denote the corresponding subsets determined by the size of the fiber.
 By a variant Losev’s Uniqueness Theorem \cite[Theorem 1]{Losev}, the invariants
\begin{equation}\label{eqn: spherical invariants}
    \Omega_\X:=(\fX,\De_{\X},\Omega^{(1)},\Omega^{(2)})
\end{equation}
determine the spherical homogeneous space $\X_{\kbar}=\rH_{\kbar}\backslash\G_{\kbar}$ of the reductive $\kbar$-group $\G_{\kbar}$ up to a $\G_{\kbar}$-equivariant isomorphism (the spherical roots $\De_{\X}$ encode the valuation cone $\mathcal{V}$).

On the other hand, our $k$-group $\G$ induces a $\ast$-action $\{\sig_\ast\}_{\sig\in \Ga}$ of $\Ga$ on $\X^\ast(\rA)$ and on $\De$.  As explained in \cite[Section 2]{BorovoiGagliardi}, the $\ast$-action induces for each $\sig\in\Ga$ a homogeneous spherical datum $\sig_\ast\Omega_\X$. A necessary condition for the existence of a $k$-model of $\X_{\kbar}$ is that $\sig_\ast\Omega_\X=\Omega_\X$ for each $\sig\in \Ga$.
\begin{Prop}\cite[Proposition 2.17]{BorovoiGagliardi}
    Suppose that $\G$ is a connected reductive $k$-group and $\X=\rH\backslash\G$ is a spherical homogeneous $k$-variety. The induced $\ast$-action on $\X^\ast(\rA)$ and $\De$ preserves the invariants $\Omega_\X$.
\end{Prop} 

When $\G$ is quasi-split, Borovoi and Gagliardi prove the following converse of this result (see \cite{BorovoiGagliardi} for their more general results). Fix also a $k$-rational Borel subgroup $\B\subset \G$, and let $\D(\X)$ denote the colors of $\X$. This comes equipped with map $\rho\times \varsigma$, defined in \eqref{eqn: color function}. Note that the $k$-rationality of $\B$ implies that the Galois action on $\overline{\X}$ induced by $\mu_\X$ preserves the set $\B$-orbits and $\rho\times \varsigma$ is $\Ga$-equivariant.

 When $\sig_\ast\Omega_\X=\Omega_\X$ for each $\sig\in \Ga$, we obtain a continuous $\Ga$-action on $\Omega$, and one may always choose a continuous lift
\[
\al_{\D}: \Ga\times \D(\X)\lra \D(\X),
\]
which amounts to giving an action of $\Ga$ on the fibers over $\Omega^{(2)}$.
\begin{Prop}\label{Prop: quasirational}\cite[Theorem 6.13 and 6.15]{BorovoiGagliardi}
    Suppose that $\G$ is quasi-split and that $\overline{\X}=\overline{\rH}\backslash\G_{\kbar}$ is a spherical homogeneous space of $\G_{\kbar}$. Assume that the corresponding $\ast$-action $\{\sig_\ast\}_{\sig\in \Ga}$ determined by $\G$ preserves $\Omega_\X$. Then for any lift $\al_{\D}$ of the $\Ga$-action from $\Omega$ to a continuous action on $\D(\X)$, there exists a $\G$-equivariant $k$-model $\X$ of $\overline{\X}$ inducing $\al_{\D}$. Moreover, $\X(k)\neq \emptyset$, so that $\X=\rH\backslash\G$ for a $k$-rational subgroup $\rH\subset \G$.
\end{Prop}

\quash{Combining this result with Lemma \ref{Lem: outer forms classify}, we obtain the following corollary.
\begin{Cor}\label{Cor: up to G-inner}
   Let $\G$ be a quasi-split reductive group over $k$. To a spherical homogeneous $\G$-variety $\X=\rH\backslash\G$, consider the \textbf{combinatorial outer data}
\[
\Psi_\X:=(\Omega_\X,\mu_\X),\qquad\text{where }\Omega_X:=(\fX,\De_{\X},\Omega^{(1)},\Omega^{(2)})
\]
 $\mu_\X:\Ga\lra \Out_\X(\rH)$ is a $1$-cocycle, and $\Ga$ acts on the combinatorial data. This data determines $\X$ up to $\G$-inner twist. More precisely, two spherical $\G$-varieties $\X$ and $\Y$ inducing isomorphic combinatorial outer data (with the obvious notion of isomorphism) are $\G$-inner forms.
\end{Cor}}

The following is a simple example where distinct lifts $\al_{\D(\X)}$ of the Galois action exist. It illustrates the issue of determining what possible $\G$-outer forms may exists.

\begin{Ex}\label{Ex: unitary example}
 Let $E/k$ be a quadratic extension of fields and let $(V,\la\cdot,\cdot\ra)$ be a $2$-dimensional Hermitian $E$-vector space containing an isotropic line. Let $\G= \U(V)$ denote the corresponding quasi-split unitary group. We assume that the Hermitian form is represented by the matrix
    \[
    J = \ep\begin{psmatrix}
        &1\\-1&
    \end{psmatrix},\qquad \ep\in E_{tr=0}.
    \]
    Thus, $\G = \{g\in \GL(V): J{}^T\overline{g}^{-1}J^{-1} = g\}.$ Now we consider the two involutions 
    \[
    \theta_t = \Ad\begin{psmatrix}
        &t^{-1}\\t&
    \end{psmatrix},\qquad t\in \{1,\ep\}.
    \]
    We compute
       \begin{align*}
    \rH_t:=\G^{\theta_t}=\left\{\begin{psmatrix}
    a&b\\t^2 b&a
\end{psmatrix}: \Nm(a)-t^2 \Nm(b)=1,\quad a\overline{b}=\overline{a}b\right\};
    \end{align*}
    so that $\rH_1\simeq  \Res_{E/k}(\Gm)$ while $\rH_\ep \simeq \Nm^1_{E/k}(\Gm)^2$.
    We have
       \[
   \X_t =\rH_t\backslash\G= \left\{\begin{psmatrix}
    x&y\\-t^2 {y}&z
\end{psmatrix}: x,y,z\in k,\: xz+t^2 y^2=1\right\}.
 \] 
 Then $\Ax\simeq \rA$ with the canonical map $\rA\lra \Ax$ being the squaring map. In particular, ${\De}^n_\X=\{2\al\}$. Since $-I\in \X(k)$, we see that $\al\in \fX$. The simple calculation now shows that the divisor $\{z=0\}$ is stable under the upper triangular Borel subgroup with cocharacter $\check{\omega} = \frac{\check{\al}}{2}\in X_\ast(\Ax)$, geometrically with two irreducible components $\{D_1,\D_2\}=\mathcal{D}(\X_t)$.
 
The two symmetric $k$-varieties $\X_t$ become isomorphic upon base changing to $E$, so that they have isomorphic homogeneous spherical data
 \[
 \Omega_\X =(2X^\ast(\rA), \{\al\},\emptyset,\{(\check{\omega},\{\al\})\})
 \]
 and the fiber $(\rho\times \varsigma)^{-1}(\check{\omega},\{\al\})=\{D_1,D_2\}$ consists of the two colors in $\X_{\kbar}$. Note that
 \begin{align*}
     X_{t,\{z=0\}}(k) 
&=\begin{cases}
    \left\{\pm \begin{psmatrix}
    x&1\\-1&
\end{psmatrix}\right\}&: t=1,\\
\qquad\emptyset&: t = \al, \text{ since } \ep^2\notin (k^\times)^2.
\end{cases}
 \end{align*}
 Thus, the $\Ga$-action on $\{D_1,D_2\}$ is trivial when $t=1$ and acts non-trivially through the quotient $\Gal(E/k)$ when $t=\ep$. \qed 
\end{Ex}
\subsection{Automorphisms and forms of varieties}For this section, we let $\overline{\G}$ be a reductive group over $\kbar$ and $\overline\X=\overline\rH\backslash\overline{\G}$ be a homogeneous spherical $\overline{\G}$-variety, and let $\G$ (respectively, $\X$) denote a given $k$-form of $\overline{\G}$ (resp. $\overline\X$). 
 \quash{As in \cite[Section 2]{BorovoiGagliardi}, we thus obtain a $\Ga$-semilinear equivariant action $\mu:\Ga\lra \mathrm{SAut}^{\overline{\G}}(\X)$ satisfying that 
\[
\mu_\sig(g\cdot x) = \sig(g)\cdot \mu_\sig(x)\qquad\text{ for all $\sig\in \Ga$, $g\in \overline{\G}(\kbar)$ and $x\in \X(\kbar)$.}
\]
Note that any algebraic semilinear action on $\X_0$ descends to a $k$-model \cite[Section 2]{BorovoiGagliardi}.}

Recall from \cite{KnopAutomorphisms} that the group $\Aut^{\overline{\G}}(\overline\X)=\overline\rH\backslash N_{\overline{\G}}(\overline\rH)$ of $\overline{\G}$-automorphisms is diagonalizable and comes equipped with a canonical inclusion $\Aut^{\overline{\G}}(\overline\X)\hra \rA_{\overline\X}$. The following result considers the $k$-rational structure on this group determined by the quasi-split form.

\begin{Prop}\cite[Appendix B]{BorovoiGagliardi}
    Suppose that $\G$ is a quasi-split $k$-form of $\overline{\G}$ and suppose that $\X$ is a $\G$-model of $\overline\X$. The model $\X$ induces a $\Ga$-action on the automorphism group $\Aut^{\overline{\G}}(\overline{\X})$, with respect to which the embedding
      \begin{equation}\label{eqn: embedding auts}
           \Aut^{\overline\G}(\overline\X)\hra \rA_{\overline{\X}}
      \end{equation}
  is equivariant. In particular, $\cala_\X:=\Aut^\G(\X)$ is a $k$-form of $\Aut^{\overline\G}(\overline\X)$ with a $k$-rational structure and the embedding \eqref{eqn: embedding auts} descends to an embedding of diagonalizable group $k$-schemes $\mathcal{A}_\X\hra \rA_{\X}$. 
\end{Prop}
We note that the $\Ga$-action on $X^\ast(\Ax)$ depends on the $k$-form $\X$, but the $\Ga$-action on $\Lam_\X$ depends only on the $\ast$-action associated to $\G$.

We now assume that $\G$ is quasi-split over $k$ for the remainder of this section and fix a $k$-rational Borel $\B\subset \G$. Assume that the $\ast$-action preserves the combinatorial data of $\overline{\X}$. Proposition \ref{Prop: quasirational} implies the existence of a $\G$-form $\X$ for any choice $\al_\cald$, and we set $\mathcal{A}_\X:=\Aut^{\G}(\X)$ for the canonical $k$-group described above.
The $k$-form $\G$ induces a homomorphism $\mu:\Ga\lra \rS\Aut(\overline{\G})$ (see Section \ref{Sec: forms and action}) and an exact sequence 
\begin{equation}\label{eqn: sequence on auts}
    1\lra \cala_\X(\kbar)\lra \rS\cala_\X\lra \Ga,
\end{equation}
where $\rS\cala_\X$ is the group of $\mu$-semi-linear automorphisms of $\overline{\X}$.

\subsubsection{Doubled and distinguished roots}\label{Sec: dist roots}
One of the main results of \cite{Losev} is to clarify the distinction between $\De_\X$ and $\De_\X^n$, which effectively calculates the root lattice $\Lam_\X=\zz\De_\X^n$. This requires the introduction of so-called \emph{distinguished spherical roots}.
\begin{Lem}\label{Lem: spherical roots distinguished}
    If $\X= \rH\backslash\G$ is a homogeneous spherical variety, then ${\De}^{n}_{\X}$ is obtained from $\De_{\X}$ by replacing any $\ga\in \De_{\X}$ satisfying \eqref{a}, \eqref{b}, \eqref{c} or \eqref{d} below by ${2}\ga$
    \begin{enumerate}[(A)]
        \item \label{a} (Type $A$, or doubled root) $\ga=\al\in\De$ such that there exists $D\in \D$ with $\rho(D) = \frac{1}{2}\al^\vee|_{\fa_{\X}}$;
        \item \label{b} (type $B$) there is a subset $\Sigma\subset \De$ of type $B_k$ with $k\geq2$ such that 
        \[
        \ga= \al_1+\al_2+\cdots+\al_k,
        \]
        and $\al_i\in \De^p_\X$ for $i>1$;
        \item \label{c} (type $\G_2)$  there is a subset $\{\al_1,\al_2\}\subset \De$ of type $\G_2$ with $\al_1$ short such that 
        \[
        \ga= 2\al_1+\al_2;
        \]
        \item \label{d} $\ga\in X^\ast(\Ax)$ but $\ga\notin \zz\Phi$,
    \end{enumerate}
\end{Lem}
\begin{proof}
    This is \cite[Theorem 2]{Losev}.
\end{proof}
More precisely, Losev identifies roots of type \eqref{a}, \eqref{b}, and \eqref{c} as \emph{distinguished roots}. Note that roots of type \eqref{c} do not appear as spherical roots of symmetric varieties.  We thus refer only to type  \eqref{a} or \eqref{b} distinguished roots in Section \ref{Section: color autos}. 

\begin{Rem} Roots of type \eqref{d} correspond to automorphisms arising from $Z(\G)$, and so do not play the same role. 
\end{Rem}

Distinguished roots of type \eqref{a} play a special role in the general theory. Knop defines a natural subgroup
\[
\mathcal{A}^\sharp_\X:=\{\phi\in \cala_\X: \phi\text{ stabilizes each } D\in \D(\X)\},
\]
which may be identified with the subgroup of $\cala_\X$ inducing the trivial automorphism all the data $(\fX,\De_\X,\D({\X}))$. Since our Borel subgroup $\B$ is $k$-rational, this subgroup is defined over $k$.
\quash{This subgroup may be characterized as follows.
\begin{Lem}\cite[Lemma 7.3]{KnopAutomorphisms}\label{Lem: characterize the center}
    For $\phi\in \cala_\X$, the following are equivalent:
    \begin{enumerate}
        \item $\phi\in \mathcal{A}^\sharp_\X$,
        \item $\phi$ acts trivially on $\mathrm{Pic}(\X)$,
        \item $\phi$ acts trivially on $\mathrm{Pic}^{\G}(\X)\simeq X^\ast(\rH)$.
    \end{enumerate}
    In particular, $\cala_\X= \mathcal{A}^\sharp_{\X}$ if $X^\ast(\rH)=0$.
\end{Lem}
}

Set $\Aut_{\Omega}(\D(\X))\subset \Aut(\fX,\De_\X,\D({\X}))$ to be the subgroup of automorphisms acting trivially on $\Omega_\X$, but acting on the fibers of $\rho\times \varsigma$ over $\Omega^{(2)}$. Then $\Aut_{\Omega}(\D(\X))$ consists of automorphisms that swap two \emph{undetermined} colors $\{D^+,D^-\}$ that lie over the same $\check{\ga}$. This is a finite \'{e}tale group scheme, and we have a short exact sequence of diagonalizable $k$-group schemes \cite[(7.7)]{BorovoiGagliardi}
\begin{equation}
    1\lra \mathcal{A}_{\X}^\sharp\lra \Aut^{G}(\X)\lra \Aut_{\Omega}(\D(\X))\lra 1.
\end{equation}

This quotient is well understood: we have the dual sequence of $\Ga$-modules
\[
0\lra X^\ast( \Aut_{\Omega}(\D(\X)))\lra X^\ast(\Ax)/\Lam_{\X}\lra X^\ast(\Ax)/\Lam_{\X}^\sharp\lra 0,
\]
where $\Lam_{\X}^\sharp$ is the sublattice spanned by the set $\De_\X^\sharp$ obtained from $\De_{\X}^n$ by replacing $2\al$ with $\al$ for every distinguished root of type \eqref{a}.
\begin{Rem}
    It turns out that $\De_\X^\sharp$ is the set of spherical roots of the \emph{spherical closure} $\rH^\sharp$ of $\rH$. More generally, $\X$ is said to be \emph{spherically closed} when $\cala_\X^\sharp = \{1\}$.
\end{Rem}
\subsection{The group of doubling automorphisms}\label{Section: doubling aut}
We generalize this by introducing a quotient $\Aut_d(\X)$ of $\cala_\X$ which plays an important role in the rational theory of spherical varieties. For functoriality reasons, our definition will proceed in two steps: first the case when $\rH$ is connected, and then the general case. Section \ref{Section: geometric cocycle}, discusses the relationship between this group and the issue of $\G$-outer forms to the geometry of $\X$.
\subsubsection{The connected case}
We first assume that $\X=\rH\backslash\G$ with $\rH$ connected. Let $\De_\X^{dist}\subset \De_\X$ denote the subset of distinguished spherical roots, and consider the set $\Sigma^{d}$ obtained from the normalized spherical roots $\De_\X^{n}$ by replacing $2\al$ with $\al$ for every $\al\in \De_\X^{dist}$ (so that only those roots of type \eqref{d} are doubled), and set $\Lam^d:=\zz\Sigma^d$. This is clearly $\Ga$-stable, so that we obtain a short exact sequence of $\Ga$-modules
\[
0\lra \Lam_\X^d/\Lam_\X\lra \fX/\Lam_\X\lra \fX/\Lam_\X^d\lra0
\]
and a dual sequence of diagonalizable $k$-schemes
\begin{equation}\label{eqn: distinguished group ses}
    1\lra\cala_\X^d\lra \cala_\X\lra \Aut_d(\X)\lra 1;
\end{equation}
we refer to the quotient group $\Aut_d(\X)$ as the group of \emph{doubling automorphisms}. To motivate the definition, recall that $\cala_\X\subset \Ax$ so that for any $\ga\in \fX = X^\ast(\Ax)$, we set $\la \ga, a\ra\in \kbar^\times$ for $a\in \Ax(\kbar)$.
\begin{Def}\cite[Definition 4.1.6]{Losev}
 For $\al\in \De_\X^{dist}$, an automorphism $a\in \cala_\X(\kbar)$ is said to \emph{double} $\al$ if $\la \al,a\ra=-1$.
\end{Def}
In particular, each $\al\in \De_\X^{dist}$ determines a canonical character $\mu_\al=\la\al,-\ra:\cala_\X(\kbar)\lra \mu_2(\kbar)$ via this pairing. In particular,
\[
\Aut_d(\X)_{\kbar}\overset{\prod_\al\mu_\al}{\lra}\prod_\al\mu_2
\]
is a $\kbar$-isomorphism. The $k$-structure on $\cala_\X$ induces a $k$-form on the quotient. 
 Set $I_\X$ be a finite index set for the $\Ga$-orbits on $\De_\X^{dist}$. We have the $\Ga$-orbit decomposition
\begin{equation}\label{eqn: decomp of dist}
\De_\X^{dist}=\bigsqcup_{i\in I_\X} \calo_i=\bigsqcup_{i\in I_\X} \Ga/\Ga_i\cdot \ga_i,
\end{equation}
where $\ga_i\in \calo_i$ and $\Ga_i = \Gal(\kbar/k_i)$ is the stabilizer of $\ga_i$. This orbit structure completely determines the $\Ga$-action on 
\[
\Lam_\X^d/\Lam_\X\simeq \zz\De_\X^{dist}/2\zz\De_\X^{dist}.
\]
More precisely, 
\begin{equation}\label{eqn: decomp of Out general}
    X^\ast(\Aut_d(\X))=\bigoplus_{i\in I_\X}\Ind_{\Ga_i}^{\Ga}(\zz/2\zz\ga_i)\simeq (\zz/2\zz)^{\De_\X^{dist}}.
\end{equation}
This induces a product formula
\begin{equation}\label{eqn: formula for Aut conn}
    \Aut_d({\X}) =  \prod_{i\in I_\X} \Res_{k_i/k}(\mu_2).
\end{equation}
The next lemma asserts that $\Aut_d({\X})$ depends only on $\G$, and not on the specific $k$-form $\X$.

\begin{Lem}\label{Lem: unique on aut}
  Suppose that $\G$ is quasi-split and that $\overline{\X}=\overline{\rH}\backslash\G_{\kbar}$ is a spherical homogeneous space of $\G_{\kbar}$. Assume that the corresponding $\ast$-action $\{\sig_\ast\}_{\sig\in \Ga}$ determined by $\G$ preserves $\Omega_\X$. The $k$-structure on the group $\Aut_d(\X)$ is independent of the choice of $\G$-equivariant $k$-model $\X$ of $\overline{\X}$.
\end{Lem}
\begin{proof}
The assumption that $\Ga$ stabilizes $\Omega_\X$ induces a $\Ga$-action on $\De_\X^n$ independent of the choice of $k$-form. This determines the $\Ga$-actions on both $\Lam_\X=\zz\De_\X^n$ and $\De_\X^{dist}\subset \De\X$. In particular, this determines the action on $\Lam_X^d$, hence on 
\[
\Lam^d_\X/\Lam_\X.
\]
By duality, this determines the diagonalizable $k$-group $\Aut_d(\X)$ up to isomorphism. 
\end{proof}
Moreover, we have the following surjectivity result.
\begin{Lem}\label{Lem: surjective on dist cohom}
   Suppose that $\X=\rH\backslash\G$ is spherical and that $\rH$ is geometrically connected. In the long exact sequence on cohomology induced by \eqref{eqn: distinguished group ses}, the map 
    \[
    H^1(k,\cala_\X)\lra H^1(k,\Aut_d(\X))
    \]
    is surjective.
\end{Lem}
\begin{proof}
    We show that 
    \[
    i:H^2(k, \cala_X^d)\to H^2(k,\cala_\X)
    \]
    is injective. This proves the claim via the long exact sequence in cohomology. For this, let recall that $\fX = X^\ast(\Ax)$, and let $\rA^n_{\X}$ (resp., $\rA^d_{\X}$) be the $k$-torus with character group $\Lam_\X$ (resp. $\Lam_\X^d$). We have the commutative diagram (with exact rows)
    \[
    \begin{tikzcd}
        0\ar[r]&\Lam_X\ar[r]\ar[d]&\fX\ar[r]\ar[d,"="]&X^\ast(\cala_X)\ar[r]\ar[d]&0\\
        0\ar[r]&\Lam^d_X\ar[r]&\fX\ar[r]&X^\ast(\cala^d_X)\ar[r]&0,
    \end{tikzcd}
    \]
    which is dual to the commutative diagram (with exact rows)
        \[
    \begin{tikzcd}
        1\ar[r]&\cala_\X^d\ar[r]\ar[d]&\Ax\ar[r]\ar[d,"="]&\rA^n_{\X}\ar[r]\ar[d]&1\\
        1\ar[r]&\cala_X\ar[r]&\Ax\ar[r]&\rA^d_{\X}\ar[r]&1.
    \end{tikzcd}
    \]
    Passing to cohomology, we have the following commutative diagram (with exact rows)
            \[
    \begin{tikzcd}
        H^1(k,\rA^d_{\X})\ar[r]\ar[d]&H^2(k,\cala_\X^d)\ar[r]\ar[d,"i"]&H^2(k,\Ax)\ar[d,"="]\\
        H^1(k,\rA^n_{\X})\ar[r]&H^2(k,\cala_X)\ar[r]&H^2(k,\Ax).
    \end{tikzcd}
    \]
    Since each of the lattices $\Lam_\X$ and $\Lam_\X^d$ possess a $\Ga$-stable $\zz$-basis ($\De^n_\X$ and $\De_\X^d$, respectively), the first cohomology groups of the tori $\rA^n_{\X}$ and $\rA^d_{\X}$ vanish by Shapiro's lemma and Hilbert's theorem 90. This forces $i$ to be injective, proving the claim.
\end{proof}

\subsubsection{The general case}\label{Section: doubling aut gen}
 Now suppose that $\pi_0(\rH)$ is not the trivial group $k$-scheme, set $\rH^\circ$ to be the connected component of the identity and set $\X^\circ= \rH^\circ\backslash\G$, so that the map $\X^\circ\to \X$ is finite \'etale (recall our assumption that $\rH$ is smooth in the positive characteristic setting). The definition given above works just as well for $\rH$, but the group it produces is generally ``too small'' in the sense to be made precise below  (see Remark \ref{Rem: interpretting Aut gen}). The following example illustrates one of the deficiencies of such an approach.

\begin{Ex}\label{Ex: SLxSL disconnect}
    Let $\G=\SL_2\times\SL_2$ and let $\rH = (T_0\times T_0)\sqcup(T_0w\times T_0w)$, where $T_0$ is the maximal torus fixed pointwise by $\theta(g) ={}^Tg^{-1}$ and $w\in N_{\SL_2}(T_0)\setminus{T_0}$. Then $\rH^\circ=T_0\times T_0$ is the connected component of the identity and the induced map
    \[
    \X^\circ=\rH^\circ\backslash\G\lra \X=\rH\backslash\G
    \]
    is an \'{e}tale double cover.   Let $\rA = T\times T$ be the diagonal torus, so we have the two canonical tori $\rA_\X^\circ$ and $\Ax$ with surjective morphisms
    \[
\rA\lra \rA_{\X^\circ}\lra \Ax.
    \]
    If $X^\ast(\rA) = \zz\omega_1\oplus\zz\omega_2$, then  $X^\ast(\rA_{\X^\circ}) = \zz(2\omega_1)\oplus\zz(2\omega_2)$ and 
    \[
     X^\ast(\rA_\X) = \zz(4\omega_1)+\zz(4\omega_2)+ \zz(2\omega_1+2\omega_2).
    \]
    In particular, we see that $\De_{\X^\circ}^n = \De_\X^n = \{4\omega_1,4\omega_2\}$;    this implies 
  $$
  \cala_{\X^\circ} \simeq\mu_2\times \mu_2,\text{ and } \cala_{{\X}} \simeq \mu_2\times \mu_2/\De\mu_2\simeq \mu_2.
   $$  
   While both spherical roots are distinguished for $\X^\circ$, they are both of type $N$ for $\X$. This implies that  $\Aut_d(\X^\circ)\simeq \mu_2\times \mu_2$, while $\Aut^{naive}_d(\X) = 1$ if we use the na\"ive definition of mirroring the definition of $\Aut_d(\X^\circ)$ for $\X$.  \qed
\end{Ex}
The issue in this example is that while $\Aut_d(\X^\circ)$ encodes a great deal of arithmetic information about $\X^\circ$, this definition of $\Aut^{naive}_d(\X)$ forgets too much information and fails to be functorial with respect to the natural map $\X^\circ\to \X$. More specifically,  $\Out_{\X^\circ}(\rH^\circ)\simeq\Aut_d(\X^\circ)$ while $\Out_{\X}(\rH)\simeq \mu_2$ is bigger than $\Aut^{naive}_d(\X)$.

To remedy this, we define
\begin{equation}
    \Aut_d(\X):=\Aut_d(\X^\circ)/\Aut_d(\pi_0(\rH)),
\end{equation}
where we slightly abuse notation to mean that we quotient out by the image of $\pi_0(\rH)$ under the morphism to $\Aut_d(\X^\circ)$. This corresponds to setting $\Lam_\X^d:= \fX\cap \Lam_{\X^\circ}^d$, so that $\Aut_d(\X)$ is dual to
\[
(\fX\cap \Lam_{\X^\circ}^d)/\Lam_\X\subset \Lam_{\X^\circ}^d/\Lam_\X;
\]
the same argument as Lemma \ref{Lem: unique on aut} applies in this case. This definition ensures that there exists a commutative diagram of diagonalizable $k$-groups with exact rows
\begin{equation}\label{eqn: functorial on Aut}
\begin{tikzcd}
    \pi_0(\rH)\ar[d]\ar[r]&\cala_{\X^\circ}\ar[d]\ar[r]&\cala_{\X}\ar[d]\\
  \Aut_d(\pi_0(\rH))\ar[r]&  \Aut_d(\X^\circ)\ar[r]&\Aut_d(\X).
\end{tikzcd}
\end{equation}
Note that $\Aut_d(\X^\circ)\simeq\Aut_d(\X)$ if $\rH^\circ\subset \rH\subset \rH^\circ\cdot Z(\G)$.

\begin{Rem}\label{Rem: interpretting Aut gen}
This will be our main application of the functorial properties of $\Aut_d(\X)$. It may be enhanced to give a functor from the category of quasi-affine homogeneous spherical $\G$-varieties with dominant $\G$-equviariant morphisms to the category of diagonalizable group $k$-schemes, but we do not need this.
%
\end{Rem}

If $H^1(k,\Aut_d(\X^\circ))\to H^1(k,\Aut_d(\X))$ is surjective, then certainly the map $H^1(k,\cala_\X)\to H^1(k,\Aut_d(\X))$ is also surjective so that Lemma \ref{Lem: surjective on dist cohom} extends to this case. \quash{Indeed, functoriality gives a commutative square
    \[
    \begin{tikzcd}
         H^1(k,\cala_{\X^\circ})\ar[r]\ar[d]& H^1(k,\Aut_d(\X^\circ))\ar[d]\\
          H^1(k,\cala_\X)\ar[r]& H^1(k,\Aut_d(\X)).
    \end{tikzcd}
    \]
    Lemma \ref{Lem: surjective on dist cohom} implies the top horizontal arrow surjects. Combined with commutativity, the bottom horizontal arrow is forced to be surjective if the right vertical arrow is.} This surjectivity need not occur. 
\quash{    \begin{Ex}\label{Ex: components gone wild}
    Suppose that $K=\prod_{i\in I}k_i$ is a finite \'etale $k$-algebra, where the finite index set $I$ and the field extensions $k_i$ ($i\in I$) are arbitrary. Let $\G=\Res_{K/k}(\SL_2)$ with $\rH^\circ=\Res_{K/k}(\Gm)$ with its normalizer $\Res_{K/k}(N)$. Then if $\X^\circ =\rH^\circ\backslash\G$, we see
    \[
\cala_{\X^\circ} \simeq\Res_{K/k}(\mu_2)\simeq \prod_{i\in I}\Res_{k_i/k}(\mu_2).
    \]
    Let $\rH\subset \Res_{\cc/\rr}(N)$ be the $k$-rational subgroup corresponding to the subgroup $\De\mu_2\subset\Res_{K/k}(\mu_2)$. Then $I_{\X^\circ} \simeq I$, while $\pi_0(\De)/\Ga =\{\ast\}$ so that every orbit is identified in the quotient. \qed
\end{Ex}}
Nevertheless, we expect the following to hold.

\quash{
we make a choice $\mathrm{i}$ of representatives $i\in [i]\in I_{\X}$ for each equivalence class, there exist $k$-rational splitting of the quotient $\Aut_d(\X^\circ)\to \Aut_d(\X)$
\begin{align*}
    s_{\mathrm{i}}: \Aut_d(\X)&\lra \Aut_d(\X^\circ)\\
        (\ep_{[i]})_{[i]\in I_{\X}}&\longmapsto (\ep^\ast_i)_{i\in I_{\X^\circ}},
\end{align*}
where $\ep_j^\ast = \ep_{[i]}$ for $j\notin \mathrm{i}\cap[i]$ and $\ep_j^\ast =1$ otherwise. }

\begin{Conj}\label{Conj: cohom surj}
      Suppose that $\G$ is a quasi-split reductive $k$-group and $\X=\rH\backslash\G$ is a homogeneous spherical $\G$-variety. The map 
    \[
    H^1(k,\cala_\X)\lra H^1(k,\Aut_d(\X))
    \]
    is surjective.
\end{Conj}
It is not hard to show that it holds in the following cases.

\begin{Cor}\label{Lem: examples of conj}
 Conjecture \ref{Conj: cohom surj} holds in the following contexts:
    \begin{enumerate}
        \item When $\rH^\circ\subset \rH\subset \rH^\circ\cdot Z(\G)$;
        \item if $\X$ is spherically closed (this includes the case $\X$ is automorphism-free);
        \item when the $\Ga$-action on $\De_{\X^\circ}^{dist}$ is trivial.
    \end{enumerate}
\end{Cor}
\subsubsection{Geometric cocycles}\label{Section: geometric cocycle} Now let $\X=\rH\backslash\G$ be a homogeneous spherical $\G$-variety. We first assume that $\rH^\circ\subset \rH\subset \rH^\circ\cdot Z(\G)$, so that $\Aut_d(\X)$ is of the form \eqref{eqn: formula for Aut conn}. Recall the exact sequence \eqref{eqn: sequence on auts}. 
Quotienting out by the $\Ga$-stable normal subgroup $\cala_\X^d(\kbar)$, we obtain the commutative diagram

        \[
    \begin{tikzcd}
        1\ar[r]&\cala_\X(\kbar)\ar[r]\ar[d]&\rS\cala_\X\ar[r]\ar[d]&\Ga\ar[d,"="]\\
    1\ar[r]&\Aut_d(\X)(\kbar)\ar[r]&\rS\Aut_d(\X)\ar[r]&\Ga,
    \end{tikzcd}
    \]
where $\rS\Aut_d(\X) = \rS\cala_\X/\cala_\X^d(\kbar)$. The $k$-structure of $\X$ gives a splitting $\mu_\X:\Ga\to \rS\cala_\X$, which descends to a splitting $\Ga\to \rS\Aut_d(\X)$. Recalling \eqref{eqn: formula for Aut conn} and Lemma \ref{Lem: unique on aut}, we have an isomorphism $\Aut_d(\X)(\kbar)\iso\mu_2(\kbar)^{\De_\X^{dist}}$ independently of the choice of $k$-form. This gives rise to a canonical extension
\[
\begin{tikzcd}
1\ar[r]&\mu_2(\kbar)^{\De_\X^{dist}}\ar[r]&\mu_2(\kbar)^{\De_\X^{dist}}\rtimes \Ga\ar[r]&\Ga,
\end{tikzcd}\] where the group law on $\mu_2(\kbar)^{\De_\X^{dist}}\rtimes \Ga$ is given by
\[
[(\ep^\sig_\al)_\al,\sig][(\ep^\tau_\al)_\al,\tau] = [(\ep^{\sig}_\al\ep^\tau_{\sig(\al)})_\al,\sig\tau].
\]
On the other hand, when acting on $\overline{\X}(\kbar)$ by a semi-automorphism, the doubling automorphisms $\mu_\al(s)\in\mu_2(\kbar)=\{\pm1\}$ are still defined, but the semi-linearity forces $\mu_\al(s t) = \mu_\al(s)\mu_{\sig(\al)}(t)$ where $s$ lies over $\sig\in \Ga$. We obtain the isomorphism $\ep: \rS\cala_\X\lra \mu_2(\kbar)^{\De_\X^{dist}}\rtimes \Ga$ denote the morphism
\[
\ep(s) =[(\mu_\al(s))_\al,\sig],
\] where $s$ lies over $\sig\in \Ga$.

\begin{Def}\label{Def: geometric cocycle}
Suppose that $\X$ is a $\G$-form of $\overline\X$, inducing a section $\mu_\X: \Ga\lra \rS\cala_\X$, defining a $\nu$-equivariant $\Ga$-action on $\overline{\X}(\kbar)$. Composing with $\ep$, we obtain a section
\begin{equation*}
    \widetilde{\mu}_\X^{dist}:\Ga\overset{\mu_\X}{\lra}\rS\cala_\X\overset{\ep}{\lra}\mu_2(\kbar)^{\De_\X^{dist}}\rtimes \Ga.
\end{equation*}
Comparing with the trivial section $\sig\mapsto [(1)_\al,\sig]$, we thus obtain a \emph{canonical} $1$-cocycle 
\begin{equation}\label{eqn: geometric cocycle}
    \mu_\X^{d}:\Ga\lra \mu_2(\kbar)^{\De_\X^{dist}} = \prod_{\al\in \De_\X^{dist}}\{\pm1\}.
\end{equation}
We refer to this as the \textbf{geometric cocycle} of $\X$.   
\end{Def}
\begin{Rem}\label{Rem: geometric}
    The name ``geometric'' is motivated by how this cocycle records the actions of $\Ga$ on the geometric structures of $\overline{\X}$, at least when $\rH\subset \rH^\circ\cdot Z(\G)$.  For example, if $\al\in \De_\X^{(2)}\subset \De_\X^{dist}$ is a doubled root, we have $\mu_\al(\sig) = -1$ if and only if the automorphism by which $\sig$ acts on $\overline{\X}(\kbar)$ swaps the two colors $D^+$ and $D^-$ associated to $\al$ \cite[Remark 4.1.7]{Losev} (see also \cite[Proposition C.5]{BorovoiGagliardi}). 

For a distinguished root $\ga$ of type \eqref{b} or\eqref{c}, there is a unique color $D_\ga$ associated to $\ga$ so that the automorphism acts trivially on $(\fX,\De_\X,\D(\X))$. Nevertheless, in this case there exist $\B$-orbits of \emph{higher codimension} in the closure of $D_\ga$ which are permuted by automorphisms. The doubling character $\mu_\ga$ encodes whether the an automorphism permutes these orbits or stabilizes them. See also the discussion in \cite[Section 4.8]{BZSV}, where the authors use the language of ``colors of even sphere type.'' The next example illustrates this point geometric nature of a doubling automorphism of type \eqref{b}.
\end{Rem}
\begin{Ex}[Doubling automorphism for $\Sp_{2n}\times \Sp_{2n}\backslash\Sp_{4n}$\:]\label{Ex: Symplectic example}
Let $(V,\la\cdot,\cdot\ra)$ be a $4n$-dimensional symplectic space and let $\G=\Sp(V)$. Let $\X$ be the variety of decompositions $V=V_1\oplus V_2$ into $2n$-dimensional non-degenerate symplectic subspaces of $V$, so that the sum is orthogonal. Then after fixing a base point, we identify $$\X\simeq \Sp(V_1)\times \Sp(V_2)\backslash\Sp(V).$$
This symmetric variety possesses a wonderful compactification
\[
\X\hra  \mathrm{Gr}_{2n}(V)
\]
into the Grassmannian of $2n$-planes in $V$, with open orbit $\X\subset \mathrm{Gr}_{2n}(V)$ corresponds to those partial flags
\[
0\subset W\subset V
\]
such that $\la\cdot,\cdot\ra|_{W\otimes W}$ is non-degenerate. The unique $\Sp(V)$-equivariant automorphism $\varphi\in\cala_\X$ doubling the distinguished spherical root corresponds to sending $W$ to its annihilator
\begin{align*}
    \varphi:\X&\lra \X\\
            W&\longmapsto W^\perp.
\end{align*}

Fix a Borel $\B\subset \Sp(V)$, which determines a full symplectic flag
\[
0\subset U_1\subset  \cdots\subset U_{2n}\subset \cdots \subset U_{4n-1}\subset V,
\]
where $U_{2n}$ is a Lagrangian and $U_{4n-i+1}/U_{4n-i}\simeq (U_{i+1}/U_i)^\ast$.
 Let $\ga\in \De_{\X}^{dist}$ denote the unique distinguished root (which is of type \eqref{b}). The unique color $D\in \D(\X)$ associated to $\ga$ may be described as
    \[
    D=\{W\subset V: \dim(U_{2n}\cap W)=1\}.
    \]
 Note that if $W\in D$, then $W^\perp\in D$ as well since $V=W\oplus W^\perp$, reflecting that $\varphi\in \cala_\X^\sharp(\kbar)$ stabilizes the color.

One now  considers the $B$-stable subvariety $D''=\{W\subset V: U_1\subset W\}\subset D$ of codimension $2$, and notes
 \[
 \varphi(D'') =\{W\subset V: U_1\subset W^\perp\}\neq D''.
 \]
Thus, $\varphi$ acts non-trivially on Borel orbits of higher codimension which lie in the closure of $D$.
 \qed
\end{Ex}

If we now allow more general $\pi_o(\rH)$, let $\X =\rH\backslash\G$ and $\X^\circ = \rH^\circ\backslash\G$. Consider the quotient  $m_{\rH}$ defined by the diagram
\begin{equation*}
    \begin{tikzcd}
        \Aut_d(\X^\circ)(\kbar)\ar[d]\ar[r]&\Aut_d(\X)(\kbar)\ar[d]\\
\mu_2(\kbar)^{\De_{\X^\circ}^{dist}}\ar[r,"m_{\rH}"]&\mu_2(\kbar)^{\De_{\X^\circ}^{dist}}/\Aut_{d}(\pi_0(\rH))&.
    \end{tikzcd}
\end{equation*}
The $k$-rationality of $\Aut_d(\pi_0(\rH))\subset \Aut_d(\X^\circ)$ implies that the subgroup 
\[
\Aut_d(\pi_0(\rH))(\kbar) = \Aut_d(\pi_0(\rH))(\kbar) \rtimes \{1\}\subset\mathrm{S}\Aut_d(\X^\circ)(\kbar)\simeq \mu_2(\kbar)^{\De_{\X^\circ}^{dist}}\rtimes \Ga
\]
is normal so that we obtain the diagram
        \[
    \begin{tikzcd}
        1\ar[r]&\cala_\X(\kbar)\ar[r]\ar[d]&\rS\cala_\X\ar[r]\ar[d]&\Ga\ar[d,"="]\\
    1\ar[r]&\Aut_d(\X)(\kbar)\ar[r]\ar[d,"\simeq"]&\rS\Aut_d(\X)\ar[r]\ar[d,"\simeq"]&\Ga\ar[d,"="]\\
    1\ar[r]&\mu_2(\kbar)^{\De_{\X^\circ}^{dist}}/\Aut_{d}(\pi_0(\rH))\ar[r]&\left(\mu_2(\kbar)^{\De_{\X^\circ}^{dist}}/\Aut_{d}(\pi_0(\rH))\right)\rtimes \Ga\ar[r]&\Ga.
    \end{tikzcd}
    \]
We may similarly obtain a \emph{canonical} cocycle $\mu_{\X}^d$ valued in the above quotient group from the section of $\rS\cala_\X\to \Ga$ as in the preceding case.

If $\X$ and $\X'$ are two $\G$-forms of $\overline{\X}$\footnote{Recall that this requires that $\X_{\kbar}\simeq\X'_{\kbar}$ in a $\G_{\kbar}$-equivariant way (cf. Remark \ref{Rem: G-form is restrictive}).}, we say they are in the same \textbf{geometric class} if the geometric cocycles $\mu_{\X}^d$ and $\mu_{\X'}^d$ are cohomologous in $H^1(k,\Aut_d(\X))$. It is an easy check that this is the case if $\X$ and $\X'$ are $\G$-equivariantly isomorphic. When $\pi_0(\rH)$ is trivial, Lemma \ref{Lem: surjective on dist cohom} implies that there exist $\G$-forms of $\X$ associated to each geometric class as soon as any $\G$-form exists, which occurs as soon as the $\ast$-action of $(\G,\rA)$ preserves the combinatorial data $\Omega_{\overline{\X}}$ by Proposition \ref{Prop: quasirational}. Conjecture \ref{Conj: cohom surj} asserts this should hold for general $\rH$ (cf. Corollary \ref{Lem: examples of conj}).

\subsubsection{$\G$-outer forms and Well-adapted varieties}\label{Section: well adapted outer} We continue to assume that $\G$ is quasi-split and that $\X$ is a homogeneous spherical $\G$-variety. Fix $x\in \X(k)$ and let $\rH=\Stab_{\G}(x)$. Recall from Section \ref{Sec: forms and action} that there is a canonical morphism
\[
{\mathrm{out}}:  \cala_\X\lra \Out(\X),
\]
which is compatible with $\cala_\X\simeq \rH\backslash N_{\G}(\rH)\to \Out_\X(\rH)\subset \Out(\rH)$,
inducing a short exact sequence of diagonalizable groups
\[
1\lra \mathcal{A}_\X^\flat\lra \cala_\X\lra \Out(\X)\lra 1.
\]
 Since $\Out(\rH)$ is a discrete group, we have $\cala_\X^\circ\subset \mathcal{A}_\X^\flat$, though $\mathcal{A}_\X^\flat$ is generally larger. Since $\cala_\X\cong \rH\backslash N_{\G}(\rH)$, it is easy to see that $Z_{\G}(\rH)=  Z(\rH)\cdot Z(\G)$, so that $\mathcal{A}_\X^\flat$ is the image of $\rH\cdot Z(\G)\subset N_{\G}(\rH)$. 

 { 
 
}

\begin{Lem}\label{Lem: center in flat in sharp}
    We have a natural inclusions
    \[
 \mathcal{A}_\X^\flat\subset\cala_\X^d\subset \mathcal{A}_\X^\sharp.
    \]
    In particular, there is a surjective morphisms 
    \[
    \Out_\X(\rH)\lra\Aut_d(\X)\lra \Aut_\Omega(\D(\X)).
    \]
\end{Lem}
\begin{proof}
The inclusion $\cala_\X^d\subset \mathcal{A}_\X^\sharp$ is a tautology. As noted above, $\mathcal{A}_\X^\flat\simeq\rH\backslash (\rH\cdot Z(\G))\simeq Z(\G)\cap \rH\backslash Z(\G)$ consists of automorphisms arising from central elements of $\G$. But any distinguished root $\al\in \De_{\X^\circ}^{dist}$ is also an element of the root lattice $\Lam$ of $(\G,\rA)$, hence is trivial on $\mathcal{A}_{\X^\circ}^\flat$. This shows the first inclusion for $\X^\circ$, but immediately implies the same for general $\X$ since 
\[
X^\ast(\Aut_d(\X)) = (\fX\cap \Lam_{\X^\circ}^d)/\Lam_\X\subset(\fX\cap \Lam_{\G})/\Lam_\X.\qedhere
\]
\end{proof}

\begin{Def}\label{Def: well adapted}
Suppose that $\G$ is a quasi-split reductive $k$-group and $\X$ is a homogeneous spherical $\G$-variety. We say that $\X$ is \textbf{well adapted} if the natural morphism 
\[
 \Out(\X)\lra\Aut_d(\X)
\]
is an isomorphism.
\end{Def}
\begin{Rem}
 We note that well-adaptedness is a geometric property of $\overline{\X}$. For example, the preceding lemma implies that spherically-closed spherical varieties are automatically well adapted.
\end{Rem}


\begin{Lem}\label{Lem: well adapted connected}
    Suppose that $\X=\rH\backslash\G$ is a homogeneous spherical $\G$-variety, and set $\X^\circ = \rH^\circ\backslash\G$. If $\X^\circ$ is well adapted, then so is $\X$.
\end{Lem}
\begin{proof}
    Fix a $k$-rational Borel $\B$ and let $\rA$ denote the canonical torus quotient. Let $\mathring{\X}\subset \X$ (resp., $\mathring{\X}^\circ\subset \X^\circ$) denote the (geometric) open $\B$-orbit. Then $\rA$ acts on $\mathring{\X}\sslash[\B,\B]$ and $\mathring{\X}^\circ\sslash[\B,\B]$ via the quotient tori $\Ax$ and $\rA_{\X^\circ}$, respectively. This induces inclusions
    \[
    X^\ast(\Ax)\subset X^\ast(\rA_{\X^\circ})\subset X^\ast(\rA).
    \]

    Since the map $\X^\circ\to \X$ is finite, we have equality of root lattices $\Lam_\X= \Lam_{\X^\circ}$ \cite[Lemma 3.1.5 (3)]{Losev}. This implies an equality $\De_\X^n=\De_{\X^\circ}^n$ of normalized spherical roots. 
     On the level of automorphism groups, the inclusion $X^\ast(\Ax)/\Lam_\X\subset X^\ast(\rA_{\X^\circ})/\Lam_\X$ induces a short exact sequence of diagonalizable $k$-groups
    \[
   1\lra \pi_0(\rH)\lra  \cala_{\X^\circ}\lra\cala_\X\lra 1;
    \]
    Another application of \cite[Lemma 3.1.5 (3)]{Losev} applied to the spherical subgroup $\rH^\circ\cdot Z(\G)$ shows that this morphism descends to a surjection $\Out_{\X^\circ}(\rH^\circ)\to \Out_\X(\rH)$. Finally, diagram \eqref{eqn: functorial on Aut} descends to a commutative diagram with exact rows
    \[
    \begin{tikzcd}
       1\ar[r]&\out(\pi_0(\rH))\ar[d]\ar[r]& \Out_{\X^\circ}(\rH^\circ)\ar[r]\ar[d]& \Out_\X(\rH)\ar[d]\ar[r]&1\\
        1\ar[r]&\Aut_d(\pi_0(\rH))\ar[r]&    \Aut_d(\X^\circ)\ar[r]&\Aut_d(\X)\ar[r]&1;
    \end{tikzcd}
    \]
    here the kernels are denoted to indicate the image of $\pi_0(\rH)$ under the appropriate quotient map. 
    By assumption the middle vertical arrow is an isomorphism. Standard arguments now show the other two vertical arrows are isomorphisms, proving the claim for the right most arrow.
    \quash{The inclusion $\De_\X^{dist}\subset \De_{\X^\circ}^{dist}$ implies that the inclusion of lattices
    \[
    X^\ast(\Aut_d(\X)) = (\zz/2\zz)^{\De_{\X}^{dist}}\subset(\zz/2\zz)^{\De_{\X^\circ}^{dist}}=X^\ast(\Aut_d(\X^\circ)),
    \]
    is a $\Ga$-stable factor admitting a $\Ga$-stable splitting $X^\ast(\Aut_d(\X^\circ))\to X^\ast(\Aut_d(\X))$, so that there is a $k$-rational morphism $s': \Aut_d(\X)\to\Aut_d(\X^\circ)$. Composing with $(\pi_d^\circ)^{-1}$ and the map to $\Out_\X(\rH)$, we obtain a splitting $s$ of $\pi_d$.}
\end{proof}
In particular, when $\overline{\X}$ is well-adapted and $\G$ is quasi-split, the data
\[
(\Omega_\X,[\mu_\X^d])=(\fX, \De_\X,\Omega^{(1)},\Omega^{(2)},[\mu_\X^d])
\]
 with $[\mu_\X^d]\in H^1(k,\Aut_d(\X))$ determines a $k$-form $\X$ up to $\G$-inner form. The notion of well adaptedness is motivated by our applications to endoscopy for symmetric varieties and our calculation of $\Out_\X(\rH)$ in Section \ref{Section: color autos}. In particular, we show that most symmetric varieties (including all symmetric varieties with no type $N$ spherical roots) are well adapted. 

\quash{
\textcolor{red}{Is this necessary??}We record the following lemma for symmetric varieties, which is partially a consequence of Section \ref{Section: color autos}.
\begin{Lem}\label{Lem: outer group exponent}
    Suppose that $\G$ is a reductive $k$-group equipped with an involution $\theta$ and let $\X=\rH\backslash\G$ be a symmetric variety associated to $\theta$ (that is, $\G^\theta\subset \rH\subset N_\G(\G^\theta)$). Then there exists a finite $\Ga$-set $\mathcal{I}_\X$ such that
    \[
    X^\ast(\Out_\X(\rH)) \cong (\zz/2\zz)^{\mathcal{I}_\X}.
    \]
\end{Lem}
\begin{proof}
First assume that $N_\G(\rH)$ is not contained in a proper parabolic $P\subset \G$ over $\kbar$. Suppose that $g\in N_\G(\rH)$. Then \cite[Proposition 4.2.1]{Losev} shows that $g^2\in N_\G(\rH)^\circ\cdot Z(\G)\subset \rH^\flat$. This implies that the diagonalizable group $\Out_\X(\rH)_{\kbar}$ has exponent $2$. The statement about the $\Ga$-action follows from Theorem \ref{Thm: Outer in terms of dist} below.

Suppose now that  $N_\G(\rH)$ is contained in a proper parabolic $P\subset \G$ over $\kbar$. Then we claim that $\rH=N_\G(\rH)$ so that the claim holds with  $\mathcal{I}_\X=\emptyset$. Indeed, since $N_{\G}(\rH)\supset\rH$, $\rH_{\kbar}$ is contained in a proper Levi subgroup $L$ of $P$; but this forces $\rH=L$ since over $\kbar$ the only symmetric subgroups acting reducibly on $\Lie(\G)/\Lie(\rH)$ are Hermitian symmetric subgroups. The claim now follows from standard results on normalizers of Levi subgroups. 
\end{proof}}
\section{Involutions on groups and root systems}\label{Section: involutions}
 In this section, we review the combinatorial theory of involutions, first on root systems and then also on reductive algebraic groups. The important concept of admissibility of the index of an involution is reviewed. Sections \ref{gamma theta} and \ref{Section: rational involutions} recall the definitions and properties associated to an arithmetic enhancement of an index due to Helminck. The new material comes in Section \ref{Section: endo invo}, where we introduce and study the notion of an endoscopic root system with involution. We return to the geometric properties of the associated symmetric varieties in Section \ref{Section: symmetric data}.

 \subsection{Abstract root systems} Let $\Y$ be a  $\zz$-lattice with dual lattice $\check{\Y}$.
 Suppose that $\Psi = (\Y, \Phi, \check{\Y}, \check{\Phi})$ is a root datum with $\Phi\neq \emptyset$. Let $\theta\in \Aut(\Psi)$ be an involution.  Set 

\[
\Y_0(\theta) = \{\lam\in \Y: \lam=\theta(\lam)\}\text{ and }\Y_1(\theta) = \{\lam-\theta(\lam)\in \Y: \lam\in Y\},
\]
and set $\Phi_0(\theta) = \Phi\cap \Y_0(\theta)$. This gives an additively-closed sub-root system. 

Suppose that $\De$ is a basis for $\Phi$ chosen so that the induced ordering on $\Y$ satisfies 
\begin{align}\label{Theta bases}
    \text{ If $\lam>0$ and $\lam\notin \Y_0(\theta)$ then $\theta(\lam)<0$;}
\end{align}
such a basis is called a basis as \textbf{$\theta$-basis} for $\Y$. We shall see that such bases correspond to maximally $\theta$-split Borel subgroups containing our fixed maximally $\theta$-split maximal torus. The next lemma is well known.


\begin{Lem}
 Suppose that $\De$ is a $\theta$-basis for $\Phi$. Then $\theta(-\De)$ is also a $\theta$-basis. In particular, there exists $w_\theta\in W_0$ such that $w_\theta\theta(\De) = -\De$. In fact, $w_\theta$ is the longest element of $W_0$ with respect to $\De_0(\theta):=\De\cap \Phi_0(\theta)$.
\end{Lem}
With the notation as above, set $\theta^\ast := -w_\theta \theta$, so that $\theta^\ast\in \Aut(\Y,\Phi,\De)$ and $\theta^\ast(\De_0(\theta)) = \De_0(\theta)$. 
The involution $\theta^\ast|_{\De_0(\theta)}= -w_0(\ast)$ is the opposition involution of $\De_0(\theta)$, hence is completely determined by $\Phi_0(\theta)$.


\subsubsection{The index of $\theta$} Suppose now that $\Psi$ is semi-simple (equivalently, $\Phi$ spans $\Y\otimes_\zz \qq$), and let $\theta \in\Aut(\Psi)$ be an involution. Set $\De$ to be a fixed $\theta$-basis. 
\begin{Lem}\cite[2.11]{Helmincktwo}
$\theta$ is determined by the quadruple $\I=(\Y,\De,\De_0(\theta), \theta^\ast)$.
\end{Lem}
This quadruple is known as the \emph{index of $\theta$}. It naturally encodes the Satake diagram of $\theta$: consider the Dynkin diagram of $(\Y,\De)$, color in black those vertices in $\De$ associated to elements of $\De_0(\theta)$, then indicate via arrows the diagram action of $\theta^\ast$. When $\Psi$ is semi-simple, this diagram encodes all the information of the index. 

\subsection{$(\Ga,\theta)$-indices}\label{gamma theta}
  
Let $k$ be a field with $\mathrm{char}(k)\neq 2$ and set $\Ga= \Gal(k^{sep}/k)$. As we will see below, $\theta$-indices give a combinatorial object designed to classify involutions of reductive groups over an algebraically closed field. Here we recall the enhancement -- called a $(\Ga,\theta)$-index -- introduced in \cite{Helminckrational} to classify involutions over more general fields. As we explain in Lemma \ref{Lem: same on torus} below, the $(\Ga,\theta)$-index records the homogeneous spherical datum of the variety $\G^\theta\backslash\G$. 
This framework is useful in verifying that the proposed endoscopic spherical data (see Section \ref{Section: endo varieties}) satisfy Luna's axioms and descend to a $k$-rational symmetric variety.

Suppose that $\Psi$ is a root datum endowed with a $\Ga$-action and an involution $\theta$ such that $\sig\circ\theta = \theta\circ \sig$ for all $\sig\in \Ga$. 
If $\mathcal{E}\subset \Aut(\Psi)$ denotes the finite subgroup of automorphisms through which $\Ga$ acts and we set $\Ga_\theta:=\mathcal{E}\cdot \{1,-\theta\}$ to be the group generated by these two (commuting) groups of automorphisms, consider the lattices
\[
\Y_0(\Ga):= \left\{\lam\in  \Y: \sum_{\sig\in \mathcal{E}}\sig(\lam)=0\right\}\:\:\text{and}\:\:\Y_0(\Ga_\theta):=\left\{\al\in \Y: \sum_{\sig\in \mathcal{E}}\sig(\al) = \theta\left(\sum_{\sig\in \mathcal{E}}\sig(\al)\right)\right\}.
\]
When $(\Y,\Phi)$ comes from a reductive group $\G$ over $k$ (see discussion below) and $\theta$ comes from a $k$-rational involution on $\G$ restricted to an appropriate maximal torus $\rA\subset \G$  (cf. Section \ref{Section: admissible tori}), $\Y_0(\Ga)$ encodes the anisotropic kernel of $\G$ and $\Y_0(\Ga_\theta)$ encodes the semi-simple part of the centralizer of a maximal $k$-split torus $\rA_{spl}^{-}$ on which $\theta$ acts by inversion. 

Assume that a $\theta$-basis $\De\subset \Phi$ may be chosen satisfying the following two conditions:
\begin{enumerate}
    \item \label{eqn: Ga basis} { If $\lam>0$ and $\lam\notin \Y_0(\Ga)$ then $\sig(\lam)>0$ for all $\sig\in \Ga$;}
      \item \label{eqn: Ga theta basis} { If $\lam>0$ and $\lam\notin \Y_0(\Ga_\theta)$ then $\sig(\lam)>0$ for all $\sig\in \Ga_\theta$.}
\end{enumerate}
Such a basis is called a \textbf{$(\Ga,\theta)$-basis} for $(\Psi,\theta)$. In general, this notion can be quite subtle. They always exist when $(\Y,\theta)$ is induced from a $k$-rational involution on a reductive group $\G$ with the maximal torus $\rA$ chosen appropriately; see \cite[Proposition 5.26]{Helminckrational}.

Since we largely focus on quasi-split groups, we will impose the following assumption: 
\begin{equation}\label{no anisotropic kernel}
   { \Y_0(\Ga)\cap \zz\Phi=0;}
\end{equation}
this is true, for example, for the $\ast$-action induced on the root datum of a $k$-rational pair $(\G,\rA)$ (equivalently, if $\G$ is quasi-split). Under this assumption, if $\lam\in \Y_0(\Ga_\theta)\cap \zz\Phi$, then $\lam\in \Y_0(\theta)$. In particular, \cite[Proposition 5.26]{Helminckrational} implies that $(\Y,\Phi,\theta)$ possesses a $(\Ga,\theta)$-index whenever the $\Ga$-action satisfies \eqref{no anisotropic kernel}.

Given this data, we associate a sextuple
 \[
 \I= (\Y, \De,\De_0(\Ga),\De_0(\theta), \sig_\ast,\theta^\ast),
 \]
 called a \textbf{$(\Ga,\theta)$-index},  where
 \begin{enumerate}
     \item $\I_k:=(\Y, \De,\De_0(\Ga), \sig_\ast)$ is a (Tits) $\Ga$-index. Recall that Tits introduced such indices to classify reductive groups over $k$ \cite{Titsclassification}. Here, $\sig_\ast$ denotes the $\ast$-action of $\Ga$ on $\Y$ stabilizing $\De$;
     \item $\I_\theta:=(Y, \De, \De_0(\theta),\theta^\ast)$ is a $\theta$-index, where $\theta^\ast= -w_\theta\theta$ is the induced diagram automorphism.
 \end{enumerate}
 We note that if \eqref{no anisotropic kernel} holds, then $\De_0(\Ga)=\emptyset$ and $\sig=\sig_\ast$ for all $\sig\in \Ga$.
 \subsection{Involutions from a reductive group}
Let $k$ be any field of characteristic $\neq 2$ and let $\G$ is a connected reductive $k$-group. Suppose that $\theta$ is a $k$-rational involution on $\G$. A torus $\mathrm{T}\subset\G$ is called \emph{$\theta$-split} if $\theta(t)=t^{-1}$ for all $t\in \mathrm{T}(\kbar)$. It is well known that any two maximal $\theta$-split tori are geometrically conjugate by an element of $(\G^\theta)^\circ$ \cite[Theorem 7.5]{Richardson}. Let $\rA\subset \G$ be a maximal $k$-torus which is stable under $\theta$. This induces an involution on the root datum
\[
\widetilde\theta: \Psi({\G,\rA})\to \Psi({\G,\rA}),
\]
where $\Psi({\G,\rA})=(X^\ast(\rA),\Phi(\rA),X_\ast(\rA),\check{\Phi}(\rA))$, in the obvious manner.
\subsubsection{Admissible tori}\label{Section: admissible tori} For any $\theta$-stable torus $T\subset \G$, we let
\[
\text{$T^+ = (T\cap \G^\theta)^\circ$  and  $T^-=(\{t\in T: \theta(t) = t^{-1}\})^\circ$.}
\]
Set $\X_\theta:=\G^\theta\backslash \G$. We assume in this section that $\G$ is isotropic over $k$. As recalled in Section \ref{Section: generalities}, there is a symmetrization map $s:\X_\theta\hra \G$ embedding $\X_\theta$ as a closed subvariety of $\G$.
\begin{Lem}\label{Lem: good split torus}
Suppose $x\in \X_\theta^{ss}(k)$. There exists a choice of a $\theta$-stable maximal $k$-torus $T\subset \G$ such that $s(x)\in T(k)$ and $T^-$ is a maximal $\theta$-split torus.
\end{Lem}
\begin{proof}
To begin, Proposition 2.3 of \cite{HelminckWang} applied to the symmetric pair $(\G,\G^\theta)$ implies that there exists a $\theta$-stable maximal $k$-torus $T\subset \G$. On the other hand, Theorem 7.5 of \cite{Richardson} implies that $s(x)\in \G(k)$ lies in a maximal $\theta$-split torus; in particular, $s(x)$ lies in a $\theta$-split $k$-torus $A$. With out loss of generality, we may assume that $A$ is a maximal $\theta$-split $k$-torus. Lemma 11.1 of \cite{HelminckWang} implies that $A$ is, in fact, a maximal $\theta$-split torus of $\G$. In particular, any choice of maximal $k$-torus $T\supset A$ is $\theta$-stable \cite[Section 1]{vust1974operation}. 
\end{proof}
For any algebraic $k$-torus $\rA$, we let $\rA_{spl}\subset \rA$ denote the maximal $k$-split sub-torus in $\rA$.
\begin{LemDef}
There exist $\theta$-stable maximal $k$-tori $\rA\subset\G$ such that
\begin{enumerate}
    \item $\rA_{spl}$ is a maximal $k$-split torus, 
    \item $\rA^-$ is a maximal $\theta$-split torus, and
    \item $\rA^-_{spl}$ is a maximal $(\theta,k)$-split torus.
\end{enumerate}
We call such a torus $\rA$ a \textbf{$(\theta,k)$-admissible torus}.
\end{LemDef}
\begin{proof}
Let $\rA_{spl}^-$ be a maximal $(\theta,k)$-split torus of $\G$, which exists by Proposition 4.3 of \cite{HelminckWang}. Let $Z_{\G}(\rA_{spl}^-)_{ani}$ denote the anisotropic factor of $Z_{\G}(\rA_{spl}^-)$. Let $\rA_{spl}$ be a maximal $k$-split torus containing $\rA_{spl}^-$ and let $\rA^-$ denote a maximal $\theta$-split $k$-torus containing $\rA_{spl}^-$. Lemma 4.5 (iii) of \cite{HelminckWang} implies that $\rA_{spl}$ is $\theta$-stable and that $Z_{\G}(\rA_{spl}^-)_{ani}\subset Z_{\G}(\rA_{spl})$. 

Consider the decomposition $\rA^- = \rA^-_{spl}\cdot \rA^-_{ani}$, where $\rA_{ani}^-$ is anisotropic. Then $\rA_{ani}^-\subset Z_{\G}(\rA_{spl}^-)_{ani}\subset Z_{\G}(\rA_{spl})$, so that
\[
\rA^-\cdot \rA_{spl} =\rA_{spl}\cdot \rA^-
\]
is a $\theta$-stable $k$-torus. Let $\rA$ be a maximal torus in $Z_{\G}(\rA^-\cdot \rA_{spl})$. Since $\rA^-\subset \rA$ is maximal $\theta$-split, $\rA$ is $\theta$-stable \cite[Section 1]{vust1974operation} which satisfies the constraints of the lemma.
\end{proof}


Now assume that our maximal torus $\rA$ is $(\theta,k)$-admissible. There exists a minimal $\theta$-split parabolic $k$-subgroup $P$ containing $\rA$; by Proposition 4.7 of \cite{HelminckWang}, it satisfies
\[
P\cap\theta(P)=Z_{\G}(\rA_{spl}^-).
\]
Let $\B$ be a Borel subgroup defined over $\kbar$, and assume it is chosen so that $\rA\subset \B\subset P$; it is then maximally $\theta$-split among Borel subgroups of $\G$. 
\begin{LemDef}\label{Def: admissible pair}
A $(\theta,k)$-admissible maximal torus $\rA$ and a maximally $\theta$-split Borel $\B$ gives a \textbf{$\theta$-admissible Borel pair} $(\rA,\B)$. If $\G$ is quasi-split, $k$-rational $\theta$-admissible Borel pairs exist.
\end{LemDef}

Suppose that $\rA$ is a maximally $k$-split maximal torus of $\G$. We say that an involution $\theta$ is \textbf{normally related to $\rA$} if $\rA$ is a $(\theta,k)$-admissible torus.
\begin{Lem}\label{Lem: normally related}\cite[Lemma 8.3]{Helminckrational}
 Suppose that $\rA$ is a maximally $k$-split maximal torus of $\G$ and $\theta$ is a $k$-rational involution on $\G$. There exists a conjugate of $\theta$ by an element $g\in \G(k)$ such that $\theta'=\Ad(g)\circ\theta\circ \Ad(g)^{-1}$ is normally related to $\rA$.   
\end{Lem}
\quash{We will often refer to $(B,T)$ as simply an admissible pair. There is a single stable $\rH^\circ$-orbit of pairs $(B,T)$ such that $B$ is maximally $\theta$-split and $T$ is maximally $\theta$-split \cite[2.6 and 2.8]{Richardson}. 

\begin{Lem}
Suppose that $\X=\rH\backslash\G$ is excellent and that $(B,T)$ is a $(\theta,k)$-admissible pair. The $\rH(k)$-orbits of such pairs are in bijection with $\mathcal{D}(T^\theta,\rH;k)$.
\end{Lem}
\begin{proof}
The preceding facts imply that we may identify the unique $\rH$-orbit of $\theta$-admissible pairs geometrically is isomorphic to $\rH/T^\theta$. We now pass to rational points to obtain the result.
\end{proof}
The question of enumeration reduces to understanding the cohomology of $T^\theta$. The failure of $\mathcal{D}(T^\theta,\rH;F)=1$ is what produces the issues in classifying isomorphism classes of involutions in \cite{Helminckrational}.
}

\subsubsection{Tori and indices}
Suppose that $\Psi$ is an abstract root datum as above. Recalling \cite[Definition 3.9]{Helmincktwo}, we say that an involution $\theta\in \Aut(\Psi)$ is \emph{admissible} if there exists a reductive algebraic group $\G$ over $\kbar$, a maximal torus $\rA\subset \G$, and an involution $\widetilde{\theta}\in \Aut(\G,\rA)$ such that
\begin{enumerate}
    \item $\Psi$ is isomorphic to $\Psi({\G,\T})=(X^\ast(\rA),\Phi(\rA),X_\ast(\rA),\check{\Phi}(\rA))$, 
    \item $\widetilde\theta$ induces $\theta$ on $\Psi$, and
    \item $\rA$ is a maximally $\theta$-split maximal torus of $\G$
\end{enumerate}
An index $\I$ of an involution $\theta$ on $\Psi$ is called \emph{admissible} if it is associated to an admissible involution. When $\Psi$ is not semi-simple, we call an index admissible if the corresponding index on 
\[
\Psi_{ad}:=(\Y_{ad},\Phi,\check{\Y}_{ad},\check{\Phi})
\]
is admissible; here $\Y_{ad}\cap \Y$ is the saturation of $\zz\Phi$ in $\Y$. Importantly, the property of $\I$ being admissible is independent of the algebraically closed field $\kbar$ \cite{Helmincktwo}.
\begin{Prop}\cite[Theorem 3.11]{Helmincktwo}\label{Prop: classification of invo}
    Assume that $\G$ is semi-simple and $\rA$ is a maximal torus of $\G$. There is a bijection between the set of $\G(\kbar)$-conjugacy classes of involutory automorphisms of $\G_{\kbar}$ and the isomorphism classes of admissible indices of $\Psi({\G,\T})$.
\end{Prop}
Let us briefly explain how to extract an index from $(\G,\rA,\theta)$. Up to conjugating $\theta$, we may assume that $\rA$ is a $(\theta,k)$-admissible torus. As we discuss in Section \ref{Section: admissible tori}, we may choose a Borel $\B$ over $\kbar$ such that there exists a parabolic subgroup $P$ over $\kbar$ satisfying $\B\subset P$ and 
\begin{equation}\label{eqn: split parabolic}
    L=P\cap \theta(P)\supset \rA
\end{equation} is a $\theta$-stable Levi subgroup of $\G$. Such a parabolic subgroup is called $\theta$-split, and we assume that $P$ is minimal among $\theta$-split parabolics, so that $\B$ is maximally $\theta$-split among Borel subgroups in the sense that $\B\cap \theta(\B)$ has minimal dimension.

Let $\De\subset \Phi^+\subset \Phi$ denote the simple and positive roots associated to $(\rA,\B)$.Let $\De_\X^p$ denote those simple roots $\al$ such that $\theta(\al)=\al$; we have that $\De^p_\X = \De_{L}$ \cite[Theorem 2.9]{HelminckHelminck}.
It follows from \eqref{eqn: split parabolic} that for any $\al\in\De\setminus{\De^p_\X},$ $\theta(\al)\in \Phi^-$. In particular, $\De$ is a $\theta$-basis.
With notation as above, 
    \[
    (X^\ast(\rA),\De, \De^p_\X, \theta^\ast)
    \]
    is an admissible $\theta$-index.

Generalizing Proposition \ref{Prop: classification of invo}, suppose that $\G$ is reductive and $\Psi\simeq\Psi(\G,\rA)$. Suppose that we have an index $\I$ such that 
\begin{enumerate}
    \item the induced index $\I_{ad}$ for $\Psi_{ad}$ is admissible, and
    \item there is a compatible involution $\theta_{\rA}$ on the torus $\rA$.
\end{enumerate}
We claim that there exists an involution $\theta$ on $\G$ inducing the index $\I$. To see this, let $\theta_{ad}$ denote a lift of the involution $\theta$ to $\G_{ad}$. Let $\G_{sc}$ be the simply connected cover of $\G_{ad}$ and let $\theta_{sc}$ be the unique lift of $\theta_{ad}$ \cite{Steinberg}. Our compatibility assumption implies that $\theta_{sc}|_{Z(\G_{sc})}$ intertwines with $\theta_{\rA}$. Since $\G=Z(\G)\times ^{Z(\G_{sc})}\G_{sc}$ \cite[2.0.1]{deligne1979varietes}, we obtain a well-defined automorphism via pushout.

\subsubsection{Rational involutions}\label{Section: rational involutions} The preceding discussion ignored the $k$-rational structure on $(\G,\theta)$. We now remark on the action of $\Ga$.
Suppose that $(\rA,\B)$ are chosen as above. Since $\B$ need not be $k$-rational (as $\G$ may not be quasi-split), the Galois action on $X^\ast(\rA)$ need not preserve the simple roots $\De$ determined by $\B$.  Recall from Section \ref{Section: pre star} the $\ast$-action $\sigma\mapsto \sig_\ast$ obtained by twisting the $\Ga$-action on $(X^\ast(\rA), \Phi)$ as follows:
 \[
 \chi^{\sig_\ast} = w_\sig^{-1}\chi^\sig,
 \]
 where $w_\sig\in W$ is the \emph{unique} element such that $\De^\sig= w_\sig\De$. Then $\sig_\ast\De=\De$. As noted above, if \eqref{no anisotropic kernel} holds, then $\sig=\sig_\ast$; this occurs when $\B$ is $k$-rational.
 
\begin{Lem}\label{Lem: theta basis}
The involution $\theta$ commutes with the $\ast$-action. In particular, the subset $\De^p_\X$ is stable under the $\ast$-action of $\Ga$.
\end{Lem}
\begin{proof}
Note that $\theta\sig=\sig\theta$ for all $\sig\in \Ga$ by the $k$-rationality of $\theta$. In particular, if $\chi\in X^\ast(\rA)$, then $\theta(\chi^\sig) = (\theta(\chi))^\sig$. On the other hand
\[
\theta(\chi^{\sig_\ast}) = \Ad(\theta(w_\sig))^{-1}\cdot (\theta(\chi))^\sig.
\]
But \cite[Proposition 10.39]{Helminckrational} shows that $\theta(w_\sig)=w_\sig$ for all $\sig\in \Ga$, proving the claim.
\end{proof}

Given this, we may associate to $(\G,\theta,\rA, \B)$ a $(\Ga,\theta)$-index
 \[
 \I= (X^\ast(\rA), \De, \De_0(\Ga),\De_0(\theta), \sig_\ast,\theta^\ast),
 \]
 where $\De_0(\Ga)$ gives the simple roots in the anisotropic kernel of $(\G,\rA)$ (recall that $\rA$ is maximally $k$-split). A $(\Ga,\theta)$-index is called \emph{admissible} if it is isomorphic (in the sense of \cite[Section 5]{Helminckrational}) to one associated to a $k$-rational symmetric pair $(\G,\theta)$ in this way. 
 
 One of the main results of \cite{Helminckrational} is a characterization of such admissible indices. 
More precisely, \cite[Theorem 10.45]{Helminckrational} states that a $(\Ga,\theta)$-index is admissible if and only if 
    \begin{enumerate}[(a)]
        \item\label{silly1} $\I$ is a basic $\Ga_\theta$-index,
        \item\label{silly2} $(X^\ast(\rA), \De, \De_0(\Ga), \sig_\ast)$ is an admissible $\Ga$-index,
        \item\label{silly3} $(X^\ast(\rA), \De, \De_0(\theta), \theta^\ast)$ is an admissible $\theta$-index.
                \item\label{silly4} the index $(X^\ast(\rA)_0, \De_0(\Ga), \De_0(\theta), \sig_\ast|_{\De_0} \theta^\ast|_{\De_0(\theta)})$ is an admissible $(\Ga,\theta)$-index.
    \end{enumerate}
    The fourth criterion captures the notion that the involution must also be admissible ``on the anisotropic kernel.'' If we assume that $\G$ is quasi-split and $\rA\subset \B$ are $k$-rational, \eqref{silly4} is trivially satisfied. Criteria \eqref{silly2} and \eqref{silly3} are natural, so we focus on the \eqref{silly1}.
    
 Recall from \cite[Section 10]{Helminckrational}, that $\I$ is a basic $\Ga_\theta$-index if the $\Ga_\theta$-basis $\De$ determined by $\B$ (so that it satisfies \eqref{eqn: Ga theta basis}) satisfies
 \begin{enumerate}[(i)]
     \item if $\Phi'\subset \Phi_0(\Ga_\theta)$ is an irreducible component then, $\Phi'\subset \Phi_0(\theta)$ or $\Phi'\subset \Phi_0(\Ga)$,
     \item $\De_0(\Ga)$ is stabilized by $-w_0$ where $w_0$ is the long element of $W$ determined by $\De_0(\Ga)\subset \De$, and
     \item $\De_0(\Ga)$ is $\theta^\ast$-stable, and
     \item $\De_0(\theta)$ is $\sig_\ast$-stable for all $\sig\in \Ga$.
 \end{enumerate}
Here, $$\Phi_0(\Ga_\theta):=\left\{\al\in \Phi: \sum_{\sig\in \mathcal{E}}\sig(\al) = \theta\left(\sum_{\sig\in \mathcal{E}}\sig(\al)\right)\right\}= \Phi\cap \Y_0(\Ga_\theta).$$ 
In particular, the $(\Ga,\theta)$-index associated to $(\G,\theta,\rA,\B)$ above satisfy these properties \cite[Section 5]{Helminckrational}. 
{\quash{ \begin{Lem}.\label{Lem: basic order}
  Given  $(\G,\theta,\rA,\B)$ as above, the induced basis $\De\subset \Phi$ is a $(\Ga, \theta)$-order. Moreover, it gives a basic $\Ga_\theta$-basis in the sense of \cite[Definition 10.40]{Helminckrational} so that 
   \[
 \I= (X^\ast(\rA), \De, \De_0(\Ga),\De_0(\theta), \sig_\ast,\theta^\ast),
 \]is admissible.
  \end{Lem}}
  
A fundamental complication is that admissibility of a $(\Ga,\theta)$-index $\I$ determines \textbf{if there exists} an involution $\theta$ on $(\G,\rA)$ inducing $\I$, but falls short of a classification of such involutions. The main result of \cite{Helminckrational} is to characterize those involutions which induce the same index, which is accomplished up to enumerating elements in a certain cohomology group. This approach does not naturally encode the geometric cocycle (cf. Definition \ref{Def: geometric cocycle}) of $\G^\theta\backslash\G$, so we do not make use of these more refined results in this paper except in the proof of Lemma \ref{Lem: same on torus} (cf. Remark \ref{Rem: what about Helminck}).
  
\subsection{Endoscopic root systems with involution}\label{Section: endo invo}Suppose now that $\Psi=(\Y,\Phi, \check{\Y}, \check{\Phi})$ is a
root datum. A (not-necessarily additively-closed) sub-root system $\Psi_s$ is called \emph{endoscopic} if it is of the form
\[
\Psi_s=(\Y,\Phi_s, \check{\Y}, \check{\Phi}_s),
\]
where $s\in \Hom(\check{\Y},\cc^\times)$, and  
\[
\check{\Phi}_s=\{\check{\lam}\in \check{\Phi}: s(\check{\lam})=1\}.
\]
Note that there is an inclusion $W_s\leq W$ of the Weyl group of $\Psi_s$ into the Weyl group of $\Psi$.

We equip $\Psi$ with a $\Ga$-action which we assume agrees with the $\ast$-action associated with some choice of base $\De\subset \Phi$; in particular, this satisfies \eqref{no anisotropic kernel}. We additionally equip $\Psi_s$ with a basis $\De_s$ and a $\Ga$-action $\sig_s$ ($\sig\in \Ga$) agreeing with the $\ast$-action associated to $\De_s$. This implies that for each $\lam\in \Y$ and $\sig\in \Ga$,
\[
\sig_s(\lam)= w(\sig)\sig(\lam), \quad w(\sig)\in N_W(\Psi_s)\subset W;
\] here, $\mathrm{Stab}_W(\Psi_s)$ is the subgroup of $W$ preserving $\Psi_s$.

Now suppose that $\Psi$ is equipped with an involution $\theta$ such that $\sig\circ\theta = \theta\circ \sig$ for all $\sig\in \Ga$. 
Let $\underline{\check{\Y}}\subset \check{\Y}$ denote the sublattice dual to  $\Y/\Y_1(\theta)$; that is,
\[
\underline{\check{\Y}} = \{\chi\in \check{\Y}: \chi|_{Y_1(\theta)}\equiv 0\}.
\] Note that if $\check{\theta}$ denotes the induced involution on $\check{\Y}$, then $\underline{\check{\Y}} =\{\chi\in \check{\Y}:\check{\theta}(\chi) =\chi\}$. Clearly, the $\Ga$-action stabilizes $\underline{\Y}$.

We now assume that\begin{enumerate}
    \item\label{endo desi 1} $s\in\Hom(\underline{\check{\Y}}, \cc^\times)$ satisfies $\theta(s) =s$ and $\sig_s(s) = s$ for each $\sig\in \Ga$;
    \item\label{endo desi 1.5} for every $\sig\in \Ga$, $w(\sig)\circ \theta=\theta\circ w(\sig)$ 
    \item\label{endo desi 2} $\Phi_0(\theta)\subset \Phi_s$.
\end{enumerate}
With these assumptions, we say that $(\Psi_s,\theta)$ is an \textbf{endoscopic root system with involution}. By our discussion in Section \ref{gamma theta}, $(\Y,\Psi_s,\theta)$ possesses a $(\Ga,\theta)$-basis, so we assume that $\De_s$ gives such a basis.
\begin{Rem}
    The third assumption is motivated by the theory in Section \ref{Section: endoscopy defs}, and is used several times in the argument below. 
\end{Rem}

\begin{Prop}\label{Prop: endoscopic roots}
    Suppose that the root datum $\Psi=(\Y,\Phi, \check{\Y}, \check{\Phi})$ is equipped with an involution $\theta$ and a $(\Ga,\theta)$-basis $\De$ such that the $(\Ga,\theta)$-index $\I=(\Y, \De,\emptyset, \De_0(\theta),\sig^\ast,\theta^\ast)$ is admissible. Suppose that $(\Psi_s,\theta)$ is an endoscopic root datum with involution and assume that $\De_s\subset \Phi_s$ is a $(\Ga,\theta)$-basis satisfying $\De_0(\theta)\subset \De_s$. The induced $(\Ga,\theta)$-index $\I_s=(\Y, \De_s,\emptyset,\De_0(\theta),\sig_s^\ast,\theta^\ast)$ is admissible.
\end{Prop}
\begin{proof}
We use \cite[Theorem 10.45]{Helminckrational} to reduce this to the claims that 
  \begin{enumerate}
        \item\label{silly1 endo} $\I_s$ is a basic $\Ga_\theta$-index,
        \item\label{silly2 endo} $(X^\ast(\rA), \De_s, \De_0(\Ga), \sig)$ is an admissible $\Ga$-index,
        \item\label{silly3 endo} $(X^\ast(\rA), \De_s, \De_0(\theta), \theta^\ast)$ is an admissible $\theta$-index.
    \end{enumerate}
    Note that the fourth criterion \eqref{silly4} is vacuous in our setting. Criteria \eqref{silly2 endo} is classical, since if $\G$ is the quasi-split reductive $k$-group associated to $\Psi$, this gives the Tits index of the quasi-split endoscopic group $\G_s$. We therefore need only to consider \eqref{silly1 endo} and \eqref{silly3 endo}.

    We claim that \eqref{silly1 endo} is immediate by considering the definition of  a basic $\Ga_\theta$-index given in Section \ref{Section: rational involutions}. There are four properties. We verify the first by noting that since \eqref{no anisotropic kernel} holds $$\Phi_{s,0}(\Ga_\theta)=\Phi_{s,0}(\theta)= \Phi_0(\theta)=\Phi_0(\Ga_\theta),$$ so that any irreducible component $\Phi'\subset\Phi_{s,0}(\Ga_\theta)$ is contained in $\Phi_{s,0}(\theta)$. Now the next two properties are vacuous as $\De_{s,0}(\Ga)=\De_0(\Ga)=\emptyset$, and the final property is the same as that of $\I$ since $\Phi_0(\theta) = \Phi_{s,0}(\theta)$.

To verify \eqref{silly3}, we need to recall some more notation from \cite{Helmincktwo}. Recall the sublattice $\Y_0(\theta)\subset \Y$. Set $\overline{\Y}_0: = Y/\Y_0(\theta)$ and let $\pi:\Y\lra \overline{\Y}_0$ denote the natural projection.  The image $\overline{\Phi}:=\pi(\Phi-\Phi_0)$ is the set of restricted roots of $\Phi$ relative to $\theta$. This is a root system with set of simple roots $\overline{\De}:=\pi(\De-\De_0)$; its cardinality is the \emph{relative rank} of $(\Psi,\theta)$. 

For any $\lam \in \overline{\Phi}$ such that $\frac{1}{2}\lam\notin \overline{\Phi}$, consider the set $\Phi(\lam)$ consisting of roots that project to an integral multiple of $\lam$. Note $\Phi_0\subset \Phi(\lam)$ for all  $\lam\in \overline{\Phi}$. Now define $\Y(\lam)$ to be the projection of $\Y$ to $\zz\Phi(\lam)\otimes_{\zz}\rr$. Setting $\check{\Y}(\lam)\subset  \check{\Y}$ to be the dual with $\check{\Phi}(\lam)\subset \check{\Phi}$, we obtain a root datum
    \[
    \Psi(\lam):=(\Y(\lam),\Phi(\lam),\check{\Y}(\lam),\check{\Phi}(\lam));
    \]
    this system is additively closed, symmetric, and stable under $\theta$ \cite[pg.71]{Borel:Tits}. Moreover, the system $(\Phi(\lam),\theta)$ has relative rank $1$. By \cite[Proposition 4.9]{Helmincktwo}, $(\Psi,\theta)$ is admissible if and only if $(\Phi(\lam),\theta)$ is admissible for each such $\lam \in \overline{\Phi}$ if and only if it is admissible for each $\lam \in \overline{\De}$.

    Now let $\lam \in \overline{\De}_s$; we may likewise consider 
    \[
     \Psi_s(\lam):=(\Y(\lam),\Phi_s(\lam),\check{\Y}(\lam),\check{\Phi}_s(\lam));
    \] this gives a $\theta$-stable sub-root system of $\Psi(\lam)$, also with relative rank $1$. It is immediate that $(\Psi_s(\lam),\theta)$ is admissible if and only if the projection to $\zz\Phi_s(\lam)\otimes_{\zz}\rr$ is, since $\Psi(\lam)$ is admissible. Noting that $\Psi_s(\lam)$ is an endoscopic root datum with involution for $\Psi(\lam)$, the claim follows once we verify it for the relative rank $1$ case.

    We thus need to consider admissible rank $1$ root data with involution, and verify that those endoscopic data with involution satisfying \eqref{endo desi 1} and \eqref{endo desi 2} are also admissible.
    
 Recall from \cite[Section 4]{Helmincktwo} that there are $6$ types of admissible absolutely irreducible indices of restricted rank $1$ (types $A_1$, $A_l\;(l\geq 3)$, $B_l\;(l\geq 2)$, $C_l\;(l\geq 3)$, $D_l\;(l\geq 3)$, and type $F_4$), and an additional irreducible one that is not absolutely irreducible (type $D_2$).  As the only proper $\theta$-stable root systems in types $A_1$ and $D_l$ $(l\geq 2)$  contain no roots (corresponding to an involution on a torus, and so admissible), we need only consider the remaining $4$ types, allowing us to assume that $\Psi(\lam)$ is absolutely irreducible.

\begin{table}[htp]
\caption{Rank-$1$ Satake diagrams and endoscopy}
    \label{tab:satake}
    \centering
    \begin{tabular}{|c|c|c|}
	\hline 
	$\textbf{Type}$ & $(\Psi(\lam),\theta)$ & $(\Psi_s(\lam),\theta)$ \\[.25 cm] \hline
 $\mathbf{A_l\:(l\geq 3)}$ & $\dynkin[edge length=.75cm,
involutions={16}]{A}{o**.**o}$ & $\dynkin[edge length=.75cm]{A}{**.**}\quad \dynkin[edge length=.75cm]{A}{o}$ \\ \hline
 & $(\SL_{l},\GL_{l-1})$ & $(S(\GL_{l-2}\times \GL_2), \GL_{l-2}\times \Gm)$ \\[.25cm] \hline
 $\mathbf{B_l\:(l\geq 3)}$ & $\dynkin[edge length=.75cm]
{B}{o**.**}$ & $ \dynkin[edge length=.75cm]{A}{o}\quad\dynkin[edge length=.75cm]{B}{**.**}$ \\[.25cm] \hline
  & $(\Spin_{2l+1},\Spin_{2l})$ &  $(\SL_2\times\Spin_{2l-1},\Gm\times \Spin_{2l-1})$ \\[.25cm] \hline
 $\mathbf{C_l\:(l\geq 3)}$ & $\dynkin[edge length=.75cm]
{C}{*o*.**}$ & $\left[\dynkin[edge length=.75cm]
{D}{*}\dynkin[edge length=.75cm]
{D}{o}\right]\quad\dynkin[edge length=.75cm]
{C}{*.**}$ \\[.25cm] \hline
 & $(\Sp_{2l},\Sp_2\times \Sp_{2l-2})$ &  $(\SO_4\times\Sp_{2l-4},\GL_2\times \Sp_{2l-4})$ \\[.25cm] \hline
 $\mathbf{F_4}$ &$\dynkin[edge length=.75cm]
{F}{***o}$ & $\dynkin[edge length=.75cm]
{B}{***}\:\dynkin[edge length=.75cm]
{A}{o}$  \\[.25cm] \hline
 & $(F_4,\Spin_9)$ & $(\Spin_7\times^{\mu_2}\SL_2, \Spin_7\times^{\mu_2} \Gm)$ \\[.25cm] \hline
    \end{tabular}
\end{table}
 An endoscopic datum corresponds to a subset $S\subset \check{\De}\cup\{-\check{\al}_{long}\}$, where $\check{\al}_{long}$ denotes the highest long coroot. If $-\check{\al}_{long}\notin S$, then $\Psi_s(\lam)$ is a Levi sub-root system and $\De_s(\lam)\subset\De(\lam)$. However since $\De_0\subset \De_s(\lam)$, the restricted rank $1$ assumption forces $\Psi_s(\lam)=\Psi(\lam)$. We thus assume that $-\check{\al}_{long}\in S$.

 We may now handle the cases in turn: given an admissible relative rank-$1$ system $(\Psi(\lam),\theta)$, we may read off the type of $(\Psi_s(\lam),\theta)$ and verify that the Satake diagram of $(\Psi(\lam),\theta)$ is again admissible. Indeed, it is shown in \cite{RubioGelfand} that passing to descendants corresponds to removing nodes from the Satake diagram. We compute this as follows. Given the Satake diagram of $(\Psi(\lam),\theta)$, we dualize while keeping the coloring for the vertices as this data gives the appropriate Levi sub-root system, the dual of which is a Levi sub-root system. This gives the Dynkin diagram of the dual root system $\check{\Psi}(\lam)$ and encodes the dual Levi. We then pass to the affine Dynkin diagram of $\check{\Psi}(\lam)$, delete a subset of the white vertices of the non-affine diagram, and dualize the resulting (possibly disconnected) diagram. We may check each case, all of which result are easily checked to produce an admissible diagram. We reproduce the diagrams which correspond to reductive involutions along with representative symmetric pairs inducing these diagram in Table \ref{tab:satake}. 

Appealing to \cite[Theorem 4.16]{Helmincktwo}, this proves the claim.
\end{proof}

\section{Symmetric varieties and indices}\label{Section: symmetric data}
In this section, we review the geometric properties of symmetric varieties, as well as establish useful rationality properties. This includes generalizing the notion of a $z$-extension from \cite{Kottwitzrational} to the symmetric setting. We also completely explicate the spherical data for a symmetric $k$-variety. In Section \ref{Section: indices}, we relate this spherical data to the associated a $(\Ga,\theta)$-index, which is a useful variant of $\Omega_\X$ for a symmetric variety.  In Section \ref{Section: reductions}, we relate the spherical roots systems of a symmetric variety to symmetric varieties for $\G_{der}$, $\G_{sc}$, and $z$-extensions; this is important for the arguments of Section \ref{Section: color autos}.

\subsection{Generalities on symmetric pairs and varieties}\label{Section: generalities}
A {symmetric pair} denotes a triple $(\G,\rH,\theta)$ where $\rH\subset \G$ are reductive groups over $k$ and $\theta$ is an involution of $\G$ such that 
\[(\G^\theta)^\circ\subset\rH\subset N_{\G}(\G^\theta).
\]
We do not require that $\rH$ is connected, but always assume it is smooth. The quotient variety ${\X}:=\rH\backslash\G$ is called the \textbf{symmetric variety,} which we say is associated to the involution $\theta$.

Fix a symmetric pair $(\G,\rH,\theta)$. When it will not lead to confusion, we frequently refer to $(\G,\rH)$ or $(\G,\theta)$ as a symmetric pair. 
Define the \emph{symmetrization map}
\begin{align*}
    \tau:\G&\lra \G\\
        g&\longmapsto g^{-1}\theta(g).
\end{align*}
It is well known \cite[Lemma 2.4]{Richardson} that $\tau$ induces an embedding (also referred to as the \emph{symmetrization map}) of affine $\G$-varieties
\[
s:\G^\theta\backslash \G\lra \G.
\]
In particular, when $\X=\G^\theta\backslash \G$ we may identify $\X$ as a closed subvariety of $\G$. More generally, we set $\X_\theta=\G^\theta\backslash \G$. If $\rH\subset \G^\theta$, we obtain a finite morphism $s=s_\X:\rH\backslash \G\to \G$. In all cases, since $N_\G(\G^\theta)/Z(\G) = [\G_{ad}]^{\theta_{ad}}$ surjects onto the fixed point subgroup of the induced involution $\theta_{ad}$ on $\G_{ad}$, there exists a morphism $$s_{ad}:\rH\backslash\G\lra \X_{ad} \subset \G_{ad};$$ we refer to this as the \emph{adjoint symmetrization map}. 

Viewing ${\X}$ as an $\rH$-variety, let ${\X}^{ss}$ (resp. ${\X}^{rss}$) denote the semi-simple locus (resp. regular semi-simple locus) of ${\X}$. The symmetrization map allows us to relate these notions to semi-simplicity of elements in $\G$ \cite[Theorem 7.5]{Richardson}.
\quash{\begin{Lem}\cite[Theorem 7.5]{Richardson}\label{Lem: semi-simple match}
 Using the symmetrization map, we identify ${\X}\subset \G$ as a closed subvariety of $\G$. Then $x\in{\X}$ is $\rH$-semi-simple if and only if $s_{ad}(x)\in \G^{ss}_{ad}$ is semi-simple as an element of $\G$. In particular,
 \[
 {\X}^{ss}=s_{ad}^{-1}\left({\X}\cap \G_{ad}^{ss}\right).
 \]
\end{Lem}}
\quash{The relationship between ${\X}^{rss}$ and $\G^{rss}$ is more subtle. While in general they are unrelated, the certain nice symmetric varieties satisfy
\begin{equation}\label{eqn: quasi-split condition}
    {\X}^{rss}={\X}\cap \G^{rss};
\end{equation}
this is because the symmetric varieties we consider are \emph{quasi-split}: the centralizer in $\G$ of a maximal $\theta$-split torus $A\subset {\X}$ is a torus; see \cite[Section 1.2]{LeslieSpringer} for a discussion on quasi-split symmetric pairs. It is easy to see that over an algebraically closed field (\ref{eqn: quasi-split condition}) is equivalent to $(\G,\rH)$ begin quasi-split.
}

We will have occasion to consider multiple involutions, so we introduce some notation. We say that two involutions $\theta$ and $\theta'$ are conjugate if there exists $g\in \G(k)$ such that $\theta' = \Ad(g)^{-1}\circ \theta\circ \Ad(g)$. We say $\theta'$ is a \textbf{pure inner ($\G$-)twist of $\theta$} if there exists $g\in \G(\kbar)$ such that $\theta' = \Ad(g)^{-1}\circ \theta \circ \Ad(g)$ satisfying $\tau(g)=g^{-1}\theta(g)\in \G(k)$ and that the $1$-cocycle  $$\sig\mapsto g^{-1}\sig(g)\in \G^\theta(\kbar).$$ 
Clearly, conjugacy classes of pure inner twists of $\theta$ are parameterized by $\ker^1(\G^\theta,\G;k)$. 
The following is an easy Galois cohomology exercise.
\begin{Lem}\label{Lem: pure inner twist invol}
There is a natural bijection 
\[
\G^\theta\backslash \G(k)=\bigsqcup_{\theta'\in \ker^1(\G^\theta,\G;k)}  \tau(g')\cdot\G(k),
\]
where $\theta'=\Ad(\tau(g'))\circ \theta$. Moreover, there is a $\G$-equivariant isomorphism $\G^\theta\backslash\G\to \G^{\theta'}\backslash\G$ if and only if $\theta'$ is a pure inner twist of $\theta$ by a class in $\ker^1(\G^\theta,\G;k)$.
\end{Lem}

\subsubsection{Descendants}
For $x\in {\X}^{ss}(k)$, let $\G_x$ denote the connected component of the identity of the centralizer in $\G$ of the element $s_{ad}(x)\in\G_{ad}(k)$ and $\rH_x$ its stabilizer in $\rH\cap \G_x$. Then $(\G_x,\rH_x,\theta|_{\G_x})$ is a symmetric pair, which we refer to as as the \textbf{descendant} of $(\G,\rH)$ at $x$ (this is a slight generalization of the concept in \cite[Definition 7.2.2]{AizGourdescent}). We will also refer to the symmetric variety ${\X}_x:=\rH_x\backslash \G_x$ as the descendant of ${\X}$ at $x.$ 


\quash{
\textcolor{red}{CUT?}There are two natural closed immersions of ${\X}_x$ into ${\X}$. The first is simply given by restriction of $s$ to $\G_x$. For the second, define the \emph{symmetrization map at $x$} by
\begin{align*}
    s_x: \G_x&\lra {\X}\\
        g&\longmapsto gx\sig(g)=xs(x)\nonumber.
\end{align*}
The following is immediate.
\begin{Lem}\label{Lem: symmetr at x}
 There is a commutative diagram
 \[
 \begin{tikzcd}
 &\G_x\ar[dl,swap,"s"]\ar[dr,"s_x"]&\\
 {\X}_x\ar[rr,"x\cdot"]&&{\X}.
 \end{tikzcd}
 \]
In particular, $s_x$ induces a closed immersion of affine $\G_x$-varieties
\begin{align*}
    s_x: {\X}_x&\lra {\X}\\
        y&\longmapsto xy\nonumber.
\end{align*}
\end{Lem}}
\begin{Rem}
    There are two notions of degenerating a symmetric variety in the literature: semi-simple descent as just described and the notion of boundary degeneration as described in \cite[Proposition 2.4.3]{SakVenk}. A crucial aspect of descent for symmetric varieties is the ability to degenerate \emph{at any semi-simple element}; boundary degeneration only applies to degeneration with respect to tori. The ability to descend with respect to torsion elements is fundamental for elliptic endoscopy.
\end{Rem}

\subsection{Compatible $z$-extensions}
Suppose that $\G$ is a connected reductive group over a field $k$ and let $(\G,\theta)$ be a symmetric pair. Let $\G_{der}\subset \G$ denote the derived subgroup and $\pi:\G_{sc}\to \G_{der}$ the simply connected cover of $\G_{der}$. By a result of Steinberg \cite[9.16]{Steinberg}, there exists a unique involution $\theta_{sc}:\G_{sc}\to \G_{sc}$ such that the diagram
\[
\begin{tikzcd}
{\G}_{sc}\ar[d,"\pi"]\ar[r,"\theta_{sc}"]&\G_{sc}\ar[d,"\pi"]\\
{\G}_{der}\ar[r,"\theta"]&\G_{der}
\end{tikzcd}
\]
commutes. The uniqueness of this involution implies that it is defined over $k$.

Let $L/k$ be a finite Galois extension over which $\G$ splits. Recall that a $z$-extension associated to $L/k$ is a homomorphism $\al: \tilde{\G}\to \G$ such that 
\begin{enumerate}
    \item $\tilde{\G}$ is a connected reductive group over $k$ whose derived group is simply connected,
    \item $\al$ is surjective, and
     \item $\ker(\al)$ is central in $\tilde{\G}$ and isomorphic to a product of tori of the form $\mathrm{Res}_{L/k}(\Gm)$.
\end{enumerate}

\begin{Prop}\label{Prop: good cover}
    Let $L/k$ be a finite Galois extension over which $\G$ splits. There exists a $z$-extension $\al:\tilde{\G}\to \G$ associated to $L/k$ and an involution $\tilde{\theta}:\tilde{\G}\to \tilde{\G}$ such that
    \begin{enumerate}
        \item     the diagram
   \[
    \begin{tikzcd}
    \tilde{\G}\ar[d,"\al"]\ar[r,"\tilde{\theta}"]&\tilde{\G}\ar[d,"\al"]\\
    \G\ar[r,"\theta"]&\G    
    \end{tikzcd}
    \]
    commutes, and
    \item $\widetilde{\G}^{\widetilde{\theta}}$ is a $z$-extension of $\G^\theta$. In particular, $\widetilde{\G}^{\widetilde\theta}(k)$ surjects onto $\G^\theta(k)$.
     \end{enumerate}
    \quash{ Finally, we may choose $\widetilde{\G}$ such that $\al$ induces a $k$-rational isomorphism 
     \[
\widetilde{\X} =\widetilde{\G}/\widetilde{\rH}\iso \rH\backslash\G=\X.
     \]}
\end{Prop}
We refer to the $z$-extensions satisfying the claims of the proposition as ($\theta$-)compatible. We moreover note that the construction works for any finite order $k$-automorphism of $\G$.
\begin{proof}
Let $\G_{sc}\to \G_{der}$ denote the non-trivial covering map of the simply connected cover of $\G_{der}$. Let $Z_{sc}$ denote the center of $\G_{sc}$ and $Z\subset \G$ denote the center of $\G$. Then we may write $\G$ as the fiber sum $\G=\G_{sc}\times^{Z_{sc}}Z$ \cite[2.0.1]{deligne1979varietes}. We recall the following useful lemma.
    \begin{Lem}\label{Lem: milneshih}\cite[3.1]{milne1982conjugates}
        Let $\mathcal{E}$ be a finite group and let $M$ be a finitely generated $\mathcal{E}$-module. There exists an exact sequence of $\mathcal{E}$-modules 
        \[
        0 \lra P_1\lra P_0\lra M\lra 0
        \]
        such that $P_0$ is free and finitely generated as a $\mathbb{Z}$-modules and $P_1$ is free as a $\zz[\mathcal{E}]$-modules.
    \end{Lem}
    
    Now let $\rA\subset \G$ be a $\theta$-stable maximal $k$-torus of $\G$ that splits over $L/k$, and let $\rA_{sc}$ be the inverse image of $\rA$ in $\G_{sc}$; it is clearly a $\theta_{sc}$-stable maximal $k$-torus of $\G_{sc}$. By assumption, both lattices $X_{\ast}(\rA)$ and $X_{\ast}(\rA_{sc})$ are finitely generated $\mathcal{E}:=\Gal(L/k)\times \la\theta\ra$-modules.
    
Set $M=X_{\ast}(\rA)/X_\ast(\rA_{sc})$. Applying Lemma \ref{Lem: milneshih} to this module, we obtain the commutative diagram of $\mathcal{E}$-modules
\begin{equation}
    \begin{tikzcd}
            &&X_\ast(\rA_{sc})\ar[d,"\iota"]\ar[r,"="]&X_\ast(\rA_{sc})\ar[d]&\\
   0\ar[r]&P_1\ar[d,"="]\ar[r]&\widetilde{X}\ar[d]\ar[r]&X_{\ast}(\rA)\ar[d]\ar[r]&0\\
    0\ar[r]&P_1\ar[r]&P_0\ar[r]&M\ar[r]&0.
    \end{tikzcd}
\end{equation}
By construction, there is a natural involution $\theta$ on each term in this diagram, and each arrow is $\theta$-equivariant. Moreover, the central row induces the exact sequence of tori
\[
1\lra N\lra \widetilde{\rA}\lra \rA\lra 1,
\]
where $N$ is an induced torus. Note that the automorphism $\theta$ induces a unique involution $\widetilde{\theta}$ on $\widetilde{\rA}$ lifting the involution $\theta$. Pulling back along the inclusion $Z\hra \rA$, we obtain the short exact sequence
\[
1\lra N\lra \widetilde{Z}\lra Z\lra 1.
\]
Set $\widetilde{\G}=\G_{sc}\times^{Z_{sc}}{\widetilde{Z}}$; this comes equipped with a  surjective morphism
\[
1\lra N\lra \widetilde{\G}\overset{\al}{\lra} \G\lra 1,
\]
as well as an involution $\widetilde{\theta}$ satisfying the first claim of the proposition.

We verify the final claim on rational points. By construction, $N=\ker(\al)$ is a torus such that $X^\ast(N)$ possesses a free $\Gal(L/k)\times \la \theta\ra$-action. Letting $N^\theta\subset N$ be the fixed-point subgroup, if follows that $N^\theta$ is connected, since $X^\ast(N)_{\theta=\pm1}$ are saturated sub-lattices. Moreover, since $X^\ast(N)$ is a free $\Gal(L/k)\times \la \theta\ra$-module, $X^\ast(N^\theta)$ is a free $\Gal(L/k)$-module. This implies that 
\[
1\lra N^\theta(k)\lra \widetilde{\G}^{\widetilde{\theta}}(k)\lra \G^\theta(k)\lra 1.\qedhere
\]
\end{proof}

 \begin{Rem}
   Suppose $\theta$ is a $k$-rational inner involution; that is, assume there is $a\in \mathrm{Inn}(\G)(k)$ such that $\theta(g) = \Ad(a)(g)$ is an involution. There is a canonical identification
\[
\mathrm{Inn}(\widetilde{\G})\simeq \mathrm{Inn}(\G),
\] for any $z$-extension $\widetilde{\G}$ of $\G$, so there is a canonical lift of $\theta$ to $\tilde{\G}$ such that $\tilde{\theta}|_{\widetilde{Z}}\equiv id$. This does \textbf{not} correspond to the construction above, as the action of $\theta$ on $X^\ast(N)$ is not free. Such a lift will not necessary recall the connected components of $\rH$, as the following example illustrates.
 \end{Rem}

\begin{Ex}\label{Ex: still type N}
    When $\G=\PGL_2$ and \[
    \theta\begin{psmatrix}
        a&b\\c&d
    \end{psmatrix}=\begin{psmatrix}
        a&-b\\-c&d
    \end{psmatrix},\]
    we have $\rH = N_{\G}(\rA)$, where $\rA$ is the diagonal torus. The natural lift as an inner automorphism to the $z$-extension $\GL_2$ has fixed-point subgroup the diagonal torus $\widetilde{\rA}$, which does not surject onto $\rH$. On the other hand, Proposition \ref{Prop: good cover} gives
    \[
    \widetilde{\G} = \SL_2\times^{\{\pm I_2\}} \Gm^2,
    \]
    with 
    \[
    \widetilde{\theta}\left(\begin{psmatrix}
        a&b\\c&d
    \end{psmatrix},\begin{psmatrix}
        x&\\&y
    \end{psmatrix}\right)=\left(\begin{psmatrix}
        a&-b\\-c&d
    \end{psmatrix},\begin{psmatrix}
        y&\\&x
    \end{psmatrix}\right),
    \]
    so that
    \[
   \widetilde{\rH}=\left\{ \left(\begin{psmatrix}
        a&\\&a^{-1}
    \end{psmatrix},\begin{psmatrix}
        x&\\&x
    \end{psmatrix}\right): a,x\in \Gm\right\}\bigsqcup \left\{\left(\begin{psmatrix}
        &b\\-b^{-1}&
    \end{psmatrix},\begin{psmatrix}
        x&\\&-x
    \end{psmatrix}\right): b,x\in \Gm\right\}.\hfill\qed
    \]
\end{Ex}

\subsection{Spherical data, derived subgroups, and $z$-extension}\label{Section: reductions}
Suppose that $(\G,\rH,\theta)$ is a symmetric pair over $k$ and let $\X=\rH\backslash\G$. In Appendix \ref{Section: norms sym}, we compute the homogeneous spherical data 
\[
\Omega_\X=(X^\ast(\Ax),\De_\X,\Omega^{(1)},\Omega^{(2)})
\]of $\X$ in terms of the involution $\theta$. On the other hand, one also obtains involutions $\theta_{ad}:\G_{\ad}\lra \G_{ad}$ and $\theta_{sc}: \G_{sc}\lra \G_{sc}$, or may consider compatible $z$-extensions of $(\G,\theta)$ such that the natural diagrams commute. Our present goal is to clarify the relations between the corresponding spherical data. 

Let $\rH\subset \G$ be a symmetric subgroup and consider $\rH_{der}:=\rH\cap \G_{der}$. 
\begin{Lem}\label{Lem: passing to derived on roots} The pair $(\G_{der},\rH_{der})$ is a symmetric pair. There is a surjection $$X^\ast(\Ax)\lra X^\ast(\rA_{\X_{der}}),$$ inducing a $\Ga$-equivariant bijection of spherical roots
\[
\De_{\X}\iso \De_{\X_{der}},
\]
sending a spherical root $\al\in X^\ast(\rA_{\X})$ to the positive multiple that is primitive in $X^\ast(\rA_{\X_{der}})$.
\end{Lem}
\begin{proof}
    The first claim is obvious. The second is contained in Appendix \ref{Sec: functorial derived}, where it is shown that the root systems associated to $\X$ and $\X_{der}$ agree after re-normalization.
\end{proof}
\quash{\begin{Rem}\textcolor{red}{To be removed}
    The case of $(\GL_2,O_2)$ and $(\SL_2,\SO_2)$ shows that the image of $\De_\X$ in $\fX_{der}$ need not be primitive. I suspect that one can show that this is a rank one phenomenon. The idea is that the image of $\De_\X$ satisfies the axioms of a root system, so that being able to divide a simple root by $2$ while maintaining the same set of normalized roots would need to be constant along components of the Dynkin diagram. This requires all the roots to by type $N$ or distinguished, which only happens for Chevalley varieties or rank $1$ varieties. 
\end{Rem}}
Let $(\widetilde{\G},\widetilde{\theta})$ denote a $\theta$-compatible $z$-extension in the sense of Proposition \ref{Prop: good cover}; in particular, $\widetilde{\G}_{der} =\G_{sc}$. Since we have a short exact sequence
\[
1\lra N^{\widetilde{\theta}}\lra \widetilde{\G}^{\widetilde{\theta}}\lra \G^\theta\lra 1,
\]
we obtain 
\begin{equation}\label{eqn: nice normalizer}
    1\lra N\lra N_{\widetilde{\G}}(\widetilde{\G}^{\widetilde{\theta}})\lra N_{\G}(\G^\theta)\lra 1,
\end{equation}
and $\pi_0(\widetilde{\G}^{\widetilde{\theta}})\simeq \pi_0(\G^\theta)$. In particular, there is a bijection of intermediate groups
\[
\{(\G^\theta)^\circ\subset \rH\subset N_{\G}(\G^\theta)\}\leftrightarrow\{(\widetilde\G^{\widetilde\theta})^\circ\subset \widetilde\rH\subset N_{\widetilde\G}(\widetilde\G^{\widetilde\theta}): N^{\widetilde\theta}\subset \widetilde\rH\}.
\]
Thus, we obtain a $k$-rational subgroup $\widetilde\rH\subset \widetilde\G$ with a fibration
\[
N^{\widetilde\theta}\backslash N\lra \widetilde\X:=\widetilde\rH\backslash\widetilde\G\lra \X,
\]
and a short exact sequence
\begin{equation}\label{eqn: SES automorphism}
    1\lra N^{\widetilde\theta}\backslash N \lra \Aut^{ \widetilde\G}( \widetilde\X)\lra \Aut^{\G}(\X)\lra 1.
\end{equation}
 
\begin{Lem}\label{Lem: pass to a sc cover roots}
    With the notation as above, there is a natural (e.g. $\Ga$-equivariant) identification 
\[
\text{$\Lam_{\widetilde\X}=\Lam_\X$ and $\De_{\widetilde\X}=\De_\X$}.
\]
Moreover, the quotient map $\widetilde\X\lra \X$ induces a $\Ga$-equivariant bijection  of colors $\D(\widetilde{\X})\iso \D(\X)$.
\end{Lem}
\begin{proof}
   We begin by choosing a Borel $N\subset \widetilde{\B}$ of $\widetilde{\G}$ and letting $\B\subset \G$ be the image. The first claim follows from the short exact sequence \eqref{eqn: SES automorphism} and the definition of $\Lam_\X$. Indeed, there is a short exact sequence on canonical tori
    \[
 1\lra N^{\widetilde\theta}\backslash N \lra \rA_{\widetilde\X}\lra \Ax\lra 1,
    \]
    corresponding to a dual sequence
    \[
0\lra X^\ast(\Ax)\lra X^\ast(\rA_{\widetilde\X})\lra X^\ast(N^{\widetilde\theta}\backslash N)\lra 0.
    \]
    Since $X^\ast(\Aut^{\G}(\X))=X^\ast(\Ax)/\Lam_\X$, the identification of root lattices follows. For the spherical roots, since $ X^\ast(\Ax)\subset X^\ast(\rA_{\widetilde\X})$ is a saturated sublattice, the minimal ray generators in $X^\ast(\Ax)$ are also minimal in $X^\ast(\rA_{\widetilde\X})$, giving the identification of spherical roots.

    Finally, it is immediate that the map induces an order-preserving (in terms of the adherence relation) bijection between $\widetilde{B}$-orbits on $\widetilde{\X}$ and $B=\widetilde{B}/N$-orbits on $\X$.
\end{proof}

The upshot is that if $\Omega_\X=(X^\ast(\Ax),\De_\X,\Omega^{(1)},\Omega^{(2)})$ is the spherical data for $\X$, then $\Omega_{\widetilde{\X}}=(X^\ast(\rA_{\widetilde{\X}}),\De_\X,\Omega^{(1)},\Omega^{(2)})$.  
Replacing $\G$ with $\widetilde{\G}$ if necessary, we now suppose $\G_{der}=\G_{sc}$. There is an isomorphism $\G_{sc}\cong\prod_i \G_{sc,i} \times \prod_{j}(\G_{sc,j}\times \G_{sc,j})$, where each factor is $k$-simple and simply-connected and $\theta$ stabilizes each $\G_{sc,i}$ in the first product and acts as the swap automorphism on each factor in the second product. If we use the notation $\rH_{sc} = \prod_i \rH_{sc,i} \times \prod_{j}\rH_{sc,j}$ for the induced product on $\rH\cap\G_{sc}$, then $\X_{sc}=\prod_i \X_i \times \prod_{j}\X_j$ and there is a product decomposition of the spherical data
\[
\Omega_{\X,sc} =\prod_i \Omega_{\X_i} \times \prod_{j}\Omega_{\X_j}, \qquad \Omega_{\X_i} = (X^\ast(\rA_{X,i}),\De_{\X_i},\Omega_i^{(1)},\Omega_i^{(2)}).
\]

Combining Lemmas \ref{Lem: pass to a sc cover roots} and \ref{Lem: passing to derived on roots} with the preceding discussion, for any symmetric $\G$-variety $\X=\rH\backslash\G$ and any spherical root $\ga\in \De_\X$, there is a well-defined \emph{$k$-simple component} $\X_\ga$ of $\X_{sc}$ such that under the bijections
\begin{equation}\label{eqn: simple components}
 \De_{\X}\to \De_{\widetilde{\X}}\to \De_{\X_{sc}}=\bigsqcup_i\De_{\X_i}\sqcup\bigsqcup_j\De_{\X_j},   
\end{equation}
$\Ga\cdot \ga$ is sent to $\De_{\X_\ga}$.

 \subsection{Rational involutions and $(\Ga,\theta)$-indices}\label{Section: indices} 
  Suppose now that $\G$ is a quasi-split connected reductive $k$-group. Let $\theta$ be a $k$-rational involution of $\G$ and fix a $(\theta,k)$-admissible pair $(\rA,\B)$. Recall the notion of $(\Ga,\theta)$-index  
  \[
  \I=(X^\ast(\rA), \De, \emptyset,\De_\X^p, \sig_\ast,\theta^\ast).
  \]from Section \ref{Section: rational involutions}, where we note that $\De_0(\Ga)=\emptyset$ since $\G$ is quasi-split and we have exchanged the notation $\De_0(\theta)$ for $\De_\X^p$ in reference to $\X=\G^\theta\backslash\G$. This provides a useful variant of the combinatorial spherical datum $\Omega_\X=(X^\ast(\Ax),\De_\X,\Omega^{(1)},\Omega^{(2)})$ in the symmetric setting.

\begin{Rem}
    Aside from Proposition \ref{Prop: endoscopic roots}, there are other benefits of utilizing $(\Ga,\theta)$-indices. Indeed, when $\overline{\X}$ is a symmetric $\G_{\kbar}$-variety in Proposition \ref{Prop: quasirational}, it is not immediate clear from the proof of \cite{BorovoiGagliardi} that the resulting homogeneous spherical $k$-variety $\X=\rH\backslash\G$ is symmetric: that is, we do not obtain an involution $\theta$ satisfying $(\G^\theta)^\circ\subset \rH\subset N_{\G}(\G^\theta)$ in this way. Beyond this, Borovoi and Gagliardi restrict to characteristic zero. This is likely not too serious for large enough positive characteristic, but working with $(\Ga,\theta)$-indices resolves both of these concerns. 
\end{Rem}

We first explain how to recover the combinatorial spherical datum $\Omega_\X$ for $\X=\G^\theta\backslash\G$ from the information of $\I$ and the involution $\theta|_\rA$, using the calculations in Appendix \ref{Section: norms sym}.  From this data, we set $\Omega_\X= (\fX,\De_\X,\Omega^{(1)},\Omega^{(2)})$ as follows:
\begin{enumerate}
    \item Set $X^\ast(\Ax)=\{\lam-\theta(\lam): \lam\in X^\ast(\T)\}\subset X^\ast(\T)$;
    \item set $\De_{\X}^n:=\{\al-\theta(\al): \al\in \De, \theta(\al)\neq \al\}$. We then re-normalize through generators in $X^\ast(\Ax)$ to obtain $\De_{\X}$;
    \item set $\D(\X) = \bigsqcup_{\al\in \De_{\X}^n}\D(\al^\vee)$, where the sets $\D(\al^\vee)$ are color data determined over $\kbar$ as in Section \ref{Section: sym colors}. More precisely, the index $\I$ determines a symmetric pair $(\G_{\kbar},\theta)$ uniquely (up to isomorphism) over $\kbar$. Setting $\rH_{\kbar}=\G^\theta_{\kbar}$, the equivalence relation defined in Section \ref{Section: sym colors} depends on the connected component group of $\rH_{\kbar}$.   
   From this, we obtain the map
    \[
    \rho\times \varsigma: \D(\X)\lra \fa_{\X,\qq}\times \mathcal{P}(\De),
    \]
    and we let $(\Omega^{(1)},\Omega^{(2)})$ denote the corresponding sets, with the natural $\Ga$-action induced from that on $\fa_{\X}\times \mathcal{P}(\De)$.
\end{enumerate}
By Appendix \ref{Section: norms sym}, this recovers $\Omega_\X$.  It is now a straightforward check that the {admissibility} properties of the $(\Ga,\theta)$-index imply that the $\Ga$-action $\sig_\ast$ preserves $\Omega_\X$.  Since $\G$ is assumed to be quasi-split, this agrees with the action on the combinatorial data of $\X$.
 
        Note that this does not recover the $\Ga$-action on $\D(\X)$ as it relies on passing to $\kbar$ to obtain a uniquely determined (up to isomorphism) involution $\theta$. This is related to the fact that the $(\Ga,\theta)$-index $\I$ does not uniquely determine the involution $\theta$, reflecting the problem of $\G$-outer forms from Sections \ref{Section: outer forms} and \ref{Section: well adapted outer}. The results of \cite{Helminckrational} toward classifying involutions with a fixed $(\Ga,\theta)$-index unfortunately do not encode the group $\Out_{\X}(\G^\theta)$ or the geometric cocycle of $\X$. We address this in Section \ref{Section: color autos} by computing $\Out_{\X}(\rH)$ for all symmetric subgroups. 

We end this section with the following lemma regarding the failure of the $(\Ga,\theta)$-index $\I$ to uniquely determine the involution $\theta$.
\begin{Lem}\label{Lem: same on torus}
   Suppose that $\G$ is quasi-split, $\rA$ is a maximally $k$-split maximal torus, $\B$ is a Borel subgroup containing $\rA$. Assume that $\theta$ and $\theta'$ are two involutions normally related to $A$ and such that $\B$ is maximally split for both involutions. Assume also that $$\theta|_{\rA}=\theta'|_{\rA}.$$ Then $(\G,\theta, \rA,\B)$ and $(\G,\theta',\rA,\B)$ induce the same $(\Ga,\theta)$-indices $\mathrm{I}$. Moreover, there is a $\Ga$-equivariant bijection $\Omega_\X\iso \Omega_{\X'}$ between the combinatorial data for the varieties $\X=\G^\theta\backslash \G$ and $\X=\G^{\theta'}\backslash \G$. 
\end{Lem}
\begin{proof}
   As argued in \cite[Section 8]{Helminckrational}, the association of the involutions to indices factors through the restriction to $\rA$, implying the indices are congruent. More specifically, there exists $g\in \G(\kbar)$ such that $$\theta' = \Ad(g)\circ\theta\circ \Ad(g)^{-1} = \Ad(s(g))\circ \theta.$$ Theorem 8.9 of \cite{Helminckrational} proves the first claim, up to congruence.
    
    In fact, the actual proof of \emph{loc. cit.} implies that our assumption that the Borel subgroup $B$ is maximally split for both involutions forces the indices to agree on the nose. Indeed, the actions agree on $X^\ast(\rA)$ and if $\De$ denotes the simple roots associated to $B$, then $\De$ gives a $\Ga_\theta$-fundamental order for both involutions \cite[Definition 5.25]{Helminckrational}.
    
That $\X$ and $\X'$ induce the same combinatorial invariants now follows from the previous discussion.
    \end{proof}

\begin{Rem}\label{Rem: what about Helminck}
    This lemma fits naturally into the classification schema of \cite{Helminckrational}, and illustrates the extent to which it determines the variety $\G^\theta\backslash\G$. Sticking with the assumption that $\G$ is quasi-split, let
 \[
 \Ind(\Ga,\theta):= \{\text{admissible $(\Ga,\theta)$-indices on  } (X^\ast(\rA),\Phi(\rA))\}/\sim,
 \]
 where $\sim$ denotes the natural notion of {congruence} of indices \cite[Section 5]{Helminckrational}. If we set $\mathrm{Inv}_k(\G)$ to be the set of $\G(k)$-conjugacy classes of $k$-rational involutions on $\G$, there is a natural surjective map
 \[
 \rho_{\G}: \mathrm{Inv}_k(\G)\lra  \Ind(\Ga,\theta),
 \]
 given by conjugating any involution so that $(\rA,\B)$ is an admissible pair and passing to the induced $(\Ga,\theta)$-index. A central result of \cite{Helminckrational} is to classify the image of this map (cf. Section \ref{Section: rational involutions}).

To characterize the fibers of $\rho_\G$, Helminck shows that this map factors 
 \[
 \mathrm{Inv}_k(\G)\overset{\mu}{\lra} \mathrm{Inv}(\G,\mathrm{A}){\lra}\Ind(\Ga,\theta),
 \]
 where $\mathrm{Inv}(\G,\mathrm{A})$ is obtained from $\mathrm{Inv}_k(\G)$ by identifying involutions $\theta$ and $\theta'$ if $\theta|_{\rA}=\theta'|_{\rA}$, and the map $\mu$ is the natural quotient map. One then considers the fibers of each factor. Lemma \ref{Lem: same on torus} states that involutions identified in $\mathrm{Inv}(\G,\mathrm{A})$ induce $\Ga$-equivariantly isomorphic homogeneous spherical data. However, as Example \ref{Ex: unitary example} shows, two such involutions can possess distinct geometric cocycles resulting in $\G$-outer forms. 
\end{Rem}

\section{Outer data for symmetric varieties}\label{Section: color autos} 
We continue to assume that $\G$ is quasi-split over $k$ and that $\X=\rH\backslash\G$ is a symmetric variety. We compute the group $\Out_\X(\rH)$ for symmetric varieties in Theorem \ref{Thm: Outer in terms of dist} below. In particular, we isolate a particular class of symmetric varieties (referred to as Chevalley varieties; see Section \ref{Sec: Type N}) such that every symmetric variety not in this class is well-adapted (cf. Definition \ref{Def: well adapted}). Along the way, we also establish the existence of quasi-split forms in each $\G$-inner class for $\X$.
In the later Sections \ref{Section: hamiltonian endoscopy} and \ref{Section: BZSV conj}, we interpret these results for symmetric varieties without type $N$ roots in terms of the dual Hamiltonian $\check{\G}$-variety introduced in some generality by Ben-Zvi, Sakellaridis, and Venkatesh \cite{BZSV}.

Let $\theta$ be the involution such that $(\G^\theta)^\circ\subset \rH\subset N_{\G}(\G^\theta).$ Let $\rA$ be a $(\theta, k)$-admissible torus and let $\Ax$ denote the quotient by which $\rA$ acts on $\X$. We have the embedding $\cala_\X\hra \Ax$ and the natural quotients $\Out_{\X}(\rH)$ and $\Aut_d(\X)$ of $\cala_\X$. Recall that a spherical variety is called well adapted if the natural surjective morphism
\[
\Out_\X(\rH)\lra \Aut_d(\X)
\]
is an isomorphism. \quash{subgroup $\cala_\X^\flat$ be the lattice satisfying $X^\ast(\mathcal{A}_\X^\flat) = X^\ast(\Ax)/\Lam_\X^\flat$, so that 
\[
1\lra \mathcal{A}_\X^\flat\lra \cala_\X\lra \Out_\X(\rH)\lra 1
\]
is dual to
\[
0\lra \Lam_\X^\flat/\Lam_\X^n\lra X^\ast(\Ax)/\Lam_\X^n\lra X^\ast(\Ax)/\Lam_\X^\flat\lra 0.
\]
The goal of this section is to compute this lattice for symmetric varieties. }
\begin{Thm}\label{Thm: Outer in terms of dist}
    Suppose that $\G$ is quasi-split over $k$ and that $\X=\rH\backslash\G$ is a symmetric $\G$-variety that has no factors of Chevalley type (see Definition \ref{Def: chevalley type}). Then $\X$ is well adapted.
\end{Thm}

\begin{Rem}
    This statement is not as general as possible. For example, the  calculation in Lemma \ref{Lem: type N outer} shows that those symmetric varieties of Chevalley type for $G_2$, $F_4$, $E_6$, $E_8$, as well as involutions of $\Spin_{2n}$ arising from outer automorphisms are also well-adapted. We opt for this formulation as we will later adopt the assumption that $\X$ has no factors of Chevalley type (cf. Assumption \ref{assumption: no type N}).
\end{Rem}
\begin{Rem}
    The case of varieties of Chevalley type is handled in Proposition \ref{Prop: absolutely simple case}, which remarks on a striking relationship between $\Out_\X(\rH)$ in the Chevalley setting and the action of $\cala_\X$ on the irreducible components of the regular part of the nilpotent cone of $\X$. \quash{Indeed, Lemma \ref{Lem: type N outer} allows for the following result:

   Let $\X=\rH\backslash\G$ be a symmetric $\G$-variety which is not well adapted. Lemma \ref{Lem: type N outer} calculates $\Out_\X(\rH)$ in terms of $\rH$-invariant automorphisms of the regular locus of the nilpotent cone of $\T^\ast_{x_0}(\X)$; we denote this group scheme as $\Aut_{\rH}(\caln_\X)$. 

    \begin{Cor}\label{Rem: include Chevalley}
          Suppose that $\G$ is quasi-split over $k$ and that $\X=\rH\backslash\G$ be a symmetric $\G$-variety. The morphism
          \[
          \Out_{\X}(\rH)\lra \Aut_d(\X)\times \Aut_{\rH}(\caln_\X)
          \]
          is injective.
    \end{Cor}}
\end{Rem}

While not directly related to the preceding theorem, we establish the following result along the way.
\begin{Thm}\label{Thm: quasi-split G-inner form}%
    Suppose that $k$ is perfect, $\G$ is quasi-split over $k$, and that $\X=\rH\backslash\G$ is a symmetric $\G$-variety. There exists a $\G$-inner form $\X_{qs}$ of $\X$ and $x\in \X_{qs}(k)$ such that the connected component of the identity of $\mathrm{Stab}_{\G}(x) = \rH_{qs}$ is quasi-split over $k$.
\end{Thm}
The assumption that $k$ is perfect arises in the proof of Lemma \ref{Lem: Galois and inner} which relies on a lemma of Kottwitz, and can be dropped in most cases.

\quash{Combining the computations there with the previous proposition, we obtain the following statement.
\begin{Prop}
      Suppose that $\G$ is quasi-split over $k$ and $\rH$ is connected. There is a canonical decomposition $\X_{sc}=\rH_{sc}\backslash \G_{sc}= \X_d\times \X_N$ where $\X_d$ has no $k$-simple factors of type $N$ and each simple factor of $\X_N$ is of type $N$. With respect to this decomposition, 
      \[
       X^\ast(\Out_{\X}(\rH))\simeq (\zz/2\zz)^{\De_\X^{dist}}\oplus (\zz/2\zz)^{\Omega_2},
      \]
      where $\Omega_2$ is described in  Proposition \ref{Prop: absolutely simple case}.
\end{Prop}}
 
\subsection{Reduction to absolutely simple case}\label{Section: reduction to asc}
In this section, we reduce the proofs of Theorems \ref{Thm: Outer in terms of dist} and \ref{Thm: quasi-split G-inner form} to the case where $\G$ is simply connected and semi-simple. Note that Lemma \ref{Lem: well adapted connected} implies that it suffices to prove Theorem \ref{Thm: Outer in terms of dist} with the assumption $\rH$ is connected; additionally, the statement of Theorem \ref{Thm: quasi-split G-inner form} only requires consideration of this setting. We therefore may assume $\rH$ is connected in the following. 

Let $(\widetilde{\G},\widetilde{\theta})$ denote a $\theta$-compatible $z$-extension in the sense of Proposition \ref{Prop: good cover}. If $p:\widetilde\G\to \G$ is the covering map, and set $\widetilde{\rH}=p^{-1}(\rH)$ and $\widetilde\X=\widetilde\rH\backslash\widetilde\G$.
\quash{Since we have a short exact sequence
\[
1\lra N^{\widetilde{\theta}}\lra \widetilde{\G}^{\widetilde{\theta}}\lra \G^\theta\lra 1,
\]
we obtain 
\begin{equation}\label{eqn: nice normalizer}
    1\lra N\lra N_{\widetilde{\G}}(\widetilde{\G}^{\widetilde{\theta}})\lra N_{\G}(\rH)\lra 1.
\end{equation}

In particular, there is a bijection of interstitial groups
\[
\{(\G^\theta)^\circ\subset \rH\subset N_{\G}(\G^\theta)\}\leftrightarrow\{(\widetilde\G^{\widetilde\theta})^\circ\subset \widetilde\rH\subset N_{\widetilde\G}(\widetilde\G^{\widetilde\theta})\}.
\]
Thus, we obtain a subgroup $\widetilde\rH\subset \widetilde\G$ with a fibration
\[
N^{\widetilde\theta}\backslash N\lra \widetilde\X:=\widetilde\rH\backslash\widetilde\G\lra \X,
\]
and a short exact sequence
\begin{equation}\label{eqn: SES automorphism}
    1\lra N^{\widetilde\theta}\backslash N \lra \Aut^{ \widetilde\G}( \widetilde\X)\lra \cala_\X\lra 1.
\end{equation}}
 
\begin{Lem}\label{Lem: pass to a sc cover}
    With the notation as above, there is a natural bijection 
    $$
    \Out_{\widetilde{\X}}(\widetilde\rH)\iso\Out_{\X}(\rH).
    $$
  Moreover, $\X$ is well-adapted if and only if $\widetilde\X$ is. Finally, a quasi-split $\widetilde\G$-inner form $\widetilde\X$ exists over $k$ if and only if a quasi-split $\G$-inner form $\X_{qs}$ exists over $k$.
\end{Lem}
\begin{proof} 

 It follows from the short exact sequence \eqref{eqn: SES automorphism} that there is an isomorphism of group schemes 
\[
\pi_0(\Aut^{ \widetilde\G}( \widetilde\X))\simeq \pi_0(\cala_\X),
\]
Recalling that $\cala_\X^\circ\subset \mathcal{A}_\X^\flat$, we obtain a surjective map $\pi_0(\Aut^{ \widetilde\G}( \widetilde\X))\lra \Out_{\X}(\rH)$.

Since $N^{\widetilde\theta}\subset Z(\widetilde{\G})$, the image of 
$
 N_{\widetilde{\G}}(\widetilde{\G}^{\widetilde{\theta}})\to \Aut(\widetilde\rH)
$ 
lies in the subgroup $\Aut(\widetilde\rH,N^{\widetilde\theta})$ of automorphisms preserving the subgroup $N^{\widetilde\theta}$, and factors through the quotient group $N_{\G}(\rH)$. Finally, since $\mathrm{Inn}(\widetilde\rH)=\mathrm{Inn}(\rH)$, we obtain a isomorphism 
\[
\begin{tikzcd}
    \Aut^{ \widetilde\G}( \widetilde\X)\ar[d,"{\widetilde{\mathrm{out}}}"]\ar[r]&\cala_\X\ar[d,"{\mathrm{out}}"]\\
    \Out_{\widetilde{\X}}(\widetilde\rH)\ar[r,"\sim"]&\Out_{\X}(\rH).
\end{tikzcd}
\]
Recalling Lemma \ref{Lem: pass to a sc cover roots}, there is a canonical $\Ga$-equivariant bijection $\De_\X^{dist}=\De_{\widetilde\X}^{dist}$ which induces a natural isomorphism 
\[
\Aut_{d}(\widetilde\X)\iso \Aut_d(\X).
\]
Using the commutativity of the diagram
\[
\begin{tikzcd}
    \Out_{\widetilde{\X}}(\widetilde\rH)\ar[r]\ar[d]&\Out_{\X}(\rH)\ar[d]\\
    \Aut_d(\widetilde\X)\ar[r]&\Aut_{d}(\X),
\end{tikzcd}
\]
where both horizontal arrows are isomorphisms, we see that $\X$ is well-adapted if and only if $\widetilde\X$ is. 

The final claim regarding quasi-split forms is obvious as $\widetilde{\rH}$ is a $z$-extension of $\rH$.
\end{proof}
In particular, we are free to replace $(\G,\theta)$ with $(\widetilde\G,\widetilde\theta)$ and assume that $\G_{der}=\G_{sc}$.
 Let $(\G_{sc},\theta)$ denote the induced symmetric pair on the derived subgroup. Then $\G_{sc}^\theta=\G_{sc}\cap(\G^\theta)^\circ$ is connected, and we set $\X_{sc}=\G_{sc}^\theta\backslash \G_{sc}$. Recall that $N_{\G}(\rH) = N_{\G}((\G^\theta)^\circ)$ \cite[Lemma 3.1.1]{Losev}.

\begin{Lem}\label{Lem: pass to derived}
Assume that $\G_{der}=\G_{sc}$ is simply connected and $\rH$ is connected. The induced map
\begin{equation}\label{eqn: surjective on derived}
   \Aut^{\G_{sc}}(\X_{sc})\lra \Out_{\X}(\rH)
\end{equation}
is surjective and induces an isomorphism $\Out_{\X_{sc}}(\G^\theta_{sc})\simeq \Out_{\X}(\rH)$. Moreover, $\X$ is well-adapted if $\X_{sc}$ is. Finally, a quasi-split $\G_{sc}$-inner form $\X_{sc}$ exists over $k$ if and only if a quasi-split $\G$-inner form $\X_{qs}$ exists over $k$.
\end{Lem} 
\begin{proof}
The surjectivity of \eqref{eqn: surjective on derived} follows from the fact that $\G_{sc}\lra \G_{ad}$ is a surjective map of reductive group schemes over $k$. More precisely, for any $n\in N_{\G}(\rH)(\kbar)$, there exists $n'\in \G_{sc}(\kbar)$ satisfying $\Ad(n) = \Ad(n')$. But then it is clear that $n'\in N_{\G}(\rH)(\kbar)\cap \G_{sc}(\kbar)\subset N_{\G_{sc}}(\G_{sc}^\theta)(\kbar)$, so that the images of the further restriction to the adjoint action on $\rH$ agree. 

Note that the surjective map \eqref{eqn: surjective on derived} factors through $\Out_{{\X_{sc}}}(\G_{sc}^\theta)$. Setting $\rH^\flat=\rH\cdot Z(\G)$ and $\rH^\flat_{sc}= \rH^\flat\cap \G_{sc}$, this gives a surjective map
\[
\Out_{{\X_{sc}}}(\G_{sc}^\theta)\simeq (\G_{sc}^\theta\cdot Z(\G_{sc}))\backslash N_{\G_{sc}}(\G_{sc}^\theta)\to \rH_{sc}^\flat\backslash N_{\G_{sc}}(\G_{sc}^\theta)\simeq\Out_{{\X}}(\rH),
\]
so that the claim holds if $\rH^\flat_{sc} =\G_{sc}^\theta\cdot Z(\G_{sc})$. Since we have assumed that $\rH$ is connected, it follows that $(\G^\theta)^\circ\subset \rH\subset (\G^\theta\cdot Z(\G))^\circ$, which gives the equality. There is thus a commutative diagram
\[
\begin{tikzcd}
    \cala_{\X_{sc}}\ar[r]\ar[d]&\cala_\X\ar[d]\\
    \Out_{\X_{sc}}(\G_{sc}^\theta)\ar[r]&\Out_{\X}(\rH),
\end{tikzcd}
\]
where the bottom horizontal arrow is an isomorphism. Since $\rH$ is connected, $\De_{\X}^{dist}$ is in natural bijection with $\De_{\X_{sc}}^{dist}$ via Lemma \ref{Lem: passing to derived on roots}, and an argument mirroring the one in the proof of Lemma \ref{Lem: pass to a sc cover} now shows that $\X$ is well-adapted if and only if $\X_{sc}$ is.

The final claim about quasi-split forms is, again, straightforward.
\end{proof}

In particular, to prove Theorems \ref{Thm: Outer in terms of dist} and \ref{Thm: quasi-split G-inner form} it suffices to assume that $\G$ is semi-simple and simply connected and that $\rH= \G^\theta$ is connected, relying on the preceding Lemmas to reduce to this case. 

\quash{We end this subsection by noting the following useful fact.
\begin{Lem}
    Suppose that $\G$ is a quasi-split reductive group over $k$. Suppose $\X=\rH\backslash\G$ is a symmetric $\G$-variety and let $\X_{sc}=\G_{sc}/\rH_{\sc}$ be the associated $\G_{sc}$-variety. Then
    \[
    \Aut_d(\X)\simeq\Aut_d(\X_{sc}).
    \]
\end{Lem}
\begin{proof}
    When $\rH$ is connected, this statement was proved in the process of proving Lemmas \ref{Lem: pass to a sc cover} and \ref{Lem: pass to derived}.
\end{proof}}

\subsection{Varieties of Chevalley type}\label{Sec: Type N}

We now describe the notion of Chevalley type as in the statement of Theorems \ref{Thm: Outer in terms of dist} and \ref{Thm: quasi-split G-inner form}.  Recall that an involution $\theta$ of $\G$ is called a \emph{Chevalley involution} if there exists a $k$-rational maximal torus $\rA$ such that $\theta|_{\rA}$ is the inversion automorphism. It is natural to call a symmetric variety $\X=\rH\backslash\G$ is a {Chevalley symmetric variety} if the associated involution is a Chevalley involution. We introduce a slightly broader notion in order to isolate a family of symmetric varieties for which $\Out_\X(\rH)$ does not have a relation to distinguished roots. 

\begin{Def}\label{Def: chevalley type}
    Suppose that $\G$ is $k$-simple, simply-connected group over $k$ and that $\theta$ is an involution on $\G$. We say that $\theta$ is an \emph{involution of Chevalley type} if the spherical root system satisfying that $\rk(\Ax)>1$ and there is at most one Galois orbit of spherical roots $\Ga\cdot\al_\X\subset\De_\X$ not of type $N$.

More generally, suppose $\G$ is a connected reductive group, $\theta$ an involution, and $\X$ a symmetric variety associated to $\theta$. We say that $\X$ has \emph{a factor of Chevalley type} if $\X_{sc}$ has a $k$-simple factor of Chevalley type. 
\end{Def}

The next lemma explains the somewhat unnatural terminology.
\begin{Lem}\label{Lem: chevalley type}
 Suppose $\G$ is $k$-simple, simply-connected group. If $\theta$ is an involution of Chevalley type, then there is a finite extension $K/k$ and  an absolutely simple $K$-group $\G'$ with an involution of Chevalley type $\theta'$ such that $(\G,\theta) = (\Res_{K/k}(\G'),\Res_{K/k}(\theta'))$. The pair $(\G',\theta')$ fits into one of the following two cases:
\begin{enumerate}
    \item $\theta'$ is a Chevalley involution on $\G'$, and 
    \item If $\theta'$ is not a Chevalley involution, the $\G'$ is a $K$-form of $\Spin(V)$ for some quadratic space $V$ over $K$, and ${\G'}^{\theta'} = \Spin(V_1)\times^{\mu_2}\Spin(V_2)$ for some decomposition $V=V_1\oplus V_2$ such that
    \[
    \begin{cases}
        n>\min\{\dim(V_1),\dim(V_2)\}\geq 2&: \dim(V)=2n+1,\\
        n>\min\{\dim(V_1),\dim(V_2)\}\geq 3&: \dim(V) =2n.
    \end{cases}
    \]
\end{enumerate}
\end{Lem}
\begin{proof}
 Over $\kbar$, there exists a single conjugacy class of Chevalley involutions. Suppose that $(\G,\theta_{Ch})$ is such an involution, and let $\rA$ be a maximally $\theta_{Ch}$-split torus. There exists a $\theta_{Ch}$-split Borel subgroup containing $\rA$; with respect to this Borel, we compute that 
\[
\X^\ast(\Ax) =2X^\ast(\rA),\quad\text{ and }\quad\De^n_{\X}=2\De.
\]
In particular, the spherical roots are obtained by sending $2\al$ to $\al$ only if $\al\in 2X^\ast(\rA)$. This occurs only for the long simple root of type $C_n$, giving the first claim. 

    The second claim may be shown by a case-by-case analysis, checking the root systems of symmetric pairs $(\G,\theta)$ with $\G$ absolutely simple and simply connected. We simply show that the claimed varieties are of Chevalley type. Suppose that $V$ is a quadratic space of dimension $d$, and let $\G=\Spin(V)$. Suppose that $\theta$ is an involution such that $\rH = \Spin(V_1)\times_{\mu_{2}} \Spin(V_2)$ with $V=V_1\oplus V_2$ denoting an orthogonal decomposition; we may assume that $\dim(V_1)=m\leq [d/2]$. Note that $\theta$ is inner unless $\dim(V)$ is even and $m$ is odd. Let $\X$ denote the associated symmetric variety.
      
      A simple computation shows that if $\De=\{\al_1,\ldots,\al_n\}$ are the simple roots (aligning with the numbering of Bourbaki), then
    \[
    \De^n_\X = \{2\al_1,\ldots,2\al_{m-1}, \ga_m\},
    \]
    where
    \[
    \ga_m = \begin{cases}
        2(\al_m+\cdots+\al_n)&: d=2n+1,\\
        2(\al_m+\cdots+\al_{n-2})+\al_{n-1}+\al_{n}&: d=2n.
    \end{cases}
    \]
In both case, $X^\ast(\Ax) = \sspan_\zz\{2e_1,\ldots,2e_m, e_1+\cdots+e_m\}$, so that $\al_i\notin X^\ast(\Ax)$ for $1\leq i\leq m-1$ except when $m\leq 2$. When $m>2$, this implies that $\{2\al_1,\ldots,2\al_{m-1}\}$ are roots of type $N$\eqref{aN}.

When $m=2$, then $\al_1= (e_1+e_2)-2e_2\in X^\ast(\Ax)$ so that $\al_1$ is a distinguished root of type \eqref{a}. When $\dim(V)$ is even, $\X$ has no roots of type $N$, while when $\dim(V)$ is odd, $\ga_2$ is of type $N$\eqref{bN}.
    
    When $d=2n+1$ and $m=1$ (so that $\X$ has rank $1$), $\ga_1=2\ga_1'$ with $\ga_1'=\al_1+\cdots+\al_n$ a distinguished root of type \eqref{b}, while $\ga_m'\notin X^\ast(\Ax)$ as soon as $m>1$; in this range, $\ga_m$ is of type $N$\eqref{bN}. 
    
When $d=2n$, $\ga_m$ is always a root of type $G$. When $m=1$ (so that $\X$ has rank $1$), $\ga_m=2\ga_m'$ with $\ga_m'\in X^\ast(\Ax)\setminus\zz\Phi$ is of type \eqref{d}. Otherwise $\ga_m'\notin X^\ast(\Ax)$.
    \quash{
Set $\Theta= \{\ga_m\}$.  Then $\X_{\Theta}$ is parabolically induced from $\X_\Theta^L=L_\Theta/\rH_\Theta$, where 
\[
L_\Theta \simeq \frac{A_{\Theta}\times \Spin(V_2')}{\mu_2}\qquad \rH_\Theta =  {A_{\Theta}[2]\times \Spin(V_2)}.
\]
for an appropriately embedded group $\mu_2$ and where $V_2' = X\oplus V_2$ is a $dim(V_2)+1$ quadratic space.
The natural double cover $A_{\Theta}\times \Spin(V_2')\to L_\Theta$ corresponds to a cover $\widetilde{\X}^L_\Theta\to \X^L_{\Theta}$ such that $\ga_m'\in X^\ast(\rA_{\widetilde{\X}_{\{\al\}}})$ is a distinguished root for $\widetilde{\X}^L_\Theta$. }
    \quash{
Since $\G=\G_{der}=\G_{sc}$, if $B$ is a maximally $\theta$-split Borel subgroup and $\rA$ is the canonical torus quotient, then $a\in \rA(\kbar)$ has a unique expression 
\[
a=\prod_{\al}h_\al(t_\al),
\]
where $h_i$ is the simple coroot associated to $\al\in\De$. Thus, if $\theta$ is of type $N$, we see that $\De=\De_-\sqcup \De^\theta$ such that
\[
\theta(\al) =\begin{cases}
    -\al&:\al\in \De_-,\\
    \al&:\al\in \De^\theta,
\end{cases}
\]
and $\De_\X^n=\{2\al:\al\in \De_-\}$. As long as $\theta$ is not the Chevalley involution of type $C$, then $\De_\X^n=\De_\X$, and 
\[
X^\ast(\Ax) = \sspan_{\zz}\{2\omega_\al: \al\in \De_-\},\qquad \De_\X = \sspan_{\zz}\{2\al: \al\in \De_-\}
\]
where $\la\omega_\al,h_\be\ra = \de_{\al,\be}$. 
}
\end{proof}

\subsection{The simply-connected, semi-simple case}
Returning to the proofs of Theorems \ref{Thm: Outer in terms of dist} and \ref{Thm: quasi-split G-inner form}, we may assume $\G$ is simply connected and semi-simple. Write $\G=\prod_i \G_i \times \prod_{j}(\G_j\times \G_j)$, where each factor is $k$-simple, simply-connected, and $\theta$ stabilizes each $\G_i$ in the first product and acts as the swap automorphism on each factor in the second product. Then $\X=\prod_i \X_i \times \prod_{j}\X_j$ and we have decompositions
\[
\Aut^{\G}(\X)=\prod_i \Aut^{\G_i}(\X_i) \times \prod_{j}\Aut^{\G_j\times \G_j}(\X_j);
\]
the morphism $\mathrm{out}_{\rH}^{\G}$ similarly factors to give
\[
\Out_{\X}(\rH)=\prod_i \Out_{\X_i}(\rH_i) \times \prod_{j}\Out_{\X_j}(\rH_j);
\]
in fact, it is easy to see in  the latter case that
$
\Aut^{\G_j\times \G_j}(\X_j) \simeq Z(\G_j),
$
so that $\Out_{\X}(\rH)=\{1\}$. Since $\X_j$ has no distinguished roots, this is compatible with Theorem \ref{Thm: Outer in terms of dist}. Moreover, in this case $\G_j\times \G_j$ is quasi-split if and only if $\rH_j \simeq\G_j$ is so that Theorem \ref{Thm: quasi-split G-inner form} also holds in this case.

Since a similar decomposition holds for $\Aut_d(\X)$, we may project to the $i^{th}$ factor, thereby reducing the claim to the setting that $\G$ is $k$-simple and simply connected, with $\rH=\G^\theta$. There thus exists a finite separable extension $k'/k$ and an absolutely simple group $\G'$ over $k'$ such that $\G\simeq \Res_{k'/k}(\G')$ uniquely up to unique isomorphism of pairs $(k',\G')\simeq(k'',\G'')$ \cite[Example 6.4.6]{Conrad}. 
\begin{Lem}\label{Lem : types of k-simple}
  Suppose that $\G$ is $k$-simple and simply connected, and $\theta$ is an involution with $\rH=\G^\theta$. Then either
  \begin{enumerate}
      \item $(\G,\theta)$ is a Galois symmetric pair; that is, there exists quadratic extension $K/k$ and a $k$-simple group $\rH$ such that $\G=\Res_{K/k}(\rH_K)$, and $\theta$ is a Galois involution; or
      \item there is an involution $\theta'$ of $\G'$ over $k'$ such that $\theta=\Res_{k'/k}(\theta')$.
  \end{enumerate}
\end{Lem}
\begin{proof} 
This is a straightforward application of the general result of \cite[Proposition A.5.14]{ConradPseudoReductive} where $\G'=\G''$ and $k'=k''$. In the first case, $\al:k'\to k'$ is a non-trivial automorphism over a subfield of index $2$, while $\al = 1_{k'}$ in the second case.
\end{proof}

The next two lemmas complete the proof of Theorem \ref{Thm: quasi-split G-inner form}.
\begin{Lem}\label{Lem: Galois and inner}
    Suppose that $\G$ is $k$-simple, simply connected, and quasi-split over $k$ with $\theta$ a $k$-rational involution. Set $\rH=\G^\theta$ and $\X= \rH\backslash\G$. Suppose one of the following holds:
    \begin{enumerate}
        \item \label{Galois type} there exists quadratic extension $K/k$ and a $k$-simple group $\rH$ such that $\G=\Res_{K/k}(\rH_K)$, and $\theta$ is a Galois involution,
        \item \label{inner type} $k$ is perfect and $\theta$ is an inner involution. That is, we assume that $\theta = \Ad(g_\theta)\in \G_{ad}(k)$ for some $g_\theta\in \G(\kbar)$.
    \end{enumerate}
    Then there exists a $\G$-inner form $\X_{qs}\simeq\rH_{qs}\backslash\G$ with $\rH_{qs}^\circ$ quasi-split over $k$.
\end{Lem}
\begin{proof}
    We first consider \eqref{Galois type}, so that $\G=\Res_{K/k}(\rH_K)$ for a quadratic extension $K/k$. Letting $\rH^\ast$ denote a quasi-split inner form of $\rH$, we see that $\G^\ast=\Res_{K/k}(\rH^\ast_K)$ is also quasi-split over $k$ and an inner form of $\G$. This implies the existence of a $k$-rational isomorphism $\psi:\G^\ast\to \G$. The image $\psi(\rH^\ast)$ gives the desired $\G$-inner twist.

    Now suppose that $\theta$ is inner in the sense of \eqref{inner type}. Thus we have a semi-simple element $\Ad(g_\theta)\in \G_{ad}(k)$ such that the centralizer in $\G_{ad}$ is $\G_{ad,\theta}:=N_{\G_{ad}}(\rH_Z)$. Here, $\rH_Z:=\rH/(Z(\G)\cap \rH)$ is the image of $\rH=\G^\theta$ in the adjoint quotient. Lemma 3.3 of \cite{Kottwitzrational} now implies the existence of an element $\theta^\ast=\Ad(g^\ast_\theta)\in \G_{ad}(k)$ which is stably conjugate to $\theta$ and satisfies that $\G_{ad,{\theta^\ast}}^\circ=:\rH_Z^\ast$ is quasi-split over $k$. Unwinding the definitions, we see that $\theta^\ast$ gives an involution of $\G$ in the same $\G$-inner class as $\theta$ such that if $\rH^\ast=\G^{\theta^\ast}\subset \G$, then $\rH^\ast\to \rH^\ast_Z$ is an isogeny. In particular, since $\rH_Z$ is quasi-split, so is $\rH^\ast$.
%
\end{proof}
In particular, the only cases which remain to prove Theorem \ref{Thm: quasi-split G-inner form} are symmetric pairs of the form $(\Res_{k'/k}(\G'),\Res_{k'/k}(\theta'))$ such that $\G'$ is absolutely simple, simply connected, and quasi-split over $k'$ and $\theta'$ induces a nontrivial outer automorphism of $\G'$. It thus suffices to base change to $k'$ and assume that $(\G,\rH)$ is an absolutely simple symmetric pair. 
\begin{Lem}\label{Lem: outer aut}
    Suppose that $\G$ is quasi-split and absolutely simple over $k$ and $\theta$ is a $k$-involution which induces a nontrivial outer automorphism of $\G$. There exists a $\G$-inner twist of $\theta$ such that $\G^\theta$ is quasi-split.
\end{Lem}
\begin{proof}
Consider the short exact sequence
    \[
    1\lra \G_{ad}\lra \Aut(\G)\lra \Out(\G)\lra 1.
    \]
  By assumption, the involution $\theta\in\Aut(\G)(k)$ induces by assumption a non-trivial element $[\theta]\in \Out(\G)(k)$ (in particular, $\G$ is a quasi-split $k$-form of $\SL_n$, $\mathrm{Spin}_{2n}$, or $E_{6,sc}$). Choose a $k$-rational Borel $\B\subset \G$, a maximally $k$-split maximal torus $\rA\subset \B$, and a $\Ga$-stable pinning $\{x_\al\}_{\al\in \De}$. This induces a splitting $p: \Out(\G)\iso \Aut(\G,\B,\rA,\{x_\al\})\subset\Aut(\G)$ and we set $\theta^\ast=p([\theta])$. Then $\theta^\ast$ is a $k$-rational involution such that $\rH^\ast:=\G^{\theta^\ast}$ is quasi-split. Moreover, there exists $g^\ast\in \G(\kbar)$ such that $\Ad(g^\ast)\in \G_{ad}(k)$ and $\theta=\Ad(g^\ast)\circ \theta^\ast$. This implies $\theta^\ast(g^\ast) = (g^\ast)^{-1}$.

We thus consider the group cohomology $$H^1(\la\theta^\ast\ra,\G(\kbar))=\{g\in \G(\kbar): \theta^\ast(g)=g^{-1}\}/\G(\kbar)$$ to determine the possible involutions obtained in this fashion and verify the claim case-by-case.

    We proceed case-by-case:
\begin{enumerate}
    \item (type $A_{2n}$, for $n\geq 1$) The outer automorphism induces the Chevalley involution $\theta^\ast$, so that $\rH^\ast$ is a $k$-form of $\SO_{2n+1}$. In this case, $H^1(\la\theta^\ast\ra,\G(\kbar)) = \{\ast\}$ so that $\rH^\ast$ is a $\G$-inner form of $\rH$.
    \item (type $A_{2n-1}$, for $n\geq 2$) The unique outer automorphism induces the Chevalley involution $\theta^\ast$, so that $\rH^\ast$ is a $k$-form of $\mathrm{SO}_{2n}$. In this case, $H^1(\la\theta^\ast\ra,\G(\kbar)) = \{\theta^\ast, \theta_2\}$ has two elements. Here, $\G^{\theta_2}$ is of type $C_{n}$. In either case, since $\G$ is quasi-split, the split form $\Sp_{2n}$ injects into $\G$.
    \item (type $D_{n}$, for $n\geq 4$) In this case, it is simpler to work with the special orthogonal group rather than $\Spin_{2n}$. We may argue geometrically in this case. Thus, we let $(V,q)$ be a quadratic space of dimension $2n$ such that $\G=\SO(V)$ is a quasi-split group. If 
    \[
    (V,q) \simeq (V_a,q_a)\oplus (V_s,q_s)
    \]
     is the Witt decomposition of $V$, where $(V_a,q_a)$ is anisotropic and $(V_s,q_s)$ is a split quadratic space. Thus, $V_s$ is isometric to a direct sum $H_2^{\oplus k}$ of hyperbolic planes. Since $\G$ is quasi-split, $\dim(V_a)\leq 2$ so that $\dim(V_a)\in \{0,2\}$. 
     
     All outer involutions of type $D$ correspond to orthogonal decomposition $V=V_1\oplus V_2$ where both summands have odd dimension. Moreover, such decompositions are all $\G(\kbar)$-conjugate, so it suffices to exhibit an orthogonal decomposition such that the stabilizing subgroup is quasi-split. 
     But for any partition $2n = (2k_1+1)+(2k_2+1)$, it is elementary to find a decomposition with $\dim(V_i) = 2k_i+1$ and $\SO(V_i)$ quasi-split: fix an orthonormal basis $\{e_1,e_2\}$ for either $V_a$ if $\dim(V_a)=2$ or for a fixed hyperbolic plane if $\dim(V_a)=0$, and set $V_i = H_2^{\oplus k_i}\oplus k e_i$.
    \item (type $E_6$) The group $\rH^\ast$ is of type $C_4$, and $H^1(\la\theta^\ast\ra,\G(\kbar)) = \{\theta^\ast, \theta_2\}$, where $\rH_2=\G^{\theta_2}$ has type $F_4$. If $\G$ is split, then both split forms embed to $\G$. If $\G$ is the quasi-split form of $E_{6,ad}$ of rank $4$, it is also true that both the split forms of $\Sp_8$ and $F_4$ occur as symmetric subgroups. 
\end{enumerate}
This completes the proof of the lemma.
\end{proof}

Finally, we make a few remarks toward Theorem \ref{Thm: Outer in terms of dist}. Noting that $\Out_{\X}(\rH)=\{1\}$ and $\De_\X^{dist}=\emptyset$ when $(\G,\rH)$ is a Galois symmetric pair, we may restrict out attention to the case when $(\G,\rH)$ is the Weil restriction of a canonically associated (up to canonical isomorphism of pairs) absolutely simple symmetric pair $(\G',\rH')$ over a finite separable extension $k'/k$.

\begin{Lem}\label{Lem : pass to absolutely simple}
Suppose that $\G\simeq \Res_{k'/k}(\G')$ and $\theta\simeq \Res_{k'/k}(\theta')$ for a $k'$-rational involution $\theta'$ on $\G'$. There is a natural isomorphism 
    \[
    \Out_{\X}(\rH)\simeq \Res_{k'/k}(\Out_{\X'}(\rH')).
    \]
\end{Lem}
\begin{proof}
If ${\rH'}^\flat\subset N_{\G'}(\rH')$ is the subgroup corresponding to $\mathcal{A}_{\X'}^\flat$. Then $$N_{\G}(\rH)= \Res_{k'/k}( N_{\G'}(\rH'))$$ and $\rH^\flat=\Res_{k'/k}({\rH'}^\flat)$ corresponds to $\mathcal{A}_{\X}^\flat$ in $$\Aut^{\G}(\X)\simeq \Res_{k'/k}(\Aut^{\G'}(\X')).$$
Thus
\[
\Out_{\X}(\rH)\simeq N_{\G}(\rH)/\rH^\flat\simeq \Res_{k'/k}( N_{\G'}(\rH'))/\Res_{k'/k}({\rH'}^\flat) \simeq \Res_{k'/k}(\Out_{\X'}(\rH')),
\]
where the last isomorphism is shown in \cite[Corollary A.5.4(3)]{ConradPseudoReductive}.
\quash{

 \[
\Out(\rH_{\kbar})\simeq \prod_{\sig:k'\hra \kbar}\Out(\rH_{\sig}).
 \]
 On the other hand, $N_{\G}(\rH)_{\kbar} \simeq \prod_{\sig:k'\hra \kbar}N_{\G_\sig}(\rH_\sig)$ acts on $\rH_{\kbar}\cong\prod_{\sig:k'\hra \kbar}\rH_{\sig}$ via conjugation preserving the product structure.}
\end{proof}

The calculation of $\Aut_d(\X)$ in \eqref{eqn: formula for Aut conn} now implies that $\X=\Res_{k'/k}(\X')$ is well-adapted as a $\G$-variety if and only if $\X'$ is well-adapted as a $\G'$-variety.

\subsubsection{The absolutely simple case}
We have reduced the calculation to the case when $\G$ is absolutely simple and simply connected and $\rH=\G^\theta$, where we shall compute $\Out_\X(\rH)$ directly. We summarize these calculations below.

\begin{Prop}\label{Prop: absolutely simple case}
    Suppose that $\G$ is absolutely simple and simply connected over $k$ and $\theta$ is an involution. Set $\rH=\G^\theta$ and $\X=\rH\backslash\G$.
    \begin{enumerate}
        \item\label{dist} Suppose that $\X$ is not a variety of Chevalley type. Then 
        \[
        \Out_\X(\rH)\simeq \Aut_d(\X).
        \]
      \item\label{type N var} Suppose that $\X$ is a variety of Chevalley type. If $\X$ is a form of either
      \[
      \text{$\Spin_2\times^{\mu_2} \Spin_{2n-1}\backslash\Spin_{2n+1}$ or $\GL_n\backslash\Sp_{2n}$ for $n\neq1$, } 
      \]or if $\Out_\X(\rH)=\{1\}$, it is well adapted. Otherwise, the map
      \[
      \Out_\X(\rH)\to \Aut_d(\X)
      \] is not injective.  In each case, the set of $\rH$-orbits on the regular part of the nilpotent cone $\mathcal{N}_\X^{reg}$ gives a $\Out_{\X}(\rH)$-torsor. \quash{(see \cite[Section 5]{Levy} for the relevant definitions). Setting $\Lam_\X^{min}\subset X^\ast(\Ax)$ to be the lattice spanned by $\Lam_\X^n$ and $\{\omega-\theta(\omega): \omega \text{ a minuscule weight of $\G$}\}$, then
    \[
    X^\ast(\Out_{\X}(\rH)) = \Lam_\X^{min}/\Lam_\X^n.
    \]}
    \end{enumerate}
\end{Prop}
This proposition completes the proof of Theorem \ref{Thm: Outer in terms of dist}.

The rest of this section proves this proposition in a case-by-case fashion. First note that the claim holds when $N_{\G}(\rH)$ is contained in a proper parabolic subgroup $P\subset \G$, since then $\rH=N_{\G}(\rH)$ is of Hermitian type. Indeed, since $N_{\G}(\rH)\supset\rH$, $\rH_{\kbar}$ is contained in a proper Levi subgroup $L$ of $P$. This forces $\rH=L$ and the claim now follows from standard results on normalizers of Levi subgroups.  In particular, $\Out_{\X}(\rH)$ is trivial when $N_{\G}(\rH)$ is contained in a proper parabolic subgroup.  Note that Losev shows in \cite[Section 4.3]{Losev} that no distinguished roots exist in this case, so that this verifies cases of the Proposition.

We may now assume throughout this section that $N_{\G}(\rH)$ is not contained in a proper parabolic subgroup of $\G$. Recall the quotient and symmetrization maps 
\[
\G\overset{\tau}{\lra} \X\overset{s}{\lra} \G;
\]Lemma 1 of \cite{VustEmbeddings} implies that $$\Aut^{\G}(\X) \simeq s(\X)\cap Z(\G)\subset Z(\G)^-=\{a\in Z(\G): \theta(z) = z^{-1}\}.$$ More precisely, if $\rA$ is a $(\theta,k)$-admissible torus, then we obtain two identifications under the symmetrization map $s$ 
\[
\cala_\X\simeq \rA^-\cap Z(\G)\cong\{\tau(a)=a^2\in \X:a\in A^Z\},
\]
where $A^Z:=\{a\in A^-: a^2 \in Z(\G)\}$. Thus, $$\Out_{\X}(\rH) \simeq (\rA^-\cap Z(\G))/\tau(Z(\G)).$$


\begin{Lem}\label{Lem: dist case}
    Suppose $\G$ is absolutely simple and simply connected over $k$, $\theta$ is an involution and $\X=\G^\theta\backslash \G$. When $\De_\X^{dist}\neq \emptyset$, 
    \[
    X^\ast(\Out_{\X}(\rH)) \simeq \zz/2\zz,
    \]
    with the non-trivial element given by the \emph{doubling automorphisms} of \cite{Losev}. 
\end{Lem}
\begin{proof}
  
\noindent
\underline{Roots of type $B$:} Suppose that $\X$ possesses a distinguished root of type \eqref{b}. Arguing as in the proof of \cite[Proposition 4.2.1]{Losev} (this is where we use the assumption that $N_\G(\rH)$ is not contained in a proper parabolic subgroup), such roots occur only if $\X=\rH\backslash\G$ is a $k$-form of 
       \[
       \Spin_{2n}\backslash \Spin_{2n+1},\qquad\text{or}\qquad \Sp_{2n}\times \Sp_{2n}\backslash\Sp_{4n}.
       \]
       In each case, $|\De_\X^{dist}|=1$ and the doubling automorphism is outer.\quash{.}\\
       
\noindent
\underline{Double roots:} By Lemma 7.3 of \cite{KnopAutomorphisms}, $\mathcal{A}_{\X}^\sharp = \Aut^{\G}(\X)$ unless $$\mathrm{rank}(Z(\rH)/Z(\G)\cap\rH) >1.$$ The symmetric varieties of this form are called of \emph{Hermitian type} \cite[(2.4)]{VustEmbeddings} and $\rH_{\kbar}$ is a Levi subgroup of $\G_{\kbar}$. Since $\G$ is absolutely simple, $X^\ast(\rH)$ has rank $1$.
    
      Passing to $\kbar$, we may suppose now that $\rH$ is a Levi subgroup of $\G$. If $s\in N_{\G}(\rH)(\kbar)$ such that $\Ad(s)$ acts trivially on $X^\ast(\rH)$, then $s$ acts trivially on $Z(\rH)$. But this then forces $s\in \rH$ so that $\Ad(s)=1\in \cala_\X(\kbar)$. In particular, any non-trivial element of $\varphi\in\cala_\X(\kbar)$ acts non-trivially on $X^\ast(\rH)$. Lemma 7.3 of \cite{KnopAutomorphisms} now forces $\varphi\notin \cala_\X^\sharp(\kbar)$. Lemma \ref{Lem: center in flat in sharp} now gives the claim. 

In each case, one has $|\De_\X^{dist}|= 1$. 
\end{proof}

The next two lemmas complete the proof of Proposition \ref{Prop: absolutely simple case}. They also highlight the connection with the irreducible components of the regular locus of the nilpotent cone in the Chevalley-type case. More precisely, recall that the involution $\theta$ induces a grading 
\[
\fg= \Lie(\G) = \fg_0\oplus \fg_1,
\]
where $\fg_i= \{X\in \fg: d\theta(X) = (-1)^iX\}$. Then $\fg_1\simeq T^\ast_{x_0}(\X)$ is the infinitesimal symmetric variety. If $\caln\subset \fg$ denotes the nilpotent cone of $\fg$, one then defines $$\caln_\X := \fg_1\cap\caln.$$ Over $\kbar$, there is then a unique $N_\G(\rH)$-orbit of \emph{regular nilpotent elements} $\caln_\X^{reg}$, which decomposes into finitely many geometric orbits of $\rH$. In Sections 5 and 6 of \cite{Levy} (see also \cite{SekiguchiNilpotent} for the original characteristic zero calculation), the number of geometric $\rH$-orbits in $\caln_\X^{reg}$ is computed.

We saw previously that $\Out_{\X}(\rH)\simeq (\rA^-\cap Z(\G))/\tau(Z(\G))$. This quotient is computed for the relevant symmetric pairs in \cite[Section 6.4]{Levy} in the analysis of $\caln_\X^{reg}$. As shown in Section 5 of \cite{Levy}, the set of connected components  $\pi_0(\caln_\X^{reg})$ is naturally a torsor of a quotient of $\Out_{\X}(\rH)$ (in the notation of \cite[Theorem 5.17]{Levy}, this is $(\rA^-\cap Z(\G))/\tau(C)$), so that $\Out_{\X}(\rH)$ acts transitively on $\pi_0(\mathcal{N}_{X}^{reg})$. 

In particular, $(\rA^-\cap Z(\G))/\tau(Z(\G))$ is computed for many symmetric pairs in \emph{loc. cit.}.
\begin{Lem}
    Suppose that $\G$ is absolutely simple and simply connected, and that $\rH=\G^\theta$. Suppose $\theta$ is not of Chevalley type and $\De_\X^{dist}=\emptyset$. Then $\Out_{\X}(\rH)=\{1\}$.
\end{Lem}
\begin{proof}   The relevant calculations are carried out in the proofs of Lemmas 6.19 and 6.21, as well as the discussion leading up to Proposition 6.22 of \cite{Levy}. A case-by-case check establishes the claim. 
\end{proof}

\begin{Lem}\label{Lem: type N outer}
    Suppose that $\G$ is absolutely simple and simply connected, and assume that $\X=\G^\theta\backslash \G$ is of Chevalley type. Then $\Out_{\X}(\rH)$ acts simply transitively on the set of $\rH$-orbits on the regular part of the nilpotent cone $\mathcal{N}_\X^{reg}$.
    \begin{enumerate}
        \item Suppose $\theta$ is a Chevalley involution. Then
        \[
        X^\ast(\Out_\X(\rH)) = (\zz/2\zz)^{\Omega_2},
        \]
        where $\Omega_2:= \{\omega\in X^\ast(\rA) \text{ minuscule}, 2\omega\in \zz\Phi\}$.
       \item\label{spin} Suppose $\X$ is a $k$-form of $\Spin(V)/\Spin(V_1)\times^{\mu_2} \Spin(V_2)$ with $|\dim(V_1)-\dim(V_2)|\geq2$. Then 
       \[
        X^\ast(\Out_\X(\rH)) = (\zz/2\zz)^{\Omega_2},
        \]
        where 
        \[
      \Omega_2:=\begin{cases}
            \la\omega_{sp}-\theta(\omega_{sp})\ra&: \theta\text{ inner},\\
            \emptyset&: \theta\text{ outer}.
        \end{cases}
        \]
        When $\theta$ is inner, the image of the fundamental weight of the (half-)spin representation(s) generates $X^\ast(\Out_\X(\rH))$.
    \end{enumerate}
\end{Lem}
\begin{Rem}
    In the case of Chevalley involutions, the set $\Omega_2$ is equal to the set of minuscule fundamental weights except in type $A$, where it is trivial for $A_{2n}$ and consists of the fundamental representation $\wedge^n V$ of $\SL(V)$ for type $A_{2n-1}$. In particular,
    \[
    \#\Omega_2 = \begin{cases}
        1&: \G \text{ of type }A_{2n-1},B_n,C_n,D_{2n+1}, E_7,\\
        2&: \G \text{ of type }D_{2n},\\
        0&: \text{otherwise.}
    \end{cases}
    \]
\end{Rem}
\begin{proof}
We handle the cases of Chevalley involutions first. In this case, $\cala_\X\simeq Z(\G)$, where
\[
X^\ast(\cala_\X) =2X^\ast(\rA)/2\zz\Phi\simeq X^\ast(\rA)/\zz\Phi =X^\ast(Z(\G)).
\]
In particular, we see $\Out_{\X}(\rH)\simeq Z(\G)/Z(\G)^2$. These are listed in the following table, where in type $A_n$, $\de\in \{0,1\}$ with $n\equiv \de\pmod{2}$.

\begin{table}[htp]
\caption{Centers of simply connected groups}
    \label{tab:centers}
    \centering
    \begin{tabular}{|c|c|c|c|c|c|c|c|c|}
	\hline 
	$\G$ & $A_n$ & $B_n$& $C_n$& $D_{2n}$& $D_{2n+1}$&  $E_6$& $E_7$&$G_2$,$F_4$,$E_8$\\ \hline
	$Z(\G)$ & $\mu_{n+1}$ & $\mu_2$& $\mu_2$& $\mu_2\times \mu_2$& $\mu_4$&  $\mu_3$& $\mu_2$&$\{1\}$\\ \hline
 $Z(\G)/Z(\G)^2$ & $\mu_{[1+\de]}$ & $\mu_2$& $\mu_2$& $\mu_2\times \mu_2$& $\mu_2$&  $\{1\}$& $\mu_2$&$\{1\}$\\ \hline
    \end{tabular}
\end{table}
Now $\X^\ast(\Out_\X(\rH))\simeq X^\ast(Z(\G))[2]$. It is now a classical fact that $X^{\ast}(Z(\G))=X^\ast(\rA)/\zz\Phi$ is generated by the minuscule weights of $\G$.
\begin{Rem}
    In type $A_{2n+1}$, the distinguished fundamental representation is $\wedge^n V$, where $V$ is the standard representation of $\G=\SL(V)$.
\end{Rem}

We now turn to the varieties $\Spin(V_1)\times^{\mu_2} \Spin(V_2)\backslash \Spin(V)$. When $\theta$ is an inner involution (either $\dim(V)$ odd, or $\dim(V)$ even and $\dim(V)\equiv \dim(V_1)\pmod{2}$), $\cala_{\X}(\kbar)$ has a unique non-trivial element acting via the (diagonal) outer automorphism on the factors of $\rH$ of type $D$. In both cases, $X^\ast( \cala_\X)\simeq \zz/2\zz$ 
       \quash{\[
      X^\ast( \cala_\X)\simeq \la \omega_{sp}-\theta(\omega_{sp})\ra\simeq \zz/2\zz
       \]}
       is generated by the image of the fundamental weight associated to the spin representation or the two half-spin representations, depending on the parity.

       \quash{\[
      X^\ast( \cala_\X)\simeq \la  \omega_{hsp}-\theta(\omega_{hsp})\ra\simeq \zz/2\zz
       \]
       is generated by the image of the fundamental weight associated to either of the half-spin representations (whose images agree in $X^\ast(\Ax)$). When $a=n/2$, we reduce to the Chevalley setting.}

       Finally, we note that the varieties $\X=\Spin_{2a+1}\times^{\mu_2} \Spin_{2(n-a)-1}\backslash\Spin_{2n}$ for $n\neq 2a+1$ satisfies that $\cala_\X=\cala_\X^\flat\simeq\mu_2$ as the involution is outer. In particular, $\Out_\X(\rH)$ is trivial. 

       In all cases, the groups calculated agree with the claim regarding the nilpotent cone by comparing with \cite[Proposition 6.22]{Levy}.
\end{proof}
\begin{Rem}
    It is reasonable to ask if this relationship with the nilpotent cone is general for symmetric varieties. This holds when $\De_\X^{dist}=\De_\X^{(2)}$ but fails for distinguished roots of type \eqref{b}. For example, in the case of $\X= \Sp_{2n}\times \Sp_{2n}\backslash\Sp_{4n}$ from Example \ref{Ex: Symplectic example} as well as $\X=\Spin_{2n}\backslash\Spin_{2n+1}$, where the nilpotent cone is irreducible.
\end{Rem}
\subsubsection{Distinguished roots in the absolutely simple case}
We saw in the proof of Lemma \ref{Lem: dist case}, that when $\G$ is simply connected and absolutely simple, $\theta$ a $k$-involution and $\X=\G^\theta\backslash\G$, then $|\De_\X^{dist}|\leq 1$. The following lemma completely classifies those cases when $|\De_\X^{dist}|=1$ over $\kbar$. 
\begin{Lem}\label{Lem: distinguished for sym} Suppose that$\G$ is simply connected and absolutely simple, $\theta$ a $k$-involution and $\X=\G^\theta\backslash\G$ with $\De_\X^{dist}\neq \emptyset$. Then $\De^{dist}_\X = \{\al\}$ and there are two cases:
\begin{enumerate}
    \item if $\al$ is a double root, then $(\G,\G^\theta)$ is a $k$-form of one of the following pairs:
    \begin{enumerate}
        \item $(\SL_{2n},S(\GL_n\times \GL_n))$, $n\geq 1$. Here, $\Phi_{\X}$ has type $B_n$;
        \item $(\Spin_{n},\Spin_2\times^{\mu_2} \Spin_{n-2})$, $n\geq 5$; $\Phi_{\X}$ has type $B_2$;
        \item $(\Spin_{4n},\GL_{2n}),$ $n\geq 2$; $\Phi_{\X}$ has type $B_n$;
        \item $(\Sp_{2n},\GL_{n}),$ $n\geq 2$; $\Phi_{\X}$ has type $B_n$;
        \item $(E_7,\Gm\cdot E_6)$; $\Phi_{\X}$ has type $B_3$.
    \end{enumerate}
    \item If $\al$ is of type $B$, then $(\G,\G^\theta)$ is a $k$-form of one of the following pairs:
    \begin{enumerate}
        \item $(\mathrm{Spin}_{2n+1},\mathrm{Spin}_{2n})$ for $n\geq 2$. Here, $\Phi_{\X}$ has type $B_1$;
        \item $(\Sp_{4n},\Sp_{2n}\times \Sp_{2n})$  for $n\geq 1$. Here, $\Phi_{\X}$ has type $B_n$.
    \end{enumerate}
\end{enumerate}
For all pairs, the normalized root system $\Phi_\X^{sv}$ is also of type $B$, except for the pairs $(\Spin_{2n+1},\Spin_2\times^{\mu_2} \Spin_{2n-1})$ with $n\geq 2$ and $(\Sp_{2n},\GL_{n}),$ $n\geq 2$, where $\Phi_\X^{sv}$ is of type $C$. These are the only cases of Chevalley type.
\end{Lem}
\begin{proof}
    The classification of simple systems with a doubled root is precisely the content of \cite[Remark after Cor. 7.6]{KnopAutomorphisms}, and is verified by direct calculation from the classification. For the classification of systems of type $B$ roots, one may either consult the classification \cite{Springerclassification} or Losev's analysis of distinguished roots \cite[Section 4]{Losev}.
\end{proof}
\subsection{Outer data for well-adapted varieties}\label{Section: calculate cocycle}
Let $\G$ be a quasi-split reductive group over $k$ with $k$-rational involution $\theta$. We fix a $k$-rational Borel pair $(\rA,\B)$ which is $(\theta,k)$-admissible. 

On the basis of Theorems \ref{Thm: Outer in terms of dist} and \ref{Thm: quasi-split G-inner form}, we impose the following restriction.
\quash{
\begin{Def}\label{Def: very well behaved}
    We say that $\X=\rH\backslash\G$ (or $\pi_0(\rH)$) is \emph{very well-behaved} if it is well-behaved, all the fields $k_i$ are Galois extensions. In particular, this says that for each $\calo_1\sim\calo_2$, we have equality $k_{\calo_1}=k_{\calo_2}$ of the associated fields.
\end{Def} 
We note that all $\rH$ are very well-behaved over $\kbar$, as are all $\rH$ such that $\pi_0(\rH)$ is a constant group scheme over $k$.  Moreover, this is equivalent to well-behaved if $k$ is finite or $k$ non-archimedean with $\G$ unramified.}
\begin{Assumption}\label{assumption: no type N}
   Assume that the symmetric $\G_{sc}$-variety $\X_{sc}=\rH_{sc}\backslash \G_{sc}$ contains no simple factors of Chevalley type and that Conjecture \ref{Conj: cohom surj} holds for $\X$.
\end{Assumption} 
Under this assumption, we know that $\X$ is well-adapted and up to replacing $\X$ with a $\G$-inner twist, we may assume that $\X=\rH\backslash\G$ satisfies that $\rH^\circ$ is quasi-split over $k$. The current section is in service of Section \ref{Sec: symplectic rep}, so we will simplify and also assume for simplicity that $\rH$ is connected, leaving discussion of disconnected groups to future work. 

Recall from Section \ref{Section: doubling aut}, there is a natural isomorphism
\[
\Aut_d(\X) =\prod_{i\in I_\X}\Res_{{k_i/k}}(\mu_2),
\] where $I_{\X}$ is the set of $\Ga$-orbits in $\De_{\X}^{dist}$. 
Following the reductions in Section \ref{Section: reduction to asc}, there thus exists a set of finite separable extensions $k_i/k$ $(i\in I_{\X})$ and $k_j/k$ ($j\in J_\X$) (unique up to unique isomorphism) with a decomposition $\X_{sc}= \prod_i\X_i\times \prod_j\X_j$ where $\X_i=\Res_{k_i/k}(\X'_i)$ such that $\X_i' = \G'_i/\rH'_i$ is a symmetric variety associated to an absolutely simple and simply connected $k_i$-group $\G_i'$ and $\De_{\X_i'}^{dist}=\{\ga_i\}$, and where $\De_{\X_j}^{dist}=\emptyset$.

\quash{Using Proposition \ref{Prop: absolutely simple case}\eqref{dist}, $\De_\X^{dist}$ decomposes 
\begin{equation}\label{eqn: decomp of dist}
\De_\X^{dist}=\bigsqcup_i \De_{\X_i}^{dist}=\bigsqcup_i \bigsqcup_{a\in \Ga_i\backslash\Ga}\De_{\X_i'}^{dist}\cdot a,
\end{equation}
where $\Ga_i = \Gal(\kbar/k_i)$. Set $I_\X^d\subset I_\X$ to be those indices such that $\De_{\X_i}^{dist}\neq \emptyset$. Combining Lemmas \ref{Lem: pass to a sc cover}, \ref{Lem: pass to derived}, and \ref{Lem : pass to absolutely simple} with Proposition \ref{Prop: absolutely simple case}, we see that 
\begin{equation}\label{eqn: decomp of Out}
    X^\ast(\Out_\X(\rH))=\bigoplus{i\in I_\X^d}\Ind_{\Ga_i}^{\Ga}(\zz/2\zz\ga_i)\simeq (\zz/2\zz)^{\De_\X^{dist}}.
\end{equation}
This induces a product formula
\[
\Out_{\X}(\rH) =  \prod_{i\in I_\X^d}\Out_{\X_i}(\rH_i) \simeq \prod_i \Res_{k_i/k}(\Out_{\X_i'}(\rH_i'));
\]}

Fix $i\in I_{\X}$. 
Over the field extension $k_i/k$, the distinguished root $\ga_i$ induces the canonical quadratic character (cf. \cite[Proposition C.5]{BorovoiGagliardi})
\begin{equation}\label{eqn: color character}
    \ep_i:=\mu_{\ga_i}:\cala_{\X'_i}\lra \mu_2,
\end{equation} 
dual to the inclusion $\la \ga\ra/\la 2\ga\ra \subset X^\ast(\Aut_d(\X))\subset X^\ast(\cala_\X)$. 
    Taking the $\Ga$-action into account, we obtain a commutative diagram
\[
\begin{tikzcd}
\cala_\X\ar[dr,swap,"\Res(\ep_i)"]\ar[rr]&&\Aut_d(\X)\ar[dl,"p_i"]\\
&\Res_{k_i/k}(\mu_2),&    
\end{tikzcd}
\] where $\Res(\ep_i)$ is obtained functorially from $\ep_i$ via the standard adjunction \cite[Eq. A.5.1]{ConradPseudoReductive} and 
\[
p_i:\Aut_d(\X) \lra \Res_{k_i/k}(\Aut_d({\X_i'}))\overset{\Res(\ep_i)}{\iso}\Res_{k_i/k}(\mu_2)
\]
denotes the canonical projection dual to the inclusion $\Ind_{\Ga_i}^{\Ga}(\zz/2\zz\ga_i)\subset X^\ast(\Aut_d(\X)).$
\quash{Thus, there is a canonical isomorphism
\[
p=\prod_ip_i :\Out_{\X}(\rH) =\prod_i\Res_{k_i/k}(\Out_{\X_i'}(\rH_i')){\iso} \prod_i\Res_{k_i/k}(\mu_2)
\]}

Consider now the geometric cocycle $\mu_\X:=\mu_\X^d:\Ga\lra \Aut_d(\X)(\kbar)$\footnote{We drop the superscript $d$ for notational simplicity. Strictly speaking, this conflicts with the notation $\mu_\X:\Ga\to \mathrm{S}\cala_\X$, but we will not refer to this section again in the paper.} as in Section \ref{Section: geometric cocycle}. For each $i\in I_\X$, consider the composition $$\mu_{\X,i}:=p_i\circ \mu_\X:\Ga\lra \Res_{k_i/k}(\mu_2)(\kbar).$$
Clearly $\mu_\X= \prod_i\mu_{\X,i}$.

Consider the diagram
\[
\begin{tikzcd}
    \Ga\ar[r,"\mu_{\X,i}"]&\Res_{k_i/k}(\mu_2)({\kbar})\\
    \Ga_i\ar[u]\ar[r,"\mu_i"]&\mu_2(\kbar)\ar[u,"j"],
\end{tikzcd}
\]
 where $j$ is the canonical unit map (e.g. \cite[Proposition A.5.7]{ConradPseudoReductive}) and where the cocycle $\mu_i$ is determined in order that the diagram commutes. Shapiro's lemma implies that the cohomology class $[\mu_{\X,i}]\in H^1(k,\Res_{k_i/k}(\mu_2))$ is uniquely determined by the classes $[\mu_{i}]\in H^1(k_i,\mu_2)$. In particular, $\mu_X$ is determined up to coboundary by $\{\mu_i\}_{i\in I_\X}$ via the isomorphisms
\begin{align}\label{eqn: Shapiro}
    H^1(k,\Aut_d(\X))&\simeq\prod_{i\in I_{\X}} H^1(k,\Res_{k_i/k}(\mu_2))\\&\simeq  \prod_{i\in I_{\X}} H^1(k_i,\mu_2)\simeq\prod_{i\in I_{\X}} \Hom(\Ga_i, \mu_2(\kbar)).\nonumber
\end{align}
We record this in the following lemma.

\begin{LemDef}\label{Lem: outer characters}
    Suppose that $\G$ is quasi-split and $\X=\rH\backslash\G$ is a symmetric $\G$-variety satisfying Assumption \ref{assumption: no type N} and that $\rH$ is connected.  
With respect to the isomorphism \eqref{eqn: Shapiro}, the data 
\begin{equation}\label{eqn: outer data}
 \mathrm{out}_{\X}:=\{\mu_i:\Ga_{i}\lra \{\pm1\}: i\in I_{\X}\}
\end{equation} uniquely determines the cohomology class $[\mu_{\X}]\in H^1(k,\Aut_d(\X))$. We refer to $\out_\X$ as the \textbf{outer data} of $\X$.

In particular, if $\X'=\rH'\backslash\G$ is another $\G$-form of $\X$, then $\X$ and $\X'$ are $\G$-inner forms if and only if $\out_\X=\out_{\X'}$.
\end{LemDef}
\begin{proof}
    Considering our discussion above, only the final claim needs justification. But this is a direct consequence of the well-adaptedness of $\X$ (and hence of $\X'$) and Lemma \ref{Lem: outer forms classify}.
\end{proof}

\begin{Ex}
 Returning to Example \ref{Ex: unitary example}, for both varieties $\X_1$ and $\X_t$ the unique spherical root is a double root. An easy calculation shows
\[
\out_{\X_1}=\{\mu=1\},\qquad \out_{\X_t}=\{\mu=\eta_{E/k}:\Ga\lra \{\pm1\}\},
\]
where $\eta_{E/k}$ it the unique non-trivial quadratic character of $\Ga$ trivial on the normal subgroup $\Gal(\kbar/E)\subset \Ga$.
\end{Ex}

\part{The dual symmetric variety and endoscopic data}

\section{The dual group and dual symmetric variety}
We recall now the formalism of the dual group of a spherical variety and discuss corresponding notions of $L$-groups. In the case of symmetric varieties, we introduce the notion of the dual symmetric variety $\hat{\X}$ and study its properties. The key statements are Lemma \ref{Lem: split torus embedding} and Proposition \ref{Prop: quotient stack is enough} . 

\subsection{The dual group of a spherical variety}\label{Section: dual groups}
Let $\G$ be a connected reductive group over $k$. Let $\X=\rH\backslash\G$ be a spherical $\G$-variety and let $\Omega_\X=({\fX},{\De}_{\X}, \Omega^{(1)},\Omega^{(2)})$ be the homogeneous spherical datum determined by a choice of Borel pair $(\rA,\B)$. Passing to the normalized root system, Sakellaridis and Venkatesh show that $({\fX}^{sv},{\De}^{sv}_{\X},\check{\fX}^{sv},\check{\De}^{sv}_{\X})$ is a based root datum.
\begin{Def}
The dual group of the $\G$-variety $\X$ is the connected complex reductive group $\check{\G}_\X$ associated to the dual based root datum $(\check{\fX}^{sv},\check{\De}^{sv}_{\X},{\fX}^{sv},{\De}^{sv}_{\X})$.
\end{Def}

Now the surjection $\rA\lra \Ax$ (or more explicitly, the inclusion ${\fX}\subset X^\ast(\T)$) produces a canonical morphism $\check{\rA}_\X\lra \check{\rA}$. Considering the inclusions of lattices
\[
{\fX}\subset\fX^{sv}=\fX+\zz\De_\X^{sv}\subset X^\ast(\T).
\]
we obtain a sequence $$\check{\rA}_\X\overset{\nu_\rA}{\lra}\check{\rA}_\X \overset{\varphi_\rA}{\lra}\check{\rA},$$
where $\nu_\rA$ is a finite morphism. 

Equip the dual group ${\check{\G}}$ of $\G$ with a pinning $e_{\al^\vee}\in \fg_{\al^\vee}^\vee$ for each $\al\in \De$. For each $\sig\in {\De}_{\X}$, Knop defines in \cite{KnopFunctorial} a one-dimensional subspace $\fg_{\check{\sig}}^\vee$ of $\fg^\vee$ by
\begin{align}\label{eqn: root embedding}
    \fg_{\check{\sig}}^\vee=\begin{cases} \fg_{\check{\sig}}^\vee&:\sig\in R^+,\\
    [\fg_{\be^\vee}^\vee,e_{\de_1^\vee}-e_{\de_2^\vee}]&: \sig \text{ is of type $D_{n\geq 3}$},\\
     [\fg_{\be^\vee}^\vee,2e_{\de_1^\vee}-e_{\de_2^\vee}]&: \sig \text{ is of type $B_{3}''$},\\
      \cc(e_{\de_1^\vee}- e_{\de_2^\vee})&: \sig \text{ is of type $D_{2}$};
    \end{cases}
\end{align}
here $\be^\vee:=\ga_1^\vee-\de_1^\vee=\ga_2^\vee-\de_2^\vee$ when $\sig$ is of type $G$. We  remark that a choice is made in the case of type $D_2$ roots, following the convention of \cite{KnopFunctorial}. See also \cite[Section 4.8] {BZSV} for discussion. \quash{The ambiguity of sign in the type $D_2$ case will be discussed below, but we assume that some assignment has been given, which may be recorded as a function
\begin{align}\label{eqn: non-standard signs}
     \De_{\X,D_2}&\overset{\varepsilon}{\lra} \{\pm1\}\\
    \sig&\longmapsto \varepsilon(\sig),\nonumber
\end{align}
where $\De_{\X,D_2}\subset \De_\X$ is the subset of spherical roots of type $D_2$.}

\begin{Def}
A homomorphism $\varphi_\X: \check{\G}_\X\to {\check{\G}}$ is \emph{distinguished} if  $\varphi_X|_{\check{\rA}_\X}=\varphi_A$. and $\varphi(\fg_{\X,\check{\sig}}^\vee) = \fg_{\check{\sig}}^\vee$ as defined in \eqref{eqn: root embedding}.
\end{Def}

\begin{Thm}\cite[Theorem 9.7]{KnopSchalke}\label{Thm: dual group map}
Suppose that $\X$ is a spherical $\G$-variety. Then distinguished morphisms $\varphi_\X$ exist. Moreover, they are unique up to $\check{\rA}_\X$-conjugacy.\footnote{In fact, one needs to restrict to a certain subtorus in the presence of certain types of spherical root; see \cite{KnopSchalke} for details.} The image 
\[
\check{\G}^\ast_{\X}:=\varphi_\X({\check{\G}}_\X)
\]
is a well-defined subgroup of ${\check{\G}}$ independent of $\varphi_\X$.
\end{Thm}
\begin{proof}
We recall a sketch of the proof, as its construction is important to us. Recall the set of \emph{associated roots} $\hat{\Sigma}$ and coroots $\hat{\De}_\X$. This construction gives rise \cite[Theorem 7.3]{KnopSchalke} to an additively-closed based root sub-datum 
\[
(X_\ast(\rA),\hat{\De}_\X, X^\ast(\rA), \hat{\Sigma})\subset (X_\ast(\rA),\check{\Phi}, X^\ast(\rA), \Phi)
\]
Thus, there exists a unique reductive subgroup $\hat{\G}_\X\subset {\check{\G}}$ corresponding to this subroot datum; this is called the associated group. The passage from spherical roots $\sig\in {\De}_{\X}$ to the associated coroots $\{\check{\ga}_1,\check{\ga}_2\}$ corresponds to the notion of \emph{folding}, and hence gives rise to a morphism 
\[
\varphi_\X:\check{\G}_\X\lra \hat{\G}_\X\subset {\check{\G}}
\]
with the desired properties.
\end{proof}

Knop and Schalke show that the kernel of the homomorphism $\varphi_X: \check{\G}_\X\to {\check{\G}}$ agrees with the kernel of the map $\varphi_{\rA}:\check{\rA}_{\X}\lra \check{\rA}$. This isolates a useful class of spherical varieties, as named in \cite{SakICM}. 
\begin{Def}\label{Def: excellent}
    We say that a spherical variety $\X$ is \textbf{excellent} if it is affine, homogeneous, and the kernel of $\rA\lra \Ax$ is connected.
\end{Def}
We note that when $\X$ is excellent, the cotangent bundle $T^\ast\X$ is hyperspherical by Proposition 3.7.4 of \cite{BZSV}. This is stronger than having no spherical roots of type $N$, the simplest counterexample being $\X=\Gm\backslash\SL_2$.

We remark that Lemma \ref{Lem: simply connected} implies the following result. 
\begin{Cor}
    Suppose that $\G$ is connected reductive group over $k$ such that $\G_{der}$ is simply connected. If $\X=\G^\theta\backslash \G$ is a symmetric $\G$-variety, it is excellent if  and only if the pair $(\G,\G^{\theta})$ is simply connected in the sense of \cite[Section 4]{Lesliedescent}.
\end{Cor}
In particular, the results of \emph{loc. cit.} provide tools for the topological study of excellent symmetric varieties.

We also record the following useful property of the dual group.
\begin{Prop}\cite[Theorem 1]{KnopFunctorial}
    Suppose that $\X$ and $\mathrm{Y}$ are spherical $\G$-varieties and assume that there exists a dominant morphism $\X\to \mathrm{Y}$. There exists a unique homomorphism with finite kernel $\check{\G}_{\mathrm{Y}}\to \check{\G}_{\X}$ which is compatible with the homomorphisms to $\check{\G}$.
\end{Prop}

\begin{Rem}\label{Rem: z-extension on dual side}
    In particular, if $(\G,\rH)$ is a symmetric pair and $(\widetilde\G,\widetilde\rH)$ is a compatible $z$-extension, then
    \[
    \widetilde\X={\widetilde\rH}\backslash\widetilde\G\lra \rH\backslash\G=\X
    \]
    is a dominant morphism of $\widetilde\G$-varieties. Noting that there is a canonical exact sequence 
    \[
    1\lra\check{\G}\lra \check{\widetilde\G}\lra \check{N}\lra 1,
    \]
    one readily checks that diagram
    \[
    \begin{tikzcd}
        \check{\G}_{\X}\ar[r,"\varphi_{\X}"]\ar[d]&\check{\G}\ar[d]\\
        \check{\widetilde\G}_{\widetilde\X}\ar[r,"\varphi_{\widetilde\X}"]&\check{\widetilde\G}
    \end{tikzcd}
    \]
    commutes.  
\end{Rem}

The set of parabolic roots $\De_\X^p$ corresponds to a Levi subgroup $\rA\subset \mathrm{L}_\X\subset \G$ and a dual Levi subgroup $\check{\mathrm{L}}_\X\subset {\check{\G}}$. Our choice of a pinning of ${\check{\G}}$ induces one on $\check{\mathrm{L}}_\X$. In particular, the pinning determines a principal homomorphism $\iota_\X:\SL_2\lra \check{\mathrm{L}}_\X$. It was shown in \cite{KnopSchalke} that the images of $\varphi_\X$ and $\iota_\X$ commute with each other in ${\check{\G}}$. We thus obtain a morphism 
\begin{equation}\label{eqn: full dual map}
\xi_\X: \check{\G}_\X\times \SL_2\lra {\check{\G}}.
\end{equation}


The results of \cite{KnopFunctorial, KnopSchalke} indicate how to extend this notion to incorporate the action of $\Ga$ on $\check{\G}_\X$. Viewed as a subgroup of $\check{\G}$, \cite[Section 10]{KnopSchalke} shows that there is a unique $\Ga$-action on $\hat{\G}_\X$ such that the inclusion $\hat{\G}_\X$ intertwines with the $L$-action on $\check{\G}$. We denote the corresponding semi-direct product by 
\[
{}^L\hat{\G}_{\X}\hra {}^L\G,
\]
though this is arguably inappropriate as the $\Ga$-action can fail to preserve the pinning on $\hat{\G}_\X$ in some cases. The additional requirement that $\varphi_{\X}$ preserve the embeddings \eqref{eqn: root embedding} forces a \emph{unique} action of $\Ga$ on $\check{\G}_{\X}$ with the property that the morphism $\xi_\X$ intertwines the action with the above action on $\hat{\G}_\X$. We thus tentatively define the Galois form of the $L$-group of $\X$ to be
\[
{}^L\X:=\check{\G}_{\X}\rtimes \Ga
\]
with respect to this unique action. The map $\xi_\X$ thus extends uniquely to a morphism
\begin{equation}\label{eqn: spherical L-homomorphism}
    {}^L\xi_{\X}: {}^L\X\times \SL_2(\cc)\lra {}^L\G.
\end{equation}
\begin{Rem}
As noted above, Knop fixes the choice of image for $\fg_{\check{\sig}}$ for all $\sig\in \De_{\X}$ of type $D_2$ in \cite{KnopFunctorial}, which we adopt here. It is worth mentioning that other choices relate to replacing $\rH$-periods with $(\rH,\chi)$-periods for certain characters $\chi$ of $\rH$; see \cite{prasad2015relative} and \cite{BPgalois}. 
\end{Rem}
\quash{
Our tentative perspective is that the following axioms should ``define'' the $L$-group of $\X$.
\begin{Axiom}
The $L$-group ${}^L{\X}$ is equipped with the action of the Weil group in the unique way so as to satisfy the following axioms.
\begin{enumerate}
    \item The $L$-group ${}^L{\X}$ of a symmetric variety $\X= \rH\backslash\G$ respects the passage to boundary degenerations.
        \item The $L$-group ${}^L{\X}$ of a Galois symmetric variety $\X= \Res_{E/k}(\rH_E)/\rH$ is $\calh_1^{op}$.
\end{enumerate}
\end{Axiom}
}
\subsection{The dual symmetric variety}\label{Sec: dual symm space}

Note that $\hat{\De}_\X\sqcup\check{\De}_\X^p\subset \check{\Phi}^+$. There is a natural involution ${\vartheta}:\hat{\De}_\X\to \hat{\De}_\X$ determined by the decomposition $\hat{\De}_{\X}=\bigsqcup_{\al\in\De_\X}\hat{\al}$, where $\hat{\al}$ denotes the set of associated coroots. Suppose now that $\X$ is a symmetric variety associated with an involution $\theta$. The involution on the set $\hat{\De}_\X$ is compatible with an involution on $\hat{\rA}$.

\begin{Lem}\label{Lem: dual involution on torus}
 Suppose that $\X$ is a symmetric $\G$-variety associated to the involution $\theta$. Assume that $(\rA,\B)$ are a $(\theta,k)$-admissible pair. Then the involution ${\vartheta}$ of $\hat{\De}_\X$ agrees with the restriction of the involution 
\[
-\theta: X_\ast(\rA)\lra X_\ast(\rA).
\]
\end{Lem}
This follows from Lemma \ref{Lem: involution on assoc}.
Set $\check{\rA}= X^{\ast}(\rA)\otimes_{\zz}\cc^\times$ for the complex dual torus. If $\check{\theta}:\check{\rA}\lra \check{\rA}$ denotes the involution induced by the action of $\theta$ on $X^\ast(\rA)$, then ${-\theta}$ induces the involution ${\vartheta}(t) = \theta(t)^{-1}$ on $\check{\rA}$.


\begin{Prop}\label{Prop: dual involution}
Suppose that $\X=\rH\backslash\G$ is a symmetric variety associated to $\theta$ and let $\varphi_{\X}:\check{\G}_{\X}\lra \hat{\G}_{\X}$ be a distinguished morphism. There exists an involution $\vartheta$ of $\hat{\G}_{\X}$ extending the involution on $\check{\rA}$ such that 
\[\check{\G}^\ast_{\X}=\varphi_{\X}(\check{\G}_{\X})=(\hat{\G}_{\X})^{\vartheta,\circ}.\]
It commutes with the given action of $\Ga$ on $\hat{\G}_\X$. In particular, the quotient $\check{\G}_{\X}^\ast\backslash\hat{\G}_{\X}$ is a complex symmetric variety equipped with a $\Ga$-action.
\end{Prop}

\begin{Rem}
In the context of the affine Grassmannian for real groups, a similar result is discussed in \cite[Section 10.7]{NadlerReal}, with the translation that $\hat{\G}_{\X}=L_1$.
\end{Rem}
\begin{proof}
Corollary 7.9 of \cite{KnopSchalke} implies the result on the level of adjoint groups. Indeed, the morphism $\varphi_{\X}$ is there constructed via the process of ``folding''  of the root system of $\hat{\G}_{\X}$, so that there exists an involution $\check{\theta}_{ad}$ on the adjoint group $\hat{\G}_{\X,ad}$. 
We also know that $Z(\check{\G}_{\X})=\varphi_{\X}^{-1}(Z(\hat{\G}_{\X}))$, so there exists a diagram
\[
\begin{tikzcd}
\check{\G}_{\X}\ar[d]\ar[r,"\varphi_{\X}"]&\hat{\G}_{\X}\ar[d]\\
\check{\G}_{\X,ad}\ar[r]&\hat{\G}_{\X,ad},
\end{tikzcd}
\]
inducing a surjective morphism
\[
\check{\G}^\ast_{\X}\backslash\hat{\G}_{\X}\lra\check{\G}^\ast_{\X,ad}\backslash\hat{\G}_{\X,ad}.
\]
Their construction shows that $\check{\G}^\ast_{\X,ad}=[\hat{\G}_{\X,ad}]^{\check{\theta}_{ad},\circ}$.

Steinberg's theorem \cite[9.16]{Steinberg} allows us to lift this involution uniquely to the simply connected cover
\[
\hat{\G}_{\X,sc}\lra \hat{\G}_{\X,ad};
\]
call this involution $\check{\theta}_{sc}$. Recall that $\hat{\G}_\X\simeq \hat{\G}_{\X,sc}\times^{\hat{Z}_{sc}}Z(\hat{\G}_\X)$ \cite[2.0.1]{deligne1979varietes}, where $\hat{Z}_{sc}\subset \hat{\G}_{\X,sc}$ is the center. Thus to show that there exists a lift of $\check{\theta}_{sc}$ to $\hat{\G}_\X$, it suffices to show that the induced involutions on $\hat{\G}_{sc}$ and $Z(\hat{\G}_\X)\subset \check{\rA}$ agree on $\hat{Z}_{sc}$.


Recall that we fixed a $(\theta,k)$-admissible pair $(\rA,\B)$ for $\G$. 
Let $\vartheta$ denote the involution on $\check{\rA}$ as in Lemma \ref{Lem: dual involution on torus}. 
We have the four groups
\[
\hat{\G}_{sc}\to \hat{\G}_{der}\subset \hat{\G}\to \hat{\G}_{ad},
\]
 and the associated tori
 \[
\hat{\rA}_{sc}\to\hat{\rA}_{der}\subset\check{\rA}\to\hat{\rA}_{ad}.
 \]
Considering weight lattices, we obtain a $\Ga$-equivariant commutative diagram
\[
\begin{tikzcd}
0\ar[r]& \zz\hat{\De}_\X\ar[r]\ar[d,"="]&X^\ast(\check{\rA})\ar[r]\ar[d]&\coker[\zz\hat{\De}_\X\to X^\ast(\check{\rA})]\ar[r]\ar[d]&0\\
0\ar[r]&\zz\hat{\De}_\X\ar[r]&X^{\ast}(\hat{\rA}_{sc})\ar[r]&\coker[\zz\hat{\De}_\X\to X^{\ast}(\hat{\rA}_{sc})]\ar[r]&0.
\end{tikzcd}
\]

Lemma \ref{Lem: dual involution on torus} implies that the  involution $\vartheta$ on $X^\ast(\check{\rA})$ is compatible with the given involution on $\zz\hat{\De}_{\X}$. We thus have a compatible involution on the cokernel, which is precisely the character group of $Z(\hat{\G}_\X)$. Moreover, the compatibility between the involutions on $\hat{\rA}_{sc}$ and $\hat{\rA}_{ad}$ induce a compatible involution on $\coker[\zz\hat{\De}_\X\to X^{\ast}(\hat{\rA}_{sc})],$ which gives the character group of $\hat{Z}_{sc}$. Finally, the middle vertical arrow is equivariant with respect to these involutions. This shows that the involution on $Z(\hat{\G}_{\X})$ is compatible with $\check{\theta}_{sc}$ under the map $\iota: \hat{Z}_{sc}\to Z(\hat{\G}_{\X})$.
\end{proof}

\begin{Def}
Let $\G$ be a connected reductive group over $k$ and let $\X=\rH\backslash\G$ be a symmetric $k$-variety. Given a distinguished morphism $\varphi_\X$, there exists a \textbf{dual involution} (unique up to $\check{\rA}_\X^\Ga$-conjugacy) $$\vartheta: \hat{\G}_{\X}\lra \hat{\G}_{\X}$$ satisfying 
\begin{enumerate}
\item $\vartheta$ commutes with the $\Ga$-action on $\hat{\G}_\X$,
    \item $\vartheta|_{\check{\rA}} = \mathrm{inv}\circ\check{\theta},$ where $\mathrm{inv}:\check{\rA}\to \check{\rA}$ denotes the inversion map,
    \item $\varphi_\X(\check{\G}_\X) = \check{\G}^\ast_{\X}= \hat{\G}_\X^{\vartheta,\circ}$.
\end{enumerate}
We define the \emph{dual symmetric variety} $\hat{\X}:=\check{\G}^\ast_{\X}\backslash\hat{\G}_{\X}$ for $\hat{\G}_{\X}$.
\end{Def}
This variety sits in the diagram
\begin{equation}\label{eqn: dual variety}
    \begin{tikzcd}
\SL_2(\cc)\ar[r]&\check{\mathrm{L}}_\X\ar[r]&\check{\G}\\
\check{\G}_{\X}\ar[r,"\varphi_{\X}"]&\hat{\G}_{\X}\ar[ur]\ar[r,"\check{\tau}"]&\hat{\X}\ar[u, swap,"\hat{s}_\X"],
\end{tikzcd}
\end{equation}
where $\check{\tau}$ is the quotient map and $\hat{s}_\X$ is induced by the symmetrization map associated to the variety $(\hat{\G}_\X)^\vartheta\backslash\hat{\G}_\X$ into $\hat{\G}_\X\subset \check{\G}$.

\subsection{Minimality and split tori}\label{Sec dual basic}
Assume that $(\rA,\B)$ is a $(\theta,k)$-admissible pair for $(\G,\rH)$ and let $\vartheta:\hat{\G}_{\X}\lra\hat{\G}_{\X}$ denote the daul involution. Let $(\check{\rA},\check{\B})$ be the dual pair and denote by $\hat{\B}_{\X}= \check{\B}\cap\hat{\G}_{\X}$. By construction, the involution $\vartheta$ preserves $\hat{\De}_{\X}$, so that $\hat{\B}_{\X}$ gives a $\vartheta$-stable Borel subgroup of $\hat{\G}_{\X}$. Following Springer \cite{springer85}, we say that the $\vartheta$-stable Borel pair  $(\check{\rA},\hat{\B}_{\X})$ is a \emph{fundamental pair} of $(\hat{\G}_{\X},{\G}^\ast_{\X})$.

 Those symmetric varieties $\hat{\X}$ which occur in this construction are strongly constrained. Indeed, Corollary 7.9 of \cite{KnopSchalke} implies that the associated symmetric $\hat{\G}_{\X,ad}$-variety  $\hat{\X}_{ad}$ satisfies the property of \emph{minimality}, so that
\[
\mathrm{rank}(\hat{\X}_{ad})=\mathrm{rank}(\hat{\G}_{\X,ad})-\mathrm{rank}(\check{\G}^\ast_{\X}/Z(\hat{\G}_\X)).
\]
This implies that $\hat{\X}$ is also minimal \cite[Proposition 3.2]{Ressayre}. The following lemma recalls the most relevant properties of such spherical varieties.

\begin{Lem}\cite[Lemma 4 and 5]{brion2004construction}\label{Lem: minimal rank}
Suppose that $(\G,\rH,\theta)$ is a symmetric pair over an algebraically closed field and set $\X=\rH\backslash\G$. Then the following are equivalent
\begin{enumerate}
    \item $\X$ is a minimal symmetric variety.
    \item Any $\theta$-stable maximal torus of $\G$ contains a maximal $\theta$-fixed subtorus (hence a maximal torus of $\rH$).
    \item Any $\theta$-stable maximal torus of $\G$ contains a maximal $\theta$-split subtorus.
    \item Any two $\theta$-stable maximal tori of $\G$ are conjugate  by $\rH^\circ$
\end{enumerate}
Suppose $\X$ is minimal. There exists a single $\rH^\circ$-orbit of fundamental pairs $(\rA,\B)$.
\end{Lem}

In particular, every maximally $\vartheta$-split maximal torus is contained in a $\vartheta$-stable Borel subgroup and that such Borel subgroups form a single $\check{\G}_\X^\ast$-orbit.
\begin{Lem}\label{Lem: split torus embedding}
For any $k$-rational maximally $\theta$-split maximal torus $T\subset\G$, there exists an embedding $\check{T}\to \hat{\G}_{\X}$ such that the natural dual involution
\begin{align*}
    \vartheta_T:\check{T}&\lra \check{T}
\end{align*}
induced by $-\theta$ on $X^\ast(T)$ intertwines with $\vartheta$. In particular, such an embedding realizes $\check{T}$ as a maximally $\vartheta$-split torus. Any two such embeddings are conjugate by $\check{\G}^\ast_{\X}$, so there is a unique $\Ga$-invariant $\check{\G}^\ast_{\X}$-orbit of such embeddings.
\end{Lem}
\begin{proof}
Fix a maximally $\theta$-split Borel $B\supset T$. Since both $\rA$ and $T$ are maximally $\theta$-split maximal tori, there exists $h\in (\G^\theta)^\circ(\kbar)\subset\rH(\overline{k})$ such that $h\rA h^{-1}= T$ and $h\B h^{-1} = B$. This induces a $\theta$-equivariant isomorphism
\[
X_\ast(\rA)\iso X_\ast(T)
\]
carrying the root system $(\check{\Phi},\check{\De})$ to one in $X_\ast(T)$. In particular, this induces a $\vartheta$-equivariant isomorphism
\[
\check{j}:\check{T}\iso \check{\rA}\subset \hat{\G}_{\X}.
\]
A standard fact is that the embedding $\check{j}$ is unique up to $(\hat{\G}_{\X})^\Ga$-conjugacy \cite{Kottwitzrational}, but not every conjugate will have a maximally $\vartheta$-split image. On the other hand. if two embeddings 
\[
\begin{tikzcd}
\check{T}\ar[r,shift left=.75ex,"j_1"]
  \ar[r,shift right=.75ex,swap,"j_2"]&\hat{\G}_{\X},
\end{tikzcd}
\]
give two maximally $\vartheta$-split tori $j_1(\check{T}),j_2(\check{T})\subset \hat{\G}_{\X}$, they must be conjugate by an element of $\hat{\G}_{\X}^{\vartheta,\circ}=\check{\G}^\ast_{\X}$ \cite[pg.287]{Richardson}.
\end{proof}
\subsubsection{Dual torus embedding}
Conversely, we establish a dual embedding statement in Proposition \ref{Prop: quotient stack is enough} under an additional rationality assumption (Assumption \ref{Assumption: orbits}).
\begin{Assumption}\label{Assumption: orbits} 
Assume that $k$ is perfect ant that for each regular semi-simple class $a\in [\X^{rss}\sslash\rH](k)$, there exists some $\xi\in H^1(k,\rH)$ with pure inner twist $(\X^\xi,\G^\xi)$ and a regular semi-simple element $x_a\in X^\xi(k)$ with invariant $a$.
\end{Assumption}
The assumption may be stated as claiming that the map from the quotient stack $\X/\rH$ to the categorical quotient $\X\sslash\rH$ is surjective on $k$-points when we restrict to the regular semi-simple locus. This is as good as can be hoped for in general. We verify this assumption in the case that the symmetric variety $\X$ is quasi-split in \cite{LesliestabFJ}, and remark that certain general results which may obviate the need for this assumption in forthcoming work of Ng\^{o}--Morrissey \cite{NgoMorrissey}. Note that when $k$ is finite (of good characteristic), then Assumption \ref{Assumption: orbits} holds. Indeed, every point $a\in[\X^{rss}\sslash\rH](k)$ corresponds to a $k$-rational semi-simple $(\G^\theta)^\circ$-orbit, which possesses a $k$-point by Lang's theorem.

Using this assumption, we wish to establish a converse of Lemma \ref{Lem: split torus embedding}. In preparation, suppose that we have $(\G,\theta)$ as above and let $\rA\subset \G$ be a $k$-rational maximally $\theta$-split maximal torus of $\G$. Suppose we are given a $1$-cocycle $w:\Ga\lra W(\kbar)$ valued in the Weyl group satisfying
\begin{enumerate}
    \item\label{cocycle1}  $w_\sig(\rA^-)=\rA^-$ for all $\sig\in \Ga$, so that the cocyle is valued in $W_1=\{w\in W: w(\rA^-)=\rA^-\}\subset W$,
    \item\label{cocycle2} $\theta\circ w_\sig = w_\sig\circ \theta$ for all $\sig \in \Ga$.
\end{enumerate}  Let ${}^\ast\rA$ denote the $k$-torus obtained by twisting $\rA$ by this cocycle. It possesses a $k$-rational involution ${}^\ast\theta:{}^\ast\rA\lra {}^\ast\rA$ satisfying ${}^\ast(\rA^-) = ({}^\ast\rA)^-$; we denote this ${}^\ast\theta$-split subtorus by ${}^\ast\rA^-$. 

Recalling that \cite[Section 4]{Richardson}
\[
W_\X\simeq W(\rH,\rA^-)\simeq W_1/W_2,
\]
where  and $W_2=\{w\in W_1: w|_{\rA^-}\equiv Id\}$, we see that $w$ descends to a $1$-cocycle valued in $W_\X$ when we restrict to $\rA^-$. Using the fact
    \[
W_\X\simeq  N_{\rH}(\rA^-)/Z_{\rH}(\rA^-),
\]
we see that for each $\sig\in \Ga$, the twisted Galois action ${}^\ast\sig$ may be realized as
    \[
    {{}^\ast\sig}(t) = n_\sig {}^\sig t n_\sig^{-1} \qquad \text{for $t\in\rA^-(\kbar)$,}
    \] where $n_\sig\in N_\G(\rA)\cap(\G^\theta)^\circ(\kbar)$ is an element representing $w_\sig$.  Note that $n_\sig$ is well defined up to right-multiplication by an element of $Z_{\rH}(\rA^-)(\kbar)$, but we may choose $n_\sig$ so that the cochain is continuous (cf. \cite[Section 7]{BorovoiSecond}).%

Now assume $\G$ is quasi-split and let $\X=\G^\theta\backslash\G$ be the associated symmetric variety with $\rH=\G^\theta$. Assume Assumption \ref{Assumption: orbits} holds for $(\G,\X)$, and identify $\X\hra \G$ via the symmetrization map. Set $x_0\in \X(k)$ to be the base point.

\begin{Prop}\label{Prop: quotient stack is enough} 
For any $\Ga$-invariant $\G^\ast_\X$-conjugacy class of maximally $\vartheta$-split maximal torus $\check{T}\to\hat{\G}_{\X}$, let $(T,\theta_T)$ denote the corresponding $k$-torus with involution. Setting $T^-$ to the maximal $\theta_T$-split subtorus, there exists
\begin{enumerate}
    \item a class $\xi\in H^1(k,\rH)$ with pure inner twist $(\X^\xi,\G^\xi)$ 
    \item a $k$-rational embedding ${T}^-\to \X^\xi\subset \G^\xi$  realizing ${T}^-$ as a maximally ${\theta}^\xi$-split torus.
\end{enumerate}
\end{Prop}
\begin{Rem}
 The embedding $T^-\to \X^\xi$ identifies the image with the \emph{flat} $T^-\cdot x_0$ through $x_0$ \cite[Section 6]{Knop95}.
\end{Rem}
\begin{proof}
    Set $\rH=(\G^\theta)^\circ$. Fix a $k$-rational maximally $\theta$-split maximal torus $\rA\subset \G$. Since both $\check\rA$ and $\check T$ are $\vartheta$-stable maximal tori, there exists $h\in \check{\G}^\ast_{\X}$ such that $\Ad(h)(\check\rA)=\check T$ 
inducing a $\vartheta$-equivariant isomorphism
\[
X_\ast(\check\rA)\iso X_\ast(\check T)
\]
carrying the root system $(\Phi,\De)$ to one in $X_\ast(\check T)$ in a $\vartheta$-equivariant way; this forces the Galois actions differ by a $1$-cocycle valued in the little Weyl group $W_{\X}$, realized as the Weyl group of $(\check{\G}_\X,\check{\rA}_\X)$. This isomorphism induces a ${\theta}$-equivariant $\kbar$-isomorphism
\[
{j}:{T}\iso {\rA}\subset \G.
\]

Thus, we obtain a $1$-cocycle $w:\Ga\lra W(\kbar)$ satisfying \eqref{cocycle1} and \eqref{cocycle2}. In particular, if we consider
     \[
     j_-:T^-\to \rA^-,
     \] we see that ${}^\sig j_- := \sig\circ j_-\circ \sig_T^{-1}= \Ad(n_\sig^{-1})\circ j_-$ for where $n_\sig\in N_{\rH}(\rA^-)(\kbar)$ is an element representing $w_\sig$.

Let ${}^\ast\rA^-\subset {}^\ast\rA$ be the twisted form of $\rA$ as above. If $k$ is finite, consider the twisted form ${}^\ast[\rH/Z_{\rH}(\rA^-)]$ of $\rH/Z_{\rH}(\rA^-)$ by the cocycle $w$. That is, for $t\in [\rH/Z_{\rH}(\rA^-)](k)$, $\sig\in \Ga$ acts by ${}^{\sig}tn_\sig^{-1}$. By Lang's theorem, there exists $t\in {}^\ast[\rH/Z_{\rH}(\rA^-)](k)$. Thus, writing $t= h Z_{\rH}(\rA^-)(\kbar)$, we see that for all $\sig\in \Ga$
\[
 h Z_{\rH}(\rA^-)(\kbar) = {}^\sig h n_\sig^{-1}Z_{\rH}(\rA^-)(\kbar).
\]
An easy calculation now show that $\Ad(h)\circ j:T^-\lra\G$ gives a $k$-rational embedding.

We now assume that $k$ is infinite. In this case, it can happen that ${}^\ast[\rH/Z_{\rH}(\rA^-)](k)=\emptyset$, so we proceed by applying Assumption \ref{Assumption: orbits}. By the assumptions on $k$ \cite[18.3]{Borel}, ${}^\ast\rA^-(k)$ is Zariski-dense in ${}^\ast\rA^-$, so that there exists $x\in {}^\ast\rA^-(k)$ which gives a regular semi-simple element of $\X(\kbar)\subset \G(\kbar)$. By assumption, 
\[
x = n_\sig {}^\sig x n_\sig^{-1}, \qquad\text{where}\quad n_\sig \in \rH(\kbar).
\]
This implies that the $\rH$-orbit of $a\in \X(\kbar)$ is $k$-rational, giving a point $a\in [\X^{rss}\sslash\rH](k)$. By Assumption \ref{Assumption: orbits}, there is a class $\xi\in H^1(k,\rH)$, represented by a $1$-cocycle $h:\Ga\lra \rH(\kbar)$ such that if we denote  the twisted Galois action on $\G(\kbar)$ (resp. $\rH(\kbar)$) by $\sig_\xi$, we obtain a pure inner twist $(\G_\xi,\theta_\xi)$ such that $\rH_\xi=(\G_\xi^{\theta_\xi})^\circ$ and there exists an element $h\in \rH(\kbar)$ such that $hxh^{-1}\in\X_\xi^{rss}(k)\subset \G_\xi(k)$ lies over the point $a\in [\X^{rss}\sslash\rH](k)$. Note that since $h_\sig\in \rH(\kbar)$, the involution $\theta:\G\to\G$ commutes with $\sig_\xi$ for all $\sig\in \Ga$, so descends to an involution on $\G_\xi$ over $k$, which we have denoted by $\theta_\xi$.

Unwinding the claim that $hxh^{-1}$ is rational gives the equation
\begin{align}\label{eqn: cocycle calc}
    x = \Ad(h^{-1}{}^{\sig_\xi} h)\Ad(h_\sig n_\sig^{-1})(x),
\end{align}
so that $(h^{-1}{}^{\sig_\xi} h)\cdot (h_\sig n_\sig^{-1})\in Z_{\rH}(\rA^-)(\kbar).$
Note that this implies that the cocycle $\sig\in \Ga\mapsto h^{-1}h_\sig {}^\sig h$ is valued in $N_{\rH}(\rA^-)(\kbar)$ and lies over $w:\Ga\lra W_\X$ with respect to the map
\[
H^1(k,N_{\rH}(\rA^-)) \lra H^1(k,W_\X).
\]


Returning to the map $j_-:T^-\iso \rA^-\subset \G$ over $\kbar$, if we twist the Galois action on $\G$ by $h_\sig$, we obtain
\[
{}^{\sig_\xi} j_- =\Ad(h_\sig n_\sig^{-1})\circ j_-
\]
It now follows from equation \eqref{eqn: cocycle calc} that the conjugate $\Ad(h)\circ j_-$ descends to a $k$-rational embedding $T^-\lra \G_\xi$. Since $h\in \rH(\kbar)$, it follows that s $hj_-(T^-)h^{-1}$ is a maximally $\theta_\xi$-split torus.
\end{proof}

This is an analogue of Kottwitz's embedding statement for maximal tori and quasi-split groups \cite[Lemma 2.2]{Kottwitzrational}. In the relative setting, we transfer flats on $\X$ rather than tori themselves and it need not suffice to consider the quasi-split form alone. On the other hand, an easy adaptation of the proof of \emph{loc. cit.} applies whenever $\X=\rH\backslash\Res_{E/k}(\rH_E)$ is a Galois symmetric variety with  $\rH$ quasi-split. 
\begin{Lem}\label{Lem: galois dual torus embedding}
Let $k$ be a perfect field, and let $E/k$ be a quadratic extension. Assume that $\rH$ is a quasi-split reductive $k$-group and set $\G=\Res_{E/k}(\rH_E)$. Set $\X=\rH\backslash\G$ for the associated Galois symmetric variety. For any $\Ga$-invariant $\G^\ast_\X$-conjugacy class of maximally $\vartheta$-split maximal torus $\check{T}\to\hat{\G}_{\X}(=\check{\G})$, there exists a $k$-rational embedding ${T}\to \G$  realizing ${T}$ as a maximally ${\theta}$-split torus.
\end{Lem}
\begin{proof}
\quash{
 Since both $\check\rA$ and $\check T$ are $\vartheta$-stable maximal tori, there exists $h\in \check{\G}^\ast_{\X}$ such that $\Ad(h)(\check\rA)=\check T$ 
inducing a $\vartheta$-equivariant isomorphism
\[
X_\ast(\check\rA)\iso X_\ast(\check T)
\]
carrying the root system $(\Phi,\De)$ to one in $X_\ast(\check T)$ in a $\vartheta$-equivariant way; this forces the Galois actions differ by a $1$-cocycle valued in the little Weyl group $W_{\X}$, realized as the Weyl group of $(\check{\G}_\X,\check{\rA}_\X)$.}
Let $\rA$ be a $k$-rational maximally $\theta$-split maximal torus of $\G$. Arguing as above, the assumptions give a ${\theta}$-equivariant $\kbar$-isomorphisms
\[
{j}_-:{T}^-\iso {\rA}^-\subset \G,\qquad {j}_+:{T}^+\iso {\rA}^+\subset \G
\]
of $\theta$-split tori. 
As before, there exists a lift $n_\sig\in N_{\G}(\rA^-)(\kbar)$ such that ${}^\sig j_\pm = \Ad(n_\sig)^{-1}\circ j_\pm$. Since $Z_{\G}(\rA^-)=\rA$ in this setting, there is an inclusion $W_\X\subset W$, so we may further assume that the cochain $\sig\mapsto n_\sig$ takes values in $N_{\rH}(\rA^-)(\kbar)\cap N_\G(\rA)(\kbar)$. 

In this case, the maximal $\theta$-fixed torus $\rA^+\subset\rA$ is a maximal torus of $\rH$. Applying Lemma 2.1 of \cite{Kottwitzrational} to $\rH$ implies that there exists $h\in \rH(\kbar)$ such that
\[
h^{-1}\sig (h) n_\sig^{-1} \in j(T)(\kbar).
\] 
 One now checks that $\Ad(h)\circ j: T\lra \G$ is $k$-rational and realizes $hj(T)h^{-1}$ as a maximally $\theta$-split torus
\end{proof}

\section{Endoscopic data for symmetric varieties}\label{Section: endoscopy defs}
Let $\G$ be a connected reductive group over $k$. In this section, we use the previous results on outer forms and the dual symmetric variety to develop a notion of endoscopic data for a symmetric $\G$-variety $\X=\rH\backslash\G$. After recalling the basic notions, we state our main existence theorem in Theorem \ref{Thm: exists}. We formulate the definitions in Section \ref{Section: endoscopic datum}. Finally, Section \ref{Section:orbit match} proves the matching of geometric semi-simple orbits between an endoscopic symmetric variety and $\X$ in Theorem \ref{Thm: point comparison}.  

 \subsection{Endoscopic data}\label{Section: endoscopy roundup}
We first recall the general notion of endoscopic triples as in \cite{Kottwitzstableelliptic,KalethaStable}. In Section \ref{Section: stabilize}  and \cite{LesliestabFJ}, it is necessary to work with pure inner forms of a given group $\G$, so we work with the appropriate version of endoscopic triple. 

 
 \begin{Def}
 A \emph{pure-refined endoscopic datum} is a triple ${\fe}=(\G_{{\fe}},\ka,\eta)$ where $\G_{{\fe}}$ is a pure inner form of a quasi-split reductive group $\G_{{\fe}}^{\ast}$ over $k$, $\ka\in Z(\check{\G}_{\fe})$ is semi-simple and central, and 
 \[
 \eta:\check{\G}_{\fe}\lra \check{\G}
 \]
 is an algebraic morphism of dual groups subject to the conditions
 \begin{enumerate}
     \item $\eta(\check{\G}_{\fe}) =\check{\G}_{\eta(\ka)}^\circ$,
     \item the $\check{\G}$-conjugacy class of $\eta$ is $\Ga$-fixed, and 
     \item $\ka\in Z(\check{\G}_{\fe})^\Ga$.
 \end{enumerate}
 An isomorphism of such triples is an isomorphism of algebraic groups $f:\G^\ast_{{\fe}_1}\lra\G^\ast_{{\fe}_2}$ defined over $k$ such that
 \begin{enumerate}
     \item $\eta_1\circ\check{f}$ and $\eta_2$ are $\check{\G}$-conjugate.
     \item The images of $\check{f}(\ka_2)$ and $\ka_1$ in $\pi_0(Z(\check{\G}_{{\fe}_1})^\Ga)$ coincide.
 \end{enumerate}
 \end{Def}
For any such triple, there is a canonical inclusion $Z(\check{\G})\subset Z(\check{\G}_{\fe})$, and we say that ${\fe}$ is \textbf{elliptic} if $Z(\check{\G}_{\fe})^{\Ga,\circ}\subset Z(\check{\G})$.

\begin{Rem}
 The definition of endoscopic triples in \cite{KottwitzCusp} takes ${\ka}\in [Z(\check{\G}_{\fe})/Z(\check{\G})]^\Ga$ such that the image of $\tilde{\ka}$ in $H^1(\Ga, Z(\check{\G}))$ is trivial for $k$ local and everywhere locally trivial when $k$ is global. When working with non-quasi-split groups $\G$ which are pure inner forms of a quasi-split form $\G^\ast$, It becomes necessary to fix a lift in $Z(\check{\G}_\fe)^\Ga$ as above.
\end{Rem}
    
\subsection{Endoscopic varieties}\label{Section: endo varieties}
Now let $\X=\rH\backslash\G$ be a symmetric $\G$-variety associated to an $k$-rational involution $\theta$ of $\G$. We assume that $\G$ is quasi-split.
\begin{Rem}
    The assumption that $\G$ be quasi-split is not essential. Extending to a general $\G$ will require an adaptation of rigid inner forms \cite{Kaletharigid} to the relative setting. This will be the subject of future work.
\end{Rem}

Let $(\rA,\B)$ be a $(\theta,k)$-admissible pair. Then $\rA$ is maximally $\theta$-split and we have the exact sequence
\[
1\lra \widetilde{\T}_\X\lra \rA\lra \Ax\lra 1.
\] 
We set $\T_\X = (\widetilde{\T}_\X)^\circ$. This induces the exact sequence
\[
1\lra \pi_0(\widetilde{\T}_\X)^D\lra \check{\rA}_\X\lra \check{\rA}\lra \check{\T}_\X\lra 1.
\]Following Proposition \ref{Prop: dual involution}, there exists an involution $\vartheta: \hat{\G}_{\X}\lra \hat{\G}_{\X}$ and a unique $\check{\rA}_\X$-conjugacy class of distinguished morphism 
\[
\varphi_{\X}:\check{\G}_{\X}\lra \check{\G},
\]
such that $\varphi_\X(\check{\G}_\X)=\check{\G}_\X^\ast = (\hat{\G}_\X)^{\vartheta,\circ}$, independent of the choice of $\varphi_\X$.
\quash{such that the diagram
\[
 \begin{tikzcd}
 \check{\rA}_{{\X}}\ar[d]\ar[r]&\check{\rA}\ar[d]\\
 \check{G}_{{\X}}\ar[r,"\varphi_{{\X}}"]&{\G}_{\X}^\wedge\subset\check{\G}
 \end{tikzcd}
\]
 commutes.}
 Fix one such map $\varphi_{\X}$ and set $\hat{\X}=\check{\G}^\ast_{\X}\backslash\hat{\G}_{\X}$. We let ${x}_0\in \hat{\X}$ denote the tautological base point fixed by $\check{\G}^\ast_{\X}$. 
We also have the finite map 
\[
\hat{s}=\hat{s}_\X: \hat{\X}\lra\hat{\G}^{\vartheta}_\X\backslash \hat{\G}_\X{\lra}\hat{\G}_\X,
\]
which is a closed immersion when $\hat{\G}^{\vartheta}_\X$ is connected. We shall refer to this as the symmetrization map for $\hat{\X}$.
\quash{\begin{Rem}
    Recall that the induced involution on $\hat{\G}_{\X,ad}$ is tightly constrained by the property that $\hat{\X}_{ad}/\hat{\G}_{\X,ad}/\G^{\wedge,\vartheta}_{\X,ad}$ is a minimal spherical variety, and is therefore a finite product of factors the form (\cite[Corollary 4.6]{KnopSchalke}, using the fact that $\X$ is symmetric)
    \begin{itemize}
        \item $K/K$,\qquad $K$ simple and adjoint;
        \item $K\times K/\De K$,\qquad $K$ simple and adjoint;
        \item $\PGL_{2n}/\mathrm{PSp}_{2n}$, $n\geq 2$;
        \item $\mathrm{PSO}_{2n}/\SO_{2n-1}$, $n\geq 4$;
        \item $E_{6,ad}/F_4$.
    \end{itemize}
    It is only in the case $\mathrm{PSO}_{2n}/\SO_{2n-1}$ where the fixed point subgroup is not-necessarily connected
\end{Rem}}

\quash{ is a quasi-split reductive group equipped with a $k$-rational involution $\theta$. Let $(\rA,\B)$ be a fixed $(\theta,k)$-admissible Borel pair. 
Suppose that $\rH\subset \G$ is a symmetric subgroup associated to $\theta$ and $\X=\rH\backslash\G$ gives rise to a diagram \eqref{eqn: dual variety}. The symmetrization map embeds $\hat{\X}=S(\hat{\X})\subset \hat{\G}_{\X}\subset {\check{\G}}$, so we may view $x$ as a semi-simple element of both $\hat{\G}_{\X}$ and ${\check{\G}}$.}
 Recall that $\mathrm{L}_\X\subset \G$ is a Levi subgroup determined by $\De^p_\X$ satisfies that the inclusion $\mathrm{L}_{\X}\cap \rH\subset \mathrm{L}_\X$ is normal and there is a short exact sequence
\[
1\lra \mathrm{L}_{\X}\cap \rH\lra \mathrm{L}_\X\lra \Ax\lra 1.
\]
Setting $\rA_\X^\ast=\varphi_{\X}(\check{\rA}_\X)$ and $\mathrm{M}_\X:=(\mathrm{L}_{\X}\cap \rH)^\circ$, this induces a commutative diagram with exact rows \cite[Section 1.8]{KottwitzCusp}  \[
 \begin{tikzcd}
 1\ar[r]&\Ax^\ast\ar[r]\ar[d,"="]&\check{\rA}\ar[d]\ar[r]&\check{\T}_\X\ar[d]\ar[r]&1\\
 1\ar[r]&\Ax^\ast\ar[r]&\check{\mathrm{L}}_\X\ar[r]&\check{\mathrm{M}}_\X\ar[r]&1,
 \end{tikzcd}
 \]
 where the bottom row follows from the normality of the map $\mathrm{L}_{\X}\cap \rH\to \mathrm{L}_\X$. Moreover, it is shown in \cite[Theorem 9.12]{KnopSchalke} that $\check{\mathrm{L}}_\X$ is the centralizer of $\Ax^\ast$ in ${\check{\G}}$, so the above sequence specializes to
 \[
 1\lra {\rA}^\ast_\X\lra Z(\check{\mathrm{L}}_\X)\lra Z(\check{\mathrm{M}}_\X)\lra 1.
 \]
In particular, this gives a map $Z(\check{\mathrm{M}}_\X)\to \check{\X}$ via $\hat{\tau}$ satisfying that 
\begin{equation}\label{aut as center}
   \cala_{\hat{\X}}\simeq \check{s}(\hat{\X})\cap Z(\hat{\G}_\X)\hra  Z(\check{\mathrm{M}}_\X).
\end{equation}

 Suppose now that $x\in \hat{\X}(\cc)$ is semi-simple. There exists a $\check{\G}^\ast_{\X}$-conjugate $x_0$ of $x$ lying in $\check{\rA}$  \cite[Theorem 7.5]{Richardson}, so we may assume that $x$ lies in the flat $\check{\tau}(\check{\rA})\simeq \check{\T}_\X$. Note that this may not be a $\Ga$-equivariant twist. Then $\hat{s}(x)$ is semi-simple as an element of ${\check{\G}}$ and $\check{\mathrm{L}}_\X$. Considering the centralizer $\check{\G}_{x}:=\check{\G}_{\hat{s}(x)}$, we obtain a diagram
 \begin{equation}
     \begin{tikzcd}
&\check{\mathrm{L}}_{\X,x}\ar[r]&\check{\G}_x\\
\check{\G}_{\X,x}\ar[r,"\varphi_{\X,x}"]&\hat{\G}_{\X,x}\ar[ur]\ar[r,"\hat{\tau}_x"]&\hat{\X}_x\ar[u, swap,"\hat{s}"],
\end{tikzcd}
 \end{equation}
 where $\hat{\G}_{\X,x}\subset{\check{\G}}_x$ are the connected components of the identity of the corresponding centralizers, the involution $\vartheta$ preserves $\hat{\G}_{\X,x}\subset \hat{\G}_{\X}$, and $\check{\G}_{\X,x}\subset \check{\G}_{\X}$ is defined so that
  $$[\hat{\G}_{\X,x}]^{\vartheta,\circ} = \G_{\X,x}^\ast:=\varphi_{\X,x}(\check{\G}_{\X,x})$$ and the diagram
\begin{equation}\label{diag: dual at x}
     \begin{tikzcd}
\check{\G}_{\X,x}\ar[d,"\eta_\X"]\ar[r,"\varphi_{\X,x}"]&\hat{\G}_{\X,x}\ar[d,"\eta"]\ar[r,"\hat{\tau}_x"]&\hat{\X}_x\ar[d, "\eta"]\\
\check{\G}_{\X}\ar[r,"\varphi_{\X}"]&\hat{\G}_{\X}\ar[r,"\hat{\tau}"]&\hat{\X},
\end{tikzcd}
\end{equation}
commutes. 
The following lemma is an immediate check.
\begin{Lem}
 The subgroup $\check{\mathrm{L}}_{\X,x}\subset {\check{\G}}_x$ is a Levi subgroup equal to the centralizer of $\Ax^\ast$ in $\check{\G}_x$.
 If $\check{s}(x)\in Z(\check{\mathrm{M}}_\X)$, then
   $ 
    \check{\mathrm{L}}_{\X,x}=\check{\mathrm{L}}_\X\subset {\check{\G}}_x.
    $

\end{Lem}
Recall that $\hat{\X}^{ss} = \check{\G}_\X^\ast\cdot\hat{s}^{-1}(\check{\T}_\X)$ by \cite[Proposition 6.3]{Richardson}. Motivated by the preceding lemma and the inclusion $Z(\check{\mathrm{M}}_\X)\subset \check{\T}_\X$, we let $\hat{\X}^{\heart}:=\check{\G}_\X^\ast\cdot\hat{s}^{-1}(Z(\check{\mathrm{M}}_\X))$ be the subvariety of semi-simple $x\in \hat{\X}^{ss}$ such that $\hat{s}(x)$ is centralized by a $\check{\G}_\X^\ast$-conjugate of $\check{\mathrm{M}}_\X$. More generally, let $S\subset \G$ be a maximally $\theta$-split maximal torus and let $j:\check{S}\to \hat{\G}_\X$ be a representative of the canonical $\check{\G}_\X$-orbit of embeddings of $\check{S}$ as a maximally $\vartheta$-split torus (cf. Lemma \ref{Lem: split torus embedding}). We note that $\check{S}^\heart:=\check{S}\cap [j^{-1}\circ\check{\tau}^{-1}(\hat{\X}^\heart)]$ is independent of the choice of $j$. 
\quash{$(\G^\theta)\circ(\kbar)$-conjugate to $\rA$, and we let $S_\X$ denote the quotient of $\S$ corresponding to $\Ax$; this is well-defined as the choice of $g\in (\G^\theta)\circ(\kbar)$ only affects the $\kbar$-isomorphism $S\simeq\rA$ up to the action by $W_\X$, preserving the splitting $\rA=\rA^+\cdot \rA^-$. Set $T\subset S$ for the connected component of the kernel of $S\to S_\X$.

We may conjugate by an element of $\check{\G}_\X^\ast$ and obtain a commutative diagram
\[
     \begin{tikzcd}
\check{S}_\X\ar[r]\ar[d]&\check{S}\ar[d,"\eta"]\ar[r]&\check{T}\ar[d]\\
\check{\rA}_\X\ar[r]&\check{\rA}\ar[r]&\check{\T}_\X.
    \end{tikzcd}
\] 
The vertical arrows are not $\Ga$-equivariant, but if $\sig_S$ denotes the action on $\check{S}$, then there exists a $1$-cocycle $w_S\in Z^1(\Ga,W_{\hat{\X}})$, where 
\[
W_{\hat{\X}}\simeq N_{\check{\G}_\X^\ast}(\check{\T}_\X)/Z_{\check{\G}_\X^\ast}(\check{\T}_\X)
\]
is the little Weyl group for $\hat{\X}$ with respect to the maximally $\vartheta$-split torus $\check{\rA}$, such that for all $s\in \check{S}$ and $\sig\in \Ga$
\[
\eta(\sig_S(s)) = w_S(\sig)\sig(\eta(s)).
\]
Through this identification, the quotient $\hat{\tau}$ induces a map $\hat{\tau}_S:\check{T}\to \hat{\X}$, which depends on the choice of conjugation $\check{S}\to \check{\rA}$. We are interested primarily in those elements $t\in \check{\T}$ such that $\hat{\tau}_S(t)\in Z(\check{L}^\circ_\X)$.}

 \subsubsection{Existence} We show how to produce symmetric varieties of endoscopic groups of $\G$ from $\X=\rH\backslash\G$; the critical point is existence and uniqueness up to $\G$-inner forms of such varieties. 
 
 Let $x\in \hat{\X}^{\heart}$ and let $g\in \check{\G}_\X^\ast$ such that $g\cdot\hat{s}(x)\in Z(\check{\mathrm{M}}_\X)\subset \check{\T}_\X$. We assume further that $\sig(x)=x$ for all $\sig\in \Ga$\footnote{ This assumption is intended to be compatible with the notion of \emph{pure refined endoscopic data} discussed above. More generally, we may allow $\check{s}(x)\in Z(\check{\mathrm{M}}_\X)$ to project to a $\Ga$-invariant element in $Z(\check{\mathrm{M}}_\X)/Z(\check{\G}_\X)$.}. 
 Let
\[
\text{$\hat{\eta}_x:\hat{\G}_{\X,x}\subset\hat{\G}_\X$ and $\eta_x: {\check{\G}}_x\subset {\check{\G}}$}
\]denote the inclusions of connected components of the identities of the two centralizers. Recall that the $L$-action of $\Ga$ on $\check{\G}$ induces a unique action on $\hat{\G}_\X$, and we denote the corresponding semi-direct product as ${}^L\hat{\G}_\X$. The assumption that $x\in \hat{\X}^{\heart,\Ga}$ implies that the centralizers are $\Ga$-stable, so that the maps above extend to maps
\[
{}^L\hat{\eta}_x:\hat{\mathcal{G}}_x\lra {}^L\hat{\G}\text{ and }{}^L\eta_x:\calg_x\lra  {}^L\G,
\]
where 
\[
1\lra \check{\G}_x\lra \calg_x\lra \Ga\lra 1
\]
is the extension induced by the endoscopic datum $\fe_x$, and similarly for $\hat{\calg}_x$. These extensions need not be the $L$-group of a reductive group over $k$. Note that $\hat{\G}_{\X,x}\subset\hat{\G}_\X$ is stable under $\vartheta$ and $\Ga$. Let $\G_\fe=\G_x$ denote the quasi-split reductive group over $k$ dual to ${\calg}_x$ in the sense of \cite[Section 1.2]{LanglandsShelstad}, so that $\fe_x:=(\G_\fe,\hat{s}(x),\eta)$ gives a (not-necessarily elliptic) pure-refined endoscopic triple of $\G$. 

\begin{Thm}\label{Thm: exists}
 Assume that $\X$ is well-adapted. Let $x\in \hat{\X}^{\heart,\Ga}$ be as above. Let $\G_\fe$ be the (quasi-split) endoscopic group associated to $\fe_x$. There exists an $k$-rational involution $\theta_\fe:\G_\fe\lra \G_\fe$ and a $k$-rational subgroup $(\G_\fe^{\theta_\fe})^\circ\subset \rH_\fe\subset N_{\G_\fe}(\G^{\theta_\fe}_\fe)$ such that if $\X_\fe:=\rH_\fe\backslash\G_\fe$  is the associated symmetric variety, 
    \begin{enumerate}
    \item\label{item lattice}if $\rA_\fe\subset \G_\fe$ is a $(\theta_\fe,k)$-admissible torus with $\rA_{\X_\fe}$ the corresponding quotient torus, there exists a $k$-rational embedding ${\rA_\fe}^-\to \X^\xi\subset \G^\xi$  realizing ${\rA_\fe}^-$ as a maximally ${\theta}^\xi$-split torus of a pure inner twist $(\X^\xi,\G^\xi)$;
        \item\label{item diagram} a $\check{\G}$-conjugate of $\eta$ identifies the sequence of connected reductive groups
        \[
\check{\G}_{\X_\fe}\lra\hat{\G}_{\X_\fe}\lra\check{\G}_\fe
        \]of the $\G_\fe$-variety $\X_\fe$ with the top row of \eqref{diag: dual at x};
        \item\label{final thing} there is a canonical morphism $\mathrm{dist}_\fe:\Aut_d(\X)\to \Aut_d(\X_\fe)$.
\end{enumerate}
If Conjecture \ref{Conj: cohom surj} holds for $\X_{\fe}$, then $\X_\fe$ may be chosen so that the geometric $1$-cocycle $\mu_{\X_\fe}:\Ga\lra \Aut_d(\X_\fe)$ constructed in Definition \ref{Def: geometric cocycle} is cohomologous to $\mu_{\X,\fe}$ defined by
 \begin{equation}\label{item rep}
\begin{tikzcd}
    \Ga\ar[dr,"\mu_{\X,\fe}"]\ar[d,"\mu_\X"]&\\
    \Aut_d(\X)(\kbar)\ar[r,"\mathrm{dist}_\fe"]& \Aut_d(\X_\fe)(\kbar).
\end{tikzcd}
 \end{equation}
Such a variety $\X_\fe$ is unique up to $k$-rational isomorphism if $H^1(k,\mathcal{A}_\X^d)$ is trivial.
\end{Thm}

\begin{proof}
  Let us first assume that $\rH=\G^\theta$. Using the notation from Section \ref{Section: indices}, the involution $\theta$ on $\G$ induces an admissible $(\Ga,\theta)$-index $\I=(X^\ast(\rA),\De, \emptyset, \De_\X^p,\sig_{\ast},\theta^\ast)$, and the data $(\I,\theta|_{\rA})$ determines the combinatorial spherical datum $\Omega_\X$. For ease of notation, we conflate $\hat{s}(x)$ with $x$.

Suppose that $(\check{\rA},\check{\B})$ is a $\Ga$-stable Borel pair for $\check{\G}$. We further assume that $\check{\rA}\subset \hat{\G}_\X$ is maximally $\vartheta$-split. Let ($\check{\rA}_\fe,\check{\B}_\fe)$ be similarly defined for $\check{\G}_\fe$. Replacing $\hat{s}(x)$ with $gxg^{-1}\in Z(\check{\mathrm{M}}_\X)$ as above (which amounts to conjugating everything), we may assume that $\eta^{-1}(\check{\B})=\check{\B}_\fe$ and $\eta^{-1}(\check{\rA})=\check{\rA}_\fe$. Note now that $\eta$ is no longer $\Ga$-equivariant but rather that $\sig\circ\eta\circ\sig_x^{-1} = \Ad(\sig(g)g^{-1})\circ\eta$. Since $g\in \check{\G}_\X^\ast$, $\eta$ is $\vartheta$-equivariant when restricted to $\hat{\G}_{\X,x}$. We thus obtain a $\vartheta$-equivariant identification 
\[
X_\ast(\check{\rA}_{\fe})\simeq X_\ast(\check{\rA})
\]
such that the $\Ga$-actions differ by a $1$-cocycle $w\in Z^1(\Ga,W^\vartheta)$, where $W^\vartheta$ is the fixed points of $\vartheta$. Noted that if $W_{\hat{\X}}$ is the little Weyl group for $\hat{\X}$ relative to the maximally $\vartheta$-split maximal torus $\check{\rA}$, then $W_{\hat{\X}} \simeq W^\vartheta/W_2$ with $W_2=\{w\in W^\vartheta: w|_{\check{\rA}_{\X}}=1\}$. This follows since the cocycle takes values in the Weyl group for $(\check{\G},\check{\rA})$, commutes with $\vartheta$, and \cite[Section 4]{Richardson}.

In particular, we may assume that $\check{\rA}_\fe$ is maximally $\vartheta$-split and that $\hat{\B}_\fe=\eta^{-1}(\check{\B}_\fe\cap \hat{\G}_{\X,x})$ is $\vartheta$-stable.
Recall that there is a $k$-rational inclusion $\rA_\fe\to \G_\fe:=\G_x$, so that the root system $\Psi_\fe$ for $\G_\fe$ is contained in $X_\ast(\check{\rA}_{\fe})=X^\ast({\rA}_{\fe})$. Let $\B_\fe$ be a $k$-rational Borel containing $\rA_\fe$ and inducing the basis $\De_\fe\subset X^\ast({\rA}_{\fe})$ corresponding to $\check{\B}_\fe$. Setting $\theta_\fe|_{\rA_\fe}$ to be the $k$-rational involution on $\rA_\fe$ induced by the involution
\[
-\vartheta: X^\ast({\rA}_{\fe})=X_\ast(\check{\rA}_{\fe})\lra X_\ast(\check{\rA}_{\fe})=X^\ast({\rA}_{\fe}),
\] we claim that $\De_\fe$ is a $(\Ga,\theta_\fe)$-basis. Indeed, this follows from the fact that $\hat{\B}_\fe$ is both $\vartheta$ and $\Ga$-stable. We may thus associate to the quasi-split endoscopic group $\G_\fe$ over $k$ the $(\Ga,\theta)$-index
    \[
    \I_\fe=(X^\ast(\rA_\fe),\De_\fe, \emptyset, \De_\X^p,\sig_{\fe,\ast},\theta^\ast),
    \]
    where
    \begin{enumerate}
        \item we have the inclusion $\De_\X^p\subset \De_\fe$ since $x\in Z(\check{\mathrm{M}}_\X)$,
        \item $\sig_{\fe,\ast}$ denotes the $\ast$-action of $\Ga$ on $\De_\fe\subset \Phi_\fe\subset X^\ast(\rA_\fe)$ determined by $\G_\fe$, and
        \item $\theta^\ast =-w_\X\theta$ is the induced diagram automorphism, which is well-defined as the Weyl elements $w_\X\in W({\check{L}_{\X}},\check{\rA}_\fe)\subset W(\check{\G}_\fe,\check{\rA}_\fe)$.
    \end{enumerate}
    That this index is admissible is precisely the statement of Proposition \ref{Prop: endoscopic roots}.
     
     In particular, there exists a $k$-rational involution $\theta_\fe$ on $\G_\fe$, normally related to $(\rA_\fe,\B_\fe)$ and restricting to $\theta_\fe|_{\rA_\fe}$ on $\rA_\fe\subset \G_\fe$, that induces the $(\Ga,\theta_\fe)$-index $\I_\fe$. Fix such an involution $\theta_\fe$ and let $\rH_\fe=\G^{\theta_\fe}_{\fe}$.  Setting $\X_\fe:=\rH_\fe\backslash\G_\fe$, this variety satisfies \eqref{item lattice} and \eqref{item diagram} for any choice of $\theta_\fe$ by construction and Proposition \ref{Prop: quotient stack is enough}. Note that Lemma \ref{Lem: same on torus} shows that any such choice for $\theta_\fe$ induce the same combinatorial data $\Omega_{\X,\fe}$ with $\Ga$-action.\quash{is dual to the data 
 \begin{equation*}
     \begin{tikzcd}
&\check{L}_{\X}\ar[r]&\check{\G}_x\\
\check{\G}_{\X,x}\ar[r,"\varphi_{\X}"]&\hat{\G}_{\X,x}\ar[ur]\ar[r,"\check{s}"]&\hat{\X}_x\ar[u, "\iota_\X"],
\end{tikzcd}
 \end{equation*}
 by construction.} 

%
Using our $\vartheta$-equivariant identification $X^\ast(\rA_\fe)=X^\ast(\rA)$, the roots $\Phi_\fe$ of $\G_\fe$ form a (potentially non-additively closed) root sub-system of $\Phi$ as subsets of $X^\ast(\rA)$.  It follows that there is a natural inclusion of root lattices $\Lam_{\X_\fe}\subset \Lam_\X$, so we obtain a canonical inclusion
\[
\cala_\X\hra \cala_{\X_\fe}.
\]
 Since $Z(\G)\subset Z(\G_\fe)$, this descends to a morphism $\Out_{\X}(\rH)\to \Out_{\X_\fe}(\rH_\fe)$. Composing with the further quotient to $\Aut_{d}(\X_\fe)$ and taking into account the well-adaptedness of $\X$, we obtain a $k$-rational morphism 
\[
{\mathrm{dist}_\fe}:\Aut_d(\X) {\lra}\Aut_d(\X_\fe),
\]
establishing \eqref{final thing}.  Lemma \ref{Lem: unique on aut} implies that $k$-structure of the diagonalizable group $\Aut_d(\X_\fe)$ depends only on $\G_\fe$ and not our choice of $\X_\fe$.

We thereby obtain a $1$-cocycle 
\[
\mu_{\X,\fe}:\Ga\overset{\mu_\X}{\lra }\Aut_d(\X)(\kbar)\overset{\mathrm{dist}_\fe}{\lra}{\Aut_d(\X_\fe)}(\kbar).
\]
Assuming Conjecture \ref{Conj: cohom surj} holds for $\X_\fe$, the class $[\mu_{\X,\fe}]\in H^1(k,\Aut_d(\X_\fe))$ lifts to a class in $H^1(k,\cala_{\X_\fe})$ 
 As discussed in Section \ref{Sec: forms and action}, this implies that there exists a $k$-rational form of $\X_\fe$, call it $\X_\fe'$, such that $\mu_{\X_{\fe}}$ is cohomologous to $\mu_{\X,\fe}$. If $\X_\fe'=\rH_\fe'\backslash\G_\fe$, let $\psi:\X_\fe\lra \X_\fe'$ be a $\G_{\kbar}$-isomorphism, so that there exists $s\in \G_\fe(\kbar)$ such that
 \[
 \psi(\rH_\fe(\kbar)g) = \rH_\fe'(\kbar)sg.
 \]
 It is now easy to verify that $\rH_\fe'=\G_\fe^{\theta_\fe'}$ where the conjugated involution
 \[
 \theta_\fe'(g) = s(\theta(s^{-1}gs))s^{-1}, \quad\text{where }g\in \G_\fe(\kbar)
 \]
 is defined over $k$. Replacing $s$ with an element of the form $hs$ with $h\in \G_\fe(k)$ if necessary (see Lemma \ref{Lem: normally related}), we may assume that $\theta'_\fe$ is normally related to $\rA_\fe$, and $\theta'_\fe|_{\rA_\fe}\ = \theta_\fe|_{\rA_\fe}$. This procedure does not change the index $\I_\fe$, up to isomorphism. Thus, replacing $\X_\fe$ and $\theta_\fe$ with $\X_\fe'$ and $\theta_\fe'$, we may assume that the $1$-cocycle $\mu_{\X_\fe}$ satisfies \eqref{item rep}. This completes the proof when $\rH=\G^\theta$.

Now consider the case that $\rH = (\G^\theta)^\circ$. There is a natural $\G$-equivariant dominant morphism $\X\to\X_\theta:=\G^\theta\backslash \G$, so by functoriality \cite[Theorem 1]{KnopFunctorial}, there is a commutative diagram
 \[
     \begin{tikzcd}
\check{\G}_{\X_\theta}\ar[r]\ar[d]&\hat{\G}_{\G^\theta\backslash \G}\ar[d]\ar[r,"\hat{s}_\theta"]&\ar[d]\hat{\X}_{\theta}\\
\check{\G}_\X\ar[r]&\hat{\G}_\X\ar[r,"\hat{s}"]&\hat{\X}.
    \end{tikzcd}
 \] However, since $\X\to\X_\theta$ is an isogeny, $\hat{\G}_{\G^\theta\backslash \G}=\hat{\G}_{\X}$ so that $\hat{\X}_\theta\simeq \hat{\X}$. Thus, $x\in\hat{\X}$ corresponds a unique element $x'\in \hat{\X}_{\theta}$. Running the preceding argument for the group $\G^\theta$ we obtain the involution $\theta_\fe:\G_\fe\to \G_\fe$; setting $\X_{\fe,\theta}:=\G_{\fe}^{\theta_{\fe}}\backslash\G_{\fe}$, this induces a canonical embedding $\cala_{\X_\theta}\subset \cala_{\X_{\fe,\theta}}$. Recall now the notation from Section \ref{Section: norms sym}: passing to the group $\rH=(\G^\theta)^\circ$ induces a short exact sequence of $\Ga$-modules
 \[
 0\to X^\ast_\theta/\Lam_\X\lra X^\ast_{\theta,\circ}/\Lam_\X\lra X^\ast_{\theta,\circ}/X^\ast_\theta\lra 0,
 \]
 dual to the sequence
 \[
 1\lra \pi_0(\G^\theta)\lra \cala_{\X}\lra \cala_{\X_{\theta}}\lra 1.
 \]
 We obtain a similar short exact sequence for $\rH_\fe=(\G^{\theta_\fe}_\fe)^\circ$, with $\pi_0(\G^{\theta_\fe}_\fe)$ dual to $X^\ast_{\theta_\fe,\circ}/X^\ast_{\theta}$. We now show that 
 \[
X^\ast_{\theta,\circ}\subset X^\ast_{\theta_\fe,\circ}.
 \]
Using \cite[Corollary 1.5]{Hofscheier}, this follows by showing that if $\lam\in X^\ast_{\theta,\qq}\cap \X^\ast(\rA)$ satisfies that $\la\rho(D),\lam\ra\in \zz$ for all colors $D\in \cald({\X_\theta})$, then the same holds for all $D\in \cald({\G_\fe^{\theta_\fe}\backslash G_\fe})$. To see this, note that since $\Lam_{\G_\fe^{\theta_\fe}\backslash\G_\fe}\subset \Lam_{\X_\theta}$, we have
\[
 \check{\Lam}_{\X_\theta}\subset \check{\Lam}_{\G_\fe^{\theta_\fe}\backslash\G_\fe}.
\]
By Proposition \ref{Prop: sym colors}, $\rho(\cald({\X_\theta}))=\check{\De}_\X^n$ is the set of coroots for the simple relative roots by Lemma \ref{Lem: root lattice} and similarly for $\rho(\cald({\G_\fe^{\theta_\fe}\backslash\G_\fe}))=\check{\De}_{\X_\fe}^n$. Since $\De_\fe\subset \Phi^+$, it follows that
\[
\De_{\X_\fe}^n\subset \{\al-\theta(\al):\al\in \Phi\}\setminus\{0\}=\widetilde{\Phi}_\theta.
\]
Since $\lam\in X^\ast_{\theta,\qq}\cap \X^\ast(\rA)$ pairs integrally with the coroots associated to elements of $\widetilde{\Phi}_\theta$, it follows that it pairs integrally with $\rho(\cald({\G_\fe^{\theta_\fe}\backslash\G_\fe}))$.

This implies that there is a unique subgroup $(\G^{\theta_\fe}_\fe)^\circ\subset \rH_\fe\subset \G^{\theta_\fe}_\fe$ satisfying that if $\X_\fe=\rH_{\fe}\backslash\G_\fe$, then \eqref{item lattice} and \eqref{item diagram} hold. We note that a similar argument now holds for any $(\G^\theta)^\circ\subset\rH\subset\G^\theta$ by replacing $X^\ast_{\theta,\circ}$ with the appropriate sublattice. In particular, $G^\theta/\rH\simeq\G_\fe^{\theta_\fe}/\rH_\fe$.
 
\quash{ inducing an inclusion  commutative diagram
 \[
\begin{tikzcd}
    \pi_0(\rH)\ar[d]\ar[r]& \cala_{\X}\ar[d]\ar[r]& \cala_{\X_{\theta}}\ar[d]\\
    \pi_0(\rH_\fe)\ar[r]& \cala_{\X_\fe}\ar[r]& \cala_{\X_{\theta_\fe}}\\
\end{tikzcd}
 \]}

Finally, if $\G^\theta\subset \rH\subset N_\G(\G^\theta)$, we may similarly reduce to the case of $\G^\theta$. More precisely, there is a natural $\G$-equivariant quotient $\X_\theta:=\G^\theta\backslash \G\to \X$, so by functoriality there is a commutative diagram
 \[
     \begin{tikzcd}
       \check{\G}_\X\ar[r]\ar[d]&\hat{\G}_\X\ar[d]\ar[r,"\hat{s}"]&\hat{\X}\ar[d]\\
\check{\G}_{\X_\theta}\ar[r]&\hat{\G}_{\G^\theta\backslash \G}\ar[r,"\hat{s}_\theta"]&\hat{\X}_{\theta}
    \end{tikzcd}
 \]so that $x\in\hat{\X}$ induces a unique element $x'\in \hat{\X}_{\theta}$. Running the preceding argument we obtain the involution $\theta_\fe:\G_\fe\to \G_\fe$. The group $\rH$ corresponds to the $k$-rational subgroup 
 \begin{equation}\label{eqn: subgroup of endo auto}
   \rH/\G^\theta\subset \cala_{\X_\theta} \subset \cala_{\G^{\theta_\fe}_{\fe}\backslash\G_{\fe}},
 \end{equation}
 
 Thus there exists a unique $k$-rational subgroup $\G_\fe^{\theta_\fe}\subset \rH_\fe\subset N_{\G_\fe}(\G^{\theta_\fe}_\fe)$  corresponding to $\rH$. By construction, the variety $\X_\fe=\rH_\fe\backslash\G_\fe$ satisfies properties \eqref{item lattice} and \eqref{item diagram} of the proposition. It is also immediate that property \eqref{item rep} holds since it holds for ${\G^\theta\backslash \G}$.
\end{proof}

 For $x\in \hat{\X}$ as in the theorem, the $k$-rational symmetric pair constructed will be referred to as an \textbf{endoscopic symmetric pair} $(\G_\fe,\rH_\fe,\theta_\fe)$, and $\X_\fe=\rH_\fe\backslash\G_\fe$ as the associated endoscopic symmetric variety. In this generality, there need not exist a $k$-form of $\X_{\fe}$ that is uniquely distinguished by the construction, even up to $\G$-inner form. A first constraint is whether Conjecture \ref{Conj: cohom surj} holds for $\X_\fe$. 
 Let us assume this conjecture for the moment. 
 
 It is generally not possible to isolate a single endoscopic variety up to isomorphism; see Example \ref{Ex: relative inner forms} below. This is reminiscent of the issue of the $L$-group ${}^L\G$ determining $\G$ only up to inner twist, and motivates Theorem \ref{Thm: quasi-split G-inner form}. 
 \begin{Cor}\label{Cor: quasisplit endo}
    Suppose that $\G$, $\X$, $\G_\fe$, and $\X_\fe$ are as in Theorem \ref{Thm: exists}. There is a \emph{unique} (up to isomorphism) $\G_\fe$-form satisfying \eqref{item rep} and such that $\rH_\fe^\circ$ is quasi-split.
\end{Cor}
\begin{proof}
   This is a direct consequence of Theorem \ref{Thm: quasi-split G-inner form}.
\end{proof}
We note that it is possible for $\X$ to be well-adapted, while $\X_\fe$ is not (see Table \ref{tab:Examples}).

 \subsection{Endoscopic data for symmetric varieties}\label{Section: endoscopic datum}
We now assume that $k$ is a global, local, or finite field. We continue to assume that $\mathrm{char}(k)$ is good for $\G$. 
Suppose that $\G$ is a connected reductive group over $k$. Let $\X=\rH\backslash\G$ be a symmetric $k$-variety which we assume satisfies Assumption \ref{assumption: no type N} (so that $\X$ is both {well-adapted and that Conjecture \ref{Conj: cohom surj} holds for $\X$}).

Let $(\rA,\B)$ be a $\theta$-admissible pair. This induces $\Ga$-stable Borel subgroups for $\check{\G}$, $\hat{\G}_\X$ and $\check{\G}_\X$ denoted $\check{\B},$ 
 $\hat{\B}_\X$, and $\check{\B}_\X$, respectively. Moreover, $\hat{\B}_\X$ is $\vartheta$-stable and for any choice of distinguished morphism $\varphi_\X:\check{\G}_\X\lra \hat{\G}_\X$ compatible with $\check{\rA}_\X\to \check{\rA}$, we have $\varphi_\X(\check{\B}_\X)\subset \hat{\B}_\X$.  Fix a choice of pinning $\{x_{\check{\al}}\}$ for $\check{\fg}$, so that $\Ga$ acts as explained in Section \ref{Section: dual groups}.

\begin{Def}
We define an \emph{endoscopic datum} of $\X$ to be a quintuple $${\fe}=(\G_\fe,\theta_\fe,\X_\fe,\ka,\eta)$$  where $\G_\fe$ is an endoscopic group equipped with a $k$-rational involution $\theta_\fe$ and a symmetric variety $\X_\fe$ associated to $\theta_\fe$ satisfying the requirements of Theorem \ref{Thm: exists} such that 
\begin{enumerate}
\item if $\hat{s}_\fe:\hat{\X}_\fe\to \hat{\G}_\fe\subset \check{\G}_\fe$ is the symmetrization map of the dual variety, then $(\G_\fe,\hat{s}_\fe(\ka),\eta)$ gives a pure refined endoscopic datum for $\G$,
        \item $\eta$ restricts to an isomorphism $\hat{\G}_{\fe,\X_\fe}\simeq (\hat{\G}_{\X,\eta(\hat{s}_\fe(\ka))})^\circ$ and intertwines the involutions $\vartheta_\fe$ and $\vartheta$,
        \item  $\hat{s}_\fe(\ka)\in \hat{\X}_\fe^\heart$ and $\eta(\hat{s}_\fe(\ka))\in\hat{s}(\hat{\X})^{\heart}\subset \hat{\G}_\X$, 
        \item  the dual variety $\hat{\X}_\fe$ fits into a commutative diagram
         \begin{equation}\label{eqn: commutative dual diag}
  \begin{tikzcd}
      \check{\G}_{\X_\fe}\ar[d]\ar[r,"\varphi_{\X_\fe}"]&\hat{\G}_{\fe,\X_\fe}\ar[d,"\eta"]\ar[r]&\hat{\X}_\fe\ar[d,"\eta_\ka"]\\  
\check{\G}_{\X}\ar[r,"\varphi_{\X}"]&\hat{\G}_\X\ar[r]&\hat{\X}.
\end{tikzcd}
 \end{equation}
Here, $\eta_\ka$ realizes $\hat{\X}_\fe$ as the $\hat{\G}_{\fe,\X_\fe}$-orbit at $\eta(\hat{s}_\fe(\ka))$ and 
    \[
 \check{\G}_{\X_\fe}\overset{\varphi_{\X_\fe}}{\lra} \check{\G}_\fe
    \]
    is the unique distinguished morphism associated to $\X_\fe$ such that the diagram commutes.
    \end{enumerate}
    We say that $\fe$ is \textbf{elliptic} if $\cala_{\hat{\X}_\fe}^{\Ga,\circ}= \cala_{\hat{\X}}^{\Ga,\circ}$, where $\cala_{\hat{\X}} =\Aut^{\hat{\G}_\X}(\hat{\X})$ is the $\hat{\G}_\X$-equivariant automorphism group of $\hat{\X}$ and $\cala_{\hat{\X}_\fe}$ is defined similarly.
\end{Def}

\begin{Rem}
     For any distinguished morphism $\varphi_{\X}$, we have $\ker(\varphi_\X)=\ker(\check{\rA}_\X\to \check{\T})$ \cite[Remark 7.8]{KnopSchalke}. In particular, for a given endoscopic datum ${\fe}$ and morphism $\varphi_{\X}$, the morphism $\varphi_{\X^\fe}$ is uniquely determined by the commutativity of \eqref{eqn: commutative dual diag} and has the same kernel as $\varphi_\X$.
\end{Rem}
 \begin{Def}\label{Def: endo iso}
 An isomorphism of such data 
    \[
    \fe_1=(\G_1,\theta_1,\X_1,\ka_1,\eta_1)\lra \fe_2=(\G_2,\theta_2,\X_2,\ka_2,\eta_2)
    \]
    is an isomorphism $f:(\G_1,\hat{s}_1(\ka_1),\eta_1)\to (\G_2,\hat{s}_2(\ka_2),\eta_2)$ of the endoscopic triples satisfying
    \begin{enumerate}
        \item $f$ intertwines $\theta_1$ with a pure inner twist of $\theta_2$. That is, the involution $\theta_2':=f^{-1}\circ \theta_1\circ f$ is a pure inner twist of $\theta_2$;
        \item the morphism $\eta_1\circ\check{f}$ and $\eta_2$ are $\check{\G}^\ast_{\X}$-conjugate;
        \item the images $\check{f}(\ka_2)$ and $\ka_1$ coincide in $\pi_0(\cala_{\hat{\X}_1}^\Ga)$.
     \end{enumerate}
      Let $\Aut(\fe)$ denote the automorphism group of the endoscopic datum. Note that there is a natural inclusion $\G_\fe/(\rH_\fe\cap Z(\G_\fe))(k)\subset\Aut(\fe)$, which we refer to as the group of inner automorphism.
 \end{Def} 
 We remark that the requirement that $f$ intertwines $\theta_1$ with a pure inner twist of $\theta_2$ forces $f$ to descend to an isomorphism
 \[
 \X_1=\rH_1\backslash\G_1 \lra f(\rH_1)\backslash\G_2\simeq\rH_2\backslash\G_2=\X_2.
 \]
 This follows from Lemma \ref{Lem: pure inner twist invol}.
 
\begin{Rem}
 While the definition of an endoscopic datum includes the involution, we see that this involution is only determined up to pure inner twist. In light of Lemma \ref{Lem: pure inner twist invol}, the data of $[\theta_\fe]$ is encoded by $\X_\fe$. We nevertheless find it psychologically nicer to keep track of the involution.
\end{Rem}

\begin{Def}
    Suppose that $\G$ is a connected reductive group over $k$. Let $\X=\rH\backslash\G$ be a symmetric $k$-variety which we assume is {well-adapted}. Let ${\fe}=(\G_\fe,\theta_\fe,\X_\fe,\ka,\eta)$ denote an endoscopic datum for $(\G,\theta,\X)$. We say that $\fe$ is the \textbf{quasi-split endoscopic datum} if $\X_{\fe}$ is the quasi-split form of Corollary \ref{Cor: quasisplit endo}.
\end{Def}

We will always assume that for a given endoscopic datum $\fe$ that $\X_\fe$ is the quasi-split form; this only plays a role in Section \ref{Section: stabilize}.


Theorem \ref{Thm: exists} shows that for any $\ka\in \hat{\X}^{\heart,\Ga}$, there exists an endoscopic datum $\fe=(\G_\fe,\theta_\fe,\X_\fe,\ka,\eta)$ containing $\ka$. On the other hand, Lemma \ref{Lem: outer forms classify} implies that any two endoscopic varieties $\X_{1}=\rH_1\backslash\G_\fe$ and $\X_2=\rH_2\backslash\G_\fe$ associated to $\ka$ in this fashion satisfy that $\rH_1$ and $\rH_2$ are inner twists of each other by a cocycle with values in $\G_\fe$. In particular, the construction of Theorem \ref{Thm: exists} {does not} distinguish between non-isomorphic $\G_\fe$-inner forms $\X_{\fe,1}$ and $\X_{\fe,2}$ in general. 

\begin{Ex}\label{Ex: relative inner forms}
    Let $\G =\Res_{E/k}(\rH_E)$ and $\rH\subset \G$ is the fixed point subgroup of a Galois involution. If $\rH'\subset \G$ is any inner form of $\rH$ such that $\rH'_E\simeq\rH$ but $\rH'$ is not a pure inner twist, then the varieties $\X=\rH\backslash\G$ and $\X'=\rH'\backslash\G$ are $\G$-inner twists of each other and not isomorphic. Theorem \ref{Thm: exists} cannot distinguish between such pairs.
\end{Ex}

\subsubsection{Extended endoscopic data}\label{Section: extended data}
For the purposes of (relative) functoriality, it is necessary to work with morphisms of $L$-groups rather than dual groups. In this subsection, we discuss the problem of enhancing an endoscopic datum $\fe=(\G_\fe,\theta_\fe,\X_\fe,\ka,\eta)$ to be compatible with $L$-groups. Following the discussion of the $\Ga$-actions on dual groups in Section \ref{Section: dual groups}, we may associate to the symmetric pair $(\G_\fe,\X_\fe)$ the sequence of groups 
\[
{}^L\X_\fe\lra{}^L\hat{\G}_{\X_\fe}\lra {}^L{\G}_{\fe},
\] extending the maps on dual groups. Recall that the $\Ga$-action on $\hat{\G}_{\X_\fe}$ is completely determined by that on $\check{\G}_\fe$, and that the action on $\check{\G}_{\X_\fe}$ is so determined once the convention \eqref{eqn: root embedding} is imposed. In particular, if $\eta$ extends to an $L$-homomorphism ${}^L\eta$, we obtain an embedding
\begin{equation}\label{diag: L-descent}
     \begin{tikzcd}
{}^L{\X_\fe}\ar[d]\ar[r]&{}^L\hat{\G}_{\X_\fe}\ar[d]\ar[r]&{}^L{\G}_\fe\ar[d, "{}^L\eta"]\\
{}^L{\X}\ar[r]&{}^L\hat{\G}_{\X}\ar[r]&{}^L\G,
\end{tikzcd}
\end{equation}
In this case, we refer to ${}^L\fe=(\G_\fe,\theta_\fe,\X_\fe,\ka,{}^L\eta)$ as an extended endoscopic datum. This is always possible if $\G_{der}$ is simply connected \cite[Proposition 1]{LanglandsStableconj}; more generally, one makes use of a $z$-extension to reduce to this case. 

If we fix a compatible $z$-extension $(\G_1,\theta_1)$ of $(\G,\theta)$ with 
\[
1\lra N\lra\G_1\lra \G\lra 1,
\]then by Remark \ref{Rem: z-extension on dual side} and Lemma \ref{Lem: pass to a sc cover roots} we have (for appropriately chosen distinguished morphisms) a commutative diagram
\begin{equation}
    \begin{tikzcd}
       \check{\G}_\X\ar[r]\ar[d]&\hat{\G}_\X\ar[d]\ar[r]&\hat{\X}\ar[d]\\
\check{\G}_{\X_1}\ar[r]\ar[d]&\hat{\G}_{\X_1}\ar[d]\ar[r]&\hat{\X}_{1}\ar[d]\\
\check{N/N^{\theta_1}}\ar[r]&\check{N}\ar[r]&\check{N}^{\theta_1},
    \end{tikzcd}
\end{equation}
where the bottom row is dual to the short exact sequence of induced $k$-tori
\[
1\lra N^{\theta_1}\lra N \lra N/N^{\theta_1}\lra1,
\]
and the right-most vertical column is a fibration. 
In particular, to an $x\in \hat{\X}^{ss}$, we obtain $x_1\in \hat{\X}_1^{ss}$ via the induced inclusion. Notice that $x_1$ depends only on the choice of compatible $z$-extension since $\hat{\X}$, $\hat{\X}_1$, and the inclusions are independent of the choice of distinguished morphisms. One may similarly obtain a $\theta_\fe$-compatible $z$-extension $(\G_{\fe,1},\theta_{\fe,1})$ such that
\[
1\lra N\lra\G_{\fe,1}\lra \G_\fe\lra 1,
\]
which fits into an extended endoscopic datum ${}^L\fe_1= (\G_{\fe,1},\theta_{\fe,1},\X_{\fe,1},x_1,{}^L\eta_1)$ so that there is an embedding
\begin{equation*}
     \begin{tikzcd}
{}^L{\X_{\fe,1}}\ar[d]\ar[r]&{}^L\hat{\G}_{\X_\fe,1}\ar[d]\ar[r]&{}^L{\G}_{\fe,1}\ar[d, "{}^L\eta_1"]\\
{}^L{\X}_1\ar[r]&{}^L\hat{\G}_{\X,1}\ar[r]&{}^L\G_1,
\end{tikzcd}
\end{equation*}
Such extensions play a similar role to $z$-extensions in the formulation of the matching of orbital integrals as explained in \cite{LesliestabFJ}.
Indeed, setting $N_{\theta}:={N/N^{\theta_1}}$ we obtain an exact sequence
\[
{}^L\X_{\fe,1}\to {}^L\X_1 \to {N_\theta}\rtimes \Ga.
\]
Restricting the first morphism to $\Ga$ and pulling back to the Weil group $W_k\to \Ga$ this gives a map
\[
\check{\lam}:W_k\to {N_{\theta}}\rtimes W_k.
\]
 When $k$ is local, this gives a parameter for a character $\lam: N_\theta(k)\to \cc^\times$. The main point is that if $f\in C^\infty_c(\X(k))$, we may view it as a function on $\X_1(k)$ constant on $N_\theta$-orbits. Then the appropriate $\ka$-orbital integrals of $f$ should match with the stable orbital integrals of a function $f^{\fe,1}\in C^\infty_{c}(\X_{\fe}(k),\lam)$, where we view $\X_{\fe,1}$ as a $N_\theta$-torsor satisfying that $\X_{\fe,1}(k)\to \X_{\fe}(k)$ is surjective since $N_\theta$ is an induced torus. We then form the $\cc^\times$-bundle $\X_{\fe,\lam}=\X_{\fe,1}(k)\times^{N_\theta(k)}\cc^\times_\lam$ and consider compactly-supported sections of the associated line bundle
 \[
 C_c^\infty(\X_\fe(k),\lam)=\bigoplus_{\rH_{\fe,1}'\in \ker^1(\rH_{\fe,1},\G_{\fe,1};k)}\ind_{\rH'_{\fe,1}(k)N(k)}^{\G_{\fe,1}(k)}(\overline{\lam}),
 \]%
 where $\rH_{\fe,1}'$ runs over a set of pure inner twists of $\rH_{\fe,1}$ indexed by $\ker^1(\rH_{\fe,1},\G_{\fe,1};k)$ and 
 \[
\ind_{\rH'_{\fe,1}(k)N(k)}^{\G_{\fe,1}(k)}(\overline{\lam})= \{f:\G_{\fe,1}(k)\to\cc: f(nhg) = \overline{\lam}(n)f(g)\},
 \]where $n\in N(k),\; h\in \rH_{\fe,1}(k),\: g\in \G_{\fe,1}(k)$, $\mathrm{supp}(f) \text{ is compact modulo }\rH'_{\fe,1}(k)N(k),$ and $\overline{\lam}:N(k)\to N_\theta(k)\to \cc^\times$ is the pull-back of $\lam$ to $N(k)$.

\subsection{Point matching}\label{Section:orbit match}
The following theorem gives a matching of stable semi-simple orbits of symmetric varieties.
\begin{Thm}\label{Thm: point comparison}
Suppose that $\G$ is a quasi-split reductive group over $k$. Let $\theta$ be a $k$-rational involution and let  $\X=\rH\backslash\G$ be an associated symmetric variety.  Let $\fe = (\G_\fe,\theta_\fe,\X_\fe,\ka,\eta)$ be an endoscopic datum for $(\G,\X)$. There exists a canonical map
\[
\fa_\fe:\X_\fe\sslash \rH_\fe\lra\X\sslash \rH
\]
between the categorical quotients.
\end{Thm}
A direct consequence of this is the ability to compare  objects indexed by stable orbits (such as stable orbital integrals) between $\X_\fe$ and $\X$; see Section \ref{Section: stabilize}.

We begin with the following simplification.
\begin{Lem}\label{Lem: chevalley}
    Suppose that the preceding theorem holds whenever $\rH= \G^\theta$. Then it holds for all symmetric subgroups.
\end{Lem}
\begin{proof}
Let $\rA$ be a fixed maximally $\theta$-split maximal torus, so that $s:\rA\to \rA_\theta$ is the quotient to the canonical torus of $\X_\theta:=\G^\theta\backslash \G$. For any  $(\G^\theta)^\circ\subset \rH\subset N_{\G}(\G^\theta)$, Richardson gives an isomorphism of $k$-schemes \cite[Corollary 11.5]{Richardson} 
\begin{equation}\label{eqn: Chevalley type}
    \X_{\theta}\sslash  \rH\simeq \rA_\theta/W_\X,
\end{equation}
where $W_\X$ denotes the little Weyl group of $\X$ and $\X_\theta=\G^\theta\backslash\G$.
   
    Suppose now that $\G^\theta\subset \rH\subset N_{\G}(\G^\theta)$. By the proof of Theorem \ref{Thm: exists} and the notation therein, we obtain an involution $\theta_\fe$ on $\G_\fe$ such that there is a unique $\G_{\fe}^{\theta_\fe}\subset \rH_\fe$ satisfying the desired properties of the theorem. By assumption there is a a canonical map
\[
\fa_\fe^{\theta}:  \X_{\fe,\theta}\sslash\G^{\theta_\fe}_\fe \lra\X_\theta\sslash\G^{\theta}
\]
between the categorical quotients. The Chevalley-type isomorphism \eqref{eqn: Chevalley type} implies that this in fact gives an isomorphism
\[
\fa_{\theta}:  \X_{\fe,\theta}\sslash\rH_\fe\lra\X_\theta\sslash \rH
\]

The inclusion of groups $\G^\theta\subset \rH$ induces a unique morphism $\X_\theta\sslash \rH\to \X\sslash \rH$ realizing $\X\sslash \rH$ as the categorical quotient of $\X_\theta\sslash \rH$ by the automorphism group $$\mathcal{Z}:=\G^\theta\backslash\rH\subset \cala_{\X_\theta};$$
here we rely on Lemma \ref{Lem: mod out by autos} with $\mathcal{Z}$ corresponding to $\cala$ and $\rH$ corresponding to $\G$. Let $\pi_{\mathcal{Z}}$ denote this quotient map. On the other hand, the inclusion of $\G_\fe^{\theta_\fe}\subset \rH_\fe$ and the inclusion \eqref{eqn: subgroup of endo auto}, $\mathcal{Z}\simeq \G_\fe^{\theta_\fe}\backslash\rH_\fe$ also acts on $\X_{\fe,\theta}\sslash\rH_\fe$ with categorical quotient given by $\X_{\fe}\sslash\rH_\fe$.

Noting that $$\pi_{\mathcal{Z}}\circ\fa_\theta:  \X_{\fe,\theta}\sslash\rH_\fe \to \X\sslash \rH$$ is $\mathcal{Z}$-invariant, it follows that there is a unique morphism $\fa_\fe:\X_{\fe}\sslash\rH_\fe \to \X\sslash \rH$ so that the diagram
\[
    \begin{tikzcd}
     \X_{\fe,\theta}\sslash\rH_\fe\ar[d]\ar[r,"\fa_\theta"]& \X \ar[d,"\pi_{\mathcal{Z}}"]\\ 
       \X_{\fe}\sslash\rH_\fe\ar[r,"\fa_\fe"]&\X\sslash \rH.
    \end{tikzcd}
\]commutes. This proves the claim when $\G^\theta\subset \rH\subset N_{\G}(\G^\theta)$.
\begin{Rem}\label{Rem: Chevalley isom}
    Note that 
    \begin{equation}\label{eqn: more general Chevalley}
    \X\sslash \rH\simeq [\X\sslash \rH_{\theta}]\sslash\mathcal{Z}\simeq \rA_\X/W_\X,
    \end{equation}
    where $\Ax \simeq \rA_\theta/\mathcal{Z}$ is the canonical torus for $\X$. In particular, the preceding argument extends Richardson's result to $\X=\rH\backslash\G$ when $\G^\theta\subset \rH\subset N_{\G}(\G^\theta)$.
\end{Rem}

Suppose now that $(\G^\theta)^\circ\subset \rH\subset \G^\theta$. Motivated by the preceding remark, we claim that the lemma will follow from extending the Chevalley-type isomorphism to $\X\sslash \rH$ by replacing $\rA_\theta$ with its finite cover $\Ax$. Indeed given such an isomorphism, a similar argument as above gives a diagram
   \[
    \begin{tikzcd}
       \X_\fe\sslash\rH_\fe\ar[r,"\simeq"]&\Ax/W_{\X_\fe}\ar[d]& \Ax/W_{\X}\ar[r,"\simeq"]\ar[d]&\X\sslash \rH \\ 
        \X_{\fe,\theta}\sslash\rH_\fe\ar[r,"\simeq"]&\rA_\theta/W_{\X_\fe}\ar[r,"\fa_\theta"]&\rA_\theta/W_\X\ar[r,"\simeq"]&\X_{\theta}\sslash\rH.
    \end{tikzcd}
    \]
    It follows that there exists a unique map $\fa_\fe:\Ax/W_{\X_\fe}\to \Ax/W_{\X}$ so that the diagram commutes.

To prove the claim, we revisit the proof of \cite[Theorem 13.2]{Richardson}: it relies solely on basic algebraic geometry and the following facts:
\begin{enumerate}
    \item Since each element $w\in W_\X$ has a representative in $N_{K}(\rA_\theta)$, where $K=(\G^\theta)^\circ$, we obtain a morphism of affine varieties
    \[
    j':\rA_\theta/W_\X\to \X_\theta\sslash K
    \]
    satisfying $\pi_{\X_\theta}\circ j = j'\circ\pi_{\rA_\theta}$, where $j:\rA_\theta\to \X_\theta$ is the orbit map. This trivially extends to $\X=\rH\backslash\G$ by the same fact for $W_\X$.
    \item Since $\pi_{\X_\theta}$ identifies closed $K$-orbits on $\X$ with points of $\X_\theta\sslash K$, Theorem 7.5 of \cite{Richardson} implies that $j'$ is surjective. Since $\X\to \X_\theta$ is the quotient by a finite subgroup scheme of $\Ax$ with quotient $\rA_\theta$, this claim immediately extends to $\X$. That is, $x\in \X(\kbar)$ is semi-simple if and only if $K\cdot x$ meets $\rA_\X\subset \X$.
    \item Finally, the map $j'$ is injective by Corollary 11.2 of \cite{Richardson}. This similarly extends to $\X$ by the surjectivity of $\X\to \X_\theta$.
\end{enumerate}
Following the argument of \cite[Theorem 13.2]{Richardson}, we therefore conclude the isomorphism \eqref{eqn: more general Chevalley} when  $(\G^\theta)^\circ\subset \rH\subset \G^\theta$ and the lemma.
\quash{By the proof of Theorem \ref{Thm: exists} and the notation therein, we obtain an involution $\theta_\fe$ on $\G_\fe$ such that there is a unique subgroup $(\G^{\theta_\fe}_\fe)^\circ\subset \rH_\fe\subset \G^{\theta_\fe}_\fe$ satisfying that if $\X_\fe=\G_\fe/\rH_{\fe}$ and $\rA\subset \G_\fe$, then $\rA_{\X_\fe}\simeq \Ax$. By our assumption and the inclusion of groups  $\rH\subset \G^\theta$ and $\rH_\fe\subset \G_\fe^{\theta_\fe}$, there is a unique morphism $\left[ \G^{\theta}_\fe\backslash\backslash\X_{\fe,\theta}\right]\to \left[\G^\theta\backslash\backslash\X_\theta \right]$ sitting in a diagram
    \[
    \begin{tikzcd}
       \left[ \X_\fe\sslash\rH_\fe\right]\ar[d]& \left[\X\sslash \rH \right]\ar[d]\\ 
       \left[ \G_\fe^{\theta}\backslash\backslash\X_{\fe,\theta}\right]\ar[r]&\left[\G^\theta\backslash\backslash\X_\theta \right].
    \end{tikzcd}
    \]
    
    there is a unique $\G_{\fe}^{\theta_\fe}\subset \rH_\fe$ satisfying the desired properties of the theorem. By assumption there is a a canonical map
\[
\mathcal{A}_\fe^{\theta}:\left[ \G^{\theta_\fe}_\fe
\backslash\backslash \X_{\fe,\theta}\right]\lra\left[ \G^{\theta}\backslash\X_\theta\right]
\]
between the categorical quotients. On the other hand, the inclusion of groups $\G^\theta\subset \rH$ and $\G_\fe^{\theta_\fe}\subset \rH_\fe$ imply that there is a unique morphism $\left[ \G^{\theta}\backslash\X_\theta\right]\to \left[\X\sslash \rH \right]$.}
\end{proof}
We may thus assume that $\rH=\G^\theta$ and $\rH_\fe = \G_\fe^{\theta_\fe}$. Fix $(\theta,k)$-admissible Borel pairs $(\rA,\B)$ and $(\rA_\fe,\B_\fe).$ This gives rise to a short exact sequence
\[
1\lra {\rA}^{\theta}\lra \rA\lra \mathrm{A}_{\X}\lra 1,
\]
and a commutative diagram (recall that  $\T_\X = ( {\rA}^{\theta})^\circ$)
\[
 \begin{tikzcd}
 \check{\mathrm{A}}_{{\X}}\ar[d]\ar[r]&\check{\rA}\ar[d]\ar[r]&{\check{\T}_\X}\ar[d]\\
 \check{\G}_{{\X}}\ar[r,"\varphi_{\X}"]&\hat{\G}_{\X}\ar[r,"\hat{s}"]&\hat{{\X}},
 \end{tikzcd}
\]
where $\varphi_{\X}$ is  our fixed distinguished morphism. Recall from Proposition \ref{Prop: dual involution} that $\check{\G}_\X^\ast=\varphi_\X(\check{\G}_\X)$ is the connected component of the identity of $\hat{\G}_\X^{\vartheta}$. As shown in Section \ref{Sec dual basic}, the dual pair $(\check{\rA},\hat{\B}_\X)$ is $\vartheta$-stable, so it gives a fundamental pair for $(\hat{\G}_\X,\hat{\X})$.

We have a similar diagram for $(\G_\fe,\X_\fe)$ and $(\rA_\fe,\B_\fe)$; we let $\vartheta_\fe$ denote the corresponding involution on $\hat{\G}_{\X_\fe}$. By assumption, we have a commutative diagram
\begin{equation}
      \begin{tikzcd}
\check{\G}_{\X_{\fe}}\ar[d, "\eta_\X"]\ar[r,"\varphi_{\X_\fe}"]&\hat{\G}_\fe\ar[d,"\eta"]\ar[r]&\hat{\X}_\fe\ar[d,"\eta_\ka"]\\  
\check{\G}_{\X}\ar[r,"\varphi_X"]&\hat{\G}\ar[r]&\hat{\X},
\end{tikzcd}
\end{equation}
realizing $(\hat{\G}_\fe,\hat{\X}_\fe)$ as the descendant of $(\hat{\G},\hat{\X})$ at $\ka\in \hat{\X}$.

\begin{Lem}\label{Lem: important simplification}
We may conjugate our choice of endoscopic data $\fe$ by an element of $\check{\G}^\ast_{\X_\fe}$ such that we have 
\begin{enumerate}
    \item\label{property1} $\eta^{-1}(\check{\rA},\check{\B})=(\check{\rA}_\fe,\check{\B}_\fe)$,
    \item\label{property2} $\eta_\X(\check{\mathrm{A}}_{{\X^\fe}})=\check{\mathrm{A}}_{{\X}}$.
\end{enumerate}
In particular, the map $\eta:\check{\rA}_\fe\lra \check{\rA}$ may be chosen so that it intertwines $\vartheta_\fe$ with $\vartheta$.
\end{Lem}
\begin{proof}
By the properties of semi-simple descent, the commutativity of the diagram already implies that 
\begin{equation}\label{eqn: important endoscopic commute}
    \eta\circ\vartheta_\fe = \vartheta\circ \eta.
\end{equation}
In particular, if we set $(\check{\rA}'_\fe,\check{\B}'_\fe)=\eta^{-1}(\check{\rA},\check{\B})$, then $(\check{\rA}'_\fe,\check{\B}'_\fe)$ is a fundamental pair for $\vartheta_\fe$. Since $\hat{\X}_{\fe}$ is a spherical variety of minimal rank, Lemma \ref{Lem: minimal rank} implies that there exists $h\in \check{\G}_{\X_\fe}^\ast$ such that
\[
\Ad(h)(\check{\rA}'_\fe,\check{\B}'_\fe)=(\check{\rA}_\fe,\check{\B}_\fe).
\]
Thus, replacing $\eta$ (resp. $\eta_{\X}$) with $\eta\circ\Ad(h^{-1})$ (resp. $\eta_{\X}\circ\Ad(h^{-1})$), we may assume that (\ref{property1}) holds. With this assumption, (\ref{property2}) follows from the commutativity relation \eqref{eqn: important endoscopic commute} and identifying $\check{\mathrm{A}}_\X=\check{\rA}^{\vartheta,\circ}$.
\end{proof}

We now assume that $\fe$ satisfies the properties of the preceding lemma so that $\eta$ induces a $(\vartheta_\fe,\vartheta)$-equivariant isomorphism $X^\ast(\check{\rA}_\fe)\iso X^\ast(\check{\rA})$. This induces an $\kbar$-isomorphism $\psi_\fe: \rA_\fe\iso \rA$. Lemma \ref{Lem: dual involution on torus} now implies that $\psi_\fe$ intertwines the $k$-rational involutions on $\rA_\fe$ and $\rA$:
\[
\psi_\fe(\theta_\fe(t)) = \theta(\psi_\fe(t)).
\]
In particular, we see that $\psi_\fe$ induces an isomorphism 
\begin{equation}\label{eqn: iso on subtori}
    \rA_{\fe}^{+}\iso \rA^+,\qquad \text{and}\qquad \rA_{\fe}^{-}\iso \rA^-
\end{equation}
between the maximal fixed and split subtori.

The morphism $\psi_\fe$ need not be a $k$-isomorphism: for each $\sig\in \Ga$, and $s\in \rA_{\fe}(\kbar)$ we obtain
\[
\psi_\fe(\sig_\fe(s)) = \Ad(w_\sig)(\sig(\psi_\fe(s)),
\]
and the map $\sig\mapsto w_\sig$ is a $1$-cocycle valued in $W(\kbar)$. On the other hand, since the involutions $\theta_\fe$ and $\theta$ are defined over $k$, they commute with the Galois actions on the two tori. It follows that the $1$-cocycle $[w_\sig]\in Z^1(\Ga,W)$ 
preserves the isomorphisms \eqref{eqn: iso on subtori}.
This implies that $\{w_\sig\}$ is valued in the subgroup $W_1=\{w\in W: w(\rA^-)=\rA^-\}$.
\begin{Lem}
Consider the isomorphism $\psi_\fe^{-}:\rA_\fe^{-}\iso \rA^-,$ and let $\{w_\sig\}$ denote the $\Ga$-cocycle valued in $W_1(\kbar)$ as above. This descends to a cocycle valued in $W_\X$.
\end{Lem}
\begin{proof}
Recall that $\Ax\simeq \rA^-$, so that $W_\X\simeq W(\rH,\rA^-)$. But we have already seen the isomorphism \cite[Section 4]{Richardson}
\[
W(\rH,\rA^-)\simeq W_1/W_2,
\]
where $W_2=\{w\in W_1: w|_{\rA^-}\equiv Id\}$.
\end{proof}
It follows  that $\psi_\fe$ induces a $k$-rational map
\[
\rA_{\X_\fe}/W_{\X_\fe}\lra \rA_{\X}/W_{\X}.
\]
The Chevalley-style isomorphism for $\X$ and $\X_\fe$ in \eqref{eqn: Chevalley type} thus induces the morphism
\[
\fa_\fe:\X_\fe\sslash\rH_\fe\simeq\rA_{\X_\fe}/W_{\X_\fe}\lra \rA_{\X}/W_{\X}\simeq \X\sslash \rH.
\]
It is easy to check that $\fa_\fe$ is independent of all choices, completing the proof of Theorem \ref{Thm: point comparison}.\qed

\begin{Cor}\label{Cor: quotient coherence}
Suppose that $\chi_\fe:\G_\fe\sslash\G_\fe\lra \G\sslash\G$ is the map between categorical quotients of groups. The diagram
\begin{equation}\label{eqn: commuting quotients}
    \begin{tikzcd}
{[\X_\fe\sslash\G^{\theta_\fe}_\fe](k)}\ar[d,"s_\fe"]\ar[r,"{\fa_\fe}"]&{[\X\sslash\G^\theta](k)}\ar[d,"s"]\\
{[\G_\fe\sslash\G_\fe](k)}\ar[r,"{\chi_\fe}"]&{[\G\sslash\G](k)}
\end{tikzcd}
\end{equation}
commutes. Here the vertical arrows are the natural maps induced by the symmetrization maps.
\end{Cor}

A similar result -- with a simpler argument -- follows for the infinitesimal (or Lie algebra) setting. That is, if $\fg=\Lie(\G)$ denote the Lie algebra, then $\theta$ induces a $\zz/2\zz$-grading on $\fg=\fg_0\oplus \fg_1$ by setting $\fg_i:=\{X\in\fg: d\theta(X) = (-1)^iX\}$. Then the $\G^\theta$-representations
\[
T^\ast_{x_0}\X \simeq \fg_1
\]
is referred to as an infinitesimal symmetric variety for $(\G,\theta)$ at $x_0$.
\begin{Prop}
  Suppose that $\G$ is a reductive group over $k$. Let $\theta$ be a $k$-rational involution and let  $\fg_1$ be the associated infinitesimal symmetric variety. Let $\fe = (\G_\fe,\theta_\fe,\X_\fe,\ka,\eta)$ be an endoscopic datum for $(\G,\X)$. There exists a canonical $k$-rational morphism
  \[
  \fa_\fe:  \fg_{\fe,1}\sslash\G^{\theta_\fe}\lra \fg_{1}\sslash \G^{\theta}.
  \]
\end{Prop}

As noted in the introduction, this result will play a role in any analysis of relative Hitchin fibrations associated to symmetric varieties \cite{LeslieSpringer,Garcia-Prada, HameisterMorrissey} toward generalizations of Ng\^{o}'s results \cite{Ngo06,NgoFL}. We discuss the application of these results toward stabilization of relative trace formulae in Section \ref{Section: Example} and prove these results in \cite{LesliestabFJ}.

\section{A symplectic representation and the dual Hamiltonian variety}\label{Sec: symplectic rep}
Suppose that $\G$ is a quasi-split connected reductive group over $k$. Let $\X=\rH\backslash\G$ be a symmetric $k$-variety. The goal of this section is to show how, for sufficiently nice symmetric varieties, the geometric cocycle $\mu_\X$ encodes a symplectic representation of ${}^L{\X}$ so that the resulting data determines $\X$ up to $\G$-inner form. Stated more directly, we show how to incorporate the outer data $\out_\X$ on the dual side in terms of a symplectic representation of ${}^L\X$. We then explain how to interpret our notion of endoscopic varieties in the context of the relative Langlands duality conjectures and the dual Hamiltonian $\check{\G}$-variety of Ben-Zvi--Sakellaridis--Venkatesh \cite{BZSV}.
\subsection{The symplectic representation $S_\X$}\label{Section: symplectic}
We now make the assumption that $\X$ satisfies Assumption \ref{assumption: no type N} and that $\rH$ is geometrically connected. 
\begin{Rem}
While much of the results in this section can be established in much more generality than this, the statements become difficult to formulate if disconnected groups are allowed. In any case, the assumptions of \cite{BZSV} are much more stringent than these. For example, we do not generally impose that $\X$ have no type $N$ roots until Section \ref{Section: hamiltonian endoscopy}.
\end{Rem}Suppose that $\theta$ is a $k$-rational involution of $\G$ associated to $\X$ and let $(\rA,\B)$ be a $(\theta,k)$-admissible pair. We fix a choice of distinguished morphism $\varphi_\X:\check{\G}_\X\lra \hat{\G}_\X$ compatible with $\check{\rA}_\X\to \check{\rA}$ and pinning induced of $\check{\G}$. This induces $\Ga$-stable Borel subgroups for $\check{\G}$, $\hat{\G}_\X$ and $\check{\G}_\X$ denoted $\check{\B},$ 
 $\hat{\B}_\X$, and $\check{\B}_\X$, respectively. Moreover, $\hat{\B}_\X$ is $\vartheta$-stable and we have $\varphi_\X(\check{\B})\subset \hat{\B}_\X$.

Let $\D(\X)$ be the set of $\B$-colors of $\X$ and let
\[
\rho\times \varsigma:\D(\X)\lra X_\ast(\Ax) \times \mathcal{P}(\De)= X^\ast(\check{\rA}_\X)\times \mathcal{P}(\check{\De})
\]
be the cocharacter map from \eqref{eqn: color function}. Our assumption on $(\rA,\B)$ implies that $\Ga$ acts on $\De$ and hence on $\mathcal{P}(\check{\De})$. 

\quash{Recall that $\Aut_{\Omega}(\D(\X))\subset \Aut(\fX,\De_{\X}, \Omega^{(1)},\Omega^{(2)})$ is the subgroup acting trivially on all data except the fibers of $\rho\times \varsigma$ over $\Omega^{(2)}$. Then $\Aut_{\Omega}(\D(\X))$ consists of automorphisms that swap two \emph{undetermined} colors $\{D^+,D^-\}$ that lie over the same $\check{\ga}$. 
 Let $\De^{(2)}_{\X} = \{\ga\in \De_{\X}\cap\De:\rho(D^+)=\rho(D^-)\}$, and let $\D^{(2)}(\X)=(\rho\times\varsigma)^{-1}(\Omega^{(2)})$ be the undetermined colors. 
Using Proposition \ref{Prop: sym colors}, there is a corresponding subset $\check{\Sigma}^{(2)}_{\X}\subset \rho(\D(\X))$ such that for every $\al\in {\De}^{(2)}_{\X}$, $\frac{1}{2}\check{\al}\in \check{\Sigma}^{(2)}_{\X}$.
\begin{Def}
        We define $\check{\De}^{(2)}_{\X}= \{\frac{c_\al}{2}\check{\al}:\al\in {\De}^{(2)}_{\X}\}$ by re-scaling elements $\frac{1}{2}\check{\al}$ to the unique minimal multiple $\frac{c_\al}{2}\check{\al}\in \check{\fX}$.
\end{Def}}
 Recall the set $\De_{\X}^{dist}$ of distinguished roots, so that by Theorem \ref{Thm: Outer in terms of dist}, 
 \[
 \Out_{\X}(\rH) \simeq \prod_{i\in I_\X}\Res_{k_i/k}(\mu_2),
 \]
 where $I_\X$ is the set of $\Ga$-orbits in $\De_\X^{dist}$. Using Proposition \ref{Prop: sym colors}, there is a corresponding subset $\check{\Sigma}^{dist}_{\X}\subset \rho(\D(\X))$ such that for every $\al\in {\De}^{dist}_{\X}$, $\frac{1}{2}\check{\al}\in \check{\Sigma}^{dist}_{\X}$.
\begin{Def}
        We define $\check{\De}^{dist}_{\X}= \{\frac{c_\al}{2}\check{\al}:\al\in {\De}^{dist}_{\X}\}$ by re-scaling elements $\frac{1}{2}\check{\al}$ to the unique minimal multiple $\frac{c_\al}{2}\check{\al}\in \check{\fX}^{sv}$.
\end{Def}

\begin{Rem}
   Re-normalizing the spherical roots to obtain the lattice $\fX^{sv} = \fX+\zz\De^{sv}_\X$ may alter the relationship between $\rho(\D(\X))\subset \X^\ast(\Ax)$ and the representation theory of $\check{\G}_\X$.
\end{Rem}

 We define $S_\X$ as the unique representation of $\check{\G}_{\X}$ satisfying
\begin{equation*}
    S_{\X}= \bigoplus_{\check{\lam}\in \mathfrak{s}_\X} V(\check{\lam})\otimes M(\check{\lam})
\end{equation*} 
where
 \begin{enumerate}
     \item\label{axiom1} $\mathfrak{s}_{\X}=\{ \text{highest weights contained in $W_{\X}\cdot \check{\De}^{dist}_{\X}$}\},$
    \item\label{axiom2} the multiplicity space $M(\check{\lam})$ has a basis indexed by colors $D\in \D(\X)$ satisfying $\rho(D) = \frac{1}{2}\check{\al}$ for any $\check{\al}$ lying in the $W_{\X}$-orbit of $\frac{2}{c_\al}\check{\lam}$.
\end{enumerate}
We will construct an action of ${}^L{\X}$ on $S_\X$. 

Following the notation from Section \ref{Section: calculate cocycle}, there is a canonical decomposition
\[
\De_\X^{dist}=\bigsqcup_{i\in I_\X}\calo_i = \bigsqcup_{i\in I_\X}\bigsqcup_{a\in \Ga/\Ga_i}a\cdot\De_{\X_i}^{dist}
\]
and for each $i\in I_\X$, there is a canonical quadratic character 
$
    \mu_i:\Ga_i\lra \{\pm1\}.
$ 
 For any $\al\in \calo_i$, we let $\mu_\al:\Ga_\al\to \{\pm1\}$ be the corresponding character on the stabilizer $\Ga_\al = a^{-1}\Ga_i a$. 
Set $\D(\X)^{dist}\subset \D(\X)$ to be those colors $D$ such that $\rho(D)\in \check{\Sigma}_\X^{dist}$. 
\begin{Prop}\label{Prop: symplectic rep}
    Suppose that $\X=\rH\backslash\G$ is a symmetric variety satisfying Assumption \ref{assumption: no type N} with $\rH$ geometrically connected. There exists a unique ${}^L{\X}$-representation on $S_\X$ extending the algebraic action of $\check{\G}_{\X}$ such that
    \begin{enumerate}
        \item there is an isomorphism of $\Ga$-representations 
    \[
    S_\X^{(\check{\B}_{\X})}\simeq \cc[\D(\X)^{dist}],
    \]
    where for any $\al\in\De_\X^{dist}\setminus\De_\X^{(2)}$ and $\sig\in \Ga_\al$, we have $\sig\cdot D_\al=     \mu_\al(\sig)D_\al$ where $D_\al$ is the unique color satisfying $\rho(D_\al) = \frac{1}{2}\check{\al}$,
    \item there is a symplectic structure on $S_\X$ such that ${}^L\X$ acts via symplecto-morphisms.
    \end{enumerate}
\end{Prop}
The proof of this relies on the classification of symmetric varieties in a crucial way. In particular, we use the calculation of colors of symmetric varieties in Section \ref{Section: sym colors} to reduce to the case of type $C$ root systems in the Lemma \ref{Lem: minuscule}, since these are the only cases where a distinguished root occurs for symmetric varieties. In particular, we exploit a tight relationship between distinguished roots and symplectic quotients of $\check{\G}_\X$.
\begin{Rem}
    As the next lemma shows, this has the added outcome that each irreducible representation $V(\check{\lam})$ occurring in $S_\X$ is \emph{minuscule}, establishing cases of a conjecture of Ben-Zvi, Sakellaridis, and Venkatesh; see Proposition \ref{Prop: BZSV}.
\end{Rem}
The next lemma does not assume that $\pi_0(\rH)$ is trivial.
\begin{Lem}\label{Lem: minuscule}
    Suppose that $\G$ is quasi-split and $\X=\rH\backslash\G$ is a symmetric variety\quash{\textcolor{red}{(with $\rH$ connected?)}}. Fix ${\al}\in \De_\X^{dist}$ and let $\Ga_\al\subset \Ga$ denote its stabilizer; set $k_\al/k$ be the associated field extension.  Then  there exists a unique $k_\al$-rational reductive normal subgroup $\G_\al\subset \G$, stabilized by $\theta$ such that
    \begin{enumerate}
        \item\label{item 1: al} if $\rH_\al:=\rH\cap \G_\al$ and $\X_\al:=\rH_\al\backslash\G_\al$, there exist surjective morphisms $\pi_\al$, $\pi_{\al,\X}$ fitting into a commutative diagram 
        \[
        \begin{tikzcd}
    \check{\G}_\X\ar[r,"\varphi_{\X}"]\ar[d,"\pi_{\al,\X}"]&\check{\G}\ar[d,"\pi_{\al}"]\\
    \check{\G}_{\X_\al}\ar[r,"\varphi_{\X_\al}"]&\check{\G}_{\al};
\end{tikzcd}
\]
\item \label{item 2: al} let $\check{\lam}_\al\in\mathfrak{s}_{\X}$ be the unique dominant weight of $\check{\G}_\X$ associated to $\al$ as in \eqref{axiom1}. Then the $\check{\G}_\X$-action on the highest weight module $V(\check{\lam}_\al)$ factors through $\pi_{\al,\X}$.
\item Assume that $\X$ satisfies Assumption \ref{assumption: no type N}. The derived subgroup of $\check{\G}_\al$ is of type $C$. The corresponding highest weight representation of $\check{\G}_{\X}$ is minuscule and symplectic.
    \end{enumerate} 
\end{Lem}
\begin{proof}
First note that if $\G'$ is a compatible $z$-extension of $\G$, we have the morphism $p: \check{\G}_\X\lra \check{\G}'_{\X'}$ as in Remark \ref{Rem: z-extension on dual side} with an equality of spherical roots by Lemma \ref{Lem: pass to a sc cover roots}. Since $\G$ and $\G'$ have isomorphic root lattices and $\check{\fX}_{\G}\subset \check{\fX}_{\G'}$, we see that $\check{\De}_{\X'}^{sv}=\check{\De}_{\X}^{sv}$. This implies that the representation $V(\check{\lam}_\al)$ extends to $\check{\G}_{\X'}$ uniquely up to character twist. Finally, it is easy to see that if $\G'_\al\subset \G'$ is a reductive group satisfying the properties above, then $\G_\al:=p(\G'_\al)\subset \G$ also satisfies the properties. It therefore suffices to prove the claim in the case $\G=\G'$ has simply-connected derived subgroup.

Suppose $\al\in \De_\X^{dist}$.  This lemma does not rely on the $k$-rational structure of $\G$ and $\X$, so we may base change to $k_\al$ and assume that $k=k_\al$ so that $\al$ is $\Ga$-fixed.

Recall from \eqref{eqn: simple components} that there is a well-defined $k$-simple sub-root system $\Phi_\al\subset \Phi_{\X}$ containing the associated simple root. This corresponds to a $k$-simple factor $\G_{der,\al}\subset\G_{der}$. Since we have assumed that $\al$ is $\Ga$-fixed, this implies that $\Phi_\al$ is absolutely simple since otherwise $|\De_{\X_\al}^{dist}|>1$ by the proof of Theorem \ref{Thm: Outer in terms of dist}.  
Consider the commutative diagram
\[
\begin{tikzcd}
    \G_{der}\ar[r]&\G\\
    \G_{der,\al}\ar[r]\ar[u]&\G_\al\ar[u],
\end{tikzcd}
\]
where $\G_\al$ is the connected reductive group generated by $\G_{der,\al}$ and $Z(\G)^\circ$. Each morphism is normal in the sense of \cite[1.8]{KottwitzCusp}, so that if we fix compatible pinnings for each group, there is a commutative diagram of dual groups given by the front square of \eqref{diagram: cube} below.
\quash{\[
\begin{tikzcd}
    \check{\G}\ar[r]\ar[d]&\check{\G}_{der}\ar[d]\\
    \check{\G}_{\al}\ar[r]&\check{\G}_{der,\al}.
\end{tikzcd}
\]}

Noting that $\theta$ preserves $\G_\al$, we set $\rH_\al:=\rH\cap\G_\al$ and $\rH_{der,\al} = \rH\cap{\G}_{der,\al}$. Compatibly with the previous diagram, we have a commuting diagram of injective equivariant morphisms
\[
\begin{tikzcd}
    \X_{der}\ar[r]&\X\\
    \X_{der,\al}\ar[r]\ar[u]&\X_\al\ar[u],
\end{tikzcd}
\]
where $\X_\al=\rH_\al\backslash\G_\al$. Fix once and for all pinnings for the four groups and distinguished morphisms for the four varieties (with the natural notations). We claim that there is a commutative cube
\begin{equation}\label{diagram: cube}
\begin{tikzcd}[row sep=1.5em, column sep = 1.5em]
\check{\G}_\X \arrow[rr] \arrow[dr, swap,"\varphi_\X"] \arrow[dd,dashed,swap,"\check{\pi}_{\al,\X}"] &&
{\:\:\check{\G}_{\X_{der}}} \arrow[dd] \arrow[dr] \\
& \check{\G}\arrow[rr] \arrow[dd]&&
\check{\G}_{der} \arrow[dd] \\
\check{\G}_{\X_\al} \arrow[rr,] \arrow[dr, swap,"\varphi_{\X,\al}"] && {\quad\check{\G}_{\X_{der,\al}}} \arrow[dr]\\
& \check{\G}_{\al} \arrow[rr] && \check{\G}_{der,\al}
\end{tikzcd}
\end{equation}
We need only focus on the back face as every other arrow has been justified. The horizontal arrows are shown to exist in Appendix \ref{Appendix: derived subgroup}, while the right-most vertical arrow is straightforward from the decompositions
\[
\G_{der}=\prod_i\G_{der,i}\times \prod_j(\G_{der,j}\times \G_{der,j}),\qquad\X_{der} = \prod_i\X_{der,i}\times \prod_j\X_{der,j}
\]
from Section \ref{Section: reductions}. Thus we need only justify the left vertical (dashed) arrow. Its existence follows from the commutativity of the corresponding diagram of tori. That is, the image of the composition $\check{\pi}_\al\circ \varphi_\X$ has the same derived subgroup as $\G_{\X_\al}^\ast=\varphi_{\X_\al}(\check{\G}_{\X_\al})$. Since the corresponding diagram on dual tori does indeed commute, the composition must factor through $\varphi_{\X_\al}$ as indicated.

Set $\check{\lam}_\al\in\mathfrak{s}_\X$ for the unique dominant weight of $\check{\G}_\X$ associated to $\al\in \De_\X^{dist}$. Let $\rho_\al:\check{\G}_\X\to \GL(V_\al)$ be the corresponding representation. We claim that $\rho_\al$ factors through $\check{\pi}_{\X,\al}:\check{\G}_\X\to \check{\G}_{\X_\al}.$ Indeed, since the root lattices of $\G$ and $\G_{der}$ agree,
\[
\{\al\}=\De_{\X_\al}^{dist}\subset\bigsqcup_i\De_{\X_i}^{sv}\sqcup \bigsqcup_j\De_{\X_j}^{sv}=\De_\X^{sv}
\]
so that $\check{\De}_{\X_\al}^{dist}\subset\check{\De}^{sv}_{\X}$. In particular, since $\pi_{\al,\X}$ restricts to a quotient of tori $\check{\rA}_\X\to \check{\rA}_{\X_\al}$ and we have
\[
\check{\lam}_\al\in X^\ast(\check{\rA}_{\X,\al})\subset X^\ast(\check{\rA}_{\X}),
\]
so the representation $V_\al$ factors through the quotient by highest weight theory. This completes the proof of \eqref{item 1: al} and \eqref{item 2: al}.
\quash{is orthogonal to $\be\in \De_\X^{sv}\setminus \De_{\X_\al}^{sv}$. Representation theory of reductive groups implies that $\check{\G}_\X$ acts on $V(\check{\lam})$ through the quotient... \textcolor{blue}{The argument is to note that if $\check{\rA}_\X\to \check{\rA}_{\X_i}$ is the induced maps of maximal tori, then
\[
\check\lam\in X^\ast(\check{\rA}_{\X_i})\subset X^\ast(\check{\rA}_{\X}),
\]
so the representation factors through by highest weight theory.}} 

We now impose Assumption \ref{assumption: no type N}. Replacing $(\G,\X)$ with $(\G_\al,\X_{\al})$, we may thus assume that $\G_{der}$ is simply connected and absolutely simple. In particular, we may assume $\De_\X^{dist}= \{\al\}$ and rely on Lemma \ref{Lem: distinguished for sym}.


As discussed in Section \ref{Section: sym colors}, if $\al\in \De^{(2)}_\X$, then ${\De}^n_{\X}$ is of type $C_n$ for some $n$ such that $2{\al}$ is the unique long simple root. After normalizing the roots, there are two possibilities: If $\De_\X$ is of type $B_n$, then $\frac{1}{2}\check{\al}^\vee\in \check{\fX}$ lies in the $W_\X$-orbit of the unique minuscule fundamental weight of the dual group. In this case, the derived subgroup of $\check{\G}_\X$ is $\Sp_{2n}(\cc)$ for $n=\rk(\rA_\X)$ and $V(\check{\lam})$ is the standard symplectic representation. On the other hand, if $\De_\X$ is of type $C_n$, then $\check{\al}=\frac{2}{2}\check{\al}^\vee\in \check{\fX}$ lies in the $W_\X$-orbit of unique dominant short root of the dual group, which is quasi-minuscule. In this case, the derived subgroup is of type $B_n$ and $V(\check{\lam})$ is the standard orthogonal representation of $\SO_{2n+1}(\cc)$. However, this case only occurs for the symmetric varieties of Chevalley type in Lemma \ref{Lem: distinguished for sym}\footnote{The difference between $\Spin_2\times^{\mu_2}\Spin_{2n-1}\backslash\Spin_{2n+1}$ and $\Spin_2\times^{\mu_2}\Spin_{2n-2}\backslash\Spin_{2n}$ is the presence a root of type $N$ in Lemma \ref{Lem: chevalley type}. In the even case, there are no such roots, and $\al$ is the short simple root in $\Phi_\X$, implying that $\frac{\check{\al}}{2}\in \fX=\fX^{sv}$. In the odd case, renormalization of the spherical root of type $N$ ensures that $\al$ is the \emph{long simple root} of $\Phi^{sv}_\X$, so that $\frac{\check{\al}}{2}\notin\fX^{sv}$.} and therefore does not occur under Assumption \ref{assumption: no type N}.

   If $\al\in \De_{\X}^{dist}\setminus\De_\X^{(2)}$, then $\X$ must be a $k$-form of either $\Spin_{2n}\backslash\Spin_{2n+1}$ or $\Sp_{2n}\times \Sp_{2n}\\Sp_{4n}$. Both cases have $\check{\Phi}^{sv}_{\X}$ of type $C$ and $\frac{1}{2}\check{\al}^\vee\in \check{\fX}$ lies in the $W_\X$-orbit of the unique minuscule fundamental weight of the dual group.
\end{proof}

We continue to impose Assumption \ref{assumption: no type N}. Fix $\al\in \De_\X^{dist}$ as in the previous lemma. Since $(\G_\al, \X_\al)$ is $k_\al$-rational, the morphism $\check{\pi}_{\al,\X}$ is only $\Ga_\al$-equivariant with respect to the given action on $\check{\G}_\X$ and a unique $\Ga_\al$-action on $\check{\G}_{\X_\al}$. We now extend the $\Ga_\al$-action on $\check{\G}_{\X_\al}$ to the representation $V_\al$ of $\check{\G}_{\X_\al}$. We may thus assume that $k=k_\al$ and $\G=\G_\al$, thereby assuming that $\G_{der}$ is absolutely simple. In particular, $\De_\X^{dist}=\{\al\}$.

By Lemma \ref{Lem: minuscule}, $\check{\G}_{\X,der}$ is simple of type $C$ and acts on the standard representation, so that we must have $\check{\G}_{\X,der}=\Sp_{2n}(\cc)$. In this case, the $\Ga$-action on $\check{\rA}_{\X,ad}$ is trivial: since $\Sp_{2n}(\cc)$ has trivial outer automorphism group, the $\ast$-action on $\Phi_{\X}$ is trivial. In particular, the $\Ga$-action on $\check{\G}_\X$ preserves our fixed pinning $\{x_{\check{\ga}}\}_{\check{\ga}\in \check{\De}^{sv}_{\X}}$ induced by $\check{\B}_\X$ up to sign, so that the action is completely determined by a unique set of characters $\chi_{\ga}:\Ga_\ga\lra \{\pm1\}$ such that $\sig x_{\check{\ga}} = \chi_\ga(\sig)x_{\check{\ga}}$  for $\sig\in \Ga_\ga = \mathrm{Stab}_\Ga(\check{\ga})$ and $\check{\ga}\in \check{\De}^{sv}_{\X}$. Since pinnings of $\check{\G}_\X$ with respect to $\check{\rA}_\X$ all lie in a single (free) $\check{\rA}_{\X,ad}$-orbit, these characters combine to give a single character $\chi:\Ga\to \check{\rA}_{\X,ad}$.
 


\begin{Lem}
Let  $\chi:\Ga\to \check{\rA}_{\X,ad}$ denote the character uniquely determined by 
\[
\sig x_{\check{\ga}} = \chi_\ga(\sig)x_{\check{\ga}} = \Ad(\chi(\sig))x_{\check{\ga}}
\]
for all $\check{\ga}\in \check{\De}^{sv}_{\X}$. For any quadratic character $\ep:\Ga\lra \{\pm1\}$, there exists a natural lift $$\widetilde{\chi}_\ep\in \Hom_{cts}(\Ga,\check{\rA}_{\X}[2])\simeq H^1(\Ga,\check{\rA}_{\X}[2]).$$
\end{Lem}
\begin{proof}
Consider the map $\check{\rA}_\X[2]\to \check{\rA}_{\X,ad}[2]$, which is not surjective. We first claim that for any choice of non-standard action, $\chi(\Ga)$ lies in the image of $\check{\rA}_\X[2]\to \check{\rA}_{\X,ad}[2]$. Considering the simple root $\check{\al}$ of $\check{\G}_\X$ associated to $\al\in \De_\X^{dist}$, we must have $\chi_{{\al}}\equiv 1$  as $\al$ is a root of type $A_1$ \cite[Lemma 10.4]{KnopSchalke}. That is, for any $\sig\in \Ga_\al$, $\Ad(\chi(\sig))x_{\check{\al}} = x_{\check{\al}}.$ This implies that the composition
\[
\Ga\overset{\chi}{\lra} \check{\rA}_{\X,ad}[2]\simeq X_\ast(\check{\rA}_{\X,ad})/2X_\ast(\check{\rA}_{\X,ad})\to X_\ast(\check{\rA}_{\X,ad})/X_\ast(\check{\rA}_{\X})\simeq \la{\omega}_{\al}\ra,
\]
where ${\omega}_{\check\al}$ is the fundamental coweight associated to ${\al}$, is trivial. This forces $\chi$ to take value in 
\[
X_\ast(\check{\rA}_{\X})/[X_\ast(\check{\rA}_{\X})\cap2X_\ast(\check{\rA}_{\X,ad})],
\]
implying our first claim. 

We now construct a lift. Let $V_\al$ denote the representation of $\check{\G}_\X$ from Lemma \ref{Lem: minuscule}, and let $l_{\B}\subset V$ be the unique $\check{\B}_\X$-stable line in $V_\al$. For any fixed quadratic character $\ep:\Ga\lra \{\pm1\}$, choose the unique lift 
\[
\widetilde{\chi}_\ep: \Ga\lra \check{\rA}_\X[2]
\]
of $\chi:\Ga\to \check{\rA}_{\X,ad}[2]$ satisfying that $\widetilde{\chi}_\ep(\sig) v_{\B} = \ep(\sig)v_{\B}$ for any $v_{\B}\in l_{\B}$. This gives a lift.
\end{proof}
\quash{ Using $\Hom(\Ga_\ga, \{\pm1\})\simeq  H^1(\Ga_\al,\{\pm1\})$ and Shapiro's lemma, the characters $(\chi_\al)$ determine a class in
 \[
 \bigoplus_{\al\in \De_\X/\Ga}H^1(\Ga_\al,\{\pm1\})\simeq \bigoplus_{\al\in \De_\X/\Ga}H^1(\Ga,\check{\rA}_{\calo}[2]])=H^1(\Ga,\check{\rA}_{\X,ad}[2])
 \]
  Fix an orbit $\calo$.  
  Let $\check{\rA}_{\calo}=\Ind_{\Ga/\Ga_\al}(\check{\rA}_\al)$ denote the quotient of $\check{\rA}_{\X}$ that gives a maximal torus for $\check{\G}_{\X,\calo}$.
Incorporating the $k$-rational structure, we use}

An important feature of this is that for any choice of $\ep$, the lift $\widetilde{\chi}_\ep$ induces an action of $\Ga$ on $V$. 

Continuing with the assumptions of the lemma, let $\check{\lam}_\al\in X^\ast(\check{\rA}_\X)$ be the dominant weight associated to $\al\in \De_\X^{dist}$. We define $\D(\al)\subset \D(\X)$ to be those colors satisfying $\rho(D) = \frac{1}{2}\check{\al}$. We define the $\Ga$-representation
\begin{equation}\label{eqn: perm rep on color}
    M(\al):= \cc[\D(\al)]
\end{equation}
where 
\begin{enumerate}
    \item if $\al\in \De_\X^{(2)}$, $M(\al)$ is determined by the $\Ga$-action on the basis $\D(\al)$. In this case, we take $\ep\equiv 1$ in the Lemma and use $\widetilde{\chi}_1$ to give a $\check{\G}_{\X}\rtimes \Ga$-action on $V_\al\otimes_\cc M(\al)$ by letting $\sig\in \Ga$ act by
    \[
    \sig(v\otimes m) = \widetilde{\chi}_{1}(\sig)v\otimes \sig(m).
    \]
    In this case, the representation acts on $V_\al\otimes M(\al)\simeq V_\al\oplus V_\al$ through 
    \[
    \check{\G}_{\X}\rtimes \Ga\to \Sp(V_\al)\times \mathrm{O}(M(\al))\subset \Sp(V_\al\otimes M(\al)).
    \]
    \item If $\al\in \De_\X^{dist}\setminus{\De_\X^{(2)}}$, then $|\D(\al)|=1$ so that $M(\al) =\cc$. In this case, we take $\ep=\mu_\al$ as in Lemma \ref{Lem: outer characters}, so that $\check{\G}_{\X}\rtimes \Ga$ acts on $V_\al\otimes_\cc M(\al)=V_\al$ by letting $\sig\in \Ga$ act by
    \[
    \sig(v) = \widetilde{\chi}_{\mu_\al}(\sig)v.
    \]
    This representation is clearly symplectic.
\end{enumerate}

 We now return to the general setting of $(\G,\X)$ satisfying Assumption \ref{assumption: no type N} and $\rH$ connected. For each  $\Ga$-orbit  $\calo\subset \De_\X^{dist}$ of distinguished roots and $\al\in \calo$, let $\check{\G}_{\X,\calo}$ be the induced group (cf. \cite[1.4,1.5]{BorelAutomorphic}
\[
 \check{\G}_{\X,\calo} = \Ind_{\Ga/\Ga_\al}(\check{\G}_{\X_\al})\simeq \prod_{\Ga/\Ga_\al}\check{\G}_{\X_\al}
\]
 such that $\check{\G}_{\X,\al,der}=[\check{\G}_{\X,\al},\check{\G}_{\X,\al}] = \Sp(V_\al)$, where $V_\al=V(\check{\lam}_{\al})$ is the vector space of the associated representation.  It follows from Lemma \ref{Lem: minuscule} that there exists a $\Ga$-equivariant quotient map 
\begin{equation}\label{eqn: distinguished quotient}
    \check{\G}_\X\overset{p}{\lra} \prod_{\calo\subset\De_{\X}^{dist} }\check{\G}_{\X,\calo}.
\end{equation}

For each $\Ga$-orbit $\calo\subset \De_\X^{dist}$, we thus obtain a $\check{\G}_{\X,\calo}\rtimes\Ga$-representation
\begin{equation}\label{eqn: calo rep}
    S_\calo:=\Ind_{\Ga_\al}^\Ga(V_\al\otimes M(\al)).
\end{equation}
\quash{To see what this representation is, fix representatives $[\tau_\be]\in \Ga_\al\backslash\Ga$ satisfying $\tau_\be(\be) = \al$ for each $\be\in \calo$. A vector $f\in V_\calo$ is thus a function $f:\Ga\lra V_\al$ satisfying
\[
f(\ga\lam) = \mu_\al(\ga)\ga\cdot f(\lam), \quad\text{ for any }\ga\in \Ga_\al,\:\lam\in \Ga.
\]For $ (g,\sig)=((g_\be)_\be, \sig)\in \check{\G}_{\X,\calo}\rtimes_\varepsilon \Ga$, one defines $[(g,\sig)\cdot f]$ as the unique vector satisfying that  $[(g,\sig)\cdot f](\ga) = [g\cdot f](\ga\sig)$ and for each $\be\in\calo$
\[
[(g,\sig)\cdot f](\tau_\be) = \tau_\be(g)\cdot f(\tau_\be).
\]
It is easy to check that this gives a representation of $\check{\G}_{\X,\calo}\rtimes_\varepsilon \Ga$.}
More canonically, there exists a $\Ga$-equivariant symplectic local system $$\mathcal{V}_\calo\lra \calo = \Res_{k_\al/k}(\Spec(k_\al)),$$ where $\Ga_\al=\Gal(\kbar/k_i)$ and we have a $\Ga$-representation $V_\calo =H^0(\calo,\mathcal{V}_\calo)$ (evidently, this construction is called \emph{geometric induction} for an action groupoid \cite{AubertGI}.) The group $\check{\G}_{\X,\calo}$ is then the automorphism group of $\mathcal{V}_\calo$, which naturally acts on $V_\calo$ stalk-wise. This action extends to $\check{\G}_{\X,\calo}\rtimes \Ga$. 


To complete the proof of Theorem \ref{Prop: symplectic rep}, note that the ${}^L\X$-representation on
\[
S_\X=\bigoplus_{\calo\subset \De_\X^{dist}}S_\calo
\]
obtained via pulling back along 
\[
\check{\G}_\X\rtimes\Ga\overset{p\otimes \De}{\lra}\prod_\calo(\check{\G}_{\X,\calo}\rtimes\Ga)
\]gives the sought-for extension. 


\subsubsection{Application to rationality}

The theorem below compares two $k$-forms $\X$ and $\X'$ of a given symmetric $\G_{\kbar}$-variety $\overline{\X}$. For ease of notation, we say that $\X$ and $\X'$ normally related to $(\rA,\B)$ if the rational involutions $\theta$ and $\theta'$ associated to $\X$ and $\X'$ are both are normally related to $(\rA,\B)$. Since we may conjugate $\theta$ and $\theta'$ by elements of $\G(k)$ to ensure this, it leads to no loss in generality.

\begin{Thm}\label{Thm: unique}
    Suppose that $\G$ is quasi-split over $k$ and suppose that $\overline{\X}=\overline{\rH}\backslash\G_{\kbar}$ is a symmetric $\G_{\kbar}$-variety satisfying Assumption \ref{assumption: no type N} and that $\overline{\rH}$ is connected. Consider two $k$-rational $\G$-forms $\X=\rH\backslash\G$ and $\X'=\rH'\backslash\G$ of $\overline{\X}$. Assume that $\X$ and $\X'$ are both normally related to $(\rA,\B)$.
    \begin{enumerate}
        \item  We have $\hat{\G}_{\X}=\hat{\G}_{\X'}$.
        \item Given a pair of distinguished morphisms $\varphi_\X$ and $\varphi_{\X'}$, there is a canonical isomorphism $f_\X:{}^L\X\simeq {}^L{\X'}$ inducing a commutative diagram
        \[
 \begin{tikzcd}
{}^L{\X}\ar[rr,"f_\X"]\ar[rd,swap,"\varphi_{\X}"]&& {}^L{\X'}\ar[ld,"\varphi_{\X'}"]\\
 & {}^L{\G}.&
 \end{tikzcd}
\]
    \end{enumerate} 
    Suppose further there exists an $f_\X$-equivariant isomorphism $f_S:S_{\X}\lra S_{\X'}$. Then $\X$ and $\X'$ are $\G$-inner forms. If $H^1(k,\mathcal{A}_\X^\flat)=0$, then $\X$ and $\X'$ are $k$-isomorphic. 
\end{Thm}

\begin{proof}
Let $\theta$ and $\theta'$ be associated involutions for $\X=\rH\backslash\G$, $\X'={\rH'}\backslash\G$, respectively. Replacing $\theta$ and $\theta'$ by $k$-rational conjugates if necessary, we may assume that both are normally related to $(\rA,\B)$. Since we have assumed that there is a $\G_{\kbar}$-equivariant isomorphism $\X_{\kbar}\simeq \X'_{\kbar}$, this implies  
\[
\De_\X=\De_{\X'},\quad \hat{\De}_\X=\hat{\De}_{\X'},\: \text{ and }\: \theta|_\rA=\theta'|_\rA.
\] Moreover, $\rA_\X\simeq\rA_{\X'}$.

The construction of the associated group relies only on the inclusion of root system generated by the coroots of the associated roots of $\X$ into the root system of $\check{\G}$. Thus, we obtain a canonical identification $\hat{\G}_\X= \hat{\G}_{\X'}$. Similarly, the construction of $\check{\G}_\X$ and the pinned $\Ga$-action defined in terms of \eqref{eqn: root embedding} are combinatorial and depend only on the $k$-group structure of $\G$, so that there exists $f_\X: {}^L{\X}\simeq {}^L{\X'}$. The definition of $\vartheta$ (resp., $\vartheta'$) is determined by $(\rA,\B)$ and by the involution $\theta|_\rA=\theta'|_{\rA}$ (cf. Lemma \ref{Lem: dual involution on torus} and Proposition \ref{Prop: dual involution}).  By construction, we find that $\vartheta=\vartheta'$ so that $\check{\G}_{\X}^\ast=\check{\G}_{\X'}^\ast$. Since distinguished morphisms have a finite kernel determined by the kernel of $\check{\rA}_\X\to \check{\rA}$, we see that for any choice of such distinguished morphisms there exists a unique $f_\X$ such that
        \[
 \begin{tikzcd}
{}^L{{\X}}\ar[rr,"f_\X"]\ar[rd,swap,"\varphi_{\X}"]&& {}^L{\X'}\ar[ld,"\varphi_{\X'}"]\\
 &{}^L\hat{\G}_\X&
 \end{tikzcd}
\]
commutes.

 \quash{$\Ga$-representations
 \[
 f_S|_{\Ga}:S_{\X}^{(\check{\B}_\X)}\iso S_{\X'}^{(\check{\B}_{\X'})};
 \] w}
 Assume now we are given an $f_\X$-equivariant isomorphism $f_S:S_{\X}\lra S_{\X'}$. In particular, this restricts to an isomorphism of representations of algebraic groups  (Proposition \ref{Prop: symplectic rep})
  \begin{equation}\label{eqn: each factor}
    \bigoplus_{\al\in \check{\De}_\X^{dist}} V(\check{\lam}_\al)\otimes M(\al)\simeq  \bigoplus_{\al'\in \check{\De}_{\X'}^{dist}} V(\check{\lam}_{\al'})\otimes M(\al'),
\end{equation} 
which induces a bijection $\De_{\X}^{dist}\iso \De_{\X'}^{dist}$ by highest weight theory; this is uniquely determined by the isomorphism $f_\X$. This is $\Ga$-equivariant and we obtain an identification of Galois orbits $\Ga\cdot \ga_i\mapsto \Ga\cdot \ga'_i$. Recall that there is an isomorphism of $\Ga$-modules
   \[
   S^{(\check{\B}_{\X})}_\X\simeq \bigoplus_{\calo_i\subset\De_\X^{dist}}\Ind_{\Ga_i}^\Ga(M(\al_i)).
   \]
 To recover the character $\mu_i$, not that the $\Ga_i$-action on $M(\al_i)$ is uniquely determined by this character and there is a canonical $\Ga_i$-equivariant morphism
\begin{align*}
\Ind_{\Ga_i}^\Ga(M(\al_i))&\lra M(\al_i)\\
[f:\Ga\to M(\al_i) ]&\longmapsto f(1),
\end{align*}
from which we may compute $\mu_i$. Now passing to highest weight spaces in \eqref{eqn: each factor}, we obtain a $\Ga_i=\Ga_i'$-equivariant isomorphism $M(\al_i)\iso M(\al_i')$. In particular, $\mu_i=\mu_i'$.
\quash{an isomorphism
 \begin{equation}\label{eqn: each factor}
      \Ind_{\Ga_i}^\Ga(V_i\otimes M(\check{\lam}_i))\simeq\Ind_{\Ga_i'}^\Ga(V_i'\otimes M(\check{\lam}_i')).
 \end{equation}
 We claim that this collection of isomorphisms induces a bijection 
$
\mathrm{out}_{\X} =\mathrm{out}_{\X'}.$ Note that this suffices to prove the theorem by Lemma \ref{Lem: outer characters}.
}

All together, we see that $\out_\X=\out_{\X'}$, which implies that $\X$ and $\X'$ are inner forms. The final claim is now a direct consequence of Lemma \ref{Lem: outer forms classify}.
\end{proof}

 \quash{provided uniqueness theorem \cite{Losev} now implies there  exists $g\in \G(\kbar)$ such that $\Ad(g): \X\lra \X'$ is a $\kbar$-isomorphism. This induces a cocycle
 \[
 [\sig\mapsto g^{-1} {}^\sig g]\in H^1(k,N_{\G}(\G^\theta)).
 \]
 We claim that this induces the trivial class under $\mathrm{out}^{\G}_{\G^\theta}$ in Lemma \ref{Lem: factors through to color}.
On the other hand, for each 
  Lemma \ref{Lem: factors through to color} thus implies that we obtain a non-trivial class
\[
d(\theta,\theta')\in H^1(k,\Out(\rH)).
\]}
\begin{Rem}
In general, one can only recover $\X$ up to $\G$-inner twist from the data of the dual group and $S_\X$ (c.f Example \ref{Ex: relative inner forms}). On the other hand, we do obtain uniqueness in the following cases:
\begin{enumerate}
\item If $\X$ is automorphism free, then the $k$-rational form $\X$ is uniquely determined up to isomorphism by the data of ${}^L\X\to {}^L\G$ and $S_\X$.
\item More generally, Lemma \ref{Lem: center in flat in sharp} implies that if $\rH$ is spherically closed (equivalently, $\mathcal{A}_\X^\sharp=\{1\}$), then $\X$ uniquely determined by the data of ${}^L\X\to {}^L\G$ and $S_\X$, up to isomorphism.
\item The weaker property that $Z(\G)\subset \rH$ also suffices, so that uniqueness holds for all inner symmetric spaces (including all hermitian symmetric varieties).
\end{enumerate}
\end{Rem}

\subsection{Relation to endoscopy}\label{Section: hamiltonian endoscopy}
We now reinterpret the notion of endoscopic data in terms of the symplectic representation $S_\X$. For simplicity, we will now impose the following conditions:
\begin{enumerate}
    \item $\G_{der}$ is simply connected,
    \item $\X$ is \emph{excellent} (cf. Definition \ref{Def: excellent}).
\end{enumerate}
We note that this forces $\X$ to have no spherical roots of type $N$ and that $\rH$ is geometrically connected (cf. Lemma \ref{Lem: simply connected}). In particular, $\X$ satisfies Assumption \ref{assumption: no type N} and any distinguished morphism $\varphi_\X$ is injective.

Recall that Lemma \ref{Lem: minuscule} gives a $\Ga$-equivariant quotient map 
\begin{equation*}
    \check{\G}_\X\overset{p}{\lra} \prod_{\calo\subset\De_{\X}^{dist} }\check{\G}_{\X,\calo};
\end{equation*}
here, for each orbit $\calo\subset \De_\X^{dist}$ and $\al\in \calo$, there is a unique isomorphism such that
\[
 \check{\G}_{\X,\calo} = \Ind_{\Ga/\Ga_\al}(\check{\G}_{\X_\al})\simeq \prod_{\Ga/\Ga_\al}\check{\G}_{\X_\al}
\]
 such that $\check{\G}_{\X,\al,der}=[\check{\G}_{\X,\al},\check{\G}_{\X,\al}] = \Sp(V_\al)$. 

Suppose now that $\fe=(\G_\fe,\theta_\fe,\X_\fe,\ka,{}^L\eta)$ is an extended quasi-split endoscopic datum (cf. Section \ref{Section: extended data}). Clearly $\X_\fe$ is also excellent since the kernel of $\rA_\fe\to{\rA}_{\X_\fe}$ is a $k$-form of the kernel of $\rA\to {\rA}_{\X}$. Thus, $\X_\fe$ possesses no roots of type $N$, is well adapted, and Conjecture \ref{Conj: cohom surj} holds for $\X_\fe$. In particular, we obtain a ${}^L{\X_\fe}$-representation $S_{\X_\fe}$ from Proposition \ref{Prop: symplectic rep}. 

Under the identification $\eta(\check{\G}_{\X_\fe})=\check{\G}_{\X,\ka}$, we obtain a commutative diagram
\[
\begin{tikzcd}
    \check{\G}_{\X_\fe}\ar[d,"p"]\ar[r,"\eta"]&\check{\G}_{\X}\ar[d,"p"]\\
    \prod_\calo\check{\G}_{\X,\calo,\ka}\ar[r]&\prod_\calo\check{\G}_{\X,\calo}.
\end{tikzcd}
\]
Here, the lower left-hand entry is the image $p(\check{\G}_{\X,\ka})$; in particular, for each $\calo\subset \De_\X^{dist}$, we obtain a subgroup $\check{\G}_{\X,\calo,\ka}\subset\check{\G}_{\X,\calo}.$ For each $\calo\subset \De_\X^{dist}$, recall from \eqref{eqn: calo rep} that we have the $\check{G}_{\X,\calo}\rtimes \Ga$-representation $S_\calo$. When $\calo\subset \De_\X^{(2)}$, this representation decomposes into two irreducible factors
\[
S_\calo = S_{\calo,1}\oplus S_{\calo,\mu_i}
\]
according to the $\Ga_i$-equivariant decomposition $M(\al_i) = \cc\oplus \cc(\mu_i)$; when $\calo\in \De_\X^{dist}\setminus{\De_\X^{(2)}}$, $S_\calo$ is irreducible.
\begin{Lem}
    The restriction of each irreducible factor of $S_\calo$ to $\check{\G}_{\X,\calo,\ka}\rtimes \Ga$ remains irreducible.
\end{Lem} 
\begin{proof}
In fact, since each factor restricts to the irreducible $\check{G}_{\X,\calo}$-representation $S_\calo =\prod_{\al\in \calo}S_\al$, it suffices to show that it remains irreducible upon restriction to $\check{\G}_{\X,\calo,\ka}$; that is, we may ignore the $\Ga$-action. In this geometric setting, each irreducible factor of $S_\calo$ is isomorphic to $\bigoplus_{\al\in \calo}V_\al$.

\quash{
Appealing to 
Recalling that the derived subgroup of $\check{\G}_{\X,\calo}$ is $ \prod_\al \Sp(V_\al)$, the derived subgroup of $\check{\G}_{\X,\calo,ka}$ is of the form $\prod_\al \Sp(V_\al)_{x_\al}\subset \prod_\al \Sp(V_\al)$, where $x_\al$ is the appropriate semi-simple element of the associated group $\hat{\G}_{\X}$;}

Moreover, the representation restricts irreducibly to $\check{\G}_{\X,\calo,\ka}$ if and only if it restricts irreducibly to $\check{\G}_{\X,\calo,\ka}\cap\check{\G}_{\X,\calo,der}$. We thus assume that $\check{\G}_{\X,\calo}=\check{\G}_{\X,\calo,der}= \prod_\al \Sp(V_\al)$ and consider the restriction of $S_\calo$ to
\[
\check{\G}_{\X,\calo,\ka}\cap\check{\G}_{\X,\calo,der} =: \prod_\al \Sp(V_\al)_{\ka_\al},
\]
where $\Sp(V_\al)_{\ka_\al}\subset \Sp(V_\al)$ denotes the corresponding subgroup.

We claim that for each $\al\in \calo$, the subgroup $\Sp(V_\al)_{\ka_\al}\subset \Sp(V_\al)$ is the stabilizer of a semi-simple element $\ka_\al$ of a symmetric variety of the form $\Sp(V_\al)\backslash\GL(V_\al)$, up to central torus. Indeed, for any $\al\in \calo$ the projection $\check{\G} \check{\G}_\al$ from Lemma \ref{Lem: minuscule} induces a commutative diagram
\[
\begin{tikzcd}
    \hat{\G}_\X\ar[r,"\check{s}"]\ar[d]&\hat{\X}\ar[d]\\
        \hat{\G}_{\X_\al}\ar[r,"\check{s}_\al"]&\hat{\X}_\al,
\end{tikzcd}
\] where the notations on the bottom row have the obvious meaning. This implies that if $x=\eta(\hat{s}_\fe(\ka))\in \hat{\X}$, then $\ka_\al\in\hat{\X}_\al$ and up to center, $\Sp(V_\al)_{\ka_\al}$ is the stabilizer of $\ka_\al$. Appealing to Lemma \ref{Lem: distinguished for sym}, the symmetric space $\hat{\X}_{\al,der}$ of the form $\Sp(V_\al)\backslash\SL(V_\al)$\footnote{This relies on Assumption \ref{assumption: no type N}, since $(\Sp_{2n},\GL_n)$ gives a counterexample.}. A simple calculation implies that $\Sp(V_\al)_{\ka_\al}$ respects the orthogonal decomposition
\[
V_\al=\bigoplus_{a}V_a,
\]
where $\{a\}\subset \cc$ are the roots of the minimal polynomial of $\ka_\al$, and  $V_a$ is the corresponding eigenspace, so that $\ka_\al|_{V_i} = aI_{V_a}$. Then $\Sp(V_\al)_{\ka_\al} = \prod_{a}\Sp(V_a)$, and the restriction is irreducible.
\end{proof}

In particular, the dual group $\check{\G}_{\X_\fe}$ possesses two symplectic representations $S_\X$ and $S_{\X_\fe}$ via the two quotients
\[
\begin{tikzcd}
   & \check{\G}_{\X_\fe}\ar[dr,"p_\fe"]\ar[dl,swap,"p"]&\\
    \prod_{\calo\subset \De_\X^{dist}}\check{\G}_{\X,\calo,\ka}\ar[rr,dashed]&&\prod_{\mathcal{U}\subset \De_{\X_\fe}^{dist}}\check{\G}_{\X_\fe,\mathcal{U}}.
\end{tikzcd}
\]
Verifying the existence of the dotted arrow is now an exercise in highest weight theory. It is compatible with the corresponding map $$\prod_{\calo\subset \De_\X^{dist}}\Res_{k_i/k}(\mu_2)\simeq \Aut_d(\X)\overset{\mathrm{dist}_\fe}{\lra}{\Aut_d(\X_\fe)}\simeq \prod_{\mathcal{U}\subset \De_{\X_\fe}^{dist}}\Res_{k_i/k}(\mu_2)$$ constructed in the proof of Theorem \ref{Thm: exists}. Note that $\mathrm{dist}_\fe$ need not be surjective. 

Since $S_\X$ is the standard minuscule representation of the derived subgroup of $\prod_{\calo\subset \De_\X^{dist}}\check{\G}_{\X,\calo,\ka}$, and likewise for $\prod_{\mathcal{U}\subset \De_{\X_\fe}^{dist}}\check{\G}_{\X_\fe,\mathcal{U}}$, the existence of the dashed morphism of algebraic groups implies the existence of a a $\check{\G}_{\X_\fe}$-linear injective homomorphism
    \[
    \rho_\fe: S_{\X}\lra S_{\X_\fe}.
    \]

\begin{Cor}\label{Cor: induced map on reps}
    Suppose that $\X=\rH\backslash\G$ is an excellent symmetric $\G$-variety and $\fe$ is a quasi-split endoscopic datum associate to the $\G_\fe$-variety $\X_\fe=\rH_\fe\backslash\G_\fe$. There is a ${}^L{\X_\fe}$-linear injection of representations
    \[
    \rho_\fe: S_{\X^\fe}|_{{}^L{\X_\fe}}\lra S_{\X}.
    \]
\end{Cor}
\begin{Ex}
    Note that the map need not be surjective. For example, if $\G=\GL_{2n+1}$ and $\X= \GL_{2n+1}/\GL_{n}\times \GL_{n+1}$, then $\Aut_d(\X)=\{1\}$. This variety has 
    \[
    \mathrm{L}_\X=\rA,\qquad\hat{\G}_\X=\GL_{2n}\times \Gm,\text{ and } \hat{\X} = \GL_{2n}\times \Gm/[\Sp_{2n}\times \{1\}].
    \]
    An easy exercise shows that for any $a+b=n$ with $a,b>0$, $\G_{\fe}=\GL_{2a}\times \GL_{2b+1}$ is an endoscopic group with endoscopic variety $\X_\fe = \GL_{2a}/[\GL_a\times \GL_a]\times \GL_{2b+1}/[\GL_b\times \GL_{b+1}]$. Here $\Aut_d(\X_\fe)\simeq\mu_2$.
\end{Ex}

\subsection{The dual Hamiltonian variety and Lagrangian correspondences}\label{Section: BZSV conj}
We may now remark on the relationship between the results of the previous section on recent work and conjectures of Ben-Zvi, Sakellaridis, and Venkatesh \cite{BZSV}.

\quash{ First we recall the following conjecture of Sakellaridis.
\begin{Conj}
Suppose that $k$ is a local field, $\G$ a connected reductive group over $k$, and let $\X$ be an excellent spherical $\G$-variety.
 There exists a $\zz$-graded finite dimensional representation $V_{\X}=\bigoplus_dV_\X^d$ of the spherical ``$L$-group'' ${}^L{\X}=\check{\G}_\X\rtimes W_k$ such that if 
\begin{enumerate}
    \item $\Phi^0$ is the ``IC function'' associated to $\X(\calo_k)$,
    \item $\pi=\iota_\ast(\sig)$ with both representations unramified principal series,
    \item $v_0\in \pi$ is a $\G(\calo_k)$-fixed vector normalized so that $||v_0||^2=1$,
\end{enumerate}
then one has
\[
\al_{\Phi^0,\Phi^0,\pi}(v_0,v_0) = L_{\X,k}^\sharp(1/2,\sig):=\De(0)\frac{L_{\X,k}(1/2,\sig)}{L(1,\sig,\Ad_{\check{\G}_{\X}})},
\]
where $\De(s)$ is a product of local zeta factors depending only on $\X$, and 
\[
L_{\X,k}(s,\sig)= \prod_dL(s+\frac{d-1}{2},\sig,V_\X^d).
\]
\end{Conj}

\begin{Rem}
    In many cases, this is essentially proved. Indeed under certain rationality assumptions on both $\G$ and $\X$, Sakellaridis established this conjecture for $\X=\rH\backslash\G$ affine (equivalently, $\rH$ reductive) in \cite{Sakellaridisspherical} up to obtaining a $\zz$-graded representation (he produces a virtual representation of $\check{\G}_{\X}$). More recently, Sakellaridis and Wang established the majority of this conjecture when $\mathrm{char}(k)>0$ for the case where $\check{\G}_{\X} = \check{\G}$ and $\X$ has no spherical roots of type $N$.
\end{Rem}

The 
}

We continue to assume that $\X=\rH\backslash\G$ is excellent, so we may regard $\check{\G}_\X=\check{\G}_\X^\ast$ as a subgroup of $\check{\G}$. Let $\rho_{\X}: \SL_2(\cc)\lra \check{\mathrm{L}}_{\X}$ be the principal $\SL_2(\cc)$ morphism determined by a fixed pinning of $\check{\G}$. This gives rise to an $\fsl(2)$-triple $(e,h,f)$ in $\check{\mathfrak{l}}_{\X}$
\[
e= d\rho_{\X}\begin{psmatrix}
    0&1\\0&0
\end{psmatrix},\quad h= d\rho_{\X}\begin{psmatrix}
    1&0\\0&-1
\end{psmatrix},\quad f= d\rho_{\X}\begin{psmatrix}
    0&0\\1&0
\end{psmatrix}.
\]
The element $h$ induces a grading $\check{\fg}= \bigoplus_{i\in\zz}\check{g}_i$; let $\check{\U}_{\X}$ to be the unipotent group whose Lie algebra is $\check{\fu}_{\X}:=\bigoplus_{i\geq 1}\check{g}_i$. By construction, the nilpotent orbit $\calo_e =\check{\G}\cdot e$ is induced from the principal orbit of the Lie subalgebra $\check{\fl}_{\X}$. It may happen that $\check{\fg}_1\neq0$; set $\check{\fu}_\X^+ = \bigoplus_{i\geq2}\check{\fg}_i$ and let $\check{\U}_\X^+$ denote the associated unipotent group. Using the Killing form $\ka(-,-)$, we obtain a character $\psi_{\X}: \check{\U}^+_{\X}\lra \mathbb{G}_a$ via
\[
\psi_{\X}(\exp(u)) = \ka(f,u);
\]
this gives an element in $\check{\fu}_{\X}^\ast$. Let $\check{\mathrm{R}}_\X= Z_{\check{\G}}(e,f,h)$ denote the centralizer of the $\fsl(2)$-triple. It is well known that the vector space $\check{\fu}_\X/\check{\fu}^+_\X$ carries a $\check{\mathrm{R}}_\X$-invariant symplectic form
\[
\ka(v,w):=\ka(f,[v,w]).
\]
Noting that $\check{\G}_\X\subset \check{\mathrm{R}}_\X$, we may thus view this as a Hamiltonian $\check{\G}_\X\check{\U}_\X$-space, where $\check{\U}_\X$ acts additively via the isomorphism $\check{\U}_\X/\check{\U}_\X^+\simeq\check{\fu}_\X/\check{\fu}^+_\X$ and the moment map 
\[
\mu_{\check{\U}_\X}:\check{\U}_\X/\check{\U}_\X^+\lra \check{\fu}_\X^\ast
\]
is shifted by $\iota_\X$; that is $\mu_{\check{\U}_\X}(u) = \ka(f,[u,-])+\iota_\X$. Finally, \cite[Section 4.3]{BZSV} defines a representation $S_\X$ of $\check{\G}_\X$, which they conjecture to be symplectic, giving rise to a moment map
\[
\mu_\X:S_\X\to \check{\fg}_\X^\ast.
\]

With this preparation, Ben-Zvi, Sakellaridis, and Venkatesh consider the Hamiltonian $\check{\G}$-space
\[
\check{\mathcal{M}}:=[S_\X\times (\check{\fu}_\X/\check{\fu}^+_\X)]\times_{(\check{\fg}_\X+ \check{\fu}_\X)^\ast}^{\check{\G}_\X\check{\U}_\X}T^\ast\check{\G},
\] 
equipped with a moment map $\mu:\check{\mathcal{M}}\to \check{\fg}^\ast$ induced by the moment map of $T^\ast\check{\G}$. More precisely, note that we also have
\[\check{\mathcal{M}}\simeq [S_\X\times (\check{\fu}_\X/\check{\fu}^+_\X)\times_{(\check{\fg}_\X+ \check{\fu}_\X)^\ast}\check{\fg}^\ast]\times^{\check{\G}_\X\check{U}_\X}\check{\G},
\] since $T^\ast\check{\G} = \check{\fg}^\ast\times \check{\G}$. Then
\[
\mu(([v,u],X),g) = X.
\]

Now suppose that $\X$ is an excellent symmetric $\G$-variety. We claim that the symplectic representation $S_\X$ constructed in Section \ref{Section: symplectic} gives precisely the representation conjectured in \cite{BZSV}. In fact, by inspection of Section 4.1 of \cite{BZSV}, their definition of $S_\X$ is intimately related to the notion of distinguished spherical roots of \cite{Losev}. Indeed, recall the notation $\cald(\X)^{dist}\subset \cald(\X)$ from Section \ref{Section: symplectic}, giving the subset of colors $D$ of $\X$ such that $\rho(D)$ is dual to $2\al$ for some $\al\in \De_\X^{dist}$. These are precisely the \emph{colors of even sphere type} \cite[Section 4.3]{BZSV}; see \cite[Definition 4.3.4.]{BZSV} for the motivation for this name.

In particular, our construction of $S_\X$, motivated entirely by rationality concerns, recovers their definition exactly. Moreover, the content of Section \ref{Section: symplectic} proves the following special case of their conjectures. We leave the details to the interested reader; they essentially follow from Proposition \ref{Prop: symplectic rep}.
\begin{Prop}\label{Prop: BZSV}
    Suppose that $\G$ is quasi-split over $k$ and that $\X=\rH\backslash\G$ is an excellent symmetric $\G$-variety. Then $\mathcal{D}_\X^{max}=\mathcal{D}_\X$ and Conjecture 4.3.16 of \cite{BZSV} holds for $\X$.
\end{Prop}

 The previous sections show that $S_\X$ comes equipped with a natural $\Ga$-action determined by the geometric class of $\X$ as a $\G$-variety over $k$. This action is compatible with the $\Ga$-actions on $T^\ast\G$, $\check{\G}_\X$, and $\check{\fu}_\X$, and therefore induces a $\Ga$-action on $\check{\mathcal{M}}$.
We note that both the moment map $\mu$ and the bundle map 
\[
\begin{tikzcd}
    \check{\mathcal{M}}\ar[r] &   \check{\X}:= \check{\G}_\X\check{\U}_\X\backslash\check{\G}
\end{tikzcd}
\]
 are $\Ga$-equivariant with respect to this $\Ga$-action on $\check{\mathcal{M}}$.

For simplicity, let us now assume that $\G_{der}$ is absolutely simple. In Section \ref{Sec: dual symm space}, we considered the \emph{symmetric variety} $\hat{\X}=\check{\G}_\X\backslash\hat{\G}_\X$, where $\hat{\G}_\X$ is the associated group of Knop--Schalke. 
One can verify that $\hat{\G}_\X\subset \check{\mathrm{R}}_\X$, with equality holding if $\X_{sc}$ is not a $k$-form of $\SL_{2n+1}/S(\GL_n\times \GL_{n+1})$ (see also \cite[Section 10.7]{NadlerReal}). In particular, $\check{\G}_\X\backslash\check{\mathrm{R}}_\X$ is always spherical.  Moreover, $\hat{\G}_\X\cap \check{\U}_\X = \{1\}$, so that there is a distinguished closed embedding
\[
\hat{\X}=\check{\G}_\X\backslash\hat{\G}_\X\subset \check{\G}_\X\check{\U}_\X\backslash\check{\G}=\check{\X}.
\]
When $\hat{\G}_\X=\check{\G}$, we have $\hat{\X}=\check{\X}$. 

Suppose now that $\fe$ is an endoscopic datum for $(\G,\X)$ associated to $(\G_\fe,\X_\fe)$; the discussion in Section \ref{Section: hamiltonian endoscopy} shows that $\X_\fe$ also does not have type $N$ roots, so that $\calm_\fe=T^\ast(\X_\fe)$ is hyperspherical. Based on the expectation that relative functoriality is mediated by the properties of the dual Hamiltonian variety, it is natural to expect a symplectic map $\check{\mathcal{M}}\to \check{\mathcal{M}}_\fe$ such that the diagram
\[
\begin{tikzcd}
    \check{\mathcal{M}}\ar[d]\ar[r,"\mu"]&\check{\fg}^\ast\ar[d]\\
    \check{\mathcal{M}}_\fe\ar[r,"\mu_{\fe}"]&\check{\fg}_\fe^\ast;
\end{tikzcd}
\] commutes. This is true in the ``the category of Hamiltonian spaces
and Lagrangian correspondences,'' compatibly with the expectations of \cite[Remark 3.1.3]{BZSV}.
\begin{Prop}\label{Prop: langrangan}
Suppose that $(\G,\X)$ are as above and assume that $\X$ is a quasi-split symmetric variety. 
For an endoscopic datum $\fe$, suppose that $\check{\mathcal{M}}_{\fe}$ is the Hamiltonian $\check{\G}_\fe$-variety dual to $(\G_\fe,\X_{\fe})$. There is a Lagrangian correspondence of Hamiltonian varieties
\[
\begin{tikzcd}
    &\mathcal{L}_{\fe}\ar[ld]\ar[rd]&\\
    \check{\mathcal{M}}_{\fe}&&\check{\mathcal{M}}.
\end{tikzcd}
\]
covering the map $\check{\fg}^\ast\to \check{\fg}^\ast_\fe$.
\end{Prop}
We expect there is an analogue of this statement in the non-quasi-split setting which respects the Whittaker reduction to $\check{\calm}_{slice}$ in \cite[Section 14.3]{BZSV}, but have not checked this. 
\begin{proof}
    The only symmetric varieties satisfying that $\X$ is excellent and that $\hat{\G}_\X=\check{\G}$ are Galois symmetric pairs and forms of $\GL_{2n}/\GL_n\times \GL_n$ (the type $A_{2n-1}$ case). Dropping the assumption $\hat{\G}_\X=\check{\G}$ includes forms of $\GL_{2n+1}/\GL_n\times \GL_{n+1}$ (the type $A_{2n}$ case). In all cases, we have that $\mathcal{M}\simeq T^\ast(\check{\X}^\ast)$ where 
    \[
    \check{\X}^\ast = \begin{cases}
       \quad \check{\X}&: \X\text{ is Galois or type $A_{2n}$,}\\
        \check{\X}\times \A^{2n}&: \X\text{ is type $A_{2n-1}$.}
    \end{cases}
    \]
    Moreover, for an endoscopic datum $\fe$, there is an induced embedding $\check{\X}_{\fe}^\ast\subset \check{\X}^\ast$; we note that these embeddings \emph{do not respect the gradings of these varieties}. In particular, we take the fiber product
    \[
    \mathcal{L}_{\fe} :=\check{\X}_{\fe}^\ast\times_{\check{\X}^\ast}\check{\mathcal{M}}.
    \]
    Now Corollary \ref{Cor: induced map on reps} implies that there exists a symplectic map $ \mathcal{L}_{\fe}\to \check{\mathcal{M}}_{\fe}$. Since $\check{\mathcal{M}}_{\fe}$ always has a similar polarization $\check{\mathcal{M}}_{\fe}=T^\ast(\check{\X}_\fe^\ast)$, it is easy to check that the induced map $\mathcal{L}\to \check{\mathcal{M}}_{\fe}\times \check{\mathcal{M}}$ is Lagrangian.
\end{proof}

\section{Examples and applications}\label{Section: Example}
In this final section, we give examples of endoscopic symmetric varieties as well as some applications. We first discuss an example related to unitary periods. In Section \ref{Section: stabilize}, we establish the prestabilization of the regular elliptic part of the relative trace formula for a symmetric variety, under certain simplifying assumptions to streamline the presentation, deferring the technical general statement to \cite{LesliestabFJ}. The key point is Proposition \ref{Prop: important}. We end with a table of examples in Section \ref{Section: example table}.

\subsection{Example with unitary groups}\label{Section: Example unitary} 
As noted in the introduction, we expect that endoscopy for the relative Langlands program will involve, in its global applications, cases where periods of automorphic forms are given by sums of Euler products. We discuss this in a context that arises naturally in the theory of endoscopy for the symmetric variety $\X=\U_{n}\times \U_n\backslash\U_{2n}$ \cite{LeslieUFJFL,Lesliedescent}. 

Among other steps, the proof of the endoscopic fundamental lemma in this case involves reducing the matching of orbital integrals on $\X_n:=\U_{n}\times \U_n\backslash\U_{2n}$ and its endoscopic varieties $\X_a\times \X_b$ with $a+b=n$ to a matching of $\ka$-orbital integrals on $\Y_n=\U_n\backslash \Res_{E/k}(\GL_n)$ to stable orbital integrals on $\Y_a\times \Y_b$ for $a+b=n$, for entire modules of the relevant spherical Hecke algebras; we refer the interested reader to \cite[Section 2]{LeslieUFJFL}. Such a matching of orbital integrals plays the role of the fundamental lemma for the stabilization of the relative trace formula of $\U_n$-periods on $\Res_{E/k}(\GL_n)$. On the other hand, a great deal of information is known about these period integrals in the literature. The goal of this section is to discuss this comparison in terms of the numerical conjectures of \cite[Section 14]{BZSV} and the known evaluation of unitary periods for Eisenstein series due to Lapid--Rogawski for $n=3$ \cite{LapidRogawski} and Feigen--Lapid--Offen in general \cite{FLO}. 

With the discussion in Section \ref{Section: BZSV conj}, the dual Hamiltonian variety for $\mathcal{M}=T^\ast\Y_n$ is the ${}^L\Res_{E/k}(\GL_n)$-variety $\check{\mathcal{M}}=T^\ast(\GL_n\backslash(\GL_n)^2)$; here the Galois action factors through $\Ga_{E/k}=\Gal(E/k)$ and acts via swapping the two factors of $\GL_n^2$ fixing the diagonally embedded $\GL_n$ point-wise. Set $\G=\Res_{E/k}(\GL_n)$ and $\rH= \U_n$, so that $\Y_n=\rH\backslash\G$. Noting that the endoscopic varieties $\Y_a\times \Y_b$ are associated to Levi subgroups, one anticipates that endoscopic properties of $\Y_n$ only involve induced representations. Here, $\hat{\X}=\check{\X}=\GL_n(\cc)\backslash(\GL_n(\cc))^2\simeq\GL_n(\cc)$, on which $\Ga_{E/k}$ acts via inversion. In particular,
\[
\hat{\X}^{\heart,\Ga}=\hat{\X}^{\Ga} = \GL_n(\cc)[2]
\]
is just the $2$-torsion, so that $\Y_a\times \Y_b$ for $a+b=n$ is a complete list representatives of the equivalence classes of endoscopic varieties.

We now consider the conjectural recipe for the global period problem: suppose for the sake of simplicity that to an automorphic representation of $\G(\A)$ we may associate a global parameter $\phi:L_k\lra \check{\G}$; this is clearly conjectural, but works everywhere locally. In any case, we are merely interpreting a conjecture of \cite{BZSV} in terms of the notion of endoscopy. In this setting, the global conjecture of \cite{SakVenk} states that an automorphic representation $\pi$ of $\G(\A_k)=\GL_n(\A_E)$ with Arthur parameter $\phi:L_k\to{}^L\Res_{E/k}(\GL_n)$ is $\rH$-distinguished if and only if $\phi$ factors through 
\[
{}^L\Y_n=\GL_n(\cc)\times \Ga\lra {}^L\Res_{E/k}(\GL_n).
\]
Generalizing slightly the conjectures of \cite[Section 14]{BZSV}, one anticipates  that the $\rH$-period of an automorphic form $\varphi\in \pi$ is related to the sum of $L$-values
\begin{equation}\label{eqn: period formula}
    \sum_x\frac{L(\phi,T_x^\ast(\hat{\X})}{L(\phi,\check{\fg}_x)},
\end{equation}
where $x$ runs over $\phi$-fixed points of $\hat{\X}= \GL_n(\cc)\backslash(\GL_n(\cc))^2$, $T_x^\ast(\hat{\X})$ is the cotangent space of $x$, and $\check{\fg}_x$ is the Lie algebra of the stabilizer of $x$. The precise $L$-values  $L(\phi,V)$ are defined in \cite[Section 14.1]{BZSV}.

By the preceding statement, we may assume that the base point $x=x_0$ is fixed by $\phi$; this records that $\pi$ arises as a base change from $\GL_n$. In this case, the stabilizer is $\GL_n(\cc)$, $\check{\fg}_x\simeq\fgl_n$, and $T_x^\ast(\hat{\X})=\fgl_n(\cc)\otimes \cc_\omega$, where $\cc_\omega$ is the one dimensional $\Ga$-representation where $\Ga$ acts through the quadratic character $\omega:=\omega_{E/k}$ associated to the extension $E/k$ by local class field theory. The same calculation holds for the central point $(I_n,-I_n)\in \check{\X}$, producing the same period. This is compatible with \cite[Theorem 10.2]{FLO} expressing the unitary period of a cuspidal automorphic form as an Euler product times $2L(1,\pi\times \check{\pi}\cdot\omega)$, where $\check{\pi}$ is the contragredient. In particular, we see the factor of $2$ in the period calculation as relating to the two fixed points $(I_n,\pm I_n)$.

Now if we assume that $\phi=\phi_a\times \phi_b$ also factors through the Levi subgroup $\GL_a\times \GL_b\subset \GL_n$ (so corresponding to an Eisenstein series), then the four additional points
\[
\{(\pm I_a,\mp I_b),\pm(\pm I_a,\mp I_b)\}
\]
are fixed by $\phi$. In these cases
\[
T_x^\ast(\hat{\X}) = [\fgl_a\otimes \cc_\omega]\oplus [\fgl_b\otimes \cc_\omega]\oplus [\mathrm{Std}_a\otimes \mathrm{Std}_b^\vee\otimes \cc_\omega]\oplus [\mathrm{Std}^\vee_a\otimes \mathrm{Std}_b^\vee\otimes \cc_\omega],
\]
while 
\[
\check{\fg}_x = [\fgl_a\oplus \fgl_b]\oplus [\mathrm{Std}_a\otimes \mathrm{Std}_b^\vee\otimes \cc_\omega]\oplus [\mathrm{Std}^\vee_a\otimes \mathrm{Std}_b^\vee\otimes \cc_\omega];
\]
In particular, the corresponding quotient term \eqref{eqn: period formula} gives
\[
\frac{L(\phi_a,\fgl_a\otimes \cc_\omega)L(\phi_b.\fgl_b\otimes \cc_\omega)}{L(\phi_a,\fgl_a)L(\phi_b.\fgl_b)};
\]
These additional terms correspond precisely to the $L$-values for the ``endoscopic periods'' associated to the endoscopic variety $\Y_a\times\Y_b$, but the precise point of evaluation needs to be corrected as these will correspond to the ``non-tempered parameter'' case in  \cite[Section 14]{BZSV}. This is compatible with the observation of Section \ref{Section: BZSV conj} that the Lagrangian correspondence will not preserve the gradings. Once appropriated corrected in this way, the conjecture agrees with the evaluation of (regularized) unitary periods of Eisenstein series \cite[Theorem 11.2]{FLO}, which expresses the period as a sum of $L$-values indexed by the fiber of the quadratic base change map from $\GL_a\times \GL_b$ to $\Res_{E/k}(\GL_a\times \GL_b)$. 

\subsection{Pre-stabilization for the regular elliptic part}\label{Section: stabilize}
In this subsection, we establish the stabilization of the regular elliptic part of the relative trace formula associated to a symmetric variety. For this, we assume several preparatory results from \cite{LesliestabFJ}. For brevity, we omit the definition of several notations (eg. abelianized cohomology, etc.).

We now assume that $k$ is a global field of characteristic zero. We also assume that $\G$ is quasi-split over $k$, $\G_{der}$ is simply connected, and $Z(\G)$ is anisotropic. We assume that $\X=\G^\theta\backslash\G$ is excellent in the sense of Definition \ref{Def: excellent}. Lemma \ref{Lem: simply connected} implies that this is enough to ensure that all semi-simple stabilizers in $\X$ are geometrically connected. Set $\rH=\G^\theta$. The assumptions on $\G$ are not necessary and are adopted here only to avoid unhelpful technical complications. More importantly, we are forced to adopt the Assumption \ref{Assumption: orbits}, so that we have access to Proposition \ref{Prop: quotient stack is enough}. In particular, we need to work with the quotient stack
\[
[\X\times \X]/\G \simeq \X/\rH.
\]
One of purposes of \cite{LesliestabFJ} is the formulate the cohomological prestabilization of the ``stack-theoretic'' relative trace formula for $\X$ discussed below.

\subsubsection{The regular elliptic part of the  $\X$-trace formula}
 We say that an element $x\in {\X}^{ss}(k)$ is $\rH$-elliptic if $Z(\rH_x)^\circ/(\rH_x\cap Z(\G))^\circ$ is anisotropic over $k.$ We say it is regular elliptic if it is $\rH$-elliptic and regular semi-simple.

Let $\A$ denote the adele ring of $k$ and let $f\in C_c^\infty(\X(\A))$. Consider the theta function
\[
\Theta_f(h) = \sum_{x\in \X(k)}f(x\cdot h);
\]
this is a locally-finite sum as $\X$ is smooth \cite[Lemma 17.6.11]{GetzHahn}. It gives an automorphic function on $\rH(k)\backslash\rH(\A)$.    In particular, for any $f\in C_c^\infty(\X(\A))$ there exist $f_\xi\in C^\infty_c(\G(\A))$ for $\xi\in\ker^1(\rH,\G;k)$ such that
    \[
    \Theta_f(h) = \sum_{\xi\in\ker^1(\rH,\G;k)}\displaystyle\int_{[\rH^\xi]}\left(\sum_{\ga\in \G(k)}f_\xi(h'\ga h)\right)dh',
    \]
where $\rH^\xi$ is a pure inner twist of $\rH$ representing the class $\xi$, $[\rH^\xi]=\rH^\xi(k)\backslash\rH^\xi(\A)$, and $dh'$ is an appropriately normalized Haar measure. For our purposes, we restrict to the sub-sum
\[
\Theta^{re}_f(h) = \sum_{x\in \X^{re}(k)}f(x\cdot h)
\]
over the regular elliptic locus of $\X(k)$.
\begin{Rem}\label{Rem: locus too big}
    We deal with the full elliptic locus in \cite{LesliestabFJ}, where certain new features arise. We point out where in the argument this occurs soon.
\end{Rem}The period integral we are interested in is 
\[
\TF^{re}_0(f)=\displaystyle\int_{[\rH]}\Theta^{re}_f(h)dh
\]
where 
$
[\rH]=\rH(k)\backslash\rH(\A).
$
If $\X$ is quasi-split, a standard unfolding of the right-hand side (subject to constraints on $f$) relates this sum to automorphic periods of automorphic forms on $\G(\A)$, at least under local assumptions on $f$ so that the respective spectral terms converge: this is the context of many ``simple'' relative trace formulae in the literature. More generally, one needs a regularization procedure to analyze the spectral side.

For cohomological reasons, we enlarge this sum. Due to the difference between global and local inner forms,  we include a sum over the cohomology set
    \[
    \ker^1(k,\rH):=\ker[H^1(k,\rH)\overset{Has}{\lra} H^1(\A,\rH)],
    \]
 which is trivial if $\rH$ satisfies the Hasse principle. For $f\in C^\infty_c(\X(\A))$, we now consider the sum
\[
\TF^{re}_\X(f)=\sum_{\be\in \ker^1(k,\rH)}\TF^{re}_\be(f),
\]
where $\TF^{re}_\be(f)$ is the analogue of $\TF^{re}_0(f)$ with the sum over the regular elliptic orbits of $\X^{\xi}(k)$.
Since $\rH$ is connected, the cohomology set $\ker^1(k,\rH)$ is finite, implying that this sum is absolutely convergent. This step is quite natural and occurs, largely for purposes of exposition, in the stabilization of the trace formula as discussed in \cite[Section 1.13]{NgoFL}.

The necessity of Assumption \ref{Assumption: orbits} motivates us to further extend the sum as follows. Define
\[
\im^1(F,\rH) :=\im[H^1(F,\rH)\overset{Has}{\lra} H^1(\A_F,\rH)].
\]
Let $\underline{f}:=(f^\xi)\in \bigoplus_{\xi\in \im^1(F,\rH)}C_c^\infty(\X^\xi(\A))$, and consider the sum
    \begin{align}\label{eqn: stack sum}
           \TF^{re}_{\X}(\underline{f})&:=\sum_{\xi\in \im^1(F,\rH) }\TF^{re}_{\X^\xi}(f^\xi),
    \end{align}
    where
    \[
    \TF^{re}_{\X^\xi}(f^\xi)=\sum_{\be\in Has^{-1}(\xi)}\TF^{re}_\be(f^\xi),
    \]
    By definition, for any collection $\underline{f}$ there are only finitely many non-zero terms on the right-hand side. In particular, this sum is absolutely convergent as well.

The natural unfolding gives the formula
\[
 \TF_{\X}^{re}(\underline{f})= \sum_{a\in (\X^{re}//\rH)(k)}\tau(\rH_{x})\sum_{\xi\in H^1(F,\rH)}\sum_{\stackrel{x\in \X^{\xi,re}(k)/\sim}{\car^\fe(x)=a}}\Orb_x(f^\xi),
\]
where $\car^\xi:\X^\xi\to \X^\sslash\rH$ is the categorical quotient, $\tau(\rH_x)$ is the Tamagawa number, and 
\[
\Orb_x(f^\xi):=\displaystyle\int_{\rH_x^\xi(\A)\backslash\rH^\xi(\A)} f^\xi(x\cdot h)dh
\]
is the adelic orbital integral. This is absolutely convergent by our ellipticity assumption. Note that for any two classes $[x]$ and $[x']$ lying over $a\in [\X^{rss}\sslash\rH](k)$, the stabilizers $\rH_{x}$ and $\rH_{x'}$ are pure inner forms. Thus, the Tamagawa numbers agree \cite{KottwitzTamagawa} and we set the common values as $\tau(\rH_a)$. 

The following theorem is the first pre-stabilization of this sum, and is a special case of the formula established in \cite{LesliestabFJ}.
\quash{As a further extension, the lack of a version of Kottwitz' theorem of transfer to the quasi-split inner form motivates us to introduce the \textbf{stack-theoretic trace formula} by summing over all of $H^1(k,\rH)$, indexing global pure inner forms of $\rH$.
Letting $\underline{f}:=\{f^\xi\}_{\xi\in H^1(k,\rH)}$ be a collection of smooth test functions such that $f^\xi\in C_c^\infty(\X^\xi(\A))$ and $f^\xi\equiv 0$ for all but finitely many $\xi$, we consider the sum
    \begin{align}\label{eqn: stack sum}
           \TF^{re}_{\X}(\underline{f})&:=\sum_{\xi\in H^1(k,\rH) }\TF^{re}_\xi(f^\xi).
    \end{align}
    By definition, for any collection $\underline{f}$ there are only finitely many non-zero terms on the right-hand side. In particular, this sum is absolutely convergent as well.}

\begin{Thm}
 Suppose $\underline{f}=(\prod_v f^\xi_v)\in \bigoplus_\xi C^\infty_c(\X^\xi(\A))$ is a family of pure tensors.  With notation as above, we have
    \begin{align}\label{eqn: prestabilized 1}
    \TF_{\X}^{re}(\underline{f}) &=\sum_{a\in(\X^{re}\sslash\rH)(k)}\sum_{\ka\in H^1_{ab}(\A/k,\rH_x)^D}\Orb_x^\ka(\underline{f}),
\end{align}
where for $\ka\in H^1_{ab}(\A/k,\rH_x)^D$, we define the adelic $\ka$-orbital integral by
\begin{equation*}
    \Orb_x^\ka(\underline{f})= \prod_v\Orb_{x}^{\ka_v}(\underline{f}_v)
\end{equation*}
is a product of local $\ka$-orbital integrals (here, $e(\rH_{x'})$  denotes the Kottwitz sign)
\begin{align*}
    \Orb^\ka_x(\underline{f}_v)=\sum_{[x']\in [\X/\rH]_x(k)}e(\rH_{x'})\la\ka,\inv(x,x')\ra\Orb_{x'}(f^{\xi(x')}_v).
\end{align*}
\end{Thm}
The necessary definitions and details will appear in  \cite{LesliestabFJ}. The important upshot is that the regular elliptic part of the relative trace formula has been expressed as a sum  \eqref{eqn: prestabilized 1} of $\ka$-orbital integrals indexed by pairs $(a,\ka)$ where
\begin{itemize}
    \item an elliptic invariant $a\in (\X^{re}\sslash\rH)(k)$ encodes a regular elliptic stable orbit, and
    \item a character $\ka\in H^1_{ab}(\A/k,\rH_x)^D$ of the abelianized cohomology of the centralizer of $x$ in $\rH$.
\end{itemize}
 Our present goal is to re-index this sum in terms of our notion of relative endoscopic data.
 \begin{Rem}
     We hasten to remark that this some is simpler than the relative trace formula associated to a single period of the form
     \[
\displaystyle\int_{[\rH]}\int_{[\rH]}\sum_{\ga\in \G(k)}f(h_1 \ga h_2) dh_1 dh_2;
     \]
     Isolating individual terms of this form in the stabilization is delicate and introduces contributions of new endoscopic varieties (such contributions can be seen locally in \cite{WanBPfuture}). This will be clarified in future work, and ultimately relies on the existence results of the present article.
 \end{Rem}
\subsubsection{Pre-stabilization of the RTF}
 We recall the relevant statement of Tate--Nakayama duality.
\begin{Lem}\cite[Proposition 1.7.3]{LabesseBook}\label{Prop: Tate-Nakayama}
Suppose $\G$ is a connected reductive group over $k$. If $k$  is local, there is a canonical injection
\[
H^1_{ab}(k,\G)\lra \pi_0(Z(\check{\G})^\Ga)^D
\]
which is bijective when $k$ is non-archimedean. If $k$ is global, there is a canonical bijection
\[
H_{ab}^1(\A/k,\G)\iso \pi_0(Z(\check{\G})^\Ga)^D.
\]\qed
\end{Lem}
Now suppose that $x=x_a\in \X^{\xi,re}(k)$ lies over $a$ for some $\xi\in H^1(k,\rH)$ (this is viable by Assumption \ref{Assumption: orbits}). One notes that the adelic $\ka$-orbital integral $\Orb_x^\ka(\underline{f})$ only depends on $a$, so we write this $  \Orb_a^\ka(\underline{f})$.  Let $\rH_x$ be the (connected, reductive) centralizer. 
Tate--Nakayama duality allows us write our sum (using the elliptic assumption) as
\begin{equation}\label{eqn: relabled sum}
      \TF_{\X}^{ell}(\underline{f}) =\sum_{a\in(\X^{re}\sslash\rH)(k)}\sum_{\ka\in Z(\check{\rH}_a)^\Ga}\Orb_a^\ka(\underline{f}),
\end{equation}
where the notation $Z(\check{\rH}_a)$ indicates that even though the $k$-form $\rH_x$ depends on our choice of lift $x=x_a$, this center does not since all choices yield inner forms.

Suppose that $\fe=(\G_\fe,\theta_\fe,\X_\fe,\ka,\eta_\fe)$ is an elliptic endoscopic datum for $(\G,\X)$ and let $x_\fe\in \X^{ss}_\fe(k)$. Let $\fa_\fe: \X_\fe\sslash\rH_\fe\to \X\sslash\rH$ denote the morphism from Theorem \ref{Thm: point comparison}. We say that $x_\fe$ is \textbf{$(\X,\X_\fe)$-regular or just $\X$-regular} if the class $\fa_\fe([x_\fe])\in(\X\sslash\rH)(k)$ is regular semi-simple in the sense that any lift $x$ in $\X(k)$ or to any pure inner twist of $\X$ such that $[x]=\fa_\fe([x_\fe])$ is a regular semi-simple element. Let $[\X_\fe\sslash\rH_\fe]^{\X}(k)$ denote the set of $\X$-regular classes. Note that an $\X$-regular class $a$ gives rise to a $(\G,\G_\fe)$-regular element $s_\fe(a)\in \G_\fe\sslash\G_\fe(k)$ in the sense of \cite{Kottwitzstableelliptic}, as the regular stabilizers in $\X_\fe$ are inner forms of those in $\X$.

\begin{Prop}\label{Prop: important} 
    For each pair $(a,\ka)\in (\X^{re}\sslash\rH)(k)\times Z(\check{\rH}_a)^\Ga$, there exists an elliptic endoscopic datum $\fe=(\G_\fe,\theta_\fe,\X_\fe,\ka,\eta)$ and a regular elliptic $b\in [\X_\fe\sslash\rH_\fe](k)$ such that $\fa_\fe(b)= a$. The pair $(\fe,b)$ is uniquely determined up to inner isomorphism in the sense that if $(\fe_1,b_1)$ is another pair corresponding to $(a,\ka)$, then there is an isomorphism $\fe\simeq \fe_1$ identifying $b=b_1$, and that this isomorphism is uniquely determined up to $[\G_{\fe}/(\rH_\fe\cap Z(\G_\fe))]/(k)$-conjugacy.
\end{Prop}
 \begin{proof}
Fix a pair $(a,\ka)$ and let $x=x_a$ be a representative as above; for notational simplicity, we assume that $x\in \X^{re}(k)$. We let $(\G_x,\rH_x)$ denote the descendent at $x$ and let $S$ denote a maximally $\theta$-split maximal torus of $\G_x$.  Set $\T_x=S^{\theta}\subset\rH_x$ and  
\[
1\lra \T_x\lra S\lra S_{\X}\lra 1
\]
denote the quotient. Since $x$ is regular, we also have $S_\X\simeq\G_x/\rH_x$.

Fix also a $(\theta,k)$-admissible pair $(\rA,\B)$; there is an analogous short exact sequence
\[
1\lra{\T}_\X\lra \rA\lra \Ax\lra 1.
\] 
where $\T_\X = \rA_\X^\theta$. Following Proposition \ref{Prop: dual involution}, there exists an involution $\vartheta: \hat{\G}_{\X}\lra \hat{\G}_{\X}$ and a unique $\check{\rA}_\X$-conjugacy class of distinguished morphism 
\[
\varphi_{\X}:\check{\G}_{\X}\lra \check{\G},
\]
such that $\varphi_\X(\check{\G}_\X) = \hat{\G}_\X^{\vartheta,\circ}$; fixing such a morphism gives a commutative diagram 
\[
     \begin{tikzcd}
\check{\rA}_\X\ar[r]\ar[d]&\check{\rA}\ar[d]\ar[r]&\ar[d]\check{\T}_\X\\
\check{\G}_\X\ar[r]&\hat{\G}_\X\ar[r,"\check{\tau}"]&\hat{\X}.
    \end{tikzcd}
\] 

As $S$ is maximally $\theta$-split, Lemma \ref{Lem: split torus embedding} now implies the existence of a unique $\Ga$-invariant $\check{\G}^\ast_{\X}$-orbit of such embeddings $\check{S}\lra \hat{\G}_\X$ realizing $\check{S}$ as a maximally $\vartheta$-split maximal torus of $\hat{\G}_\X$. For any such embedding, the quotient $\check{\tau}$ induces an orbit map $\check{\T}_x\to \hat{\X}$, the image of which is a flat in $\hat{\X}$.

Since $\G_x/\rH_x$ is a torus and $S\subset \G_x$ is maximal, we see that $T_x\subset\rH_x$ is a maximal torus of $\rH_x$, so that duality gives rise to the canonical $\Ga$-equivariant embedding $Z(\check{\rH}_x)\subset \check{\T}_x$.
\begin{Rem}
    For a general $x\in \X^{ell}(k)$, one can obtain a $\Ga$-equivariant morphism $Z(\check{\rH}_x)\to \check{\T}_x$, which need not be injective (cf. Remark \ref{Rem: locus too big}).
\end{Rem}
By Tate-Nakayama duality and the ellipticity assumption, we obtain $\ka\in Z(\check{\rH}_x)^\Ga\to \check{\T}_x\to \hat{\X}$. 
    Passing to the descendant at $\ka$, we obtain a commutative diagram
    \[
         \begin{tikzcd}
\check{\G}_{\X,\ka}\ar[d]\ar[r]&\hat{\G}_{\X,\ka}\ar[d]\ar[r]&\hat{\X}_\ka\ar[d]\\
\check{\G}_{\X}\ar[r,"\varphi_{\X}"]&\hat{\G}_{\X}\ar[r]&\hat{\X}.
\end{tikzcd}
    \]This gives rise to an endoscopic datum $\fe=(\G_\fe,\theta_\fe,\X_\fe,\ka,\eta_\fe)$ where $\check{\G}_\fe\simeq \check{\G}_\ka$ and the symmetric variety $\X_\fe$ satisfying the conditions of Theorem \ref{Thm: exists}. Note that since $\X$ is assumed excellent, it has no type $N$ roots. This implies the same for $\X_\fe$ as discussed in Section \ref{Section: hamiltonian endoscopy}. By Corollary \ref{Cor: quasisplit endo}, we may assume that $\X_\fe$ is the quasi-split endoscopic variety. 
    
    Moreover, we also obtain a morphism $\check{S}\to \hat{\G}_{\X,\ka}\simeq \hat{\G}_{\X_\fe}$ with image a maximally $\vartheta_\fe$-split maximal torus, the $\check{\G}_{\X,\ka}$-orbit of which is $\Ga$-stable. Proposition \ref{Prop: quotient stack is enough} thus implies that up to replacing $\X_\fe$ by a pure inner twist $\X_\fe^\xi$ corresponding to a class in $H^1(k,\rH_\fe)$, we obtain a $k$-rational embedding $S^-\to \G^\xi_\fe$ as a maximal $\theta_\fe^\xi$-split maximal torus; replacing $\X_\fe$ by such a twist if necessary, we drop the superscript $\xi$ for notational simplicity. Let $S_\fe\subset \G_\fe$ be a maximally $\theta_\fe$-split maximal torus such that $S_\fe^-$ agrees with the image of $ S^-$.
\quash{\[
\X_\fe(k)=\bigsqcup_{\be\in \ker^1(\rH_\fe,\G_\fe;k)} x_\be\cdot\G_\fe(k).
\]
For later use, let $S_1\subset \G_{\fe}$ denote the image and let $j_1: S_1\to S$ denote the $k$-rational isomorphism inverse to the above embedding.}

Let $x_\fe\in \X_{\fe}(k)$ correspond to $\theta_\fe$ in the sense that $\rH_\fe = \G_{\fe,x_\fe} = \G_\fe^{\theta_\fe}$ and consider the flat $S^-_\fe\cdot x_\fe = S_\X\cdot x_\fe$. Recall that if $x_0\in \X(k)$ is the distinguished base point of $\X$ with respect to $\theta$, our point $x\in (S_\X\cdot x_0)(k)$ so there exists $t\in S_{\X}(k)$ such that $x=t\cdot x_0$. Now set $x_1:=t\cdot x_\fe\in \X_\fe(k)$.
Via the Chevalley isomorphism and Theorem \ref{Thm: point comparison} we see that $b:=[x_1]\mapsto [x]=a$ under the map
     \[
     \X_\fe\sslash\rH_\fe\simeq S_\X/W_{\X_\fe}\overset{\fa_\fe}{\lra}S_\X/W_\X\simeq \X_\fe\sslash\rH_\fe.
     \]
   Thus $(\fe,b)\mapsto (a,\ka)$.

   To establish uniqueness, we first show  how to associate a pair $(a,\ka)$ as in \eqref{eqn: relabled sum} to  a pair $(\fe,b)$ of an endoscopic datum $\fe$ and an $\X$-regular class $b\in[\X_\fe\sslash\rH_\fe]^{\X}(k)$. Set $a:=\fa_\fe(b)$, where $\fa_\fe$ is the morphism of Theorem \ref{Thm: point comparison}. By Assumption \ref{Assumption: orbits}, up to replacing $\X$ and $\X_\fe$ by pure inner twists, we may let $x=x_a\in \X(k)$ be a representative of $a$ in $\X^{rss}(k)$ and $y\in \X_{\fe}(k)$ be a representative of $b$ in $\X_{\fe}^{rss}(k)$. The $\X$-regular assumption implies that the stabilizers $\rH_{\fe,y}$ and $\rH_{x}$ are inner forms of one another, so that $x$ is elliptic if $y$ is and there is a canonical isomorphism
 \[
 Z(\check{\rH}_{\fe,b})\simeq Z(\check{\rH}_{a}).
 \]
By assumption, the endoscopic datum $\fe$ induces a diagram
 \[
         \begin{tikzcd}
\check{\G}_{\X_\fe}\ar[d]\ar[r,"\varphi_{\X_\fe}"]&\hat{\G}_{\X_\fe}\ar[d,"\eta"]\ar[r,""]&\hat{\X}_\fe\ar[d, "\eta_\X"]\\
\check{\G}_{\X}\ar[r,"\varphi_{\X}"]&\hat{\G}_{\X}\ar[r]&\hat{\X},
\end{tikzcd}
    \]
    with $\ka\in Z(\check{\G}_\fe)^\Ga\subset Z(\hat{\G}_{\X_\fe})$ giving an element of $\cala_{\hat{\X}_\fe}\simeq \hat{s}_\fe(\hat{\X}_\fe)\cap  Z(\hat{\G}_{\X_\fe})$ \cite[Lemma 1]{VustEmbeddings}. Using \eqref{aut as center}, we obtain 
    $$\ka\in \cala_{\hat{\X}_\fe}^\Ga\subset  Z(\check{\rH}_{\fe,b})^\Ga\simeq Z(\check{\rH}_{a})^\Ga,$$
    giving our pair $(a,\ka)$.
     
 Write $(\fe_1,b_1) =(\fe,b)$. We now show that $(\fe_2,b_2)\mapsto (a,\ka)$ if and only if there is an isomorphism $\fe_1\simeq \fe_2$ identifying $b_1=b_2$. We fix a lift $x_1\in \X_{\fe_1}(k)$ as above and let $x_2\in \X_{\fe_2}(k)$ be a lift of $b_2$. Fix also maximally $\theta$-split maximal tori $S_1$ and $S_2$  in  $\G_{\fe,1}$ and $\G_{\fe,2}$ respectively containing these lift.

 Using $\eta_2$ and Proposition \ref{Prop: quotient stack is enough}, we may fix embeddings 
 \[
 j_1: S_1\to \G^{\xi_1}, \qquad\text{ and }\qquad j_2: S_2\to \G^{\xi_2},
 \]
 where $\xi_1,\xi_2\in H^1(k,\rH)$ may be distinct classes, each realizing $j_i(S_i^-)$ as a $k$-rational maximal $\theta^{\xi_i}$-split maximal torus. Arguing as above, this gives points $x_1\in \X^{\xi_1}(k)$ and $x_2\in \X^{\xi_2}(k)$ over the invariant $a\in (\X\sslash\rH)(k)$. But this implies that we may arrange for $\xi_1=\xi_2$ and $x_1=x_2$ by replacing $j_2$ with a $\rH(\kbar)$-conjugate. At the cost of dropping the assumption that $\G$ is quasi-split, we may drop the superscript $\xi$ and denote by $\theta$ the involution on $\G$ with respect to which the above embeddings give rise to maximal $\theta$-split tori. We thus have a diagram (equivariant with respect to the various involutions) of $\kbar$-morphisms
 \begin{align*}
     S_1\overset{j_1}{\lra} &\G \overset{j_2}{\longleftarrow}S_2\\
     s_1(x_1) \longmapsto &s(x) \longmapsfrom s_2(x_2).
 \end{align*}
 This implies that $j(S_2)\subset \G_x$ is a maximally $\theta$-split maximal torus, so that there exists $h\in \rH_x(\kbar)$ such that $\Ad(h)\circ j_2(S_2)=j_1(S_1)$, allowing us to identify $S_1$ and $S_2$ over $\kbar$ via $j_2^{-1}\circ j_1$. Since $\ka$ determines the root systems for $\G_{\fe,1}$ and $\G_{\fe,2}$ as transferred to $X^\ast(S)$, we see that there is a $\kbar$-isomorphism $\phi_0: \G_{\fe,1}\to \G_{\fe,2}$ extending $j_2^{-1}\circ j_1 :S_1\to S_2.$ Moreover, since $j_2^{-1}\circ j_1$ intertwines the two involutions, it follows that $\phi_0$ may be chosen to be $(\theta_1,\theta_2)$-equivariant.

 Since $j_2$ is $(\theta_{2},\theta)$-equivariant $j_2(s_2(x_2)) = s(x)$, we see that for all $\sig\in \Ga$, ${}^\sig j_2 = \Ad(h)\circ j_2$ for some $h\in \rH_x(\kbar)$; here we have used the rationality of $x$ and the assumption that $x\in \X^{rss}(k)$. On the other hand, $\X$-regularity of $b_2\in[\X_\fe\sslash\rH_\fe](k)$ implies that $s_2(x_2)$ is $(\G,\G_{\fe,2})$-regular in the sense of \cite{Kottwitzstableelliptic}, so that ${}^\sig j_2 =  j_2\circ \Ad(h_2)$ for $h_2\in \rH_{\fe_2,x_2}(\kbar)$. This shows that ${}^\sig(j_2^{-1}\circ j_1)$ is conjugate to $j_2^{-1}\circ j_1$ by an element of $\rH_{\fe,2,x_2}(\kbar)$. It follows that the $\rH_{\fe,2,x_2}$-conjugacy class of $\phi_0:  \G_{\fe,1}\iso\G_{\fe,2}$ is defined over $k$. 

  Recall that the pure inner twists which arise in Assumption \ref{Assumption: orbits} only correspond to twists of the symmetric pair. Fixing pure inner twists $\psi_1$ and $\psi_2$ to the quasi-split inner forms, we have a diagram
 \[
\begin{tikzcd}
    (\G_{\fe,1},\theta_1)\ar[d,"\psi_1"]\ar[r,"\phi_0"]&(\G_{\fe,2},\theta_2)\ar[d,"\psi_2"]\\
    (\G_{\fe,1}^\ast,\theta_1^\ast)& (\G_{\fe,2}^\ast,\theta_2^\ast),
\end{tikzcd}
 \]
 where each arrow is equivariant with the involutions on the domain and codomain. An argument as in the proof of Lemma \ref{Lem: split torus embedding}, we see that a $\rH_{\fe,2}^\ast(\kbar)$-conjugate $\phi$ of $\psi_2\circ\phi_0\circ \psi_1^{-1}$ is defined over $k$. This implies that $\phi\circ \psi_1(s_1(x_1))$ is stably conjugate to $\psi_2(s_2(x_2)),$ and we see that $\phi$ induces an isomorphism $(\G_{\fe,1},\hat{s}_1(\ka_1),\eta_1)\lra(\G_{\fe,2},\hat{s}_2(\ka_2),\eta_2)$. 

We claim that $\phi$ intertwines $\theta_1^\ast$ with a pure inner twist of $\theta_2^\ast$; that is $\theta_1':=\phi\circ \theta_1^\ast\circ \phi^{-1}$ is a pure inner twist of $\theta_2^\ast$. To see this, note that they are both $k$-rational involutions on $\G^\ast_{\fe,2}$ which agree on the maximally $\theta$-split torus\footnote{By construction, the involutions all agree on $\psi_2(S_2)$.}. By \cite[Theorem 3.7]{Helmincktwo}, there exists $h\in \G_{\fe,2}^\ast(\kbar)$ satisfying $\theta'_1 = \Ad(h)^{-1}\circ\theta_2^\ast\circ\Ad(h)$ and that $\Ad(h^{-1}\theta_2^\ast(h))\in S_{ad}(k)$, where $S_{ad}^\ast\subset \G^\ast_{\fe,2,ad}$ is the image in the adjoint group. This implies that the involutions are geometrically conjugate, so that the symmetric varieties 
\[
\text{$X_1':=\G_{\fe,2}^{\ast,\theta_1'}\backslash\G_{\fe,2}^\ast$ and $X_2^\ast:=\G_{\fe,2}^{\ast,\theta^\ast_2}\backslash\G^\ast_{\fe,2}$}
\]
are $\G_{\fe,2}^\ast$-forms of each other. Thus, there is a canonical isomorphism 
$
\Aut_d({X_1'})\simeq \Aut_d({\X^\ast_2}), 
$
as this group depends only on $\G^\ast_{\fe,2}$ by Lemma \ref{Lem: unique on aut}. Since both forms have geometric cocycles determined by \eqref{item rep} and they are well-adapted, it follows that they are $\G_{\fe,2}$-inner forms. Since they are both quasi-split in the sense of Corollary \ref{Cor: quasisplit endo}, it follows that they are isomorphic. Lemma \ref{Lem: pure inner twist invol} now implies that this forces $\theta_1'=\phi\circ \theta_1^\ast\circ \phi^{-1}$ to be a pure inner twist of $\theta_2^\ast$, proving the claim. It is now straightforward to verify that $\phi$ induces an isomorphism of endoscopic data, since it intertwines maximally $\theta$-split tori, so that maps $\eta_1\circ\check{\phi}$ is conjugate by $\check{\G}_{\X_1}$.


   Finally, we verify that this isomorphism is unique up to inner automorphisms, which amounts to showing that any automorphism $\phi\in \Aut(\fe)$ which preserves the $\X$-regular semi-simple class $b$ (in the sense that if $y\in \X_{\fe}(k)$ lies in the stable orbit over $b$, then so does $\phi_\X(y)$ where $\phi_\X:\X_\fe\to \X_\fe$ is the induced automorphism) is automatically inner. 
   
   The observation (already used above) is that the matching $a=\fa_\fe(b)$ induces a matching of semi-simple orbits $\chi_\fe(s_\fe(b))=s(a)$ as in Corollary \ref{Cor: quotient coherence}. Since $b$ is $\X$-regular, it follows that $s_\fe(b)$ is a $(\G,\G_\fe)$-regular class. Since $\phi$ preserves $b$, it follows that $f_\phi$ preserves $s_\fe(b)$. The proof of \cite[Lemma 9.2]{Kottwitzstableelliptic} now goes through, showing that $\phi$ must be inner.
 \end{proof}
In particular, we may rewrite the sum
\[
  \TF_{\X}^{re}(\underline{f}) =\sum_{\fe/\sim}\left(\sum_{b\in [\X_\fe^{re}\sslash\rH_\fe]^{\X}(k)}\Orb_{\fa_\fe(b)}^\ka(\underline{f})\right),
\]
where $[\X_\fe^{re}\sslash\rH_\fe]^{\X}$ denotes those orbits in $\X_\fe^{re}$ which are $(\X,\X_\fe)$-regular. At this stage, one is in a position to formulate the local conjectures necessary to relate the inner sum to (stable) orbital integrals on $\X_\fe$. We defer to \cite{LesliestabFJ} for the details.

\subsection{Examples}\label{Section: example table}
Finally, Table \ref{tab:Examples} gives examples of symmetric varieties and associated endoscopic varieties under the assumption that $\G_{der}$ is $k$-simple. Note that several are already listed when $\X$ has rank $1$ in Table \ref{tab:satake}; we reproduce a few of those examples here as cases of infinite families.

\quash{
\begin{table}[htp]
\caption{Rank-$1$ endoscopic varieties}
    \label{tab:satake2}
    \centering
    \begin{tabular}{|c|c|c|}
	\hline 
	$\textbf{Type}$ & $(\Psi(\lam),\theta)$ & $(\Psi_s(\lam),\theta)$ \\[.25 cm] \hline
 $\mathbf{A_l\:(l\geq 3)}$ & $\dynkin[edge length=.75cm,
involutions={16}]{A}{o**.**o}$ & $\dynkin[edge length=.75cm]{A}{**.**}\quad \dynkin[edge length=.75cm]{A}{o}$ \\ \hline
 & $(\SL_{l},\GL_{l-1})$ & $(S(\GL_{l-2}\times \GL_2), \GL_{l-2}\cdot T)$ \\[.25cm] \hline
 $\mathbf{B_l\:(l\geq 3)}$ & $\dynkin[edge length=.75cm]
{B}{o**.**}$ & $ \dynkin[edge length=.75cm]{A}{o}\quad\dynkin[edge length=.75cm]{B}{**.**}$ \\[.25cm] \hline
  & $(\Spin_{2l+1},\Spin_{2l})$ &  $(\SL_2\times\Spin_{2l-1},T\times \Spin_{2l-1})$ \\[.25cm] \hline
 $\mathbf{C_l\:(l\geq 3)}$ & $\dynkin[edge length=.75cm]
{C}{*o*.**}$ & $\left[\dynkin[edge length=.75cm]
{D}{*}\dynkin[edge length=.75cm]
{D}{o}\right]\quad\dynkin[edge length=.75cm]
{C}{*.**}$ \\[.25cm] \hline
 & $(\Sp_{2l},\Sp_2\times \Sp_{2l-2})$ &  $(\SO_4\times\Sp_{2l-4},\GL_2\times \Sp_{2l-4})$ \\[.25cm] \hline
 $\mathbf{D_l\:(l\geq 3)}$ & $\dynkin[edge length=.75cm]
{D}{o**.***}$ &  $\dynkin[edge length=.75cm]
{D}{o**.***}$ \\[.25cm] \hline
 & $(\Spin_{2l},\Spin_{2l-1})$ &  $(\Spin_{2l},\Spin_{2l-1})$ \\[.25cm] \hline
 $\mathbf{F_4}$ &$\dynkin[edge length=.75cm]
{F}{***o}$ & $\dynkin[edge length=.75cm]
{B}{***}\:\dynkin[edge length=.75cm]
{A}{o}$  \\[.25cm] \hline
 & $(F_4,\Spin_9)$ & $((\Spin_7\times\SL_2)/\mu_2, (\Spin_7\times T)/\mu_2)$ \\[.25cm] \hline
    \end{tabular}
\end{table}}
In the table below, we note by an $\ast$ examples when the endoscopic variety is not well adapted, even if the original variety is.

\setlength{\extrarowheight}{20pt}
\begin{table}[htbp]
    \label{tab:Examples}
    \centering
    \caption{Examples of endoscopic symmetric varieties}
    \begin{tabular}{|c|c|}
	\hline 
	$\underline{\X=\rH\backslash\G}$ &  $\underline{\X_{\fe}=\rH_\fe\backslash\G_\fe}$ \\[.5 cm] \hline
 $\GL_n\times \GL_{n+r}\backslash\GL_{2n+r}$ & $\left[\GL_a\times \GL_a\backslash\GL_{2a}\right]\times \left[\GL_b\times \GL_{2b+r}\backslash\GL_{2b+r}\right]$ \quad$(a+b=n)$\\ \hline
   $\U_n\times \U_{n}\backslash\U_{2n}$ & $\left[\U_a\times \U_{a}\backslash\U_{2a}\right]\times \left[\U_b\times \U_{b}\backslash\U_{2a}\right]$\quad $(a+b=n)$\\ \hline
  $\U_n\times \U_{n+r}\backslash\U_{2n+r}$, $r\geq1$ & $\left[\Res_{E/k}(\GL_a)\backslash\U_{2a}\right]\times \left[\U_b\times \U_{b+r}\backslash\U_{2b+r}\right]$\quad $(a+b=n)$\\ \hline
   $\Res_{E/k}(\GL_n)\backslash\U_{2n}$ & $\left[\Res_{E/k}(\GL_a)\backslash\U_{2a}\right]\times \left[\Res_{E/k}(\GL_b)\backslash\U_{2b}\right]$\quad $(a+b=n)$\\ \hline
  $\rH\backslash\Res_{E/k}(\rH_E)$ & \qquad\qquad \qquad \qquad\:$\rH_\fe\backslash\Res_{E/k}(\rH_{\fe,E})$ \hfill($\rH_\fe$ endoscopic for $\rH$) \\[.25cm] \hline
 $\Sp_{2p}\times \Sp_{2p+2r}\backslash\Sp_{4p+2r}$ & $\left[\GL_m\backslash\SO_{2m}\right]\times \left[\Sp_{2(p-m)}\times \Sp_{2(p-m)+2{r}}\backslash\Sp_{4(p-m)+2r}\right]\: (m\leq p)$  \\[.25cm] \hline
 $\GL_{2n+1}\backslash\SO_{4n+2}$ & $\GL_{2n}\backslash\SO_{4n}\times \SO_2\backslash\SO_2$  \\[.25cm] \hline
$A_1\times A_5\backslash E_6^{sc}$ &  $[\Gm\backslash\Gm]\times^{\mu_4}[\Spin_{6}\times^{\mu_2}\Spin_{4}\backslash\Spin_{10}]^\ast$ \\[.25cm] \hline
  $\Spin_9\backslash F_4$ & $[\Spin_7 \backslash\Spin_7]\times^{\mu_2}[\Gm\backslash\SL_2]$ \\[.25cm] \hline
    \end{tabular}
\end{table}

\newpage
\begin{appendix}

\section{Functoriality of the dual group}\label{Appendix: derived subgroup}\label{Sec: functorial derived}
Suppose that $\G$ is reductive over $k$ and let $\G_{der}$ denote the derived subgroup. Suppose that $\rH\subset \G$ is a spherical $k$-subgroup and set $\rH_{\G,der}:=\rH\cap\G_{der}$. The purpose of this section is to clarify the effect of viewing a spherical $\G$-variety as a (potentially non-spherical) $\G_{der}$-variety.
\begin{Lem}
    The subgroup $\rH_{\G,der}\subset\G_{der}$ is spherical.
\end{Lem}
\begin{proof}
    The assumption that $\X=\rH\backslash\G$ is spherical is equivalent to the existence of a Borel subgroup $\B\subset \G$ such that $B\cdot \rH\subset \G$ is Zariski-open. Since $\G=\G_{der}\cdot Z(\G)^\circ$ and $Z(\G)^\circ\subset \B$, it is immediate that $\B_{der}\cdot \rH_{\G,der}\subset \G$ is Zariski-open.
\end{proof}
Setting $\X_{der}=\G_{der}/\rH_{\G,der}$, we obtain a closed embedding $\iota:\X_{der}\to \X$ sitting in a sequence
\[
\X_{der}\lra \X\lra \G_{ab}/\rH_{\G,ab},
\]
where $\rH_{\G,ab}$ is the schematic image of $\rH$ under the abelianization map. In general, $\X$ need not be spherical as a $\G_{der}$-variety. Happily, the dual group of Knop--Schalke exists for a general (normal) $\G_{der}$-variety. In particular, there exists a dual group $\check{\G}_{\X,der}$ and a distinguished morphism $\varphi_{\X,der}:\check{\G}_{\X,der}\to \check{\G}_{der}$. 

\begin{Lem}\label{Lem: restrict to derived}
    For any choice of distinguished morphisms $\varphi_{\X}$ and $\varphi_{\X,der}$, there exists a unique surjective morphism $p_\X:\check{\G}_{\X}\lra \check{\G}_{\X,der}$ such that
      \[
\begin{tikzcd}
    \check{\G}_\X\ar[r,"\varphi_{\X}"]\ar[d,"p_\X"]&\check{\G}\ar[d]\\
    \check{\G}_{\X,der}\ar[r,"\varphi_{\X,der}"]&\check{\G}_{der}
\end{tikzcd}
    \]
  commutes.
\end{Lem}
\begin{proof}
    Suppose that $\rA=\B/[\B,\B]$ is the canonical maximal torus for $\G$, and let $\rA\lra \Ax$ be the quotient through which $\rA$ acts on $\mathring{\X}\sslash[\B,\B]$ where $\mathring{\X}$ denotes the open $\B$-orbit. \quash{ By definition of a distinguished morphism, we have a commutative diagram
    \[
\begin{tikzcd}
    \check{\rA}_\X\ar[r]\ar[d]&\check{\rA}\ar[d]\\
    \check{\G}_{\X}\ar[r,"\varphi_{\X}"]&\check{\G}.
\end{tikzcd}
    \]}
    Now suppose that $\rA_{der}= \B_{der}/[\B_{der},\B_{der}]$ is the canonical torus of $\G_{der}$, and let $\rA_{der}\to \rA_{\X,der}$ be the corresponding quotient. The local structure theorem implies that $X^\ast(\rA_{\X,der})$ is the lattice of $\B_{der}$-characters associated to $\B_{der}$-semiinvariant functions on $\X$. In particular, since there is a natural embedding $ \rA_{\X,der}\subset\Ax$. It follows that the restriction map $\fX\lra \fX_{der}$ is surjective.
    
    Passing to dual side, we obtain a commutative diagram
      \begin{equation}\label{diagram: commute on tori}
\begin{tikzcd}
    \check{\rA}_\X\ar[r]\ar[d]&\check{\rA}\ar[d]\\
    \check{\rA}_{\X,der}\ar[r]&\check{\rA}_{der}.
\end{tikzcd}
      \end{equation}
 Considering now the dual groups, a choice of distinguished morphism (unique up to $\check{\rA}_{\X,der}$-conjugacy) induces a natural extension of three of these morphisms
       \[
\begin{tikzcd}
    \check{\G}_\X\ar[r,"\varphi_{\X}"]&\check{\G}\ar[d,"p"]\\
    \check{\G}_{\X,der}\ar[r,"\varphi_{\X,der}"]&\check{\G}_{der},
\end{tikzcd}
    \]
and the claim is that $p\circ\varphi_\X$ factors through $\varphi_{\X,der}$ and a morphism $p_\X$ extending the left vertical arrow in \eqref{diagram: commute on tori}. But this follows since the sublattice $X^\ast(\check{\rA}_{\X,der})\subset X^\ast(\check{\rA}_\X)$ contains the set of roots $\check{\Phi}_\X$ for $\check{\G}_\X$, inducing such a morphism. Moreover, the commutativity of \eqref{diagram: commute on tori} and the definition of distinguished morphisms imply that $p\circ\varphi_\X$ and $\varphi_{\X,der}$ have a common image in $\check{\G}_{der}$.
\quash{
More precisely, the second requirement of a distinguished morphism stated in \cite{KnopFunctorial} involves root spaces in $\check{\fg}_{\X}$ and $\check{\fg}$, relying on a pinning of both.  Note that a pinning of $\check{\fg}$ determines a unique pinning on $\check{\fg}_{der}$ as they share a common root system, differing only in the center. As discussed in Section \ref{Section: dual groups}, there is an ambiguity in a choice for roots of type $D_2$, but as long as these choices are fixed then the two distinguished morphisms will possess a common image in the direct sum of root spaces in that case. Once this choice is (coherently) fixed, the image $\G_{\X,der}^\ast\subset\check{\G}_{der}$ is independent of the choice of distinguished morphisms. \cite[Theorem 2.5 (iv)]{KnopFunctorial}.

   The upshot is that we have a pair of surjective morphisms of complex reductive groups with central kernels
   \[
   \begin{tikzcd}
    \check{\G}_\X\ar[rd,swap,"p\circ\varphi_{\X}"]&&\check{\G}_{\X,der}\ar[ld,"\varphi_{\X,der}"]\\
    &{\G}^\ast_{\X,der},
\end{tikzcd}
   \]
   as well as a commuting diagram of maximal tori, so that there exists a unique morphism $p_\X:\check{\G}_\X\to\check{\G}_{\X,der}$ satisfying the claim of the lemma.}
\end{proof}

Recalling the functoriality properties of dual groups \cite[Theorem 2]{KnopFunctorial}, the injective $\G_{der}$-morphism $\iota$ induces a dual map 
\begin{equation}\label{eqn: dual map on derived}
    \check{\iota}:  \check{\G}_{\X_{der}}\lra\check{\G}_{\X,der}.
\end{equation}

Upon fixing distinguished morphisms $\varphi_{\X,der}$ and $\varphi_{\X_{der}}$, $\check{\iota}$ is unique and has a finite kernel.
\begin{Lem}
    The morphism \eqref{eqn: dual map on derived} is an isomorphism.
\end{Lem}
\begin{proof}
    The tori $\rA_{\X,der}$ and $\rA_{\X_{der}}$ are isomorphic as they share a common character lattice. Since the kernel of $\check{\iota}$ agrees with  the kernel of the natural map $\check\rA_{\X,der}\to \check\rA_{\X_{der}}$, the claim follows.
\end{proof}
In particular, if  $\check{Z}_\X$ denotes the kernel of $p_\X$, then $\check{Z}_\X$ is dual to $\G_{ab}/\rH_{\G,ab}$, which has the character lattice $\fX_Z$ satisfying
\[
0\lra \fX_Z\lra \fX\lra \fX_{der}\lra 0.
\] Finally, Lemma \ref{Lem: restrict to derived} implies that there is commutative diagram,
\[
    \begin{tikzcd}
   1\ar[r]&\check{Z}_\X\ar[r]\ar[d]& \check{\G}_\X\ar[r,"p_\X"]\ar[d,"\varphi_{\X}"]&\check{\G}_{\X_{der}}\ar[d]\ar[r]&1\\
    1\ar[r]&\check{\G}_{ab}\ar[r]& \check{\G}\ar[r,"p"]&\check{\G}_{der}\ar[r]&1.
\end{tikzcd}
\]

\section{Spherical data of symmetric varieties}\label{Section: norms sym} In this appendix, we make explicit the spherical data of a symmetric variety. Let us assume that $\X=\rH\backslash\G$ is a symmetric $k$-variety associated to a $k$-rational involution $\theta$. For simplicity, we assume that $\G$ is quasi-split, though this is not essential. We fix a $(\theta,k)$-admissible Borel pair $(\rA,\B)$ as in Section \ref{Section: admissible tori}; in particular, $\B$ is maximally $\theta$-split so that the associated set of simple roots gives a $(\Ga,\theta)$-basis for $X^\ast(\rA)$.
\begin{Rem}
    We worked with the canonical torus $\rA=\B/[\B,\B]$ in Section \ref{Section: spherical data}. In the current setting, we opt to utilize maximally $\theta$-split maximal tori as their existence is very useful. If one preferred to work with the canonical torus, when $\B$ is chosen as above, then $\rA=\B/[\B,\B]$ is isomorphic to any $(\theta,k)$-admissible torus $A\subset \B$ in a $\theta$-equivariant way. 
\end{Rem}
There is a short exact sequence of diagonalizable groups
\[
1\lra \rA^{\theta}\lra \rA\lra \rA_\theta\lra 1.
\]
Let $X^\ast_\theta:=X^\ast(\rA_\theta)= \{\lam-\theta(\lam):\lam\in X^\ast(\rA)\}$ and
\[
\widetilde{\De}_{\X}:=\{\al-\theta(\al):\al\in \De,\:\theta(\al)\neq \al\}.
\]
We compute the little Weyl group \cite[Section 4]{Richardson}
\[
W_{\X}\simeq N_{\G}(\rA^-)/Z_{\G}(\rA^-)\simeq N_{(\G^\theta)^\circ}(\rA^-)/Z_{(\G^\theta)^\circ}(\rA^-),
\]
and the relative roots $\widetilde{\Phi}_{\X} = W_{\X}\cdot\widetilde{\De}_{\X}$; this is the reduced root system associated to the (possibly non-reduced) relative root system of $\X$ with base $\widetilde{\De}_{\X}$. With respect to the $k$-structure of $\rA\subset \G$, the associated $\Ga$-action preserves $\widetilde{\De}_{\X}$.
\begin{Lem}\label{Lem: root lattice}
    For $(\G^\theta)^\circ\subset \rH\subset N_{\G}(\G^\theta)$, let $\X=\rH\backslash\G$. The root lattice $\Lam_\X$ is given by $\zz\widetilde{\De}_\X$.
\end{Lem}
\begin{proof}
    When $\rH=\G^\theta$, this follows from the claim  $\widetilde{\De}_{\X}={\De}^{n}_{\X}$ which is \cite[Theorem 6.7]{KnopAutomorphisms}. Lemma 3.1.5 (4) of \cite{Losev} implies the claim for any $\rH$ such that $[\rH:(\G^\theta)^\circ]$ is finite. On the other hand, $N_{\G}(\G^\theta)\backslash\G$ descends to a symmetric variety of $\G_{ad}$ so that the root lattice is also given by the root lattice of the relative root system, that is by $\zz\widetilde{\De}_\X$. Finally, since
    \[
   (\G^\theta)^\circ\backslash \G\to\X\to N_{\G}(\G^\theta)\backslash\G
    \]
    are each surjective $\G$-equivariant morphisms with the first and last varieties possessing the same root lattice, the claim follows.
\end{proof}
The quotient torus $\rA_\theta$ is isomorphic to the canonical torus $\rA_{\X_\theta}$ when $\rH=\G^\theta$. In general, the quotient through which $\rA$ acts freely on the distinguished base point $x_0\in \X(k)$ is isomorphic to $\Ax$. When $\X^\circ=(\G^\theta)^\circ\backslash \G$, the combination of \cite[Corollary 1.5]{Hofscheier} (see also \cite{knop2024fundamental}) with Proposition \ref{Prop: sym colors} below shows that 
\begin{equation}\label{eqn: connected lattice}
   X^\ast_{\theta,\circ}:= X^\ast(\rA_{\X^\circ})= \{\lam\in X^\ast_{\theta,\qq}\cap X^\ast(\rA): \la \check{\ga}^n,\lam\ra\in \zz, \text{ for all }\ga^n\in{\De}^{n}_{\X}\}.
\end{equation}
\begin{Ex}
    If $\G=\GL_n$ and $\theta(g) = {}^Tg^{-1}$ so that $\G^\theta=\mathrm{O}_n$, then $X^\ast_\theta= 2X^\ast(\rA)$ and 
    \[
    \check{{\De}}^{n}_{\X}=\left\{\frac{1}{2}\check{\al}:\al\in \De\right\}.
    \]
    Thus $X^\ast_{\theta,\circ} = 2X^\ast(\rA)+\zz(e_1+\cdots+e_n)$ gives the weight lattice of $\GL_n/\SO_n$.\qed
\end{Ex}
More generally, let $\X=\rH\backslash\G$ with any $(\G^\theta)^\circ\subset \rH\subset N_{\G}(\G^\theta)$. Since $(\G^\theta)^\circ\backslash \G\to \X$, we have a $\Ga$-stable chain of lattices
\[
\zz\widetilde{\De}_{\X} \subset \fX \subset X^\ast_{\theta,\circ}
\] corresponding uniquely to the quotient $\Aut^{\G}((\G^\theta)^\circ\backslash \G) \to \cala_{\X}$.

\quash{we now compute $\X^\ast(\Ax)$. First we use \cite[Corollary 1.5]{Hofscheier}, which states that if $\X^\circ = \G\sslash\rH^\circ$, then 
\[
X^\ast(\rA_{\X^\circ})= \{\lam\in X^\ast(\Ax)_{\qq}\cap X^\ast(\rA): \la \rho(D),\lam\ra\in \zz, \text{ for all }D\in\D_{\X}\}.
\]}

As an application, we recall from \cite[Section 4]{Lesliedescent} that a symmetric pair $(\G,\G^\theta)$ is called \emph{simply connected} if for each semi-simple $x\in \X_\theta^{ss}(k)$ the stabilizer $\rH_x$ in $\rH=\G^\theta$ is geometrically connected. 
 \begin{Lem}\label{Lem: simply connected}
     Suppose that $\G_{der}\subset \G$ is simply connected and that $\X=\rH\backslash\G$ is a symmetric variety with $\rH=\G^\theta$. The stabilizers $\rH_x$ are geometrically connected for all $x\in \X^{ss}(k)$ if and only if $\rA\lra \Ax$ has a geometrically connected kernel.
 \end{Lem}
 \begin{proof}
    Clearly we may pass to $\kbar$ and assume that $k$ is algebraically closed. For a general $x\in \X^{ss}(k)$, 
    we denote the descendent pair $(\G_x,\rH_x)$. Note that any $x\in \X^{ss}(k)$ lies in a maximally $\theta$-split torus $A^-$. Letting $A$ be a maximal torus containing $A^-$, we may assume that $x\in s(A)(k)$ lies in the image of $A$ under the symmetrization map $s$; for notation set $\Ax=s(A)$ for this quotient. Setting $\rA_0=\ker(A\lra \Ax)$, clearly, $\rA_0\subset \rH_x.$ In particular, for any $x\in \Ax(k)$, $A\subset \G_x$ is a maximally $\theta$-split maximal torus of the descent $(\G_x,\rH_x)$. 
    
Sufficiency follows since if $x$ is regular semi-simple, then $\rH_x\supset \rA_0$ satisfies \cite[pg.9]{grinberg2018nearby}
 \[
 \pi_0(\rH_x) = \pi_0(\rA_0).
 \]

  Passing to the derived subgroup, we consider the $\theta$-stable maximal torus of the derived subgroup $A_{der}\subset \G_{der}$. By Steinberg's theorem \cite{Steinberg}, the fixed-point subgroup $\rH\cap \G_{der}$ of the semisimple automorphism $\theta|_{\G_{der}}$ is connected.
  
 Now assume that $\pi_0(\rA_0)=\{\ast\}$. Arguing by induction on $\dim(\X_x)$, the claim reduces to seeing that $\pi_0(\rH)=\{\ast\}$. As stated above, we have that the lattice for $(\G^\theta)^\circ\backslash\G$ is $X^\ast_{\theta,\circ}\subset X_{\theta,\qq}^\ast\cap X^\ast(\rA)$. The assumption forces $X^\ast(\Ax) = X_{\theta}^\ast$ to be saturated as a sublattice of $X^\ast(\rA)$, so that $$X^\ast_{\theta,\circ}=X_\theta^\ast=X^\ast(\Ax).$$ This forces $(\G^\theta)^\circ= \G^\theta$, proving the claim.
 \end{proof}
\subsection{Normalized and associated roots}
We make a few remarks on the normalized roots $\De_\X^{sv}$ of a symmetric variety. 

\begin{Lem}\label{Lem: involution on assoc}
Suppose that $(\G^\theta)^\circ\subset \rH\subset N_{\G}(\G^\theta)$ is a symmetric subgroup and $\X=\rH\backslash\G$ the associated symmetric variety. Suppose that $\ga\in \De_\X$ is a spherical root of type $G$ and
\[
\ga^{sv} = \ga_1+\ga_2
\]
are the associated roots. Then $\theta(\ga_1) = -\ga_2$.
\end{Lem}
\begin{proof}

The proof of Lemma \ref{Lem: root lattice} shows that the spherical roots are obtained by re-normalizing the reduced roots 
\[
\widetilde{\De}_\X=\{\al-\theta(\al): \al\in \De-\De^p_{\X}\},
\]
so that there exists $c\in \{1/2,1\}$ and $\al\in\De-\De^p_{\X}$ such that
\[
\ga= c(\al-\theta(\al)).
\]
In particular, $-\theta(\ga)=-\theta(\ga_1)-\theta(\ga_2) = \ga_1+\ga_2=\ga$. But note that $\{-\theta(\ga_1),-\theta(\ga_2)\}$ are then two strongly orthogonal roots of $\G$ summing to $\ga$.

By the duality between $\Phi$ and $\check{\Phi}$, we see that ${(\theta(\ga))}^\vee = \check{\theta}(\check{\ga})$. Thus,
\[
(-\theta(\ga_2))^\vee -(-\theta(\ga_1))^\vee =(-\theta(\de_2))^\vee -(-\theta(\de_1))^\vee.
\]
An easy calculation handling the cases of $D_2$ and $D_{n\geq 3}$ separately shows that 
\[
(-\theta(\de_2))^\vee -(-\theta(\de_1))^\vee = \check{\de}_2 -\check{\de}_1.
\]
By uniqueness, this implies the claim. 
\end{proof}


Suppose now that $\ga\in \De_\X$ is of type $N$. That is, $\ga=2\be$ for some $\be\in \Phi^+$, but $\be\notin\fX$. Using Lemma \ref{Lem: root lattice}, we see that $\ga=c(\al-\theta(\al))$ for some simple root $\al\in \De$ and for $c\in \{1,1/2\}$. We claim that $c=1$. Otherwise, we have \quash{ We may now verify that $c=1$ since otherwise}
\[
4\be=\al-\theta(\al).
\]
But this is impossible since the simple root $\al$ occurs in the support of the right-hand side with multiplicity $2$. Thus when $\ga$ is a root of type $N$, then
 \[
 2\be=\ga=\al-\theta(\al),\text{ but }\be\notin \fX.
 \]
 This implies a strong connection between roots of type $N$ and distinguished roots (cf. Lemma \ref{Lem: spherical roots distinguished}).
\begin{Prop}\label{Prop: type N roots sym}
    Suppose that $\X=\rH\backslash\G$ is a symmetric variety and that $\ga=\al-\theta(\al)\in \De_\X$ is a spherical root of type $N$. Then
    \begin{enumerate}[(A)]
        \item \label{aN} (Type $N$\eqref{aN})  $\ga=2\al$  such that there exists $D\in \D$ with $\rho(D) = \frac{1}{2}\al^\vee|_{\fa_{\X}}$;
        \item \label{bN} (Type $N$\eqref{bN}) there is a subset $\Sigma\subset \De$ of type $B_k$ with $k\geq2$ such that 
        \[
        \ga= 2(\al_1+\al_2+\cdots+\al_k),
        \]
        and $\al_i\in \De^p$ for $i>1$;
    \end{enumerate}
\end{Prop}
\begin{proof}
   Let $\De_\al\subset \De$ be the set of simple roots satisfying that $\al\in \De_\al$ and if we write
\[
\theta(\al)=-\sum_{\be\in \De}c_\be\be,
\]
then $\be\in \De_\al\setminus\{\al\}$ if and only if $c_\be\neq 0$. The sub-root system generated by $\De_\al$ is stable under $\theta$, and corresponds to a $\theta$-stable Levi $\mathrm{L}_\al\supset\rA$ with the corresponding boundary degeneration (cf. \cite[Section 2]{SakVenk}) the parabolic induced variety from the symmetric Levi variety $\X_\al:=\rH_{\De_\al}\backslash\mathrm{L}_\al$. This gives a symmetric varieties of rank $1$ with $\De_{\X_\al}=\{2\be\}$. Passing to this Levi variety $\X_\al$, the proposition may be verified via the classification of rank $1$ symmetric varieties.
\end{proof}
For example, roots of type N\eqref{bN} occur for the varieties $\Spin_{2b}\times^{\mu_2} \Spin_{2a+1}\backslash\Spin_{2n+1}$, for any $a,b\in \zz_{>0}$ satisfying $a+b=n$. If we let $a=0$, the variety $\Spin_{2n}\backslash\Spin_{2n+1}$ has a distinguished root of type \eqref{b}, while $\mathrm{Pin}_{2n}\backslash\Spin_{2n+1}$ has a root of type $N$.

\subsection{Colors of symmetric varieties}\label{Section: sym colors} The reduced root system $\widetilde{\Phi}_\X$ almost encodes the  color data of $\X$; some analysis of the possible root systems is required. 
Fix a maximally $\theta$-split Borel subgroup $B\subset \G$ and let $\cald(\X)$ denote the corresponding set of colors. 
\begin{Prop}\label{Prop: sym colors}
For $\X=\rH\backslash\G$ symmetric over $\kbar$, the image $\rho(\D(\X))\subset \fa_{\X,\qq}$ corresponds to $\{\check{\ga}^n: \ga\in \De_\X\mapsto \ga^n\in \De_\X^n\}$. That is, the images are dual to the simple vectors in the relative root system of $\X$.
\end{Prop}
\begin{proof}
   This is \cite[Proposition 1]{VustEmbeddings}
\end{proof}
To enumerate the colors requires the consideration of different cases. 
Now for any $\ga\in \De_\X$, Proposition \ref{Prop: sym colors} indicates that there exist colors $\D(\ga)\subset \D(\X)$ satisfying that $\rho(\De(\ga)) = \check{\ga}^n$, where $\ga^n\in \De_\X^n$ is the associated $n$-spherical root and $\check{\ga}^n\in X_\ast(\Ax)$ is the associated coroot.  On the other hand, for each spherical root $\ga\in \De_\X$, we saw in Section \ref{Section: reductions} that there exists a well-defined $k$-simple factor $\X_\ga$ of $\X_{sc}$. Recalling notation from Section \ref{Section: reductions}, we set $\rH_\ga:=\rH_{sc,\ga}$.

If $\X=\rH\backslash\G$ is a symmetric variety and $\ga\in \De_\X$, there are three possibilities (cf. \cite[Section 26.8]{Timashevbook}):
    \begin{enumerate}
        \item\label{colortype1} $\X_{\ga}$ has a non-reduced restricted root system of type $BC_n$, and $\check{\ga}^n= \lam^\vee-\theta(\lam^\vee)$ is a short coroot. Then $\widetilde{\D}(\check{\ga}) =\{D_\lam,D_{\theta^\ast(\lam)}\}$, and 
        \[
        \varsigma(D_\lam)=\{\lam\}, \text{ and } \varsigma(D_{\theta^\ast(\lam)})=\{\theta^\ast(\lam)\};
        \]
         \item\label{colortype2}  $\X_{\ga}$ has a reduced restricted root system of type $B_n$, and $\check{\ga}= \frac{1}{2}\al^\vee$ is a (real) short coroot. Then $\widetilde{\D}(\check{\ga})= \{D_\al^+,D_\al^-\}$ and 
     \[
        \varsigma(D_\al^+)=\varsigma(D_{\al}^-)=\{\al\}.
        \]
        This is precisely the case of double roots;
         \item\label{colortype3} otherwise, $\widetilde{\D}(\check{\ga})=\{D_\ga\}$. Here $\varsigma(D_{\ga}) = \{\lam\}$ where $\ga=\lam-\theta(\lam)$.
    \end{enumerate} 
    In all cases $\rho(\widetilde{\D}(\check{\ga})) = \check{\ga}^n\in \check{\De}^n_{\X}$. Then 
    \begin{equation}
    \D(\X) = \bigsqcup_{\ga\in \De_{\X}^n}\D(\check{\ga}),
    \end{equation} where $\D(\check{\ga})=\widetilde{\D}(\check{\ga})/\sim$ where $\D(\al^\vee) = \widetilde{\D}(\check{\ga})$ when $|\widetilde{\D}(\check{\ga})|=1$ and
    \[
    \D(\check{\ga}) =\begin{cases}\:
        \widetilde{\D}(\check{\ga})&:|\pi_0(\rH_\ga)(\kbar)| = 1,\\
        \{D_1=D_2\}&: |\pi_0(\rH_\ga)(\kbar)| > 1,
    \end{cases}
    \]
    when $|\widetilde{\D}(\check{\ga})|=2$. We thus obtain the map
    \[
    \rho\times \varsigma: \D(\X)\lra \fa_{\X}\times \mathcal{P}(\De),
    \]
    from which we may compute the sets $(\Omega^{(1)},\Omega^{(2)})$.


\end{appendix}

\bibliographystyle{alpha}

\bibliography{bibs}
\end{document}